\documentclass[11pt,dvipsnames]{amsart}
\usepackage[utf8]{inputenc}

\title{Invariant Gibbs measures for the three-dimensional wave equation with a Hartree nonlinearity II: Dynamics}
\author{Bjoern Bringmann}
\date{\today}

\usepackage{graphics}
\newcommand*{\medcap}{\mathbin{\raisebox{1pt}{\scalebox{0.8}{\ensuremath{\bigcap}}}}}
\newcommand*{\medcup}{\mathbin{\raisebox{1pt}{\scalebox{0.8}{\ensuremath{\bigcup}}}}}
\newcommand*{\scriptmedcup}{\mathbin{\raisebox{1pt}{\scalebox{0.6}{\ensuremath{\bigcup}}}}}

\usepackage{amsmath, amssymb, amsfonts, amsthm, verbatim, color}
\usepackage{fullpage}
\usepackage{mathabx}
\usepackage{hyperref}
\usepackage[numbers]{natbib}
\usepackage{mathtools}

\usepackage{multirow}
\usepackage{array}

\usepackage{changepage}

\usepackage{enumerate}

\newtheorem{theorem}{Theorem}%[section]
\newtheorem{proposition}[theorem]{Proposition} 
\newtheorem{corollary}[theorem]{Corollary}
\newtheorem{lemma}[theorem]{Lemma}
\newtheorem{definition}[theorem]{Definition}
\newtheorem{remark}[theorem]{Remark}

\newtheorem{example}[theorem]{Example}
\numberwithin{equation}{section} %Equation numbering
\numberwithin{theorem}{section} %theorem numbering
\newtheorem*{maintheorem}{Main result}%[section]
\newtheorem*{definitionnotag}{Definition}

\newtheorem{innerassumptions}{Assumptions}
\newenvironment{assumptions}[1]
  {\renewcommand\theinnerassumptions{#1}\innerassumptions}
  {\endinnerassumptions}

\newcommand{\lcol}{:\hspace{-0.5ex}}
\newcommand{\rcol}{\hspace{-0.5ex}:}
\newcommand{\defe}{\overset{\text{def}}{=}} %To avoid clash with Wick ordering

\usepackage{stochastic_diagrams}

\usepackage{pgfgantt}
\usepackage{float}
\usetikzlibrary{shapes.geometric, arrows, shadows}
\usetikzlibrary{fit,backgrounds} % <-added
\usetikzlibrary{calc}
\usetikzlibrary{math}

\newcommand{\cA}{\mathcal{A}}

\newcommand{\bP}{\mathbb{P}}
\newcommand{\bC}{\mathbb{C}}

\newcommand{\bZ}{\mathbb{Z}}
\newcommand{\bE}{\mathbb{E}}
\newcommand{\bR}{\mathbb{R}}
\newcommand{\bT}{\mathbb{T}}
\newcommand{\cT}{\mathcal{T}}		%%Types
\newcommand{\scrP}{\mathscr{P}} 		%%Used for pairings
\DeclareMathOperator{\Law}{Law}
\newcommand{\MN}{\mathcal{M}_N}
\newcommand{\MNparald}{\MN^{\scalebox{0.7}{\parald}}}
\newcommand{\MNnparald}{\MN^{\hspace{-0.5ex}\scalebox{0.7}{\nparaldalt}}}
\newcommand{\mNparald}{\mathscr{m}_N^{\scalebox{0.7}{\parald}}}
\newcommand{\mNnparald}{\mathscr{m}_N^{\hspace{-0.5ex}\scalebox{0.7}{\nparaldalt}}}

\newcommand{\blue}{\textup{blue}}
\newcommand{\purple}{\textup{pur}}
\newcommand{\red}{\textup{red}}
\newcommand{\grn}{\textup{grn}}
\newcommand{\cub}{\textup{cub}}
\newcommand{\stime}{\textup{time}}

\newcommand{\stab}{\textup{stab}}
\newcommand{\gwp}{\textup{gwp}}
\newcommand{\tsp}{\textup{sp}}
\newcommand{\ms}{\textup{ms}}
\newcommand{\bil}{\textup{bil}}

\DeclareMathOperator{\So}{\mathbf{So}}
\DeclareMathOperator{\CPara}{\mathbf{CPara}}
\DeclareMathOperator{\RMT}{\mathbf{RMT}}
\DeclareMathOperator{\Phy}{\mathbf{Phy}}

\newcommand{\cJ}{\mathcal{J}}

\DeclareMathOperator{\modulation}{\mathscr{L}\hspace{-0.2ex}\mathscr{M}}
\newcommand{\LM}{\mathscr{L}\hspace{-0.2ex}\mathscr{M}}
\newcommand{\X}[2]{\mathscr{X}^{#1,#2}}
\newcommand{\Z}{\mathscr{Z}}
\newcommand{\chs}{\mathcal{H}_x^s}
\newcommand{\cC}{\mathcal{C}}
\newcommand{\cQ}{\mathcal{Q}}
\newcommand{\cE}{\mathcal{E}}

\DeclareMathOperator{\NL}{\mathscr{N}\hspace{-0.2ex}\mathscr{L}}

\newcommand{\LdagM}{\mathcal{L}^{\text{amb}}_M(A,\tau)}
\newcommand{\LdagMB}{\mathcal{L}^{\text{amb}}_M(B,\tau)}
\newcommand{\LAtau}{\mathcal{L}(A,\tau)}
\newcommand{\LAtautil}{\widetilde{\mathcal{L}}(A,\tau)}
\newcommand{\cL}{\mathcal{L}}

\newcommand{\etype}{\overset{\text{type}}{=}} 
\newcommand{\type}{\text{type}} 
\newcommand{\netype}{\overset{\text{type}}{\neq}} 
\newcommand{\wblue}{w^{\bluescript}}
\newcommand{\Xblue}{X^{\bluescript}}
\newcommand{\wpurple}{w^{\purplescript}}
\newcommand{\Xpurple}{X^{\purplescript}}

\DeclareMathOperator{\Duh}{I}
\DeclareMathOperator{\Para}{PCtrl}
\DeclareMathOperator{\PCtrl}{PCtrl}

\DeclareMathOperator{\supp}{supp}
\DeclareMathOperator{\med}{med}			%%Median, i.e., for the arguments the second highest

\newcommand{\util}{\widetilde{u}_N}
\newcommand{\wtil}{\widetilde{w}_N}

\newcommand{\Zcirc}{Z^{\scalebox{1.1}{$\circ$}}}
\newcommand{\Zbox}{Z^{\scalebox{1.1}{$\square$}}} 
\usepackage[scr=boondoxo]{mathalfa}
\newcommand{\cs}{\ensuremath{\mathscr{s}}}

\newcommand{\Wcos}[2]{W^{(\cos),n_{#1}}_{#2}}
\newcommand{\Wsin}[2]{W^{(\sin),n_{#1}}_{#2}}

\newcommand{\Wpm}[2]{W^{(\pm_{#1}),n_{#1}}_{#2}}
\newcommand{\Was}[2]{W_{\cs_{#2}}^{(a),#1}}
\newcommand{\Wbs}[2]{W_{\cs_{#2}}^{(b),#1}}

\newcommand{\dWs}[1]{\mathrm{d}W_\mathscr{s}^{#1}}
\newcommand{\cI}{\mathcal{I}}
\newcommand{\cH}{\mathscr{H}}
\newcommand{\cHk}{\mathcal{H}_k}
\newcommand{\cHl}{\mathcal{H}_l}
\newcommand{\dW}[2]{\mathrm{d}W^{{#1}_{#2}}_{\cs_{#2}}}
\newcommand{\C}{\mathcal{C}}		%%Correlation

\newcommand{\cG}{\mathscr{G}}

\newcommand{\HNbp}{H_N[ \bluedot \rightarrow \purpledot]}
\newcommand{\HNpb}{H_N[ \purpledot \rightarrow \bluedot]}
\newcommand{\YNbp}{Y_N[ \bluedot \rightarrow \purpledot]}
\newcommand{\YNpb}{Y_N[ \purpledot \rightarrow \bluedot]}

\newcommand{\dtprime}{\mathrm{d}t^\prime}
\newcommand{\ds}{\mathrm{d}\cs}
\newcommand{\dx}{\mathrm{d}x}
\newcommand{\dy}{\mathrm{d}y}
\newcommand{\dn}{\mathrm{d}n}
\newcommand{\dlambda}{\mathrm{d}\lambda}

\newcommand{\cS}{\ensuremath{\mathcal{S}}}

\newcommand{\mup}{\mu^{\scalebox{0.7}{$\otimes$}}}
\newcommand{\nup}{\nu^{\scalebox{0.7}{$\otimes$}}}
\newcommand{\mupscript}{\mu^{\scalebox{0.6}{$\otimes$}}}

\newcommand{\cgp}{\mathscr{g}^{\scalebox{0.7}{$\otimes$}}}
\newcommand{\cg}{\mathscr{g}}

\newcommand{\cEj}{\mathcal{E}_K(B,j\tau)}
 
\newcommand{\cEjminus}{\mathcal{E}_K(B,(j-1)\tau)}

\begin{document}

\maketitle

\begin{abstract}
In this two-paper series, we prove the invariance of the Gibbs measure for a three-dimensional wave equation with a Hartree nonlinearity. The novelty lies in the singularity of the Gibbs measure with respect to the Gaussian free field. 

\noindent In this paper, we focus on the dynamical aspects of our main result. The local theory is based on a para-controlled approach, which combines ingredients from dispersive equations, harmonic analysis, and random matrix theory. The main contribution, however, lies in the global theory. We develop a new globalization argument, which addresses the singularity of the Gibbs measure and its consequences. 
\end{abstract}

\tableofcontents

%%%%%%%%%%%%%%%%%%%%%%%%%%%%%%%%%%%%%%%%%%%%%%%%%%%%%%%%%%%%%%%%%%%%%%%%%
%%%%%%%%%%%%%%%%%%%%%%%%%%%%%%%%%%%%%%%%%%%%%%%%%%%%%%%%%%%%%%%%%%%%%%%%%%

\section*{Continuation of the series}
This paper is the second part of a two-paper series and we refer to the first part \cite{BB20a} for a more detailed introduction to the series. 

We study the renormalized wave equation with a Hartree nonlinearity and random initial data given by
\begin{equation}\label{intro_series:eq_nlw} \tag{a}
\begin{cases}
-\partial_{tt}^2 u - u +\Delta   u = \, \lcol  (V*u^2) u \rcol  \qquad (t,x)\in \bR \times \bT^3,  \\
u|_{t=0} = \phi_0, \quad \partial_t u |_{t=0} = \phi_1. 
\end{cases}
\end{equation}
Here, the three-dimensional torus $\bT^3$ is understood as $[-\pi,\pi]^3$ with periodic boundary conditions. The interaction potential $V \colon \bT^3 \rightarrow \bR$ satisfies $V(x)= c_\beta |x|^{-(3-\beta)}$ for all $x\in \bT^3$ close to the origin, where $0<\beta<3$, satisfies $V(x)\gtrsim 1$ for all $x\in \bT^3$, is even, and is smooth away from the origin. The nonlinearity $\lcol (V*u^2) u \rcol $ is a renormalization of $(V\ast u^2) u$ and defined in \eqref{intro:eq_renormalized_nonlinearity} below. 

The nonlinear wave equation \eqref{intro_series:eq_nlw} is corresponding to the Hamiltonian $H$ given by
\begin{equation*}
H[u,\partial_t u](t) = \frac{1}{2} \Big( \| u(t) \|_{L_x^2}^2  + \| \nabla u(t) \|_{L_x^2}^2   + \| \partial_t u(t) \|_{L_x^2}^2   \Big) + \frac{1}{4} \int_{\bT^3} \lcol (V\ast u^2)(t,x) u(t,x)^2 \rcol \dx,
\end{equation*}
where $L^2_x=L_x^2(\bT^3)$. 
The formal Gibbs measure $\mup$ corresponding to the Hamiltonian has been rigorously constructed in the first paper of this series. All necessary properties of this construction will be recalled in Theorem \ref{theorem:measures} below. 

The main result of this series is the invariance of the Gibbs measure $\mup$ under the flow of the nonlinear wave equation \eqref{intro_series:eq_nlw}. We first state a formal version of our main result and postpone a rigorous version until Theorem \ref{theorem:measures} and Theorem \ref{theorem:gwp_invariance} below. 

\begin{maintheorem}[Global well-posedness and invariance, formal version] 
The formal Gibbs measure $\mup$ exists and, for $0<\beta<1/2$, is singular with respect to the Gaussian free field $\cgp$. The renormalized wave equation with Hartree nonlinearity \eqref{intro_series:eq_nlw} is globally well-posed on the support of $\mup$ and the dynamics leave $\mup$ invariant. 
\end{maintheorem}

\section{Introduction}

The second paper in this series deals with the dynamical aspects of our argument. As a result, it is inspired by recent advances in random dispersive equations. The interest in random dispersive equations stems from their connections to several areas of research, such as analytic number theory, harmonic analysis, random matrix theory, and stochastic partial differential equations (cf. \cite{Nahmod16}). In fact, much of the recent progress have been fueled through similar advances in singular stochastic partial differential equations, such as Hairer's \emph{regularity structures} \cite{Hairer14} or Gubinelli, Imkeller, and Perkowski's \emph{para-controlled calculus} \cite{GIP15}. \\

The most classical problem in random dispersive equations is the construction of invariant measures for (periodic and defocusing) nonlinear wave and Schr\"{o}dinger equations. This has been an active area of research since the 1990s, and we refer the reader to Figure \ref{figure:literature} for an overview of some of the most important contributions. 

\begin{figure}
\begin{tabular}{ll|>{\centering\arraybackslash}p{4cm}|>{\centering\arraybackslash}p{4.5cm}}
\multicolumn{2}{c|}{Dimension \& Nonlinearity} & Wave  & Schr\"{o}dinger \\ \hline &&  \\[-9pt]
$d=1$ , 	&  $ |u|^{p-1} u $									&  \cite{F85,Zhidkov94}					 & 	\cite{Bourgain94} 	\\ \hline &&  \\[-9pt]
$d=2$,	&$|u|^2 u$  								&  \multirow{3}{*}{\vspace{2ex} \cite{OT18}}		 & 		\cite{Bourgain96}		\\ \cline{1-2} \cline{4-4} &&  \\[-9pt]
$d=2$, 	&$|u|^{p-1} u$ 							& 							 & 			\cite{DNY19}		\\ \hline &&  \\[-9pt]
$d=3$, 	&$(|x|^{-(3-\beta)} \ast |u|^2) \cdot u$ 			&				
\begin{tabular}{rl}
$\beta>1:$ & \cite{OOT20} \\
$\beta>0:$ & \emph{This paper.}
\end{tabular}			 & 
\begin{tabular}{rl}
$\beta >2\colon$ & \cite{Bourgain97} \\ 
$2 \geq \beta > 1/2\colon $ & Feasible. \\ 
$1/2 \geq \beta > 0 \colon$ & Open.
\end{tabular}		
 \\ \hline &&  \\[-9pt]
$d=3$,  	&$|u|^2 u$ 								&		Open 					 &  Open 
\end{tabular}
\caption{\small{Invariant Gibbs measures for defocusing nonlinear wave and Schr\"{o}dinger equations.}}
\label{figure:literature}
\end{figure}

The first results in this direction were obtained in one-spatial dimension by Friedlander \cite{F85}, Zhidkov \cite{Zhidkov94} and Bourgain \cite{Bourgain94}. Friedlander \cite{F85} and Zhidkov \cite{Zhidkov94} proved the invariance of the Gibbs measure for the one-dimensional nonlinear wave equation. Inspired by earlier work of Lebowitz, Rose, and Speer \cite{LRS88}, Bourgain \cite{Bourgain94} proved the invariance of the Gibbs measure for the one-dimensional nonlinear  Schr\"{o}dinger equations
\begin{equation*}
i \partial_t u + \partial_x^2 u = |u|^{p-1} u, \qquad (t,x) \in \bR \times \bT. 
\end{equation*}
In this seminal paper, Bourgain introduced his famous \emph{globalization argument}, which will be described in detail below. Even though Friedlander \cite{F85}, Zhidkov \cite{Zhidkov94} and Bourgain \cite{Bourgain94} consider random initial data (drawn from the Gibbs measure), the local theory is entirely deterministic. The reason is that the Gibbs measure is supported at spatial regularity $1/2-$, which is above the (deterministic) critical regularities $s_{\text{det}}= \frac{1}{2}- \frac{1}{p}$ (cf. \cite{CCT03}) and $s_{\text{det}}= \frac{1}{2}-\frac{2}{p-1}$ for the one-dimensional wave and  Schr\"{o}dinger equations (in $H^s$), respectively. 

The first result in two spatial dimensions was obtained by Bourgain in \cite{Bourgain96}. He proved the invariance of the Gibbs measure for the renormalized cubic nonlinear Schr\"{o}dinger equation
\begin{equation}\label{intro:eq_cubic_NLS}
i \partial_t u + \Delta  u = \lcol |u|^{2} u \rcol  \qquad (t,x) \in \bR \times \bT^2. 
\end{equation}
In \eqref{intro:eq_cubic_NLS}, the renormalized (or Wick-ordered) nonlinearity is given  by $ |u|^2 u -2 \| u \|_{L^2}^2 u$. In this specific case, the renormalized equation \ref{intro:eq_cubic_NLS} is related to the cubic nonlinear   Schr\"{o}dinger equation through a gauge transformation. In contrast to the one-dimensional setting, the Gibbs measure is supported at spatial regularity $0-$, which is just below the (deterministic) critical regularity $s_c=0$. To overcome this obstruction, the local theory in \cite{Bourgain96} exhibits probabilistic cancellations in several multi-linear estimates. Very recently, Fan, Ou, Staffilani, and Wang \cite{FOSW19} extended Bourgain's result from the square torus $\bT^2$ to irrational tori. 

The situation for two-dimensional nonlinear wave equations is easier than for  two-dimensional nonlinear Schr\"{o}dinger equations. While the Gibbs measure is still supported at spatial regularity $0-$, this is partially compensated by the smoothing effect of the Duhamel integral. In \cite{OT18}, Oh and Thomann prove the invariance of the Gibbs measure for
\begin{equation}\label{intro:eq_2d_NLW}
- \partial_t^2 u - u + \Delta u = \lcol u^p \rcol  \qquad (t,x) \in \bR \times \bT^2,
\end{equation}
where $p \geq 3$ is an odd integer. The renormalized nonlinearity $\lcol u^p \rcol$ in \eqref{intro:eq_2d_NLW} is the Wick-ordering of $u^p$, see e.g. \cite[(1.9)]{OT18}.  In contrast to the nonlinear Schr\"{o}dinger equation \eqref{intro:eq_cubic_NLS}, it cannot be obtained from the original equation via a gauge transformation. However, the renormalization is likely necessary to obtain non-trivial dynamics for random low-regularity data (see e.g. \cite{OOR20,ORSW21}). We emphasize that their argument for the cubic ($p=3$) and higher-order ($p\geq 5$) nonlinearity is essentially identical. Due to its clear and detailed exposition, we highly recommend \cite{OT18} as a starting point for any beginning researcher in random dispersive equations. 

In a recent work \cite{DNY19}, Deng, Nahmod, and Yue proved the invariance of the Gibbs measure for the nonlinear Schr\"{o}dinger equations
\begin{equation}\label{intro:eq_NLS}
i \partial_t u + \Delta  u = \lcol |u|^{p-1} u \rcol  \qquad (t,x) \in \bR \times \bT^2,
\end{equation}
where $p\geq 5$ is an odd integer. In contrast to the situation for the two-dimensional nonlinear wave equations, this result is much harder than its counterpart for the cubic nonlinear Schr\"{o}dinger equation \eqref{intro:eq_cubic_NLS}. The main difficulty is that all high$\times$low$\times$ $\hdots$ $\times$low-interactions between the random initial data with itself or smoother remainders only have spatial regularity $1/2-$, which is strictly below the (deterministic) critical regularity $s_{\text{det}} = 1 - \frac{2}{p-1}$. To overcome this difficulty, Deng, Nahmod, and Yue worked with random averaging operators, which are related to the adapted linear evolutions in \cite{Bringmann18}. Their framework was recently generalized through the \emph{theory of random tensors} \cite{DNY20}, which will be further discussed below. 

Unfortunately, much less is known in three spatial dimensions. The reason is that the Gibbs measure is supported at spatial regularity $-1/2-$, which is far below the deterministic critical regularity $s_{\text{det}}= \frac{3}{2} - \frac{2}{p-1}$. In fact, the invariance of the Gibbs measure for both the cubic nonlinear wave and Schr\"{o}dinger equation are famous open problems. Previous research has instead focused on simpler models, which are obtained either through additional symmetry assumptions, a (slight) regularization of the random initial data, or a (slight) regularization of the nonlinearity. 

In the radially-symmetric setting, the invariance of the Gibbs measure for the three-dimensional cubic wave and  Schr\"{o}dinger equation has been proven in \cite{BB14,S11,Xu14} and \cite{BB14b}, respectively. The radially-symmetry setting was also studied in earlier work on the two-dimensional nonlinear Schr\"{o}dinger equation \cite{Deng12,Tzvetkov06,Tzvetkov08}.

 In \cite{OPT19}, Oh, Pocovnicu, and Tzvetkov studied the cubic nonlinear wave equation with Gaussian initial data.   While the Gaussian initial data in \cite{OPT19} does not directly correspond to a Gibbs measure, the local theory in \cite{OPT19} still yields partial progress towards the (local aspects of) the Gibbs-measure problem.   The Gaussian initial data in \cite{OPT19} has regularity  $s>-1/4$ and, as a result, is  more than $1/4$-derivatives smoother than the Gibbs measure. Using some of the methods in this paper, Oh, Wang, and Zine \cite{OWZ21} very recently improved the regularity condition from $s>-1/4$ to $s>-1/2$. In particular, the Gaussian data in \cite{OWZ21} is only an $\epsilon$-derivative smoother than the support of the Gibbs measure.

 In \cite{Bourgain97}, Bourgain studied the defocusing and focusing three-dimensional Schr\"{o}dinger equation with a Hartree nonlinearity given by 
\begin{equation}\label{intro:eq_nls_hartree}
i \partial_t u + \Delta u = \pm \lcol ( V \ast |u|^2) u \rcol\qquad (t,x) \in \bR \times \bT^3,
\end{equation}
where the interaction potential $V$ behaves like $+|x|^{-(3-\beta)}$. He proved the invariance of the Gibbs measure for $\beta>2$, which corresponds to a relatively smooth interaction potential. In the focusing case, this is optimal (up to the endpoint $\beta=2$), since the Gibbs measure is not normalizable for $\beta <2$ (cf. \cite{OOT20}). From a physical perspective, the most relevant cases are the Coulomb potential $|x|^{-1}$ (corresponding to $\beta=2$) and the Newtonian potential $|x|^{-2}$ (corresponding to $\beta=1$). Since the cubic nonlinear Schr\"{o}dinger equation formally corresponds to \eqref{intro:eq_nls_hartree} with the interaction potential  $V$ given by the Dirac-measure, it is also interesting (and challenging) to take $\beta$ close to zero. After the first version of this manuscript appeared, Deng, Nahmod, and Yue \cite{DNY21} used random averaging operators (as in \cite{DNY19}) to cover the regime $\beta>1-\epsilon$ in the defocusing case, where $\epsilon>0$ is a small unspecified constant. As discussed in \cite{DNY21}, it is likely possible to use the more sophisticated theory of random tensors from \cite{DNY20} to cover the regime $\beta>1/2$.  In the regime $0<\beta=1/2$, the Gibbs measure becomes singular with respect to the Gaussian free field (see Theorem \ref{theorem:measures}).  As described in \cite[Section 1.2.1]{DNY21}, the extension of the theory of random tensors to singular Gibbs measures remains a challenging open problem (see also Remark \ref{intro:remark_DNY}). \\
 
After the completion of this series, the author learned of independent work by Oh, Okamoto, and Tolomeo \cite{OOT20}. The authors study (the stochastic analogue of) the focusing and defocusing three-dimensional nonlinear wave equation with a Hartree nonlinearity given by
\begin{equation*}
-\partial_t^2 u - u + \Delta u = \pm \lambda \lcol  ( V \ast u^2) u  \rcol\qquad (t,x) \in \bR \times \bT^3,
\end{equation*}
where $\lambda >0$. The main focus of \cite{OOT20} lies on the construction and properties of the Gibbs measures, which are discussed in the first part of the series (cf. \cite[Remark 1.2]{BB20a}). Regarding the dynamical results of \cite{OOT20}, the authors prove the invariance of the Gibbs measure in the following cases:
\begin{enumerate}[(i)]
\item focusing ($-$):  \hspace{2.4ex}$\beta>2$ or $\beta=2$ in the weakly nonlinear regime.
\item defocusing ($+$): $\beta>1$. 
\end{enumerate}
In light of the non-normalizability of the focusing Gibbs measure for $\beta<2$ and $\beta=2$ in the strongly nonlinear regime (cf. \cite{OOT20}), the result is optimal in the focusing case. In the defocusing case, however, the restriction $\beta>1$ excludes all Gibbs measures which are singular with respect to the Gaussian free field. In contrast, Theorem \ref{theorem:gwp_invariance} below covers the complete range
 $\beta>0$, which includes singular Gibbs measures. In fact, this is the main motivation behind our two-paper series.    \\
 
 \begin{figure}
\scalebox{0.9}{
\begin{tabular}{ll|>{\centering\arraybackslash}p{1cm}>{\centering\arraybackslash}p{2.5cm}>{\centering\arraybackslash}p{2cm}|>{\centering\arraybackslash}p{1cm}>{\centering\arraybackslash}p{2.5cm}>{\centering\arraybackslash}p{2cm}}
\multicolumn{2}{c|}{Dimension} & \multicolumn{3}{c|}{Wave}  &\multicolumn{3}{c}{ Schr\"{o}dinger} 
 \\ \cline{3-8} &&&&&&&  \\[-9pt]
 \multicolumn{2}{c|}{\& Nonlinearity} & $s_{\text{G}}$ &  $s_{\text{prob}}$ & $s_{\text{det}}$ & $s_{\text{G}}$ & $s_{\text{prob}}$ & $s_{\text{det}}$
  \\ \hline &&&&&&&  \\[-9pt]
$d=1$ , 	&  $ |u|^{p-1} u $							
& $\frac{1}{2}$- & $-\frac{1}{2p}$ & $\frac{1}{2}-\frac{1}{p}$  
& $\frac{1}{2}$- & $-\frac{1}{p-1}$ & $\frac{1}{2}-\frac{2}{p-1}$ 
 \\ \hline &&&&&&&  \\[-9pt]
$d=2$, 	&$|u|^{p-1} u$ 							
& $0$- & $-\frac{3}{2p}$ & $1-\frac{2}{p-1}$
& $0$- & $-\frac{1}{p-1}$ & $1-\frac{2}{p-1}$ 
 \\ \hline &&&&&&&  \\[-9pt]
$d=3$, 	&$(V  \ast |u|^2) \cdot u$ 			
& $-\frac{1}{2}$- & $-\min(\frac{2+\beta}{3},\frac{3}{2})$ & $\max(\frac{1-2\beta}{2},0)$
& $-\frac{1}{2}$- & $-\min( \frac{1+\beta}{2}, 1)$ & $\max(\frac{1-2\beta}{2},0)$	
 \\ \hline &&&&&&&  \\[-9pt]
$d=3$,  	&$|u|^2 u$ 								
& $-\frac{1}{2}$- & $-\frac{2}{3}$ & $\frac{1}{2}$
& $-\frac{1}{2}$- &  $-\frac{1}{2}$ & $\frac{1}{2}$ 
\end{tabular}
}
\caption{\protect{\small{Relevant spatial regularities for the invariance of the Gibbs measure:   $s_\text{G}$ (support of the Gibbs measure), $s_{\text{prob}}$ (probabilistic scaling), $s_{\text{det}}$ (deterministic scaling). The value of $s_{\text{prob}}$ for power-type nonlinearities can be found in \cite{DNY19}. The probabilistic critical regularity \(s_{\text{prob}}\) for the wave equation with a Hartree nonlinearity is a result of high\(\times\)high\(\times\)high\(\rightarrow\)low and (high\(\times\)high\(\rightarrow\)low)\(\times\)high\(\rightarrow\)high-interactions. For the Schr\"{o}dinger equation with a Hartree nonlinearity, \(s_{\text{prob}}\) is a result of (high\(\times\)high\(\rightarrow\)high)\(\times\)high\(\rightarrow\)high and (high\(\times\)high\(\rightarrow\)low)\(\times\)high\(\rightarrow\)high-interactions.}}}
\label{figure:regularities}
\end{figure}

In the preceding discussion, we have seen several examples of invariant Gibbs measures supported at regularities even below the deterministic critical regularity. In \cite{DNY19,DNY20}, Deng, Nahmod, and Yue describe a probabilistic scaling heuristic, which takes into account the expected probabilistic cancellations. We denote the critical regularity with respect to the probabilistic scaling by  $s_{\text{prob}}$ and the spatial regularity of the support of the Gibbs measure $s_{\text{G}}$. Based on the probabilistic scaling heuristic, we then expect probabilistic local well-posedness as long as $s_{\text{G}} > s_{\text{prob}}$. We record the relevant quantities for nonlinear wave and Schr\"{o}dinger equations in Figure \ref{figure:regularities}. For comparison, we also include the deterministic critical regularity $s_{\text{det}}$.   The probabilistic scaling heuristic, however, does not address any obstructions related to the global theory, renormalizations, or measure-theoretic aspects. As a result, it does not capture some of the difficulties for dispersive equations with singular Gibbs measures, such as the cubic nonlinear wave equation in three dimensions.

Our discussion so far has been restricted to invariant Gibbs measures for nonlinear wave and Schr\"{o}dinger equations. While this is the most classical problem in random dispersive equations, there exist many more active directions of research. Since a full overview of the field is well-beyond the scope of the introduction, we only mention a few directions and refer to the given references for more details. 

\begin{enumerate}
\item Invariance of white noise \cite{KMV19,Oh09,QV08},
\item Invariant measures (at high regularity) for completely integrable  equations \cite{TV14,TV15,DTV15},
\item Quasi-invariant Gaussian measures for non-integrable equations \cite{GOTW18,OT20,T15},
\item Non-invariance methods related to scattering, solitons, and blow-up \cite{Bringmann18b,Bringmann20,DLM17,KM19,Pocovnicu17},
\item Wave turbulence \cite{BGHS19,CG19,CG20,DH19}, 
\item Stochastic dispersive equations \cite{BD99,BD03,DW18,GKO18a,GKO18}. 
\end{enumerate}

After this overview of the relevant literature, we now turn to a more detailed description of the most relevant methods. Our discussion will be split into two parts separating the local and global aspects. As a teaser for the reader, we already mention that our contributions to the local theory will be of an intricate but technical nature, while our contributions to the global theory will be conceptual. \\

As mentioned above, the first local well-posedness result for dispersive equations relying on probabilistic methods was proven by Bourgain \cite{Bourgain96}. He considered the renormalized cubic nonlinear Schr\"{o}dinger equation 
\begin{equation}\label{intro:eq_cubic_NLS_IVP}
\begin{cases}
 i \partial_t u - u + \Delta  u = \lcol |u|^{2} u \rcol  \qquad (t,x) \in \bR \times \bT^2, \\
 u|_{t=0} = \phi 
\end{cases}
\end{equation}
The additional $-u$-term has been introduced for convenience, but can be easily removed through a gauge transformation. The random initial data $\phi$ is drawn from the corresponding Gibbs measure, which coincides with the (complex) $\Phi^4_2$-model. Since the $\Phi^4_2$-model is absolutely continuous with respect to the Gaussian free field and the local theory does not rely on the invariance of the Gibbs measure, we can represent $\phi$ through the random Fourier series
\begin{equation}\label{intro:eq_random_Fourier}
\phi = \sum_{n\in \bZ^2} \frac{g_n}{\langle n\rangle} e^{i \langle n ,x \rangle}. 
\end{equation}
Here, $\langle n \rangle \defe \sqrt{1+|n|^2}$ and $(g_n)_{n\in \bZ^2}$ is a sequence of independent and standard complex-valued Gaussians. The independence of the Fourier coefficients, and more generally the simple structure of \eqref{intro:eq_random_Fourier}, is an essential ingredient for many arguments in \cite{Bourgain96}. A direct calculation yields almost surely that $\phi\in H^s(\bT^2)\backslash L^2(\bT^3)$ for all $s<0$. Since \eqref{intro:eq_cubic_NLS_IVP} is mass-critical, $\phi$ lives below the (deterministic) critical regularity. To overcome this obstruction, Bourgain decomposed the solution by writing
\begin{equation*}
u(t) = e^{ i t (-1+\Delta)} \phi + v(t). 
\end{equation*}
This decomposition is commonly referred to as Bourgain's trick, but is also known in the stochastic PDE literature as the Da Prato-Debussche trick \cite{DPD03}. Using this decomposition, we see that the nonlinear remainder $v$ satisfies the evolution equation 
\begin{equation*}
 i \partial_t v - v + \Delta  v = \lcol |e^{ i t (-1+\Delta)} \phi + v|^{2} (e^{ i t (-1+\Delta)} \phi + v) \rcol  \qquad (t,x) \in \bR \times \bT^2. 
\end{equation*}
Through a combination of probabilistic and PDE arguments, Bourgain proved that the Duhamel integral 
\begin{equation*}
I\Big[ \lcol |e^{ i t (-1+\Delta)} \phi |^{2} e^{ i t (-1+\Delta)} \phi  \rcol \Big] 
\end{equation*}
lives at spatial regularity $1/2-$ (see also \cite{CLS19}). This opens the door to a contraction argument for $v$ at a positive (and hence sub-critical) regularity. The contraction argument requires further ingredients from random matrix theory to handle mixed terms, but can in fact be closed. We emphasize that the nonlinear remainder $v$ is treated purely deterministically and is not shown to exhibit any random structure. \\

We now discuss the more recent work of Gubinelli, Koch, and Oh \cite{GKO18a}, which covers the stochastic wave equation
\begin{equation*}
\begin{cases}
-\partial_t^2 u - u + \Delta u = \lcol u^2 \rcol + \xi \qquad (t,x) \in \bR \times \bT^3, \\
u[0]=0. 
\end{cases}
\end{equation*}
Here, $\xi$ denotes space-time white noise. Inspired by a (higher-order version of) Bourgain's trick, we decompose
\begin{equation*}
u = \<1black> + \<2Dblack> + v. 
\end{equation*}
The linear stochastic object $ \<1black>$ solves the forced wave equation
\begin{equation*}
(-\partial_t^2 - 1 + \Delta ) \,  \<1black> = \xi. 
\end{equation*}
The black dot represents the stochastic noise $\xi$ and the arrow represents the Duhamel integral. An elementary arguments shows that $\<1black>$ has spatial regularity $-1/2-$. The quadratic stochastic object $\<2Dblack>$ is the solution of the forced wave equation
\begin{equation*}
(-\partial_t^2 - 1 + \Delta ) \, \<2Dblack> = \lcol \big( \, \<1black> \, \big)^2 \rcol. 
\end{equation*}
Based on similar arguments for stochastic heat equations, one may expect that $\<2Dblack>$ has spatial regularity $2 \cdot (-1/2-) +1= 0-$, where the gain of one spatial derivative comes from the Fourier multiplier $\langle \nabla \rangle^{-1}$ in the Duhamel integral. Using multilinear dispersive estimates, however, Gubinelli, Koch, and Oh proved that $\<2Dblack>$ has spatial regularity $1/2-$. Using the definition of our stochastic objects, we obtain the evolution equation 
\begin{equation*}
(-\partial_t^2 - 1 + \Delta ) \, v = 2 \Big( \<2Dblack> + v \Big) \cdot \<1black> + \Big( \<2Dblack> + v \Big)^2
\end{equation*} 
for the nonlinear remainder $v$. In the following discussion, we let $\parall$ and $\paraeq$ be the low$\times$high and high$\times$high-paraproducts from Definition \ref{lwp:def_paraproduct}. Due to low$\times$high-interactions such as $v \parall \<1black>$, we expect $v$ to have spatial regularity at most $(-1/2-)+1=1/2-$. We emphasize that, unlike high$\times$high to high-interactions, the low$\times$high-interactions are not affected by multi-linear dispersive effects. However, this implies that the spatial regularities of $v$ and $\<1black>$ do not add up to a positive number, which means that the high$\times$high-term $v \paraeq \<1black>$ cannot even be defined (without additional information on $v$). This problem cannot be removed through a direct higher-order expansion of $u$ and persists through all orders of the Picard iteration scheme. Instead, Gubinelli, Koch, and Oh \cite{GKO18a} utilize ideas from the para-controlled calculus for singular stochastic PDEs \cite{GIP15}. We write $v=X+Y$, where $X$ and $Y$ solve 
\begin{equation}\label{intro:eq_GKO_X}
(-\partial_t^2 - 1 + \Delta ) \, X = 2 \Big( \<2Dblack> + X + Y \Big) \parall \<1black> 
\end{equation}
and
\begin{equation}\label{intro:eq_GKO_Y}
(-\partial_t^2 - 1 + \Delta ) \, Y = 2 \Big( \<2Dblack> + X+Y \Big) \parageq \<1black> + \Big( \<2Dblack> + X+ Y\Big)^2.
\end{equation}
The para-controlled component $X$ only has spatial regularity $1/2-$, but exhibits a random structure. In the analysis of the high$\times$high-interactions $X \paraeq \<1black>$, this random structure can be exploited by replacing $X$ with the Duhamel integral of the right-hand side in \eqref{intro:eq_GKO_X}. Since this leads to a double Duhamel integral in the expression for $Y$, this approach is often called the ``double Duhamel trick''. In contrast to $X$, $Y$ lives at a higher spatial regularity and can be controlled through deterministic arguments. The local theory in this paper will follow a similar approach, but relies on more intricate estimates, which will be further discussed below. \\

After this discussion of the local theory, we now turn to the global theory. We discuss Bourgain's globalization argument \cite{Bourgain94}, which uses the invariance of the truncated Gibbs measures as a substitute for a conservation law. We first recall the definition of the different modes of convergence for a sequence of probability measures, which will be needed below.

\begin{definitionnotag}[Convergence of measures]
Let $\cH$ be a Hilbert space and let $B(\cH)$ be the Borel $\sigma$-algebra on $\cH$. Furthermore, let $(\mu_N)_{N\geq 1}$ and $\mu$ be Borel probability measures on $\cH$. Then, we say that
\begin{enumerate}[(i)]
\item $\mu_N$ converges in total variation to $\mu$ if 
\begin{equation*}
\lim_{N\rightarrow \infty} \sup_{A \in B(\cH)} | \mu(A) - \mu_N(A)| = 0,
\end{equation*}
\item $\mu_N$ converges strongly to $\mu$ if 
\begin{equation*}
\lim_{N\rightarrow \infty} \mu_N(A) = \mu(A)  \qquad \text{for all } A \in B(\cH), 
\end{equation*}
\item $\mu_N$ converges weakly to $\mu$ if 
\begin{equation*}
\lim_{N\rightarrow \infty} \mu_N(A) = \mu(A)  \qquad \text{for all } A \in B(\cH) \text{ satisfying } \mu(\partial A)=0. 
\end{equation*}
\end{enumerate}
\end{definitionnotag} 
To isolate the key features of the argument, we switch to an abstract setting. Let $\cH$ be a Hilbert space and let $\Phi_N \colon \bR \times \cH \rightarrow \cH$ be a sequence of jointly continuous flow maps. Let $\mu_N$ be a sequence of Borel probability measures on $\cH$. Most importantly, we assume that $\mu_N$ is invariant under $\Phi_N$ for all $N$, i.e., 
\begin{equation*}
\mu_N( \Phi_N(t)^{-1} A ) = \mu_N(A) \qquad \text{for all } t\in \bR \text{ and } A \in B(\cH). 
\end{equation*}
In our setting, $\Phi_N$ will be the flow for a frequency-truncated nonlinear wave equation and $\mu_N$ will be the corresponding truncated Gibbs measure. Our main interest lies in the removal of the truncation, i.e., the limit of the dynamics $\Phi_N$ and measure $\mu_N$ as $N$ tends to infinity. Let $\mu$ be a limit of the sequence $\mu_N$, where the mode of convergence will be specified below. In order to construct the limiting dynamics on the support of $\mu$, we need uniform bounds on $\Phi_N$ on the support of $\mu$. At the very least, we require an estimate of the form 
\begin{equation}\label{intro:eq_bourgain_global_estimate}
\limsup_{N\rightarrow \infty} \mu\Big( \sup_{t\in [0,1]} \| \Phi_N(t) \phi \|_{\cH} \leq \epsilon^{-1} \Big) \geq 1 - o_\epsilon(1),
\end{equation}
where $0<\epsilon<1$ and $o$ is the small Landau symbol. Bourgain's globalization argument \cite{Bourgain94} proves \eqref{intro:eq_bourgain_global_estimate} in two steps. 

In a first measure-theoretic part, we use that 
\begin{equation*}
 \Big| \mu\Big( \sup_{t\in [0,1]} \| \Phi_N(t) \phi \|_{\cH} \leq \epsilon^{-1} \Big) - \mu_N\Big( \sup_{t\in [0,1]} \| \Phi_N(t) \phi \|_{\cH} \leq \epsilon^{-1} \Big) \Big| \leq  \sup_{A \in B(\cH)} | \mu(A) - \mu_N(A)| . 
\end{equation*}
As long as $\mu_N$ converges in \emph{total variation} to $\mu$, we can reduce \eqref{intro:eq_bourgain_global_estimate} to 
\begin{equation}\label{intro:eq_bourgain_global_estimate_reduced}
\limsup_{N\rightarrow \infty} \mu_N\Big( \sup_{t\in [0,1]} \| \Phi_N(t) \phi \|_{\cH} \leq \epsilon^{-1} \Big) \geq 1 - o_\epsilon(1),
\end{equation}
In a second dynamical part, we use the invariance of $\mu_N$ under $\Phi_N$ and the probabilistic local well-posedness. Let $J\geq1$ be a large integer and define the step-size $\tau= J^{-1}$. Then, 
\begin{align*}
\mu_N\Big( \sup_{t\in [0,1]} \| \Phi_N(t) \phi \|_{\cH} >  \epsilon^{-1} \Big) 
&\leq \sum_{j=0}^{J-1}\mu_N\Big( \sup_{t\in [j \tau ,(j+1)\tau]} \| \Phi_N(t) \phi \|_{\cH} >  \epsilon^{-1} \Big)  \\ 
&= \sum_{j=0}^{J-1}\mu_N\Big( \sup_{t\in [0 ,\tau]} \| \Phi_N(t)  \Phi_N(j\tau) \phi \|_{\cH} >  \epsilon^{-1} \Big) . 
\end{align*}
Using the invariance of $\mu_N$ under $\Phi_N(j\tau)$, we obtain that 
\begin{equation}\label{intro:eq_bourgain_global_estimate_reduced_2}
\mu_N\Big( \sup_{t\in [0,1]} \| \Phi_N(t) \phi \|_{\cH} >  \epsilon^{-1} \Big) \leq \tau^{-1} \mu_N\Big( \sup_{t\in [0 ,\tau]} \| \Phi_N(t)   \phi \|_{\cH} >  \epsilon^{-1} \Big) . 
\end{equation}
The right-hand side of \eqref{intro:eq_bourgain_global_estimate_reduced_2} can then be controlled through an appropriate choice of $\tau$ and the local theory (as well as tail estimates for $\mu_N$). 

In (this sketch of) Bourgain's globalization argument, the convergence in total variation played an essential role. In all previous results on the invariance of (defocusing) Gibbs measures \cite{Bourgain94,Bourgain96,Bourgain97,DNY19,OOT20,OT18,Zhidkov94}, the truncated Gibbs measures converge in total variation, so that this assumption does not pose any problems. In our case, however, the truncated Gibbs measures $\mu_N$ only converge weakly to the Gibbs measure $\mu$. The weak mode of convergence is related to the singularity of the Gibbs measure $\mu$ with respect to the Gaussian free field $\cg$, which necessitates softer arguments in the construction of $\mu$. Using the weak convergence of $\mu_N$ to $\mu$, we can only reduce \eqref{intro:eq_bourgain_global_estimate} to 
\begin{equation}\label{intro:eq_bourgain_global_estimate_reduced_3}
\limsup_{N\rightarrow \infty} \Big[ \limsup_{M\rightarrow \infty} \mu_M\Big( \sup_{t\in [0,1]} \| \Phi_N(t) \phi \|_{\cH} \leq \epsilon^{-1} \Big) \Big] \geq 1 - o_\epsilon(1),
\end{equation}
In \eqref{intro:eq_bourgain_global_estimate_reduced_3}, we will typically have $M>N$, and hence we cannot (directly) use the invariance of the truncated Gibbs measures. \\
In \cite{NORS12}, Nahmod, Oh, Rey-Bellet, and Staffilani prove the invariance of a Wiener measure for the periodic derivative nonlinear Schr\"{o}dinger equation. The truncated Wiener measures in \cite{NORS12} are defined using a frequency-truncation not only in the interaction but also in the Gaussian free field (cf. \cite[(5.13)]{NORS12}. As a consequence, the truncated Wiener measures only converge weakly (cf. \cite[Proposition 5.13]{NORS12}). In order to prove \eqref{intro:eq_bourgain_global_estimate_reduced_3}, the authors rely on the (quantitative) mutual absolute continuity of the (truncated) Wiener measure with respect to the (truncated) Gaussian free field (cf. \cite[(6.7)]{NORS12}). Unfortunately, the singularity of the Gibbs measure in this work (as stated in Theorem \ref{theorem:measures}) prevents us from using a similar approach. 

\subsection{Main results and methods}
Before we can state our main results, we need to define the renormalized and frequency-truncated Hamiltonians, wave equations, and Gibbs measures. For any dyadic $N\geq 1$, we define the renormalized and frequency-truncated potential energy by
\begin{align*}
&\frac{1}{4} \int_{\bT^3} \lcol ( V \ast ( P_{\leq N} \phi)^2 ) (P_{\leq N} \phi)^2 \rcol \dx \\
\defe& \frac{1}{4} \int_{\bT^3} \Big[  ( V \ast ( P_{\leq N} \phi)^2 ) (P_{\leq N} \phi)^2 - 2 a_N (P_{\leq N} \phi)^2 - 4 ( \MN P_{\leq N} \phi ) P_{\leq N} \phi +\widehat{V}(0) a_N^2 + 2 b_N \Big]  \dx  + c_N.
\end{align*}
Here, the renormalization constants $a_N$, $b_N$, $c_N$ are given by Definition 2.6, Definition 2.8, and Proposition 3.2 in the first paper of this series \cite{BB20a}, but their precise values are not needed in this paper. The renormalization multiplier $\MN$ is defined by
\begin{equation}\label{intro:eq_MN}
\widehat{\MN f}(n) \defe \Big( \sum_{k \in \bZ^3} \frac{\widehat{V}(n+k)}{\langle  k \rangle^2} \rho_N(k)^2 \Big) \widehat{f}(n),
\end{equation}
where $\rho_N$ is a truncation to frequencies of size $\lesssim N$. The Hamiltonian $H_N$ is then defined as 
\begin{equation}
H_N[\phi_0,\phi_1] \defe \frac{1}{2} \Big( \| \phi_0 \|_{L^2}^2 + \| \langle \nabla \rangle \phi_0  \|_{L^2}^2 + \| \phi_1 \|_{L^2}^2 \Big) + \frac{1}{4} \int_{\bT^3} \lcol ( V \ast ( P_{\leq N} \phi)^2 ) (P_{\leq N} \phi)^2 \rcol \dx. 
\end{equation}
The renormalized and frequency-truncated nonlinear wave equation corresponding to $H_N$ is given by
\begin{equation}\label{intro:eq_truncated_nlw}
\begin{cases}
(-\partial_t^2 - 1 + \Delta ) u = P_{\leq N} \Big( \lcol ( V \ast (P_{\leq N} u )^2 ) P_{\leq N} u \rcol \Big) \qquad (t,x) \in \bR \times \bT^3, \\
u|_{t=0} = \phi_0, \quad \partial_t u |_{t=0} = \phi_1,
\end{cases}
\end{equation}
where the renormalized nonlinearity is given by
\begin{equation}\label{intro:eq_renormalized_nonlinearity}
 \lcol ( V \ast (P_{\leq N} u )^2 ) P_{\leq N} u \rcol \defe  ( V \ast (P_{\leq N} u )^2 ) P_{\leq N} u - a_N \widehat{V}(0) P_{\leq N} u - 2 \MN P_{\leq N} u. 
\end{equation}
We remark that frequencies much larger than $N$ are not affected by the nonlinearity \eqref{intro:eq_renormalized_nonlinearity}. As a result, the nonlinear component of the solution of \eqref{intro:eq_renormalized_nonlinearity} is always smooth. For a fixed $N\geq 1$, the coercivity of $H_N$ implies the global well-posedness of \eqref{intro:eq_truncated_nlw}. We also define the renormalized square 
\begin{equation}
\lcol (P_{\leq N} u )^2 \rcol \defe (P_{\leq N} u )^2 - a_N, 
\end{equation}
which will simplify the notation below. The Gibbs measure $\mup_N$ corresponding to $H_N$ is given by $\mup_N = \mu_N \otimes (\langle \nabla \rangle)_{\#} \cg$, where $\mu_N$ is defined in \cite[(1.10)]{BB20a} and $ (\langle \nabla \rangle)_{\#} \cg$ is the pushforward of the three-dimensional Gaussian field (defined in the introduction of \cite{BB20a}) under $\langle \nabla \rangle$. Before we state the properties of the truncated Gibbs measures $\mup_N$, we recall the assumptions on the interaction potential from the first paper of the series. In these assumptions, $0<\beta<3$ is a fixed parameter.

\begin{assumptions}{A}\label{assumptions:V}
We assume that the interaction potential $V$ satisfies 
\begin{enumerate}
\item $V(x)= c_\beta |x|^{-(3-\beta)}$ for some $c_\beta>0$ and all $x\in \bT^3$ satisfying $\| x\|\leq 1/10$,
\item $V(x) \gtrsim_\beta 1$ for all $x\in \bT^3$, 
\item $V(x)=V(-x)$ for all $x\in \bT^3$,
\item $V$ is smooth away from the origin. 
\end{enumerate}
\end{assumptions}

The following properties of the Gibbs measures $\mup_N$ are a direct consequence of \cite[Theorem 1.1]{BB20a}, which is phrased in terms of $\mu_N$. For notational reasons related to the weak convergence instead of convergence in total variation, we use a second parameter $M$ for the frequency-truncation. Our notation for the random variables, which is based on dots, will be discussed below the theorem. 
\begin{theorem}[Gibbs measures]\label{theorem:measures}
Let $\kappa>0$ be a fixed positive parameter, let $0<\beta<3$ be a parameter, and let the  interaction potential $V$ be as in the Assumptions \ref{assumptions:V}. Then, the truncated Gibbs measures $(\mup_M)_{M\geq 1}$ weakly converge to a limiting measure $\mup_\infty$ on $\cH_{x}^{-1/2-\kappa}(\bT^3)$, which is called the Gibbs measure. If in addition $0<\beta < \frac{1}{2}$, then the Gibbs measure $\mup_\infty$ is singular with respect to the Gaussian free field $\cgp$. \\
Furthermore, there exists a sequence of reference measures  $(\nup_M)_{M\geq 1}$  on $\cH_{x}^{-1/2-\kappa}(\bT^3)$ and an ambient probability space $(\Omega,\mathcal{F},\bP)$ satisfying the following two properties: 
\begin{enumerate}
\item (Absolute continuity and $L^q$-bounds) The truncated Gibbs measure $\mup_M$ is absolutely continuous with respect to the reference measure $\nup_M$. More quantitatively, there exists a parameter $q>1$ and a constant $C\geq 1$ independent of $M$ such that
\begin{equation*}
\mup_M(A) \leq C \nup_M(A)^{1-\frac{1}{q}}
\end{equation*}
for all Borel sets $A \subseteq \cH_{x}^{-1/2-\kappa}(\bT^3)$. 
\item (Representation of $\nup_M$) Let $\gamma=\min(1/2+\beta,1)$. Then, there exist two random variables $\, \bluedot, \reddotM \colon (\Omega,\mathcal{F}) \rightarrow \cH_{x}^{-1/2-\kappa}(\bT^3)$ and a large integer $k=k(\beta) \geq 1$ satisfying for all $p\geq 2$ that
\begin{equation*}
\nup_M = \Law_\bP( \, \bluedot + \reddotM), \quad  \cgp = \Law_\bP( \,  \bluedot \, ), \quad \text{and} \quad   
\Big( \bE_{\bP} \| \, \reddotM \|_{\cH_x^{\gamma-\kappa}(\bT^3)}^p \Big)^{\frac{1}{p}} \leq p^{\frac{k}{2}}. 
\end{equation*}
\end{enumerate}
\end{theorem}
\begin{remark}
After the completion of this series, the author learned of independent work by Oh, Okamoto, and Tolomeo \cite{OOT20}, which yields an analogue of Theorem \ref{theorem:measures}. We refer to Remark 1.2 in the first part of the series \cite{BB20a} for a more detailed comparison. 
\end{remark}
 We will require that the ambient probability space $(\Omega,\mathcal{F},\bP)$ is rich enough to contain a family of independent Brownian motions, which is clear from the definition of  $(\Omega,\mathcal{F},\bP)$ in \cite{BB20a} and detailed in Section \ref{section:gaussian_processes}.

Let us further explain the notation in Theorem \ref{theorem:measures}. We use dots to represent the random data, since they can be used as building blocks in more complicated stochastic objects. We already saw this graphical notation in our discussion of \cite{GKO18a} and we refer the reader to \cite{MWX17} for a detailed discussion of similar diagrams. We use the blue dot $\, \bluedot\,$ for the Gaussian random data, since it lives at low spatial regularities and is primarily viewed as a high-frequency term. We use the red dot $\, \reddotM\,$ to denote the more regular component of the random data, since we primarily view it as a low-frequency term. Furthermore, the blue dot $\,\bluedot\,$ is filled while the red dot $\, \reddotM\,$ is not filled. The reason is that the manuscript should be accessible to colorblind readers and also readable as a black and white copy. 

In the following, we often write $\, \purpledot \,$ for a generic element $\phi \in \cH_x^{-1/2-\kappa}(\bT^3)$. The purple diamond will be used as a building block for further stochastic objects. When working with the reference measure $\nup_M$, we have that 
\begin{equation*}
\Law_{\nup_M}(\, \purpledot\,) = \Law_\bP( \, \bluedot + \reddotM).
\end{equation*}
Naturally, we chose the color purple since it is a mixture of blue and red. The change in shape, i.e., from a dot to a diamond, is primarily made for colorblind readers. We also only use diamonds for intrinsic objects in $\cH_x^{-1/2-\kappa}(\bT^3)$, while dots are used for objects defined on the ambient probability space $(\Omega,\mathcal{F},\bP)$. The significance of this distinction will be further discussed in Sections \ref{section:local} and \ref{section:global}. 

While Theorem \ref{theorem:measures} already contains the measure-theoretic results of this series, we now state the dynamical results.

\begin{theorem}[Global well-posedness \& invariance] \label{theorem:gwp_invariance}
There exists a Borel-measurable set $\cS \subseteq \cH^{-1/2-\kappa}(\bT^3)$ satisfying $\mup_\infty(\cS)=1$ and such that the following two properties hold: 
\begin{enumerate}
\item (Global well-posedness) Let $\Phi_N$ be the flow of the renormalized and frequency-truncated wave equation \eqref{intro:eq_truncated_nlw}. Then, the limit
\begin{equation*}
\Phi_\infty[t] \, \purpledot \defe \lim_{N\rightarrow \infty} \Phi_N[t] \, \purpledot  
\end{equation*}
exists in $\cH^{-1/2-\kappa}(\bT^3)$ for all $t\in \bR$ and $\, \purpledot\, \in \cS$.  \label{intro:item_gwp}
\item (Invariance)  The Gibbs measure $\mup_\infty$ is invariant under $\Phi_\infty$, i.e., it holds for all $t\in \bR$ that
\begin{equation*}
\Phi_\infty[t]_\# \mup_\infty = \mup_\infty 
\end{equation*}
\end{enumerate}
\end{theorem}

\begin{remark}\label{intro:rem_restriction_beta}
In the proof of Theorem \ref{theorem:gwp_invariance}, we restrict ourselves to the  case $\beta\in (0,1/2)$. The purpose of this restriction is purely notational. The same argument also works for $\beta\in [1/2,3)$, as long as $\beta$ in each estimate is replaced by $\min(\beta,1/2)$. 
\end{remark}

\begin{remark}
While Theorem \ref{theorem:gwp_invariance} shows that the limiting dynamics $\Phi_\infty[t]$ are well-defined, we do not obtain that  $\Phi_\infty[t]$ satisfies the group property. The author believes that the estimates in this paper (from Sections \ref{section:stochastic_object}-\ref{section:physical}) are strong enough to prove the group property, but the stability theory (Section \ref{section:stability_gaussian} and Section \ref{section:stability}) would need to be modified. Instead of working with a single flow $\Phi_N[t]$, one needs similar statements for the mixed flows $\Phi_{N_1}[t_1] \Phi_{N_2}[t_2]$. We refer the reader to \cite{ST19} for a more detailed discussion of the group property and its relation to the recurrence properties of the flow. 
\end{remark}

We now describe individual aspects of our argument. As in our discussion of the previous literature, we separate the local and global aspects. As mentioned above, our contributions to the local theory are of an intricate but technical nature, whereas our contributions to the global theory are conceptual. \\

In the local theory, we use the absolute continuity $\mup_M \ll \nup_M$ and the representation of $\nup_M$ from Theorem \ref{theorem:measures}. As a result, the reference measure $\nu_M$ serves the same purposes as the Gaussian free field in earlier results on invariant Gibbs measures. We then follow the para-controlled approach of \cite{GKO18a} and decompose the solution  $u_N(t)$ of \eqref{intro:eq_truncated_nlw} as 
\begin{equation}\label{intro:eq_decomposition}
u_N = \<1b> + \<3DN> + X_N + Y_N, 
\end{equation}
where the stochastic objects $\<1b>$ and $\<3DN>$, the para-controlled component $X_N$, and the smoother nonlinear remainder $Y_N$ are defined in Section \ref{section:local}. The smoother component $\, \reddotM\,$ in the representation of $\nup_M$ will be placed inside $Y_N$. In comparison to \cite{GKO18a}, however, there is an increase in the complexity of the evolution equation for $Y_N$. We split the terms into four different categories, which correspond to the methods used in their estimates.
\begin{itemize}
\item[$\bullet$] \emph{Stochastic objects:} These terms are explicit and include 
\begin{equation*}
\<131N>  \qquad \text{and} \qquad \<313N>. 
\end{equation*}
In contrast to the previous literature, we use multiple stochastic integrals for the non-resonant/resonant-decompositions, which significantly decreases the algebraic complexities. We also use counting estimates related to the dispersive symbol of the wave equation. 
\item[$\bullet$] \emph{Random matrix terms:} The terms include 
\begin{equation*}
\Big( V \ast \<2N>\Big) P_{\leq N} Y_N. 
\end{equation*}
They will be controlled through a recent random matrix estimates of Deng, Nahmod, and Yue \cite[Proposition 2.8]{DNY20}, which is based on the moment method. 
\item[$\bullet$] \emph{Contributions of para-controlled terms:} These terms include 
\begin{equation*}
V \ast \Big( P_{\leq N} \<1b> \paraeq P_{\leq N} X_N \Big) P_{\leq N} Y_N. 
\end{equation*}
We use the double Duhamel trick to exploit stochastic cancellations between $\<1b>$ and $X_N$. In our definition of $X_N$, we use the paradifferential operators $\parald$ and $\paraboxld$ introduced in Section \ref{section:local}, which form a technical novelty.
\item[$\bullet$] \emph{Physical terms:} These terms include 
\begin{equation*}
V \ast \Big( P_{\leq N} \<1b> \cdot P_{\leq N} Y_N \Big) P_{\leq N} \<3DN> \qquad \text{and} \qquad \Big( V \ast (P_{\leq N} Y_N)^2 \Big) P_{\leq N} Y_N. 
\end{equation*} 
The first term should be viewed as a random operator in $Y_N$, but is mainly treated through physical-space arguments. We believe that our approach is of independent interest, since it provides an alternative to the more Fourier-analytic estimates in \cite{Bourgain96,GKO18a,DNY19,DNY20}. The second term is treated deterministically and we rely on the refined Strichartz-estimates of Klainerman and Tataru \cite{KT99}. 
\end{itemize}
As we mentioned before, all stochastic objects have been based on $\, \bluedot \,$ and the smoother component $\, \reddotM \,$ is simply placed inside $Y_N$. This approach yields the convergence of the flows $\Phi_N$ on the support of $\mup_\infty$ for a short time interval (see Corollary \ref{lwp:cor_unstructured}). The structural information in the decomposition \eqref{intro:eq_decomposition}, however, cannot (directly) be carried over to the support of $\mup_\infty$, since $\,\bluedot\,$ is only defined on the ambient probability space $(\Omega,\mathcal{F},\bP)$. This defect will be addressed below, since the structural information is required for the global theory. 

\begin{remark}\label{intro:remark_DNY}
As was already mentioned in our overview of the literature, Deng, Nahmod, and Yue recently developed a theory of random tensors \cite{DNY20}, which forms a comprehensive framework for the local theory of random dispersive equations. 
The theory of random tensors (and its precursor \cite{DNY19}) rely more intricately on the independence of the Fourier coefficients than the para-controlled approach. Even under the reference measure $\nup_M$, however, the random data  $\purpledot = \bluedot + \reddotM$ has dependent Fourier coefficient. This presents a challenge for the theory of random tensors, which was already mentioned in \cite[Section 9.1]{DNY20}.
 In addition, there are further technical problems related to the switch from  Schr\"{o}dinger to wave equations, which are described in Section \ref{section:counting}. As a result, the author views the extension of the theory of random tensors to a local theory even for singular Gibbs measures and/or nonlinear wave equations as an interesting open problem. \\
\end{remark}

After this discussion of the local theory, we turn to the global dynamics on the support of the Gibbs measure $\mup_\infty$. As we have seen in our earlier discussion of Bourgain's globalization argument, its original version requires the convergence of the truncated Gibbs measures in total variation. Unfortunately, Theorem \ref{theorem:measures} only yields the weak convergence of the truncated Gibbs measures $\mup_M$ to $\mup_\infty$. We now give an informal description of our new globalization argument, but postpone a rigorous discussion until Section \ref{section:global}. 

We let $T\geq 1$ be a large time, $B\geq 1$ be a large parameter describing the size of the evolution, $K\geq 1$ be a large frequency scale, and $\tau>0$ be a small step-size. For any $j\geq 1$, we let $\cEj\subseteq \cH^{-1/2-\kappa}(\bT^3)$ be the set of initial data $\purpledot$ satisfying for all $t\in [0,j\tau]$ and $N\geq K$  that
\begin{equation}\label{intro:eq_global_decomposition}
\Phi_N(t)\, \purpledot = \<1p>(t) + \<3DNp>(t) + w_N(t),
\end{equation}
where $w_N$ has size at most $B$ in ``structured high-regularity'' norms. In our rigorous argument, $B$ will depend on $j$, but we ignore this during our informal discussion. We also omit a smallness condition for the difference of $\Phi_N(t)\, \purpledot$ and  $\Phi_K(t)\, \purpledot$. The goal is to prove by induction over $j\leq T/\tau$ that 
\begin{equation*}
\limsup_{M\rightarrow \infty} \mup_M(\, \purpledot \in \cEj) 
\end{equation*}
is close to one as long as $B$, $K$, and $\tau$ are chosen appropriately. The proof relies on four separate ingredients:
\begin{enumerate}[(i)]
\item \emph{(Structured local well-posedness)} This is the base case $j=1$. Using our local theory, we only have to convert the stochastic objects in \eqref{intro:eq_decomposition}, which are based on $\, \bluedot\,$, into stochastic objects based on $\, \purpledot\,$.
\item \emph{(Structure and time-translation)} Using the induction hypothesis, we now assume that the probability ${\mup_M(\, \purpledot \in \cEjminus)} $ is close to one. In order to increase the time-interval, we let $\, \greendot \defe \Phi_M[\tau] \, \purpledot$. Using the invariance of $\mup_M$ under $\Phi_M$, we obtain that 
\begin{equation*}
 \qquad \mup_M(\, \greendot \in \cEjminus) =  \mup_M(\, \Phi_M[\tau] \,\purpledot \, \in \cEjminus) = \mup_M(\, \purpledot \, \in \cEjminus),
\end{equation*} 
which is close to one. After unpacking the definitions, we obtain information on the mixed flow $\Phi_N[t-\tau] \Phi_M[\tau] \, \purpledot$ for $t\in [\tau,j\tau]$. It therefore remains to analyze the difference between  $\Phi_N[t-\tau] \Phi_M[\tau] \, \purpledot$ and $\Phi_N[t-\tau] \Phi_N[\tau] \, \purpledot$. 
\item \emph{(Structure and the cubic stochastic object)} The lowest regularity term in $\Phi_N(\tau) \, \purpledot- \Phi_M(\tau) \, \purpledot$ is given by a portion of the cubic stochastic object. In this step, we add the linear evolution of this portion to the mixed flow $\Phi_N[t-\tau] \Phi_M[\tau] \, \purpledot$, which yields a function $\widetilde{u}_N$. It is then shown that $\widetilde{u}_N(t)$ is an approximate solution of the nonlinear wave equation \eqref{intro:eq_truncated_nlw} for $t\in [\tau,j\tau]$. 
\item \emph{(Stability theory)} We develop a para-controlled stability theory and construct a solution $u_N$ close to the approximate solution $\widetilde{u}_N$, which also accounts for the remaining portion of  $\Phi_N(\tau) \, \purpledot- \Phi_M(\tau) \, \purpledot$. Since our stability theory preserves the structure of $\widetilde{u}_N$, this yields  \eqref{intro:eq_global_decomposition} on the time-interval $[\tau,j\tau]$. Since the base case already yields the desired structure on $[0,\tau]$, this completes the induction step. 
\end{enumerate}
As is evident from this sketch, the proof of global well-posedness is much more involved than in Bourgain's original setting \cite{Bourgain94,Bourgain96}.  While not perfectly accurate, the author finds the following comparison with the deterministic global theory of dispersive equations illustrative. Bourgain's globalization argument \cite{Bourgain94,Bourgain96} is the probabilistic version of a deterministic global theory using a (sub-critical) conservation law. The conservation law is replaced by the invariance, which implies that $t\mapsto \mu_N(\Phi_N(t)\phi \in \mathcal{E})$ is constant. In both cases, the global well-posedness is obtained by iterating the local well-posedness, but the estimates used in the local theory are no longer needed. In contrast, the new globalization argument is the probabilistic version of a deterministic global theory using almost conservation laws (cf. \cite{CKSTT02}). The place of the almost conserved quantities is taken by the functions $t\mapsto \mu_M(\Phi_N(t) \phi \in \mathcal{E})$, which should be close to a constant function. In addition, the proof of global well-posedness often intertwines the local estimates and the choice of the almost conserved quantities. For entirely different reasons, the similarity with almost conserved quantities also appears in the globalization argument of \cite{NORS12}, which proves the invariance of a Wiener measure for the periodic derivative nonlinear Schr\"{o}dinger equation. The truncated dynamics in \cite[(3.1)]{NORS12} only approximately conserve the energy (cf. \cite[Theorem 4.2]{NORS12}). Even with the same truncation parameter in the measure and the dynamics, the truncated Wiener measure is then only almost invariant (cf. \cite[Proof of Lemma 6.1]{NORS12}). 

Our globalization argument for the nonlinear wave equation also differs from the globalization argument for the parabolic stochastic quantization equation as in \cite{HM18}. While the invariant measure is singular in both situations, the dependence on the initial data  in the parabolic setting is continuous even at spatial singularity $-1/2-$. As a result, it is possible to iterate the local theory over the time-intervals $\{ [(j-1)\tau,j\tau]\}_{j=1,\hdots,J}$ using only bounds in the $C^{-1/2-}(\mathbb{T}^3)$-norm.  As can be seen from the sketch above, iterating the local theory for the nonlinear wave equation \eqref{intro:eq_truncated_nlw} requires more detailed information on the solution.

Once the global well-posedness has been proven, the proof of invariance is essentially the same as in \cite{Bourgain94}. 

\begin{remark}
A paper of this length creates both mathematical challenges and different options for the exposition. The author does not claim to have found the perfect solutions or made the best expository choice in every single instance. While we postpone a more detailed discussion to Remark \ref{intro:remark_DNY}, Remark \ref{local:rem_paracontrolled}, Remark \ref{global:rem_organization_lwp}, Remark \ref{tools:rem_why_multiple}, Remark \ref{phy:rem_lorentzian_vs_galilean}, and Remark \ref{ftg:rem_drawback}, the author wanted to make this point in a central location of the paper. The author hopes that this encourages the reader to think more about our result and related open problems. 
\end{remark}

\textbf{Acknowledgements:} The author thanks his advisor Terence Tao for his patience and invaluable guidance. The author also thanks Nikolay Barashkov, Yu Deng, Martin Hairer, Redmond McNamara, Dana Mendelson, Andrea Nahmod, Tadahiro Oh, Felix Otto, Nikolay Tzvetkov, Haitian Yue, and Guangqu Zheng for  helpful comments and discussions.

%%%%%%%%%%%%%%%%%%%%%%%%%%%
%%%%%%%%%%%%%%%%%%%%%%%%%%%%
\subsection{Overview} 
Due to the excessive length of this paper, we include a few suggestions for the reader. We also display the (main) relationship between the sections in Figure \ref{figure:organization}. 

The local and global theory are described in Section \ref{section:local} and \ref{section:global}, respectively. These sections contain the main novelties of this paper and should be interesting to most readers. As long as the reader believes several estimates, these sections are also self-contained. We therefore encourage the expert to focus on these sections. 

Section \ref{section:tools} contains a collection of tools from dispersive equations, harmonic analysis, and probability theory. The reader should be familiar with the content of each subsection before moving on, but the expert should be able to only skim most content. 

The Sections \ref{section:stochastic_object}-\ref{section:physical} contain the main technical aspects of this paper. They are concerned with separate terms in the evolution equation and rely on different methods. As a result, they can (essentially) be read independently. 

In Section \ref{section:ftg}, we extend the multi-linear estimates from Sections  \ref{section:stochastic_object}-\ref{section:physical}, which have been phrased in terms of the Gaussian initial data $\,\bluedot\,$, to random initial data $\, \purpledot \,$ drawn from the Gibbs measure. Each proof consists of a concatenation of previous results, and hence this section can safely be skipped on first reading.

%Different node styles for diagram
\tikzset{rect/.style={rectangle, rounded corners, minimum width=4cm, minimum
height=1cm,text centered,align=center, draw=black, fill=white!10},
arrow/.style={thick,->,>=stealth}}
\tikzset{title/.style={rectangle,  minimum width=3cm, minimum
height=1cm,text centered,align=center, draw=black, fill=orange!10},
arrow/.style={thick,->,>=stealth}}
\tikzset{theorem/.style={rectangle, minimum width=4cm, minimum
height=1cm,text centered,align=center, draw=black, fill=red!10},
arrow/.style={thick,->,>=stealth}}

%Double arrow
\tikzstyle{vecArrow} = [thick, decoration={markings,mark=at position
   1 with {\arrow[semithick]{open triangle 60}}},
   double distance=1.4pt, shorten >= 5.5pt,
   preaction = {decorate},
   postaction = {draw,line width=1.4pt, white,shorten >= 4.5pt}]
\tikzstyle{innerWhite} = [semithick, white,line width=1.4pt, shorten >= 4.5pt]

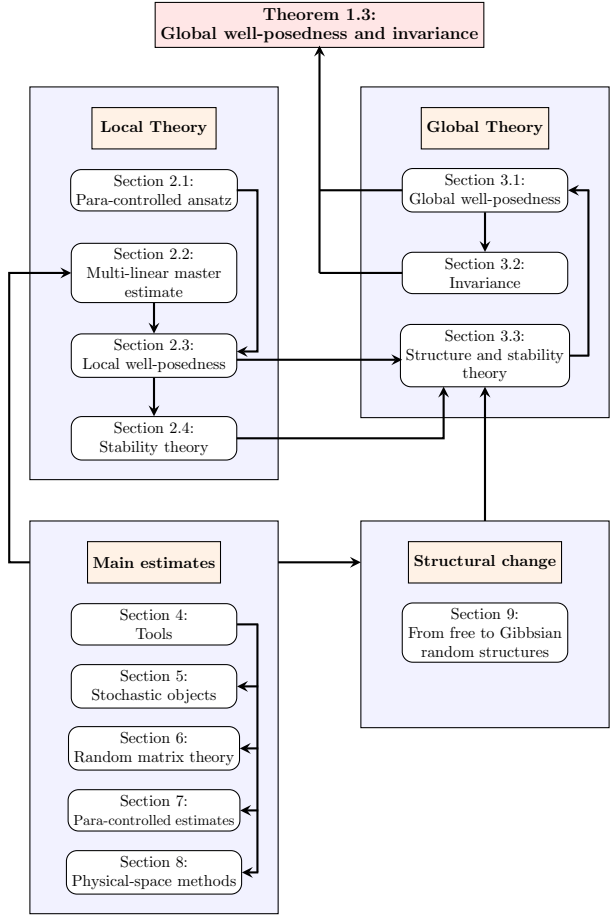
\begin{figure}[t]
\begin{center}
\begin{tikzpicture}[node distance=2cm,scale=0.55, every node/.style={scale=0.55}]

%Main theorem
\node (thm)[theorem] at (0,2) {\textbf{\large \begin{NoHyper}Theorem \ref{theorem:gwp_invariance}:\end{NoHyper}} \\ \textbf{\large Global well-posedness  and invariance}};

%Local theory
\node (local)[title]  at (-4,-0.5) {\textbf{Local Theory}};
\node (para)[rect]  at (-4,-2) {Section 2.1:\\ Para-controlled ansatz};
\node (ms)[rect]  at (-4,-4) {Section 2.2:\\ Multi-linear master \\ estimate};
\node (lwp)[rect]  at (-4,-6) {Section 2.3:\\ Local well-posedness};
\node (lstab)[rect]  at (-4,-8) {Section 2.4:\\ Stability theory};
\begin{scope}[on background layer]
\draw [fill=blue!5] (-7,0.5) rectangle (-1,-9); 
\end{scope}

%Global theory
\node (global)[title]  at (4,-0.5) {\textbf{Global Theory}};
\node (gwp)[rect]  at (4,-2) {Section 3.1:\\ Global well-posedness};
\node (inv)[rect]  at (4,-4) {Section 3.2:\\ Invariance};
\node (stab)[rect]  at (4,-6) {Section 3.3:\\ Structure and stability \\ theory};
\begin{scope}[on background layer]
\draw [fill=blue!5] (7,0.5) rectangle (1,-7.5); 
\end{scope}

%Main estimates 
\node (est)[title]  at (-4,-11) {\textbf{Main estimates}};
\node (tools)[rect]  at (-4,-12.5) {Section 4:\\ Tools};
\node (so)[rect]  at (-4,-14) {Section 5:\\ Stochastic objects};
\node (rmt)[rect]  at (-4,-15.5) {Section 6:\\ Random matrix theory};
\node (parae)[rect]  at (-4,-17) {Section 7:\\ \small{Para-controlled estimates}};
\node (phy)[rect]  at (-4,-18.5) {Section 8:\\ Physical-space methods};
\begin{scope}[on background layer]
\draw [fill=blue!5] (-7,-10) rectangle (-1,-19.5); 
\end{scope}

%From free to Gibbs
\node (struc)[title]  at (4,-11) {\textbf{Structural change}};
\node (ftg)[rect]  at (4,-12.7) {Section 9:\\ From free to Gibbsian\\ random structures};
\begin{scope}[on background layer]
\draw [fill=blue!5] (7,-10) rectangle (1,-15); 
\end{scope}

%%%%%%%%%%%%% Connecting arrows 
%Inside Local and inside global theory
\draw[arrow] (ms) -- (lwp);
\draw[arrow] (lwp) -- (lstab);
\draw[arrow] (gwp) -- (inv);
\draw[thick] (stab) -- ($(stab)+(2.52,0)$);
\draw[thick] ($(stab)+(2.5,0)$) -- ($(gwp)+(2.5,0)$);
\draw[arrow] ($(gwp)+(2.52,0)$) --(gwp); 

%Between para and lwp
\draw[arrow] (-2,-2) -- (-1.5,-2) -- (-1.5,-5.9) -- (-2,-5.9);

%Between local and global
\draw[arrow] (-2,-6.1)-- (2,-6.1);
\draw[arrow] (-2,-8) -- (3,-8) -- (3,-6.75);

%Inside main estimates 
\draw[thick] (tools) -- ($(tools)+(2.52,0)$);
\draw[arrow]  ($(so)+(2.52,0)$) -- (so); 
\draw[arrow]  ($(rmt)+(2.52,0)$) -- (rmt); 
\draw[arrow]  ($(parae)+(2.52,0)$) -- (parae); 
\draw[arrow]  ($(phy)+(2.52,0)$) -- (phy); 
\draw[thick] ($(tools)+(2.5,0)$) -- ($(phy)+(2.5,0)$);

%Section arrows
\draw[arrow] (-1,-11) -- (1,-11); 
\draw[arrow] (4,-10) -- (4,-6.75); 

%From 4-8 into 2.2
\draw[arrow] (-7,-11) -- (-7.5,-11) -- (-7.5,-4) -- (-6,-4); 

%To main theorem
\draw[arrow] (0,-4) -- (0,1.5); 
\draw[thick] (0,-4) -- (2,-4); 
\draw[thick] (0,-2) -- (2,-2); 

\end{tikzpicture}
\end{center}
\caption{\small{This figure illustrates the main dependencies between the different sections. The heart of the paper lies in the local and global theory (Section \ref{section:local} and \ref{section:global}), which, as long as the reader believes certain estimates, can be read independently from the rest of the paper. A few minor dependencies between the different sections are not included in this illustration. For instance, basic properties of $\X{s}{b}$-spaces, which are recalled in Section \ref{section:tools}, will also be used in Section \ref{section:local} and \ref{section:global}.}}\label{figure:organization}
\end{figure}

\subsection{Notation}\label{section:notation}
We recall and introduce notation that will be used throughout the rest of the paper. \\

\emph{Dyadic numbers:} Throughout this paper, we denote dyadic integers by $K,L,M$, and $N$. In limits or sums, such as $\lim_{M\rightarrow \infty}$ or $\sum_N$, we implicitly restrict ourselves to dyadic integers. \\

\emph{Parameters:} We first introduce several parameters which are used in our function spaces, in the paradifferential operators, and our estimates. We fix 
\begin{equation}\label{intro:eq_parameters}
\epsilon>0, \quad \delta_1,\delta_2>0, \quad \kappa>0, \quad  \eta,\eta^\prime>0, \quad \text{and} \quad b_+>b>1/2>b_->0. 
\end{equation}
We use $\epsilon>0$ in our para-differential operators, $\kappa>0$ to capture small losses in probabilistic estimates, $\eta,\eta^\prime>0$ to capture gains in the highest frequency-scale, and $\delta_1,\delta_2,b_+,b,b_-$ in the definition of our function spaces. 
We impose the condition
\begin{equation}\label{intro:eq_parameter_condition}
1/2 - b_- \ll b- 1/2 \ll b_+-1/2 \ll \eta^\prime \ll \eta \ll \kappa \ll \delta_2 \ll \epsilon \ll \delta_1. 
\end{equation}
In \eqref{intro:eq_parameter_condition}, the implicit constant in each ``$\ll$'' is allowed to depend on all parameters appearing to its right. We also define
\begin{equation*}
s_1 = \frac{1}{2} - \delta_1 \qquad \text{and} \qquad s_2 = \frac{1}{2}+\delta_2. 
\end{equation*}
In several statements of this paper, we will also use $0<\zeta<1$ and $C\geq 1$ as parameters. However, they may change their values between different lines and are allowed to depend on all parameters in \eqref{intro:eq_parameters}. \\

\emph{Wave equation and flows:} We denote the solution of the nonlinear wave equation \eqref{intro:eq_truncated_nlw} by $u_N(t)$. We also write 
\begin{equation*}
u_N[t] \defe \begin{pmatrix} u_N(t) , \partial_t u_N(t) \end{pmatrix}, 
\end{equation*}
which is standard in the literature on nonlinear wave equations. If $\purpledot \in \cH_x^{-1/2-\kappa}(\bT^3)$, we also write $\Phi_N(t) \, \purpledot$ and $\Phi_N[t]\, \purpledot$ for the solution with initial data $\, \purpledot$. When working with the flows $\Phi_N[t]$ and the Gibbs measures $\mup_M$, we write $\Phi_N[t]_{\#} \mup_M$ for the pushforward of $\mup_M$ under $\Phi_N[t]$. 

Furthermore, we denote the Duhamel integral operator of the wave equation by $\Duh$. More precisely, we define
\begin{equation*}
\Duh \big[ F\big] (t) \defe \int_0^t \frac{\sin((t-t^\prime) \langle \nabla \rangle)}{\langle \nabla \rangle} F(t^\prime) \dtprime. 
\end{equation*}

\emph{Fourier transform:} With a slight abuse of notation, we write $\mathrm{d}x$ for the normalized Lebesgue measure on $\bT^3 = \bR^3 / (2\pi \bZ)^3$, i.e., we require that 
\begin{equation*}
\int_{\bT^3} 1 \dx = 1. 
\end{equation*}
We then define the Fourier transform of a function $f\colon \bT^3 \rightarrow \bC$ by 
\begin{equation}
\widehat{f}(n) \defe \int_{\bT^3} f(x) e^{-inx} \dx. 
\end{equation}
For any $k\in \mathbb{N}$ and any $n_1,n_2,\hdots,n_k \in \bZ^3$, we define
\begin{equation*}
n_{12\hdots k}\defe \sum_{j=1}^k n_j . 
\end{equation*}
For example, $n_{12}=n_1+n_2$ and $n_{123}=n_1+n_2+n_3$. \\

\emph{Interaction potential:} For a given interaction potential $V$ satisfying the Assumptions \ref{assumptions:V}, we define
\begin{equation*}
\widehat{V}_S(n_1,n_2,n_3) \defe \frac{1}{6} \sum_{\pi \in S_3} \widehat{V}(n_{\pi_1}+n_{\pi_2}). 
\end{equation*}

\emph{Truncations and Littlewood-Paley operators:} For each $t\geq 0$, we let $\rho_t \colon \bZ^3 \rightarrow [0,1]$ be the same truncation to frequencies $n \in \bZ^3$ satisfying $|n|\lesssim \langle t \rangle$ as in \cite[Section 1.3]{BB20a}. For each dyadic $N\geq 1$, we define the Littlewood-Paley multiplier $P_{\leq N}$ by 
\begin{equation*}
\widehat{P_{\leq N} f}(n) = \rho_N(n) \widehat{f}(n). 
\end{equation*}
We further set 
\begin{equation*}
P_1 f = P_{\leq 1} f \qquad \text{and} \qquad P_{N} f = P_{\leq N} f - P_{\leq N/2} f \quad \text{for all } N\geq 2. 
\end{equation*}
The corresponding Fourier multipliers are denoted by 
\begin{equation*}
\chi(n) = \chi_1(n) = \rho_1(n) \qquad \text{and} \qquad \chi_N(n) = \rho_N(n) - \rho_{N/2}(n) \quad \text{for all } N\geq 2. 
\end{equation*}
We also define fattened Littlewood-Paley multipliers by 
\begin{equation*}
\widetilde{P}_N = \sum_{N/16\leq K \leq 16K} P_K.  
\end{equation*}

\emph{Function spaces:}
For any \( s \in \bR \), the \( \cC_x^s(\bT^3) \)-norm is defined as 
\begin{equation}
\| f \|_{\cC_x^s(\bT^3)} \defe \sup_{N\geq 1} N^s \| P_N f \|_{L^\infty_x(\bT^3)}.
\end{equation}
We then define the corresponding space \( \cC_x^s(\bT^3) \) by 
\begin{equation}\label{notation:eq_Cx}
\cC_x^s(\bT^3) \defe \big\{ f \colon \bT^3 \rightarrow \bR |~  \| f \|_{\cC_x^s} < \infty, \lim_{N\rightarrow \infty} N^s \| P_N f \|_{L^\infty_x(\bT^3)} = 0 \big\}. 
\end{equation}
We let $H_x^s(\bT^3)$ be the usual $L^2$-based Sobolev space. More precisely, for any $f\colon \bT^3 \rightarrow \bC$, we define the corresponding norm by
\begin{equation*}
\| f \|_{H_x^s(\bT^3)} \defe \| \langle n \rangle^s \widehat{f}(n) \|_{\ell^2_n(\bZ^3)}. 
\end{equation*}
Furthermore, we define $\cH_x^{s}(\bT^3)\defe H_x^{s}(\bT^3) \times H_x^{s-1}(\bT^3)$. In this paper, we will also use the Bourgain spaces $\X{s}{b}(\cJ)$ and the low-frequency modulation space $\LM(\cJ)$, which are defined in Definition \ref{tools:def_xsb} and Definition \ref{para:def_LM}, respectively.

\section{Local theory}\label{section:local}

In this section, we show that the truncated and renormalized nonlinear wave equations
\begin{equation}\label{eq:nlw_N}
\begin{cases}
(- \partial_t^2-1+\Delta) u_N  = P_{\leq N} \Big( \lcol ( V \ast (P_{\leq N} u_N )^2 ) P_{\leq N} u_N \rcol \Big) \\
u_N[0] = \mathbf{\phi}
\end{cases}
\end{equation}
are locally well-posed on the support of the Gibbs measures $\mup_M$ uniformly in $M$. It is important in the definition of the limiting dynamics and the globalization argument that the truncation parameter $N$ in the dynamics and the truncation parameter $M$ in the Gibbs measure $\mup_M$ are allowed to be different. \\

 Due to the truncation, a soft argument based on the coercivity of the Hamiltonian shows that \eqref{eq:nlw_N} is globally well-posed for a fixed truncation parameter $N$. We denote the corresponding flow by $\Phi_N(t)$.

\subsection{Para-controlled ansatz}\label{section:ansatz}

We now introduce our para-controlled approach. As discussed in the introduction, we will use a graphical notation for the several stochastic objects appearing in this paper. We denote the random initial data by \( \purpledot \). In the local theory, we can work with the reference measure \( \nup_M \) and, more precisely, the representation of the reference measure with respect to the ambient measure \( \bP \). 

Based on Theorem \ref{theorem:measures}, we have that \( \nup_M = \Law_\bP(\bluedot+\reddotM) \), where $\,\bluedot\,$ is the Gaussian low-regularity component and $\,\reddotM\,$ is has regularity $\min(1/2+\beta,1)-$. Naturally, we chose the color purple for the random initial data $\purpledot$ since it is a mixture of the blue and red random initial data. We emphasize that $\bluedot$ and $\reddotM$ are probabilistically dependent! Fortunately, this does not introduce any major difficulties in our treatment of the wave equation with a Hartree nonlinearity. We believe, however, that the proof of the invariance of the Gibbs measure for both the cubic wave equation and the three-dimensional Schrödinger equation with cubic or Hartree nonlinearity will require a more detailed understanding of the relationship between $\bluedot$ and $\reddotM$. This additional information is provided in the first part of the series \cite{BB20a}. \\

Before we introduce our stochastic and para-controlled objects, we discuss the following question: Should we define our stochastic objects based on 
$\,\bluedot\,$ or based on $\,\purpledot\,$? Due to the independence of the Fourier coefficient under \( \bP\) and its simple structure, it is much more convenient to work with $\,\bluedot\,$. However, the decomposition \(\, \purpledot= \bluedot+\reddotM \, \) of the samples of \( \nup_M \) is based on the ambient measure \( \bP\). It cannot be performed intrinsically on the samples of \( \nup_M \) and has \emph{no meaning} for the Gibbs measure \( \, \mup_M \). In particular, if we want to examine the probability of an event  under \( \mup_M \), we must phrase the event in terms of the full initial data 
\(\purpledot\,\). Fortunately, there is a convenient solution to our conundrum: We first carry out most of our (local) analysis in terms of $\bluedot$ and with respect to the ambient measure \( \bP \). Once all the estimates in terms of \( \, \bluedot \, \) are available, we can convert the stochastic objects and para-controlled structures from $\,\bluedot\,$ into $\,\purpledot\,$ (see Section \ref{section:ftg}). Then, the absolute continuity of \( \mup_M \) with respect to the reference measure \( \nup_M \) allows us to obtain the same stochastic objects and para-controlled structures on the support of the Gibbs measure \( \mup_M \). \\

We now begin with the construction of the stochastic objects and para-controlled structures, which were briefly discussed in the introduction. We define \( \<1b> \) as the linear evolution of the random initial data \( \bluedot\,\). More precisely, \( \<1b> \) solves the evolution equation
\begin{equation}\label{eq:1b}
(-\partial_t^2 - 1 + \Delta )\, \<1b> = 0, \quad \<1b>[0]=\bluedot. 
\end{equation}
The black line in the stochastic object reflects the linear propagator of the wave equation. For future use, we define the frequency-truncated and renormalized square of $\<1b>$ by  
\begin{equation}
\<2N> \defe \lcol ( P_{\leq N} \<1b> )^2 \rcol . 
\end{equation}
The multiplication is reflected by the joining of the two lines and the frequency-truncation is reflected in the subscript $N$. We then define the renormalized nonlinearity $\<3N>$ by
\begin{equation}
\<3N> \defe P_{\leq N} \Big( \lcol ( V \ast ( P_{\leq N} \<1b>)^2 ) ( P_{\leq N} \<1b>) \rcol \Big).
\end{equation}
The orange asterisk reflects the convolution with the interaction potential. The color orange has no significance and we only chose it for aesthetic reasons. As before, the nonlinearity is reflected in the joining of the three lines and the truncation parameter $N$ in the nonlinearity appears as a subscript. Finally, we define the Duhamel integral of $\<3N>$  by
\begin{equation}\label{local:eq_3DN}
(-\partial_t^2 - 1 + \Delta )  \<3DN> =  \<3N>, \quad  \<3DN>[0]=0. 
\end{equation}
The line with an arrow reflects the integration in the Duhamel operator. In contrast to $\<1b>$, we note that the distribution of $\<3DN>$ is not stationary in time. Naively, one may expect that  $\<3DN>$ has spatial regularity \( -1/2+\beta-\). Namely, one would expect spatial regularity \( 3\cdot (-1/2)- \) from the cube of the random initial data \( \, \bluedot \, \), a gain of one spatial derivative from the multiplier \( \langle \nabla \rangle^{-1}\) in the Duhamel operator, and a gain of \( \beta \) derivatives from the convolution with the interaction potential.  In Proposition \ref{so3:prop}, however, we will see that  $\<3DN>$ actually has spatial regularity $\beta-$, which is half of a derivative better. The additional gain is a result of multi-linear dispersive effects. We now decompose our solution $u_N$ by writing
\begin{equation}
u_N = \<1b> +  \<3DN> + w_N.
\end{equation}
The remainder $w_N$ has initial data $w_N[0]=\reddotM$ and solves the forced nonlinear wave equation
\begin{align}
&(-\partial_t^2 - 1 + \Delta ) w_N  \notag \\
=& P_{\leq N} \bigg[ \lcol  \Big( V \ast \Big( P_{\leq N} \big( \<1b> + \<3DN> + w_N \big )\Big)^2 \Big) P_{\leq N} \big( \<1b> + \<3DN> + w_N \big) \rcol - \<3N>   \bigg] \notag  \allowdisplaybreaks[3]\\
=&  P_{\leq N} \bigg[ 2 \Big( V \ast \Big( P_{\leq N} \<1b> \cdot P_{\leq N} \big( \<3DN>+w_N \big) \Big)  P_{\leq N} \<1b> - \MN P_{\leq N} \big( \<3DN>+w_N \big) \Big) \label{lwp:eq_wN_1}\\
&\qquad+ \Big( V \ast \Big( P_{\leq N} \big( \<3DN>+w_N \big) \Big)^2 \Big) P_{\leq N} \<1b>\label{lwp:eq_wN_2} \allowdisplaybreaks[3] \\
&\qquad+  2 V \ast \Big( P_{\leq N} \<1b> \cdot  P_{\leq N} \big( \<3DN>+w_N \big) \Big)  P_{\leq N} \big( \<3DN>+w_N \big) \\
&\qquad+ \Big( V \ast \<2N> \Big) P_{\leq N} \big( \<3DN>+w_N \big) \\
&\qquad+\Big( V \ast \Big( P_{\leq N} \big(  \<3DN> + w_N \big )\Big)^2 \Big) P_{\leq N} \big(  \<3DN> + w_N \big) \bigg].
\end{align}
If we intend to construct (or control) $w_N$ via a ``direct" contraction argument, we would need the following conditions on the regularity of $w_N$ (uniformly in $N$): 
\begin{enumerate}
\item Due to the high$\times$high$\rightarrow$low-interactions in factors such as $P_{\leq N} \<1b> \cdot  P_{\leq N} w_N  $, the regularity of $w_N$ needs to be greater than $1/2$. 
\item Due to ``deterministic" nonlinear terms such as $(V\ast (P_{\leq N} w_N)^2) P_{\leq N} w_N$, the regularity of $w_N$ needs to be greater than or equal to the deterministic critical regularity, which is given by $1/2-\beta$. 
\end{enumerate}
Clearly, the first regularity condition is more restrictive. Unfortunately, the contribution of the first two summands \eqref{lwp:eq_wN_1} and \eqref{lwp:eq_wN_2} has regularity at most \(1/2-\). The low$\times$low$\times$high-interaction gains one derivative from the multiplier \( \langle \nabla \rangle^{-1}\) in the Duhamel operator, but does not benefit from the convolution with $V$ and does not experience any multi-linear dispersive effects. Thus, we are ``$\epsilon$-away" from a working contraction argument. As was observed in \cite{GIP15,GKO18a}, the term responsible for the low-regularity exhibits a para-controlled structure. Even though $P_{\leq N} \<1b> \cdot  P_{\leq N} w_N  $ is not well-defined for a general $w_N$ at spatial regularity $1/2-$, we will see in Proposition \ref{para:prop_quadratic_object} below that it is well-defined for a para-controlled $w_N$ at the same regularity! We therefore decompose the solution $w_N$ into two components: A para-controlled component $X_N$ at regularity $1/2-$ and a smoother nonlinear remainder $Y_N$ at a regularity greater than $1/2$. 

Before we can define the decomposition, we need to introduce our para-product operators.
\begin{definition}[Para-product operators] \label{lwp:def_paraproduct} 
Let $\epsilon>0$ be the fixed parameter from Section \ref{section:notation} and let  $f,g,h\colon \bT^3 \rightarrow \mathbb{R} $. We define the low$\times$high, high$\times$high, and high$\times$low-paraproducts by 
\begin{align*}
f \paral g \defe & \sum_{N_1 \leq N_2/8} P_{N_1} f \cdot P_{N_2} g, \\
f \paraeq g \defe & \sum_{N_2/4 \leq N_1 \leq 4N_2} P_{N_1} f \cdot P_{N_2} g, \\
f \parag g \defe & \sum_{N_1 \geq 8 N_2} P_{N_1} f \cdot P_{N_2} g. 
\end{align*}
We also define
\begin{equation*}
f \parageq g \defe f \parag g + f \paraeq g \qquad \text{and} \qquad f \paraleq g \defe f \paral g + f \paraeq g. 
\end{equation*}
In most of this paper, it  will be convenient to replace ``low'' frequencies by ``very low'' frequencies. To this end, we define the bilinear operator
\begin{equation}
f \parald g \defe \sum_{\substack{ N_1,N_2 \colon \\ N_1 \leq N_2^\epsilon}} P_{N_1} f \cdot P_{N_2} g
\end{equation}
and the trilinear operator 
\begin{equation}\label{lwp:eq_paraboxld}
\paraboxld \Big( V \ast (fg) h \Big) \defe \sum_{\substack{N_1,N_2,N_3\colon \\ N_1,N_2\leq N_3^\epsilon}}  V \ast \big( P_{N_1} f \cdot P_{N_2} g \big) P_{N_3} h. 
\end{equation}
Furthermore, we define the negations of $\parald$ and $\paraboxld$ by 
\begin{equation*}
f \nparald g \defe f g - f \parald g \quad \text{and} \quad \nparaboxld \Big( V \ast (fg) h \Big) \defe  V \ast (fg) h -  \paraboxld \Big( V \ast (fg) h \Big). 
\end{equation*}
\end{definition}

\begin{remark} The notation ``$\lessdot$" is seldom used in the mathematical literature, which is precisely the reason why we use it in Definition \ref{lwp:def_paraproduct}. Its meaning would otherwise easily be confused with projections to $N_1\leq N_2, N_1\lesssim N_2$, or $N_1 \ll N_2$, which are again more common, but less suitable in our situation than $N_1 \leq N_2^\epsilon$. 
Comparing our notation for the operators $\parald$ and $\paraboxld$, it may seem more natural to write 
\begin{equation*}
 V \ast (fg)  \paraboxld  h 
\end{equation*}
instead of \eqref{lwp:eq_paraboxld}. We found, however, that the notation in \eqref{lwp:eq_paraboxld} is much cleaner once it is combined with the stochastic objects. We also point out that the negation of $\parald$ is not $\paragd$. 
\end{remark}

We are now ready to define $X_N$ and $Y_N$. We define the para-controlled component $X_N$ by $X_N[0]=0$ and 
\begin{equation}\label{lwp:eq_XN}
\begin{aligned}
&(-\partial_t^2 -1 + \Delta)  X_N \\
&=  P_{\leq N}\bigg[ 2\, \paraboxld \Big(  V \ast \Big( P_{\leq N} \<1b> \cdot P_{\leq N} \big( \<3DN>+X_N \big) \Big)  P_{\leq N} \<1b> \Big)\\
&+ 2 \Big( V \ast \Big( P_{\leq N} \<1b> \cdot P_{\leq N} \big(Y_N \big) \Big)  \parald P_{\leq N} \<1b> \Big) \\
&+ \Big( V \ast \Big( P_{\leq N} \big( \<3DN>+w_N \big) \Big)^2 \Big) \parald P_{\leq N} \<1b>\bigg].
\end{aligned}
\end{equation}

\begin{remark}\label{local:rem_paracontrolled}
As far as the author is aware, the operator $\paraboxld$ has not been used in previous work on random dispersive equations. The reason for introducing the operator lies in the first term in \eqref{lwp:eq_XN}, which contains 
$ P_{\leq N} \<1b> \cdot P_{\leq N}X_N$.  In order to define this term (uniformly in $N$), the spatial regularity of $X_N$ alone is not sufficient. It is also difficult to use the structure of $X_N$, since this term appears in the evolution equation for $X_N$ (and not for $Y_N$), and hence one (may) run into a circular argument. By using $\paraboxld$, however, this problem does not occur, since we can borrow a small amount of regularity from the third argument in $ \paraboxld \big(  V \ast ( P_{\leq N} \<1b> \cdot P_{\leq N} X_N ) \,   P_{\leq N} \<1b>\big)$. We mention, however, that using $\paraboxld$ has a small drawback, which is explained in Remark \ref{ftg:rem_drawback}. 

We also did not include any component of \( \MN P_{\leq N} Y_N \) in the second term of \eqref{lwp:eq_XN}. It turns out that the contribution coming from the $\parald$-portion of the renormalization can be controlled at regularities bigger than \( 1/2 \) and is therefore placed in the evolution equation for $Y_N$ below. \\
\end{remark}

As determined by our choice of $X_N$, the nonlinear remainder $Y_N$ satisfies  $Y_N[0]=\reddotM$ and 
\begin{align}
&(-\partial_t^2 - 1 + \Delta)  Y_N  \notag \\
&= 2 P_{\leq N} \bigg[  \Big( \hspace{-1ex} \nparaboxld \hspace{-0.8ex} \Big( V \ast \Big( P_{\leq N} \<1b> \cdot P_{\leq N} \big( \<3DN>\hspace{-0.4ex}+ \hspace{-0.4ex} X_N \big) \Big)  P_{\leq N} \<1b> \Big) - \MN P_{\leq N} \big( \<3DN>+X_N \big) \Big) \bigg] \allowdisplaybreaks[3]\\
&+ P_{\leq N} \bigg[2  \Big( V \ast \Big( P_{\leq N} \<1b> \cdot P_{\leq N} \big(Y_N \big) \Big)  \nparald P_{\leq N} \<1b> \Big)- \MN P_{\leq N} \big( Y_N \big) \Big)\allowdisplaybreaks[3]\\
&+ \Big( V \ast \Big( P_{\leq N} \big( \<3DN>+w_N \big) \Big)^2 \Big) \nparald P_{\leq N} \<1b> \allowdisplaybreaks[3]\\
&+   2 V \ast \Big( P_{\leq N} \<1b> \cdot  P_{\leq N} \big( \<3DN>+w_N \big) \Big)  P_{\leq N} \big( \<3DN>+w_N \big) \\
&+ \Big( V \ast \<2N> \Big) P_{\leq N} \big( \<3DN>+w_N \big) \\
&+\Big( V \ast \Big( P_{\leq N} \big(  \<3DN> + w_N \big )\Big)^2 \Big) P_{\leq N} \big(  \<3DN> + w_N \big) \bigg].
\end{align}

To facilitate the analysis in the body of this paper, we further organize the terms in the evolution equation for $Y_N$. We write 
\begin{equation}
(-\partial_t^2 -1 + \Delta)  Y_N  = \So + \CPara + \RMT + \Phy,
\end{equation}
where the stochastic objects $\So$, the contributions of the para-controlled terms $\CPara$, the random-matrix terms $\RMT$, and the physical terms $\Phy$ are defined as follows: \\

We define the individual stochastic objects by
\begin{align}
\<131N> &\defe    \Big( V \ast \Big( P_{\leq N} \<1b> \cdot P_{\leq N} \big( \<3DN> \big) \Big)  P_{\leq N} \<1b> \Big) - \MN P_{\leq N} \big( \<3DN> \big) \Big), \\
\nparaboxld \<131N> &\defe   \nparaboxld  \Big( V \ast \Big( P_{\leq N} \<1b> \cdot P_{\leq N} \big( \<3DN> \big) \Big)  P_{\leq N} \<1b> \Big) - \MN P_{\leq N} \big( \<3DN> \big) \Big), \label{local:eq_so_term2} \\
\<113N> &\defe \Big( V \ast \<2N> \Big) P_{\leq N} \<3DN>, \\
\<313N> &\defe  V \ast \Big(  P_{\leq N} \<3DN>  \cdot P_{\leq N} \<1b> \Big)  P_{\leq N}  \<3DN>,\\
\<331N> &\defe  \Big( V \ast \Big( P_{\leq N} \<3DN>\Big)^2   \Big) P_{\leq N} \<1b>, \\
 \nparald \<331N> &\defe  V \ast \Big( P_{\leq N} \<3DN>\Big)^2 \nparald P_{\leq N} \<1b>.
\end{align}
We then define 
\begin{equation}\label{lwp:eq_so}
\So = \So_N\defe P_{\leq N} \bigg[ 2 \nparaboxld \<131N> + \<113N> +  \nparald \<331N> + 2 \<313N> \bigg].
\end{equation}
In works on singular SPDEs, such as \cite{MWX17}, the para-differential operators are usually placed at the joints of the different lines. The advantage is that it works for arbitrary ``trees'' and can accommodate multiple para-differential operators. Since this level of generality will not be needed here, we prefer our notation, since it is slightly easier to read. \\

We define 
\begin{equation}\label{lwp:eq_CPara}
\begin{aligned}
\CPara &= \CPara_N(X_N,w_N)\\
&\defe 2 P_{\leq N} \bigg[  \Big( \hspace{-1ex} \nparaboxld \hspace{-0.8ex} \Big( V \ast \Big( P_{\leq N} \<1b> \cdot P_{\leq N}   X_N  \Big)  P_{\leq N} \<1b> \Big) - \MN P_{\leq N} X_N  \Big) \bigg]\\
&+   2 P_{\leq N} \bigg[  V \ast \Big( P_{\leq N} \<1b> \paraeq  P_{\leq N} X_N \Big)  P_{\leq N} w_N  \bigg] \\
&+   2 P_{\leq N} \bigg[  V \ast \Big( P_{\leq N} \<1b> \paraeq  P_{\leq N} X_N \Big)  P_{\leq N}  \<3DN> \bigg]. 
\end{aligned}
\end{equation}
In our analysis of $\CPara$, we will use the double Duhamel-trick, i.e., we will replace $X_N$ by the Duhamel-integral of the right-hand side in \eqref{lwp:eq_XN}. \\
The random matrix term is defined as 
\begin{align}\label{lwp:eq_RMT}
\RMT &=\RMT_N(Y_N,w_N) \notag \\
&\defe  P_{\leq N} \bigg[  \Big( V \ast \<2N> \Big) P_{\leq N} w_N  \bigg]\\
&+ 2 P_{\leq N} \bigg[   \Big( V \ast \Big( P_{\leq N} \<1b> \cdot P_{\leq N} \big(Y_N \big) \Big)  \nparald P_{\leq N} \<1b> \Big)- \MN P_{\leq N} \big( Y_N \big) \Big) \bigg]  .
\end{align}
Our reason for calling \eqref{lwp:eq_RMT} the random matrix term lies in the method used in its estimate. We will view the summands as random operators in $w_N$ and $Y_N$, respectively, and estimate the operator norm using the moment method (as in \cite[Proposition 2.8]{DNY20}). \\
Finally, we define the physical term by 
\begin{align}
\Phy = &\Phy_N(X_N,Y_N,w_N) \notag \\
&\defe P_{\leq N} \bigg[   2 V \ast \Big( P_{\leq N} \<1b> \cdot  P_{\leq N}  \<3DN>\Big)  P_{\leq N} w_N \\
&+ 2  \Big( V \ast \Big( P_{\leq N} \<3DN> \cdot P_{\leq N} w_N \Big) \nparald P_{\leq N} \<1b> \label{lwp:eq_phy_2} \\
&+   2 V \ast \Big( P_{\leq N} \<1b> \paraneq  P_{\leq N} w_N\Big)  P_{\leq N} \<3DN> \label{lwp:eq_phy_3} \\
&+   2 V \ast \Big( P_{\leq N} \<1b> \paraeq  P_{\leq N} Y_N\Big)  P_{\leq N} \<3DN>  \\
&+2 V \ast \Big( P_{\leq N} \<1b> \paraneq  P_{\leq N}  w_N  \Big)  P_{\leq N} w_N \\
&+2 V \ast \Big( P_{\leq N} \<1b> \paraeq  P_{\leq N}  Y_N  \Big)  P_{\leq N} w_N \\
&+ \Big( V \ast \Big( P_{\leq N} w_N  \Big)^2 \Big) \nparald P_{\leq N} \<1b> \\
&+\Big( V \ast \Big( P_{\leq N} \big(  \<3DN> + w_N \big )\Big)^2 \Big) P_{\leq N} \big(  \<3DN> + w_N \big) \bigg].
\end{align}
Similar as for $\RMT$, we call $\Phy$ the physical term due to the methods used in its estimate. We point out, however, that \eqref{lwp:eq_phy_2} and \eqref{lwp:eq_phy_3} are ``hybrid" terms and their estimates rely on both random matrix techniques and physical methods. In the estimates of the other terms in $\Phy$, we also make use of the refined Strichartz estimates by Klainerman-Tataru \cite{KT99}. \\

\subsection{Multi-linear master estimate}\label{section:master_estimate}
In this subsection, we combine all multi-linear estimates from Section \ref{section:stochastic_object}-\ref{section:physical} into a single proposition, which we refer to as the multi-linear master estimate (Proposition \ref{local:prop_master}). In particular, the multi-linear master estimate will include estimates of $\So$, $\CPara$, $\RMT$, and $\Phy$, even though the proofs of the individual estimates are quite different. Before we can state the multi-linear master estimate, however,  we require additional notation. For the definition of the function spaces $\X{s}{b}$ and $\LM$, we refer to Definition \ref{tools:def_xsb} and Definition \ref{para:def_LM}.

\begin{definition}[Types]\label{local:def_types}
Let $\cJ \subseteq [0,\infty)$ be a bounded interval and let $\varphi \colon J \times \bT^3 \rightarrow \bR$. We say that $\varphi$ is of type
\begin{itemize}
\setlength\itemsep{1ex}
\item $\,\<1b>\,$ if $\varphi=\<1b>\,$, 
\item $\, \<3D>\,$ if $\varphi=\<3DN>\,$ for some $N\geq 1$, 
\item $w$ if $\| \varphi \|_{\X{s_1}{b}(\cJ)}\leq 1$ and  $ \sum_{L_1\sim L_2} \| P_{L_1} \<1b> \cdot P_{L_2} w \|_{L_t^2 H_x^{-4\delta_1}(\cJ\times \bT^3)} \leq 1$, 
\item $X$ if $\varphi = P_{\leq N} \Duh \big[1_{\cJ_0} \PCtrl(H, P_{\leq N} \, \<1b>\, )\big] $ for a dyadic integer $N\geq 1$, a sub-interval $\cJ_0\subseteq \cJ$,  and a function $H\in \LM(\cJ_0)$ satisfying $\| H \|_{\LM(\cJ_0)}\leq 1$, 
\item $Y$ if $\| \varphi \|_{\X{s_2}{b}(\cJ)} \leq 1$. 
\end{itemize}
Let $\varphi_1,\varphi_2,\varphi_3 \colon \cJ \times \bT^3 \rightarrow \bR$ and $\cT_1,\cT_2 ,\cT_3 \in \Big\{ \,\<1b>, \, \<3D>, w , X ,Y \Big\}$. We write
\begin{equation*}
(\varphi_1,\varphi_2) \etype (\cT_1,\cT_2)
\end{equation*}
if either $\varphi_1$ is of type $\cT_1$ and $\varphi_2$ is of type $\cT_2$ or $\varphi_1$ is of type $\cT_2$ and $\varphi_2$ is of type $\cT_1$. Furthermore, we write 
\begin{equation*}
(\varphi_1,\varphi_2;\varphi_3) \etype (\cT_1,\cT_2;\cT_3)
\end{equation*}
if $(\varphi_1,\varphi_2) \etype (\cT_1,\cT_2)$ and $\varphi_3$ is of type $\cT_3$. 
\end{definition}
\begin{remark} \label{local:rem_types}
The types $w$, $X$, and $Y$ are designed for the functions $w_N$, $X_N$, and $Y_N$ from Section \ref{section:ansatz}. Our notation for the type of $(\varphi_1,\varphi_2;\varphi_3)$ respects the symmetry in the first two arguments of the nonlinearity $\big(V \ast ( \varphi_1 \varphi_2 ) \big) \varphi_3$. We also mention that the types $w$ and $X$ implicitly depend on $\,\bluedot\,$. In Section \ref{section:ftg}, we will therefore refer to the types $w$ and $X$ as $\wblue$ and $\Xblue$, respectively. 
\end{remark}
 In the next lemma, we show that functions of type $X$ and $Y$ are multiples of functions of type $w$. This allows us to prove several estimates for functions of type $X$ and $Y$ simultaneously. 

\begin{lemma}\label{local:lem_type_conversion}
Let $A\geq 1$, $T\geq 1$, and let $\zeta=\zeta(\epsilon,s_1,s_2,\kappa,\eta,\eta^\prime,b_+,b)>0$ be sufficiently small. Then, there exists a Borel set $\Theta_{\type}(A,T)\subseteq \cH_x^{-\frac{1}{2}-\kappa}(\bT^3)$ satisfying 
\begin{equation*}
\bP(\bluedot \in \Theta^{\type}_{\blue}(A,T)) \geq 1 - \zeta^{-1} \exp(-\zeta A^\zeta)
\end{equation*}
and such that the following holds for all $\, \bluedot \in \Theta^{\type}_{\blue}(A,T)$: If $\varphi\colon  \cJ \times \bT^3 \rightarrow \bR$ is of type $X$ or $Y$, then $T^{-4}A^{-1} \varphi$ is of type $w$. 
\end{lemma}

\begin{proof}
We treat the types $X$ and $Y$ separately. First, we assume that $\varphi$ is of type $X$, and hence there exists a dyadic integer $N\geq 1$, a sub-interval $\cJ_0\subseteq \cJ$, and a function  $H\in \LM(\cJ_0)$ satisfying $\| H\|_{\LM(\cJ_0)}\leq 1$ such that $\varphi= P_{\leq N} \Duh \big[ 1_{\cJ_0}\PCtrl(H, P_{\leq N} \<1b>)\big]$. Using the inhomogeneous Strichartz estimate (Lemma \ref{tools:lem_inhomogeneous_strichartz}) and Lemma \ref{para:lem_basic}, we have that 
\begin{align*}
\| P_{\leq N} X \|_{\X{s_1}{b}(\cJ)} &\lesssim \| 1_{\cJ_0} \PCtrl( H, \, P_{\leq N} \<1b>\,)\|_{L_t^{2b} H_x^{s_1-1}(\cJ\times \bT^3)} \\
&\lesssim T \| H\|_{\LM(\cJ)} \|  \,  \<1b> \, \|_{L_t^\infty H_x^{s_1-1+8\epsilon}(\cJ\times \bT^3)} \\
&\lesssim T \| \, \bluedot \, \|_{\cH_x^{-1/2-\kappa}(\bT^3)}. 
\end{align*}
This is bounded by $TA$ on a set of acceptable probability. Using Proposition \ref{para:prop_quadratic_object}, we obtain on a set of acceptable probability that 
\begin{equation*}
 \sum_{L_1 \sim L_2} \| P_{L_1} \<1b> \cdot P_{L_2} \varphi \|_{L_t^2 H_x^{-4\delta_1}(\cJ\times \bT^3)}\leq T^4 A  \| H\|_{\LM(\cJ_0)}  \leq T^4 A. 
\end{equation*}
By combining both estimates, we see that $T^{-4} A^{-1} \varphi$ is of type $w$. \\ 

Second, we assume that $\varphi$ is of type $Y$. Then, we have that $\| \varphi \|_{\X{s_1}{b}(\cJ)}\leq \| \varphi \|_{\X{s_2}{b}(\cJ)} \leq 1$. 
This implies
\begin{align*}
\sum_{L_1\sim L_2} \| P_{L_1} \<1b> \cdot P_{L_2} \varphi \|_{L_t^2 H_x^{-4\delta_1}(\cJ\times \bT^3)} 
&\lesssim T^{\frac{1}{2}} \sum_{L_1\sim L_2} L_1^{\kappa-\delta_2} \| P_{L_1}\<1b> \,  \|_{L_t^\infty \cC_x^{-1/2-\kappa}(\cJ\times \bT^3)}
 \| P_{L_2} \varphi \|_{L_t^\infty H_x^{s_2}(\cJ\times \bT^3)} \\
 &\lesssim T^{\frac{1}{2}}  \| \, \<1b> \, \|_{L_t^\infty \cC_x^{-1/2-\kappa}(\cJ\times \bT^3)}
\end{align*}

As above, this is bounded by $T^{\frac{1}{2}} A$ on a set of acceptable probability. By combining both estimates, we see that $T^{-\frac{1}{2}} A^{-1} \varphi$ is of type $w$. 
\end{proof}

In order to state the multi-linear master estimate, we need to introduce a multi-linear version of the renormalization in \eqref{intro:eq_renormalized_nonlinearity}. 

\begin{definition}[Renormalization]\label{local:def_renormalization}
Let $\cJ$ be a compact interval, let $\varphi_1,\varphi_2,\varphi_3$ be as in Definition \ref{local:def_types}, and let $N\geq 1$. Furthermore, assume that 
\begin{equation*}
(\varphi_1,\varphi_2;\varphi_3) \netype \big(\, \<1b>\,, \,\<1b>\,,\,\<1b>\, \big). 
\end{equation*}
Then, we define the renormalized and frequency-truncated nonlinearity by 
\begin{equation}
\begin{aligned}
&\lcol  V \ast \big( P_{\leq N} \varphi_1 \cdot P_{\leq N} \varphi_2 \big)  P_{\leq N} \varphi_3  \rcol \\
&\defe
\begin{cases}
\begin{tabular}{ll}
$\Big( V \ast \<2N> \Big) P_{\leq N} \varphi_3$ &\hspace{2ex} if $(\varphi_1,\varphi_2)\etype \big( \, \<1b>\,, \, \<1b>\,\big)$, \\
$ V \ast\big( P_{\leq N} \<1b> \cdot P_{\leq N} \varphi_2  \Big) P_{\leq N} \, \<1b> - \MN P_{\leq N} \varphi_2 $ &\hspace{2ex} if $(\varphi_1,\varphi_3) \etype \big( \, \<1b>\,, \, \<1b>\,\big)$, \\
$ V \ast\big(  P_{\leq N} \varphi_1 \cdot  P_{\leq N} \<1b> \Big) P_{\leq N} \, \<1b> - \MN P_{\leq N} \varphi_1 $ &\hspace{2ex} if $(\varphi_2,\varphi_3)\etype \big( \, \<1b>\,, \, \<1b>\,\big)$, \\
$ V \ast \big( P_{\leq N} \varphi_1 \cdot P_{\leq N} \varphi_2 \big)  P_{\leq N} \varphi_3 $   &\hspace{2ex} else. 
\end{tabular}
\end{cases}
\end{aligned}
\end{equation}
If  $(\varphi_1,\varphi_2)\netype \big( \, \<1b>\,, \, \<1b>\,\big)$, we define the action of the paradifferential operators $\parald$ and $\paraboxld$ on the renormalized and frequency-truncated nonlinearity by
\begin{align*}
\lcol  V \ast \big( P_{\leq N} \varphi_1 \cdot P_{\leq N} \varphi_2 \big) \parald P_{\leq N} \varphi_3  \rcol 
&\defe V \ast \big( P_{\leq N} \varphi_1 \cdot P_{\leq N} \varphi_2 \big) \parald P_{\leq N} \varphi_3 , \\ 
\paraboxld \Big( \lcol  V \ast \big( P_{\leq N} \varphi_1 \cdot P_{\leq N} \varphi_2 \big) P_{\leq N} \varphi_3  \rcol \Big) 
&\defe \paraboxld \Big(   V \ast \big( P_{\leq N} \varphi_1 \cdot P_{\leq N} \varphi_2 \big)  P_{\leq N} \varphi_3  \Big),
\end{align*}
which does not involve a renormalization. We also define the negated paradifferential operators by 
\begin{align*}
&\lcol  V \ast \big( P_{\leq N} \varphi_1 \cdot P_{\leq N} \varphi_2 \big) \nparald P_{\leq N} \varphi_3  \rcol \\
&\defe \lcol  V \ast \big( P_{\leq N} \varphi_1 \cdot P_{\leq N} \varphi_2 \big)  P_{\leq N} \varphi_3  \rcol  
-\lcol  V \ast \big( P_{\leq N} \varphi_1 \cdot P_{\leq N} \varphi_2 \big) \parald P_{\leq N} \varphi_3  \rcol , \\ 
&\nparaboxld \Big( \lcol  V \ast \big( P_{\leq N} \varphi_1 \cdot P_{\leq N} \varphi_2 \big) P_{\leq N} \varphi_3  \rcol \Big) \\
&\defe V \ast \big( P_{\leq N} \varphi_1 \cdot P_{\leq N} \varphi_2 \big)  P_{\leq N} \varphi_3  \rcol  
 -\paraboxld \Big( \lcol  V \ast \big( P_{\leq N} \varphi_1 \cdot P_{\leq N} \varphi_2 \big) P_{\leq N} \varphi_3  \rcol \Big) ,
\end{align*}
which contains the full renormalization. 
\end{definition}

Equipped with our notion of types and the renormalization, we can now state and prove the multi-linear master estimate.

\begin{proposition}[Multi-linear master estimate]\label{local:prop_master}
Let  $\zeta=\zeta(\epsilon,s_1,s_2,\kappa,\eta,\eta^\prime,b_+,b)>0$ be sufficiently small, let $A\geq 1$, and let $T\geq 1$. Then, there exists a Borel set $\Theta_{\blue}^{\ms}(A,T) \subseteq \cH_x^{-1/2-\kappa}(\bT^3)$ satisfying 
\begin{equation}
\bP(\, \bluedot \in \Theta_{\blue}^{\ms}(A,T) ) \geq 1 - \zeta^{-1} \exp(-\zeta A^{\zeta}) 
\end{equation}
and such that for all $\, \bluedot \in \Theta_{\blue}^{\ms}(A,T)$ the following estimates hold: 

Let $\cJ\subseteq [0,T]$ be an interval and let $N\geq 1$. Let $\varphi_1,\varphi_2,\varphi_3\colon \cJ\times \bT^3 \rightarrow \bR$ be as in Definition 
\ref{local:def_types} and let 
\begin{equation*}
(\varphi_1,\varphi_2;\varphi_3) \netype \big(\, \<1b>\,, \,\<1b>\,;\,\<1b>\, \big), \big( \,\<1b>\, , w ; \, \<1b>\,\big).
\end{equation*}
\begin{enumerate}[(i)]
\item \label{local:item_master_1} If $(\varphi_1,\varphi_2;\varphi_3) \etype \big(\, \<1b>\,, \,\<3D>\,;\,\<1b>\, \big), ~  \big(\, \<1b>\,  ,X;\,\<1b>\, \big)$, then
\begin{equation*}
\Big\| \nparaboxld \Big( \lcol  V \ast \big( P_{\leq N} \varphi_1 \cdot P_{\leq N} \varphi_2 \big) P_{\leq N} \varphi_3  \rcol \Big)  \Big\|_{\X{s_2-1}{b_+-1}(\cJ)}
\leq T^{30} A. 
\end{equation*}
\item \label{local:item_master_2} If $(\varphi_1,\varphi_2;\varphi_3) \etype \big(\, \<1b>\,  ,Y;\,\<1b>\, \big)$ or $\varphi_1,\varphi_2 \netype \, \<1b>$ and $\varphi_3 \etype \, \<1b>$, then 
\begin{equation*}
\Big\|\lcol  V \ast \big( P_{\leq N} \varphi_1 \cdot P_{\leq N} \varphi_2 \big) \nparald P_{\leq N} \varphi_3  \rcol  \Big\|_{\X{s_2-1}{b_+-1}(\cJ)}
\leq T^{30} A. 
\end{equation*}
\item \label{local:item_master_3} In all other cases, 
\begin{equation*}
\Big\|\lcol  V \ast \big( P_{\leq N} \varphi_1 \cdot P_{\leq N} \varphi_2 \big)  P_{\leq N} \varphi_3  \rcol  \Big\|_{\X{s_2-1}{b_+-1}(\cJ)}
\leq T^{30} A. 
\end{equation*}
\end{enumerate}
\end{proposition}

\begin{remark}\label{local:remark_master}
The frequency-localized versions of each estimate in Proposition \ref{local:prop_master} gain an $\eta^\prime$-power in the maximal frequency-scale. Furthermore, functions of the type $\<3D>$ can be replaced by  $\<3Dtau>$ as defined in \eqref{global:eq_modified_cubic_objects}. For more details on these minor modifications, we refer the reader to the proof of the individual main estimates (Section \ref{section:stochastic_object}-\ref{section:physical}). 
\end{remark}

\begin{proof}
It suffices to prove the estimates with $A$ on the right-hand side replaced by $CA^C$, where $C=C(s_1,s_2,b,b_+,\epsilon)$. Then, the desired estimate follows by replacing $A$ with a small power of itself and adjusting the constant $\zeta$. In the following, we freely restrict to events with acceptable probabilities.

\emph{Proof of \eqref{local:item_master_1}:}  
 If $(\varphi_1,\varphi_2;\varphi_3)$ has type
\begin{itemize}
\item[$\bullet$] $\big( \, \<1b> \,  , \<3D> ; \, \<1b>\,    \big)$, we use Proposition \ref{so5i:prop}, 
\item[$\bullet$] $\big( \, \<1b> \,  , X ; \, \<1b>\,    \big)$, we use Proposition \ref{para:prop_cubic} . 
\end{itemize}

\emph{Proof of \eqref{local:item_master_2}:} 
If $(\varphi_1,\varphi_2;\varphi_3) \etype \big(\, \<1b>\,  ,Y,\,\<1b>\, \big)$, this follows from  Proposition \ref{rmt:prop2}. Using Lemma \ref{local:lem_type_conversion}, we may assume in all remaining cases that $\varphi_1$ and $\varphi_2$ have type $\<3D> $ or $w$, as long as we obtain the estimate with $T^{18}$ instead of $T^{30}$. If $(\varphi_1,\varphi_2;\varphi_3)$ has type 
\begin{itemize}
\item[$\bullet$] $\big(  \<3D>  , \<3D> ;  \, \<1b> \,   \big)$, we use Lemma \ref{para:lemma_obj_1} and Proposition \ref{so7:prop},
\item[$\bullet$] $\big(    w  ,    \<3D> ; \, \<1b> \,   \big)$, we use  Lemma \ref{para:lemma_obj_2} and Proposition \ref{phy:prop_hybrid},
\item[$\bullet$] $\big(  w , w ;  \, \<1b> \,     \big)$, we use  Lemma \ref{para:lemma_obj_2} and Proposition \ref{phy:prop_1b}. 
\end{itemize}

\emph{Proof of \eqref{local:item_master_3}:} Using Lemma \ref{local:lem_type_conversion}, we may assume that all functions $\varphi_j$ are of type 
$\,\<1b>\,$, $\,\<3D>\,$, or $w$, as long as we prove the estimate with $T^{18}$ instead of $T^{30}$. If no factor is of type $\,\<1b>\,$, the desired estimate follows from Proposition \ref{so3:prop} and Proposition \ref{phy:prop_3b}. The remaining cases can be estimated as follows: If $(\varphi_1,\varphi_2;\varphi_3)$ has type 
\begin{itemize}
  \setlength\itemsep{1ex}
\item[$\bullet$] $\big( \, \<1b> \,   , \, \<1b>\,    ; \<3D>\big)$, we use  Proposition \ref{so5ii:prop},
\item[$\bullet$] $\big( \, \<1b> \,   , \, \<1b>\,    ; w \big)$, we use Proposition \ref{rmt:prop1}, 
\item[$\bullet$] $\big( \, \<1b> \,  , \<3D> ;  \<3D>    \big)$, we use Proposition \ref{so7:prop},
\item[$\bullet$] $\big( \, \<1b> \,  , \<3D> ; w     \big)$, we use Proposition \ref{phy:prop_hybrid}, 
\item[$\bullet$] $\big( \, \<1b> \,  , w  ;    \<3D>   \big)$, we use	 Lemma \ref{phy:lem_bilinear_tool} and Proposition \ref{phy:prop_hybrid},
\item[$\bullet$] $\big( \, \<1b> \,  , w ; w     \big)$, we use   Proposition \ref{phy:prop_neq} and Lemma \ref{phy:lem_bilinear_tool}. 
\end{itemize}
\end{proof}

\subsection{Local well-posedness}

In this subsection, we obtain our first local well-posedness result. It is phrased in terms of the ambient measure $\bP$ and the random structure is based on the Gaussian initial data $\, \bluedot$. 

\begin{proposition}[Structured local well-posedness w.r.t. the ambient measure]\label{lwp:prop}
Let $M\geq 1$, let $A\geq 1$, let $0<\tau\leq 1$, and  let $\zeta=\zeta(\epsilon,s_1,s_2,\kappa,\eta,\eta^\prime,b_+,b)>0$ be sufficiently small. Denote by $\LdagM$ the event in the ambient space $(\Omega,\mathcal{F})$ defined by the following conditions:
\begin{enumerate}[(i)]
\item \label{lwp:item_P1} For any $N\geq 1$, the solution of \eqref{eq:nlw_N} with initial data $\purpledot= \bluedot + \reddotM$ exists on $[0,\tau]$. 
\item \label{lwp:item_P2} For all  $N\geq 1$, there exist $w_N\in \X{s_1}{b}([0,\tau])$, $H_N \in \modulation([0,\tau])$, and $Y_N \in \X{s_2}{b}([0,\tau])$ such that
\begin{equation*}
\Phi_N(t)\,  \purpledot = \<1b>(t) + \<3DN>(t)+ w_N(t)  \quad \text{and} \quad  w_N(t) = P_{\leq N} \Duh\big[ \Para(H_N,  P_{\leq N} \<1b>)\big](t) + Y_N(t)
\end{equation*}
for all $t\in [0,\tau]$. Furthermore, we have the bounds
\begin{align*}
   \| w_N \|_{\X{s_1}{b}([0,\tau])}, \| H_N \|_{\modulation([0,\tau])}, \| Y_N \|_{\X{s_2}{b}([0,\tau])}  &\leq A \quad \text{and} \\
\sum_{L_1 \sim L_2} \| P_{L_1} \<1b> \cdot P_{L_2} w_N \|_{L_t^2 H_x^{-4\delta_1}([0,\tau]\times \bT^3)} &\leq A. 
\end{align*}
\item\label{lwp:item_P3} It holds for all \( N,K \geq 1 \) that 
\begin{equation*}
 \| \Phi_N[t]\,  \purpledot - \Phi_K[t]\,  \purpledot\, \|_{L_t^\infty \cH_x^{\beta-\kappa}([0,\tau]\times \bT^3)} \leq  A \min(N,K)^{-\eta^\prime}. 
\end{equation*}
We further require that 
\begin{equation*}
\| H_N - H_K \|_{\LM([0,\tau])}, ~
\| Y_N - Y_K \|_{\X{s_2}{b}([0,\tau])} \leq   A \min(N,K)^{-\eta^\prime}. 
\end{equation*}
\end{enumerate}
If $A \tau^{b_+-b} \leq 1 $, then \( \LdagM \) has a high probability and it holds that 
\begin{equation}\label{lwp:eq_high_probability}
\bP(\LdagM) \geq 1 - \zeta^{-1} \exp(-\zeta  A^{\zeta}). 
\end{equation}
\end{proposition}

\begin{remark}\label{lwp:remark_lwp}
The superscript ``$\text{amb}$'' in $\LdagM$ emphasizes that the event lives in the ambient probability space. The first item \eqref{lwp:item_P1} is only stated for expository purposes. Indeed, since \eqref{lwp:item_P1} is a soft statement and does not contain any uniformity in the frequency-truncation parameter, it follows from the global well-posedness of \eqref{eq:nlw_N} (which is also not uniform in $N$). The interesting portions of the proposition are included in \eqref{lwp:item_P2} and \eqref{lwp:item_P3}, which contain uniform structural information about the solution and allow us to locally define the limiting dynamics. 
\end{remark}

By combining Theorem \ref{theorem:measures} and  Proposition \ref{lwp:prop}, we easily obtain the local well-posedness of the renormalized nonlinear wave equation on the support of the Gibbs measure.

\begin{corollary}[Local well-posedness for Gibbsian initial data]\label{lwp:cor_unstructured}
Let $0<\tau<1$ and let $\zeta=\zeta(\epsilon,s_1,s_2,\kappa,\eta,\eta^\prime,b_+,b)>0$ be sufficiently small. Then, there exists a Borel set $\cL(\tau)\subseteq \cH_x^{-1/2-\kappa}(\bT^3)$ such that $\Phi_N[t] \, \purpledot$ converges in $C_t^0 \cH_x^{-1/2-\kappa}([0,\tau]\times \bT^3)$ as $N\rightarrow \infty$ and such that 
\begin{equation}\label{lwp:eq_unstructured}
\mup_\infty(\cL(\tau)) \geq 1 - \zeta^{-1} \exp(-\zeta \tau^{-\zeta}).  
\end{equation}
\end{corollary}

Corollary \ref{lwp:cor_unstructured} shows that the limiting dynamics $\Phi(t)\,\purpledot = \lim_{N\rightarrow \infty} \Phi_N(t)\, \purpledot$ are locally well-defined on the support of the Gibbs measure. However, it does not contain any structural information about the solution, which will be essential in the globalization argument (Section \ref{section:global}). The main difficulty, which was described in detail in Section \ref{section:ansatz}, is that the free component of the initial data $\, \bluedot\,$ is only defined on the ambient space. Nevertheless, in Proposition \ref{global:prop_lwp} below, we obtain a structured local well-posedness theorem in terms of $\purpledot$.

We first use the structured local well-posedness result for the ambient measure (Proposition \ref{lwp:prop}) to prove the unstructured local well-posedness for Gibbsian random data (Corollary \ref{lwp:cor_unstructured}). Then, we present the proof of Proposition \ref{lwp:prop}.

\begin{proof}[Proof of Corollary \ref{lwp:cor_unstructured}:]
Let $M\geq 1$ and let $A$ satisfy $A \tau^{b+-b} \leq 1$. We define a closed set $\LAtautil \subseteq \cH_{x}^{-\frac{1}{2}-\kappa}(\bT^3)$ by requiring that $\purpledot \, \in \LAtautil$ if and only if
\begin{enumerate}[(a)]
\item  For any $N\geq 1$, the solution of \eqref{eq:nlw_N} with initial data $\purpledot$ exists on $[0,\tau]$. 
\item  It holds for all \( N,K \geq 1 \) that 
\begin{equation*}
\| \Phi_N(t)\,  \purpledot - \Phi_K(t)\,  \purpledot \|_{L_t^\infty \cH_x^{\beta-\kappa}([0,\tau]\times \bT^3)} \leq  A \min(N,K)^{-\eta^\prime}. 
\end{equation*}
\end{enumerate}
It is clear from the definition that $\cL(\tau) \subseteq \LAtautil $. We emphasize that $\LAtautil$ is defined intrinsically through $\, \purpledot\,$ and does not refer to the ambient probability space $(\Omega,\mathcal{F},\bP)$. From the definition of $\LdagM$ in Proposition \ref{lwp:prop}, it follows that 
\begin{equation*}
\LdagM \subseteq \{ \, \bluedot + \reddotM \in \LAtautil  \} .
\end{equation*}
By using the representation of the reference measure in Theorem \ref{theorem:measures}, we have that $\Law_{\bP}( \, \bluedot + \reddotM\,)= \nup_M$. This yields
\begin{equation*}
\nup_M(\LAtautil ) = \bP(  \, \bluedot + \reddotM \in \LAtautil  ) \geq \bP( \LdagM) \geq 1 - \zeta^{-1} \exp(-\zeta A^{\zeta}). 
\end{equation*}
By using the quantitative version of the absolute continuity $\mup_M \ll  \nup_M$ in Theorem \ref{theorem:measures}, we obtain that 
\begin{equation*}
\mup_M( \cH_x^{-\frac{1}{2}-\kappa}(\bT^3)\backslash \LAtautil ) \lesssim \nup_M( \cH_x^{-\frac{1}{2}-\kappa}(\bT^3)\backslash \LAtautil)^{1-\frac{1}{q}} \lesssim   \zeta^{-1} \exp\big(-\zeta  \big( 1 - q^{-1}\big) A^{\zeta}\big).
\end{equation*}
After adjusting the value of $\zeta$, this yields the desired estimate \eqref{lwp:eq_unstructured} with $\mup_\infty$ replaced by $\mup_M$. 
Since $\LAtautil$ is closed in $\cH_x^{-1/2-\kappa}(\bT^3)$ and a subsequence of $\mup_M$ weakly converges to $\mup_\infty$, we obtain the same probabilistic estimate for the limiting measure $\mup_\infty$. 
\end{proof}

\begin{proof}[Proof of Proposition \ref{lwp:prop}:] 
As discussed in Remark \ref{lwp:remark_lwp}, \eqref{lwp:item_P1} follows from a soft argument. We now turn to the proof of \eqref{lwp:item_P2}, which is the heart of the proposition. We let $B=c A^{c}$, where $c=c(\epsilon,s_1,s_2,b_+,b)$ is a sufficiently small constant.  

Using Theorem \ref{theorem:measures}, Lemma \ref{local:lem_type_conversion}, Proposition \ref{local:prop_master}, and Proposition \ref{so3:prop}, we may restrict to the event 
\begin{equation}\label{local:eq_lwp_event_restriction}
\begin{aligned}
&\Big\{ \, \bluedot \in \Theta_{\blue}^{\ms}(B,1) \Big\} \, \medcap  \,  \Big\{ \, \bluedot \in \Theta^{\type}_{\blue}(B,1)\Big\} \, \medcap  \,  \Big\{ \big\| \, \<1b>\, \big\|_{L_t^\infty \cC_x^{-1/2-\kappa}([0,1]\times \bT^3)}\leq B \Big \} \,
  \\
&\medcap  \, \Big\{ \sup_{N} \big\| \, \<3DN>\, \big\|_{L_t^\infty \cC_x^{\beta-\kappa}([0,1]\times \bT^3)}\leq B \Big \} \, \medcap \, \Big\{ \big\| \, \reddotM  \big\|_{\cH_x^{1/2+\beta-\kappa}(\bT^3)} \leq B \Big\}. 
\end{aligned}
\end{equation}
We now define a map
\begin{equation*}
\Gamma_N = (\Gamma_{N,X}, \Gamma_{N,Y}) \colon \X{s_1}{b}([0,\tau]) \times \X{s_2}{b}([0,\tau]) \rightarrow  \X{s_1}{b}([0,\tau]) \times \X{s_2}{b}([0,\tau])  
\end{equation*}
by 
\begin{align*}
\Gamma_{N,X}(X_N,Y_N) 
&\defe  P_{\leq N} \Duh\bigg[ 2\, \paraboxld \Big(  V \ast \Big( P_{\leq N} \<1b> \cdot P_{\leq N} \big( \<3DN>+X_N\big) \Big)  P_{\leq N} \<1b> \Big)\\
&+ 2 \Big( V \ast \Big( P_{\leq N} \<1b> \cdot P_{\leq N} \big(Y_N \big) \Big)  \parald P_{\leq N} \<1b> \Big) 
+ \Big( V \ast \Big( P_{\leq N} \big( \<3DN>+w_N \big) \Big)^2 \Big) \parald P_{\leq N} \<1b>\bigg]
\end{align*}
and
\begin{align*}
&\Gamma_{N,Y}(X_N,Y_N) \\
&\defe \, \<1r>\,+ \Duh\big[\So_N + \CPara_N(\Gamma_{N,X}(X_N,Y_N),w_N) + \RMT_N(Y_N,w_N) + \Phy_N(X_N,Y_N,w_N)\big],
\end{align*}
where $w_N=X_N+Y_N$. We emphasize our use of the double Duhamel trick, which is manifested in the argument $\Gamma_{N,X}(X_N,Y_N)$ of $\CPara_N$. Our goal is to show that $\Gamma_N$ is a contraction on a ball in $\X{s_1}{b}([0,\tau]) \times \X{s_2}{b}([0,\tau])$, 
where the radius remains to be chosen. \\

Using Lemma \ref{para:lemma_obj_1} and Lemma \ref{para:lemma_obj_2}, it follows that there exists a (canonical) $H_N=H_N(X_N,Y_N)$ satisfying the identity

\begin{equation*}
\Gamma_{N,X}(X_N,Y_N) =  P_{\leq N} \Duh \big[ \PCtrl( H_N , P_{\leq N} \<1b>)\big] 
\end{equation*}
and the estimate
\begin{equation}\label{local:eq_contraction_H}
\| H_N \|_{\LM([0,\tau])} \lesssim B^2 + \| X_N \|_{\X{s_1}{b}([0,\tau])}^2 + \| Y_N \|_{\X{s_2}{b}([0,\tau])}^2. 
\end{equation}
Using the energy estimate (Lemma \ref{tools:lem_energy}), the inhomogeneous Strichartz estimate (Lemma \ref{tools:lem_inhomogeneous_strichartz}), Lemma \ref{para:lem_basic}, and $s_1- 1 +8\epsilon < - 1/2-\kappa$, we obtain that 
\begin{equation}\label{local:eq_contraction_1}
\begin{aligned}
\big\| \Gamma_{N,X}(X_N,Y_N) \big\|_{\X{s_1}{b}([0,\tau])} &\lesssim \big\| \PCtrl( H_N , P_{\leq N} \<1b>\,) \big\|_{\X{s_1-1}{b-1}([0,\tau])} \\
&\lesssim  \big\| \PCtrl( H_N , P_{\leq N} \<1b>\,) \big\|_{L_t^{2b}H_x^{s_1-1}([0,\tau]\times \bT^3)}\\
&\lesssim \tau^{\frac{1}{2b}}  \big\| \PCtrl( H_N , P_{\leq N} \<1b>\,) \big\|_{L_t^{\infty}H_x^{s_1-1}([0,\tau]\times \bT^3)}\\
&\lesssim  \tau^{\frac{1}{2b}} \| H_N \|_{\LM([0,\tau])}  \big\| \, \<1b>\, \big \|_{L_t^{\infty}H_x^{s_1-1+8\epsilon}([0,\tau]\times \bT^3)}\\
&\lesssim \tau^{\frac{1}{2b}} B \big(  B^2 + \| X_N \|_{\X{s_1}{b}([0,\tau])}^2 + \| Y_N \|_{\X{s_2}{b}([0,\tau])}^2 \big). 
\end{aligned}
\end{equation}
Using the multi-linear estimates from Proposition \ref{local:prop_master}, which are available due to our restriction to the event \eqref{local:eq_lwp_event_restriction}, and the time-localization lemma (Lemma \ref{tools:lem_localization}), we similarly obtain 
\begin{equation}\label{local:eq_contraction_2}
\begin{aligned}
&\big\| \Gamma_{N,Y}(X_N,Y_N) \big\|_{\X{s_2}{b}([0,\tau])} \\
&\lesssim \big\| \, \<1r>  \big\|_{\X{s_2}{b}([0,\tau])} +\big\| \So + \CPara + \RMT + \Phy  \big\|_{\X{s_2-1}{b-1}([0,\tau])} \\
&\lesssim B + \tau^{b_+-b} \big\| \So + \CPara + \RMT + \Phy  \big\|_{\X{s_2-1}{b_+-1}([0,\tau])} \\
&\lesssim B + \tau^{b_+-b}  \big(  B^3 + \| X_N \|_{\X{s_1}{b}([0,\tau])}^3 + \| Y_N \|_{\X{s_2}{b}([0,\tau])}^3 \big)
\end{aligned}
\end{equation}
By combining \eqref{local:eq_contraction_1} and \eqref{local:eq_contraction_2}, we obtain for a constant $C=C(\epsilon,s_1,s_2,b_+,b)$ that 
\begin{equation}
\big\| \Gamma_{N}(X_N,Y_N) \big\|_{\X{s_1}{b}([0,\tau])\times \X{s_2}{b}([0,\tau])} \leq CB + C \tau^{b_+-b}  \big(  B^3 + \| X_N \|_{\X{s_1}{b}([0,\tau])}^3 + \| Y_N \|_{\X{s_2}{b}([0,\tau])}^3 \big). 
\end{equation}
Since  $C^4 \tau^{b_+-b} B^2 \leq 1/100$, which follows from $\tau^{b_+-b} A\leq 1$ and our choice of $B$, we see that $\Gamma_N$ maps the ball in $\X{s_1}{b}([0,\tau]) \times \X{s_2}{b}([0,\tau])$ of radius $2CB$ to itself. A minor modification of the above argument also yields that $\Gamma_N$ is a contraction, which implies the existence of a unique fixed point $(X_N,Y_N)$ of $\Gamma_N$ satisfying
\begin{equation}\label{local:eq_contraction_3}
 \| X_N \|_{\X{s_1}{b}([0,\tau])}, \| Y_N \|_{\X{s_2}{b}([0,\tau])}\leq 2CB. 
\end{equation}
Using \eqref{local:eq_contraction_H}, we obtain that $X_N=P_{\leq N} \Duh \big[ \PCtrl( H_N , P_{\leq N} \<1b>)\big]$ with $H_N$ satisfying $\| H_N \|_{\LM([0,\tau])}\lesssim B^2$. Finally, using the triangle inequality and the condition $\,\bluedot \in  \Theta^{\type}_{\blue}(B,1)$ from \eqref{local:eq_lwp_event_restriction}, we obtain that $w_N=X_N+Y_N$ satisfies
\begin{equation}\label{local:eq_contraction_4}
\| w_N \|_{\X{s_1}{b}([0,\tau])} \leq 4CB \qquad \text{and} \qquad  \sum_{L_1\sim L_2} \| P_{L_1} \<1b>  \cdot P_{L_2} w_N  \|_{L_t^2 H_x^{-4\delta_1}([0,\tau])} \lesssim B^2. 
\end{equation}

Using that $B=c A^{c}$, \eqref{local:eq_contraction_3} and \eqref{local:eq_contraction_4} yield the desired estimates in \eqref{lwp:item_P2}. \\

We now turn to \eqref{lwp:item_P3}. This is a notationally extremely tedious but  mathematically minor modification of the arguments leading to \eqref{lwp:item_P2}. Similar modifications are usually omitted in the literature and we only outline the argument. In the frequency-localized versions of our estimates leading to \eqref{lwp:item_P2}, we always had an additional decaying factor $N_{\operatorname{max}}^{-\eta^\prime}$, where $N_{\operatorname{max}}$ was the maximal frequency-scale (see Remark \ref{local:remark_master} and Sections \ref{section:stochastic_object}-\ref{section:physical}). So far, this was only used to sum over all dyadic scales, but it also yields the smallness conditions in \eqref{lwp:item_P3}. Indeed, one only has to apply the same estimates as above to the difference equation
\begin{equation*}
(X_N-X_K,Y_N-Y_K)= \Gamma_N(X_N,Y_N)-\Gamma_K(X_K,Y_K). 
\end{equation*} 

\end{proof}

\subsection{Stability theory}\label{section:stability_gaussian}

In this subsection, we prove a stability estimate (Proposition \ref{local:prop_stability}) on large time-intervals. Strictly speaking, the stability estimate is part of the global instead of the local theory, but the argument is closely related to the proof of local well-posedness (Proposition \ref{lwp:prop}). While the stability estimate in this section is phrased in terms of $\,\bluedot\,$, it can be used to obtain a similar estimate in terms of $\, \purpledot\,$(Proposition \ref{global:prop_stability}). This second stability estimate will then be used in the globalization argument. \\

In order to state the stability result, we introduce the function space $\Z$, which captures the admissible perturbations of the initial data.

\begin{definition}[Structured perturbations]\label{global:definition_structured}
Let $T\geq 1$, $t_0\in [0,T]$, $N\geq 1$, and $K\geq 1$. For any $\, \purpledot \in \cH_x^{-1/2-\kappa}(\bT^3)$ and $Z[t_0]\in \cH_x^{s_1}(\bT^3)$, we define 
\begin{align*}
&\| Z[t_0] \|_{\mathscr{Z}([0,T], \scriptp;t_0,N,K)} \\
&= \inf_{Z^{\scalebox{0.9}{$\circ$}},Z^{\scalebox{0.9}{$\square$}}}  \max\Big( \| \Zbox[t_0] \|_{\mathcal{H}_x^{s_1}(\bT^3)}, ~ 
\| \Zcirc[t_0] \|_{\mathcal{H}_x^{s_2}(\bT^3)}, ~ 
 \sum_{L_1\sim L_2} \| P_{L_1} \<1p> \cdot  P_{L_2} Z \|_{L_t^2 H_x^{-4\delta_1}([O,T]\times \bT^3)}, ~ \\
&\hspace{13ex}\| \lcol V \ast \big( P_{\leq N} \<1p>  \cdot P_{\leq N} \Zcirc \big) \nparald P_{\leq N} \<1p> \rcol  \|_{\X{s_2-1}{b_+-1}([0,T])}, \\
&\hspace{13ex}\| \nparaboxld \Big( \lcol  V \ast \big( P_{\leq N} \<1p>  \cdot P_{\leq N} \Zbox \big) \, P_{\leq N} \<1p>  \rcol \Big) \|_{\X{s_2-1}{b_+-1}([0,T])} \Big),
\end{align*}
where the infimum is taken over all $\Zbox[t_0]\in \cH_x^{s_1}(\bT^3)$ and $\Zcirc[t_0]\in \cH^{s_2}_x(\bT^3)$ satisfying the identity $Z[t_0]= \Zbox[t_0]+ \Zcirc[t_0]$ and the Fourier support condition $\supp \widehat{\Zbox}[t_0](n) \subseteq  \{ n \in \bZ^3 \colon |n|\leq 8\max(N,K) \}$. Furthermore, we wrote $\Zbox$, $\Zcirc$, and $Z$ for the corresponding solutions to the linear wave equation. 
\end{definition}

The notation $\Zcirc_N$ and $\Zbox_N$ is motivated by the paradifferential operators used in their treatment. The contributions of $\Zcirc_N$ and $\Zbox_N$ are estimated using $\parald$ and $\paraboxld$, respectively.

It is clear that, for a fixed parameters $T,t_0,N$, and $K$, the maximum is jointly continuous in $\Zcirc[t_0]\in \cH_x^{s_1}(\bT^3)$ (satisfying the frequency-support condition), $\Zbox [t_0] \in \cH_x^{s_2}(\bT^3)$,  and $\purpledot \in \cH_x^{-1/2-\kappa}$. This is the primary reason for including the frequency support condition, since the sum in $L_1$ and $L_2$ would otherwise not be continuous in $\Zcirc[t_0]$. In particular, the norm
$\| Z[t_0] \|_{\mathscr{Z}([0,T], \scriptp;t_0,N,K)}$ is Borel-measurable in $Z[t_0] \in \cH_x^{s_2}(\bT^3)$ and  $\purpledot \in \cH_x^{-1/2-\kappa}$. 

\begin{proposition}[Stability estimate]\label{local:prop_stability}
Let $T\geq 1$, let $A\geq 1$, and let  $\zeta=\zeta(\epsilon,s_1,s_2,\kappa,\eta,\eta^\prime,b_+,b)>0$ be sufficiently small. There exists a constant $C=C(\epsilon,s_1,s_2,b_+,b_-)$ and a Borel set $\Theta^{\stab}_{\blue}(A,T) \subseteq \cH_x^{-1/2-\kappa}(\bT^3)$ satisfying
\begin{equation*}
\bP(\, \bluedot \in \Theta^{\stab}_{\blue}(A,T)) \geq 1 - \zeta^{-1} \exp(-\zeta A^{\zeta}) 
\end{equation*}
such that the following holds for all  $\, \bluedot \in \Theta^{\stab}_{\blue}(A,T)$: \\
Let $N\geq 1$, $B\geq 1$, $0<\theta<1$, $\cJ\subseteq [0,T]$ be a compact interval, and $t_0 \defe \min \cJ$. Let  $\util\colon \cJ \times \bT^3 \rightarrow \bR$ be an approximate solution of \eqref{eq:nlw_N} satisfying the following assumptions.
\begin{itemize}
\item[(A1)] Structure: We have the decomposition
\begin{equation*}
\util = \<1b> + \<3DNb> + \wtil. 
\end{equation*}
\item[(A2)] Global bounds: It holds that
\begin{equation*}
\| \wtil \|_{\X{s_1}{b}(J)} \leq B \quad \text{and} \quad \sum_{L_1\sim L_2} \|  P_{L_1} \<1b> \cdot P_{L_2} \wtil \|_{L_t^2 H_x^{-4\delta_1}(\cJ\times \bT^3)} \leq B. 
\end{equation*}
\item[(A3)] Approximate solution: There exists $H_N\in \LM(\cJ)$ and $F_N \in \X{s_2-1}{b_+-1}(\cJ)$ satisfying the identity
\begin{align*}
(-\partial_t^2 - 1 + \Delta) \util  = P_{\leq N} \lcol \big( V \ast ( P_{\leq N} \util )^2  \big) P_{\leq N} \util \rcol - P_{\leq N} \PCtrl(H_N , P_{\leq N} \<1b>\,) - F_N 
\end{align*}
and the estimates
\begin{equation*}
\| H_N \|_{\LM(\cJ)} \leq \theta \qquad \text{and} \qquad \| F_N \|_{\X{s_2-1}{b_+-1}(\cJ)} \leq \theta.  
\end{equation*}
\end{itemize}
Furthermore, let $Z_N[t_0] \in H_x^{s_1}(\bT^3)$ be a perturbation satisfying the following assumption.
\begin{itemize}
\item[(A4)] Structured perturbation: There exists a $K\geq 1$ such that
\begin{equation*}
\| Z[t_0] \|_{\Z(\cJ,\scriptb;t_0,N,K)}\leq \theta. 
\end{equation*}
\end{itemize}
Finally, assume that 
\begin{itemize}
\item[(A5)] Parameter condition: $C \exp\Big( C (A+B)^{\frac{2}{b_+-b}} T^{\frac{40}{b_+-b}} \Big) \theta \leq 1$. 
\end{itemize}
Then, there exists a solution $u_N\colon \cJ \times \bT^3\rightarrow \bR$ of \eqref{eq:nlw_N} satisfying the initial value condition $u_N[t_0]= \util[t_0] + Z_N[t_0]$ and  the following conclusions. 
\begin{itemize}
\item[(C1)] Preserved structure: We have the decomposition 
\begin{equation*}
u_N= \<1b> + \<3DN> + w_N. 
\end{equation*}
\item[(C2)] Closeness: The difference $u_N-\util=w_N-\wtil$ satisfies
\begin{align*}
\| u_N - \util \|_{\X{s_1}{b}(\cJ)} &\leq C \exp\big( C (A+B)^{\frac{2}{b_+-b}} T^{\frac{40}{b_+-b}} \big) \theta, \\
 \sum_{L_1\sim L_2} \|  P_{L_1} \<1b> \cdot P_{L_2} (u_N-\util) \|_{L_t^2 H_x^{-4\delta_1}(\cJ\times \bT^3)} &\leq  C \exp\big( C (A+B)^{\frac{2}{b_+-b}} T^{\frac{40}{b_+-b}} \big) \theta.
\end{align*}
\item[(C3)] Preserved global bounds: It holds that
\begin{equation*}
\| w_N \|_{\X{s_1}{b}(\cJ)} \leq B_\theta \quad \text{and} \quad \sum_{L_1\sim L_2} \| P_{L_1} \<1b> \cdot  P_{L_2} w_N \|_{L_t^2 H_x^{-4\delta_1}(\cJ\times \bT^3)} \leq B_\theta,
\end{equation*}
where $ B_\theta \defe B+ C \exp\big( C (A+B)^{\frac{2}{b_+-b}} T^{\frac{40}{b_+-b}} \big) \theta$. 
\end{itemize}
\end{proposition}

As mentioned above, the proof of Proposition \ref{local:prop_stability} is close to the proof of local well-posedness. The most important additional ingredient is a  Gronwall-type argument in $\X{s}{b}$-spaces, which is slightly technical due to their non-local nature in the time-variable. 

\begin{proof}
Let  $N,B,\theta,\cJ,t_0,\util, \wtil, H_N, F_N, Z_N, \Zbox_N$, and $\Zcirc_N$ be as in the statement of the proposition and assume that (A1)-(A5) are satisfied. We make the Ansatz
\begin{equation*}
u_N(t) = \util(t)+ v_N(t)+Z_N(t),
\end{equation*}
where the nonlinear component $v_N(t)$ will be decomposed into a para-controlled and a smoother component below. Based on  the condition $u_N[t_0]= \util[t_0] + Z_N[t_0]$, we require that $v_N[t_0]=0$. Using the assumption (A3) and that $Z_N$ solves the linear wave equation, we obtain the evolution equation
\begin{align*}
(-\partial_t^2 - 1 + \Delta) v_N  
=& P_{\leq N} \lcol \big( V \ast \big( P_{\leq N} (\util+v_N+Z_N )^2  \big) \big)P_{\leq N} (\util+v_N+Z_N) \rcol \\
& -P_{\leq N} \lcol \big( V \ast ( P_{\leq N} \util )^2  \big) P_{\leq N} \util \rcol   \\
&+P_{\leq N} \PCtrl(H_N , P_{\leq N} \<1b>\,) + F_N.
\end{align*}
Inserting the structural assumption (A1) and using the binomial formula, we obtain that 
\begin{align*}
(-\partial_t^2 - 1 + \Delta) v_N  
=&  P_{\leq N} \bigg[  2 \lcol V \ast \Big( P_{\leq N} \Big( \<1b>+ \<3DN> +w_N \Big) \cdot P_{\leq N} \Big( v_N + Z_N \Big) \Big) \, P_{\leq N} \<1b> \rcol  \\
&+   V \ast \Big(  \Big( P_{\leq N} \big( v_N + Z_N \big) \Big)^2 \Big)  \, P_{\leq N} \<1b>  \allowdisplaybreaks[4]  \\
&+ \PCtrl(H_N , P_{\leq N} \<1b>\,) \allowdisplaybreaks[4] \\
&+ 2  V \ast \Big( P_{\leq N} \Big( \<1b>+ \<3DN> +w_N \Big) \cdot P_{\leq N} \Big( v_N + Z_N \Big) \Big) \,  P_{\leq N} \Big( \<3DN> +w_N \Big) \allowdisplaybreaks[4]\\
&+ P_{\leq N}  \big( V \ast\big( \lcol  P_{\leq N} (\util+v_N+Z_N )^2 \rcol \big)  \big) P_{\leq N} (v_N+Z_N)\bigg] + F_N. 
\end{align*}
We then decompose $v_N=X_N+Y_N$, where $X_N$ is the para-controlled component and $Y_N$ is the smoother component. Since $v_N[t_0]=0$, we impose the initial value conditions $X_N[t_0]=0$ and $Y_N[t_0]=0$. Similar as in Section \ref{section:ansatz}, we define $X_N$ and $Y_N$ through the evolution equations
\begin{equation}\label{local:eq_stability_X}
\begin{aligned}
(-\partial_t^2 -1 +\Delta) X_N =&  P_{\leq N} \bigg[ 2  \paraboxld \Big(  V \ast \Big( P_{\leq N}  \<1b> \cdot P_{\leq N} \Big( X_N + \Zbox_N  \Big) \Big) \, P_{\leq N} \<1b> \Big) \allowdisplaybreaks[4]\\
		 &+  2 V \ast \Big( P_{\leq N}   \<1b> \cdot P_{\leq N} \Big( Y_N + \Zcirc_N \Big) \Big) \parald P_{\leq N} \<1b> \allowdisplaybreaks[4]\\
		 &+   2\,  V \ast \Big( P_{\leq N} \Big(  \<3DN> +w_N \Big) \cdot P_{\leq N} \Big( X_N+Y_N+Z_N \Big) \Big) \, \parald P_{\leq N} \<1b>\allowdisplaybreaks[4] \\
		 &+   V \ast \Big(  \Big( P_{\leq N} \big( v_N + Z_N \big) \Big)^2 \Big)  \parald  \, P_{\leq N} \<1b> + \PCtrl(H_N , P_{\leq N} \<1b>\,)  \bigg]   
\end{aligned}
\end{equation}
and 
\begin{equation}\label{local:eq_stability_Y}
\begin{aligned}
(-\partial_t^2 -1 +\Delta) Y_N =&  P_{\leq N} \bigg[ 2 \nparaboxld \Big(  \lcol  V \ast \Big( P_{\leq N}  \<1b> \cdot P_{\leq N} \Big( X_N + \Zbox_N  \Big) \Big) \, P_{\leq N} \<1b> \rcol \Big) \allowdisplaybreaks[4]\\
		 &+  2 \lcol V \ast \Big( P_{\leq N}   \<1b> \cdot P_{\leq N} \Big( Y_N + \Zcirc_N  \Big) \Big) \nparald P_{\leq N} \<1b>  \rcol \allowdisplaybreaks[4]\\
		 &+   2\,  V \ast \Big( P_{\leq N} \Big(  \<3DN> +w_N \Big) \cdot P_{\leq N} \Big( v_N + Z_N \Big) \Big) \, \nparald P_{\leq N} \<1b>\allowdisplaybreaks[4] \\
		 &+   V \ast \Big(  \Big( P_{\leq N} \big(  v_N +Z_N \big) \Big)^2 \Big)  \nparald \, P_{\leq N} \<1b> \allowdisplaybreaks[4] \\   
		 &+ 2  V \ast \Big( P_{\leq N} \Big( \<1b>+ \<3DN> +w_N \Big) \cdot P_{\leq N} \Big( v_N+Z_N \Big) \Big) \,  P_{\leq N} \Big( \<3DN> +w_N \Big)\allowdisplaybreaks[4] \\
		&+   \big( V \ast\big( \lcol  P_{\leq N} (\util+v_N+Z_N )^2 \rcol \big)  \big) P_{\leq N} (v_N+Z_N)\bigg] + F_N.
\end{aligned}
\end{equation}
Since the nonlinearity in \eqref{local:eq_stability_X} and \eqref{local:eq_stability_Y} is frequency-truncated, a soft argument yields the local existence and uniqueness of $X_N$ and $Y_N$ in $C_t^0 \cH_x^{s_1}$ and $C_t^0 \cH_x^{s_2}$, respectively. Since $\X{s}{b}(\cJ)$ embeds into $C_t^0 \cH_x^{s}(\cJ\times \bT^3)$ for all $s\in \bR$, the solutions exist as long as the restricted $\X{s_1}{b}$ and $\X{s_2}{b}$-norms  stay bounded. 

In order to prove that $X_N$ and $Y_N$ exist on the full interval $\cJ$ and satisfy the desired bounds, we let $T_\ast$ be the maximal time of existence of $X_N$ and $Y_N$ on $\cJ$. We now proceed through a Gronwall-type argument in $\X{s}{b}$-spaces.  We first define
\begin{equation*}
f_N\colon [t_0,T^\ast) \rightarrow [0,\infty), ~ t \mapsto \| X_N \|_{\X{s_1}{b}([t_0,t])} + \| Y_N \|_{\X{s_2}{b}([t_0,t])}.
\end{equation*}
We emphasize that we neither rely on nor prove the continuity of $f_N$. Using Lemma \ref{tools:lem_restricted_continuity} and Lemma \ref{tools:lem_energy}, there exists an implicit constant $C_{\text{En}}=C_{\text{En}}(s_1,s_2,b)$ such that 
\begin{equation*}
g_N(t) \defe C_{\text{En}}  \Big( \big\| 1_{[t_0,t]} (-\partial_t^2 -1 +\Delta) X_N \big\|_{\X{s_1-1}{b-1}(\bR)} +  \big\| 1_{[t_0,t]} (-\partial_t^2 -1 +\Delta) Y_N \big\|_{\X{s_2-1}{b-1}(\bR)} \Big)
\end{equation*}
satisfies $f_N(t) \leq g_N(t)$ for all $t\in [t_0,T_\ast)$. Due to Lemma \ref{tools:lem_restricted_continuity}, $g_N(t)$ is continuous. Now, let $\tau>0$ be a step-size which remains to be chosen and assume that $t,t^\prime\in [t_0,T_\ast)$ satisfy $t\leq t^\prime \leq t +\tau$. Using Lemma \ref{tools:lem_localization}, we obtain that for an implicit constant $C=C(s_1,s_2,b,b_+)$ that 
\begin{align*}
&g_N(t^\prime) \\
&\leq  C_{\text{En}}  \Big( \big\| 1_{[t_0,t]} (-\partial_t^2 -1 +\Delta) X_N \big\|_{\X{s_1-1}{b-1}(\bR)} +  \big\| 1_{[t_0,t]} (-\partial_t^2 -1 +\Delta) Y_N \big\|_{\X{s_2-1}{b-1}(\bR)} \Big) \\
&+ C_{\text{En}}  \Big( \big\| 1_{(t,t^\prime]} (-\partial_t^2 -1 +\Delta) X_N \big\|_{\X{s_1-1}{b-1}(\bR)} +  \big\| 1_{(t,t^\prime]} (-\partial_t^2 -1 +\Delta) Y_N \big\|_{\X{s_2-1}{b-1}(\bR)} \Big) \\
&\leq g_N(t) + C \tau^{b_+-b}  \big( \big\| (-\partial_t^2 -1 +\Delta) X_N \big\|_{\X{s_1-1}{b_+-1}((t,t^\prime])}  + \big\| (-\partial_t^2 -1 +\Delta) Y_N \big\|_{\X{s_2-1}{b_+-1}((t,t^\prime])} \big)  \\
&\leq g_N(t) + C \tau^{b_+-b}  \big( \big\| (-\partial_t^2 -1 +\Delta) X_N \big\|_{\X{s_1-1}{b_+-1}([t_0,t^\prime])}  + \big\| (-\partial_t^2 -1 +\Delta) Y_N \big\|_{\X{s_2-1}{b_+-1}([t_0,t^\prime])} \big) . 
\end{align*}
Similar as in the proof of local well-posedness (Proposition \ref{lwp:prop}), we can use Lemma \ref{local:lem_type_conversion}, Proposition \ref{local:prop_master}, and Proposition \ref{so3:prop} to restrict to the event 
\begin{equation}\label{local:eq_stability_event_restriction}
\begin{aligned}
&\Big\{ \, \bluedot \in \Theta_{\blue}^{\ms}(A,T) \Big\} \, \medcap  \,  \Big\{ \, \bluedot \in \Theta^{\type}_{\blue}(A,T)\Big\} \, \medcap  \,  \Big\{ \big\| \, \<1b>\, \big\|_{L_t^\infty \cC_x^{-1/2-\kappa}([0,1]\times \bT^3)}\leq A \Big \} \,
  \\
&\medcap  \, \Big\{ \sup_{N} \big\| \, \<3DN>\, \big\|_{L_t^\infty \cC_x^{\beta-\kappa}([0,1]\times \bT^3)}\leq T^3 A \Big \}. 
\end{aligned}
\end{equation}
By combining the assumption (A2), (A3), (A4), and the multi-linear master estimate, a similar argument as in the proof of Proposition \ref{lwp:prop} yields 
\begin{align*}
 &\tau^{b_+-b}  \big( \big\| (-\partial_t^2 -1 +\Delta) X_N \big\|_{\X{s_1-1}{b_+-1}([t_0,t^\prime])}  + \big\| (-\partial_t^2 -1 +\Delta) Y_N \big\|_{\X{s_2-1}{b_+-1}([t_0,t^\prime])} \big)  \\
 &\lesssim T^{30}\tau^{b_+-b} ((A+B)^2+f_N(t^\prime)^2) (\theta+f_N(t^\prime)). 
\end{align*}
All together, we have proven for all $t,t^\prime \in [t_0,T_\ast)$ satisfying $t\leq t^\prime \leq t+\tau$ the estimate
\begin{equation*}
f(t^\prime)\leq g(t^\prime) \leq g(t) + C  T^{30} \tau^{b_+-b} ((A+B)^2+f_N(t^\prime)^2) (\theta+f_N(t^\prime)). 
\end{equation*}
Using $g(t_0)=0$, using a continuity argument (Lemma \ref{tools:lem_continuity}), iterating the resulting bounds, and assuming the conditions 
\begin{equation}\label{local:eq_stability_condition}
C (A+B)^2 e^{\frac{T}{\tau}} \theta \leq 1/2 \qquad \text{and} \qquad 2  C T^{30} \tau^{b_+-b} ((A+B)^2+6) \leq 1/4,
\end{equation}
we obtain that 
\begin{equation}\label{local:eq_stability_f_bound}
\sup_{t\in [t_0,T_\ast)} f(t) \leq \sup_{t\in [t_0,T_\ast)} g(t) \leq C (A+B)^2 e^{\frac{T}{\tau}} \theta. 
\end{equation}
Using the case of equality in the second condition in \eqref{local:eq_stability_condition} as a definition for $\tau$, the first condition follows from our assumption (A5). Recalling the definition of $f$, we obtain that
\begin{equation*}
\sup_{t\in [t_0,T_\ast)} \Big(  \| X_N \|_{\X{s_1}{b}([t_0,t])} + \| Y_N \|_{\X{s_2}{b}([t_0,t])} \Big) \leq C \exp\big( C (A+B)^{\frac{2}{b_+-b}} T^{\frac{40}{b_+-b}}  \big)
\end{equation*}
This estimate rules out finite-time blowup on $\cJ$ and implies that $T_\ast=\sup \cJ$. Together with a soft argument, which is based on the integral equation for $X_N$ and $Y_N$ as well as the time-localization lemma (Lemma \ref{tools:lem_localization}), we obtain that
\begin{equation}\label{local:eq_stability_bound}
\| X_N \|_{\X{s_1}{b}(\cJ)} + \| Y_N \|_{\X{s_2}{b}(\cJ)}  \leq C \exp\big( C (A+B)^{\frac{2}{b_+-b}} T^{\frac{40}{b_+-b}}  \big). 
\end{equation}
With this uniform estimates in hand, we now easily obtain the desired conclusions (C1), (C2), and (C3). In order to obtain (C1), we (are forced to) choose 
\begin{equation*}
w_N = \widetilde{w}_N + X_N + Y_N + Z_N. 
\end{equation*}
The conclusions (C2) and (C3) follow from (A4), \eqref{local:eq_stability_bound}, and the condition $\,\bluedot \in \Theta^{\type}_{\blue}(A,T)$ in our event \eqref{local:eq_stability_event_restriction}. 
\end{proof}

\section{Global theory}\label{section:global}

In this section, we prove the global well-posedness of the renormalized nonlinear wave equation and the invariance of the Gibbs measure. As mentioned in the introduction, the heart of this section is a new form of Bourgain's globalization argument. In Section \ref{section:gwp}, we prove the global well-posedness for Gibbsian initial data. We focus on the overall strategy and postpone several individual steps to Section \ref{section:stability} below. In Section \ref{section:invariance}, we prove the invariance of the Gibbs measure. Using the global well-posedness from Section \ref{section:gwp}, the proof of invariance is similar as in Bourgain's seminal paper \cite{Bourgain94}. 

\subsection{Global well-posedness}\label{section:gwp}

We now prove the (quantitative) global well-posedness of the renormalized nonlinear wave equation for Gibbsian initial data. In particular, we show that the structure 
\begin{equation*}
\Phi_N[t] \, \purpledot = \<1p> + \<3DNp> + w_N
\end{equation*}
from the local theory (see Proposition \ref{global:prop_lwp}) is preserved by the global theory. Here, the linear and cubic stochastic objects are defined exactly as in \eqref{eq:1b} and \eqref{local:eq_3DN}, but with $\,\bluedot\,$ replaced by $\,\purpledot\,$. 

\begin{proposition}[Global well-posedness]\label{global:prop_gwp}
Let $A\geq1$, let $T\geq 1$, let $C=C(\epsilon,s_1,s_2,\kappa,\eta,\eta^\prime,b_+,b)\geq 1$ be sufficiently large, and let $\zeta=\zeta(\epsilon,s_1,s_2,\kappa,\eta,\eta^\prime,b_+,b)>0$ be sufficiently small. We assume that $B,D\geq 1$ satisfy 
\begin{equation}\label{global:eq_BD_condition}
B\geq B(A,T)\defe C \exp( C (A+T)^C) \quad \text{and} \quad D\geq D(A,T) \defe C \exp( \exp( C (A+T)^C)). 
\end{equation}
Furthermore, let $K\geq 1$ satisfy the condition
\begin{equation}\label{global:eq_gwp_parameter}
C \exp( C (A+B+T)^C) K^{-\eta^\prime} \leq 1.
\end{equation}
Then, the Borel set 
\begin{alignat*}{3}
\cE_K(B,D,T) &= \bigcap_{N\geq K}
 \bigg( && \Big\{ \, \purpledot \in \cH_x^{-1/2-\kappa}(\bT^3)\big| ~ w_N(t)= \Phi_N(t)\, \purpledot - \,  \<1p> - \, \<3DNp> \text{ satisfies } \\
&&& \|w_N\|_{\X{s_1}{b}([0,T])}\leq B ~ \text{and} ~  \sum_{L_1\sim L_2} \| P_{L_1} \<1p> \cdot P_{L_2} w_N \|_{L_t^2 \cH_x^{-4\delta_1}([0,T]\times \bT^3)} \leq B \Big\}  \\
&&&\medcap \Big\{ \, \purpledot \in \cH_x^{-1/2-\kappa}(\bT^3)\big|
 \| \Phi_N[t]\, \purpledot\,-  \Phi_K[t]\, \purpledot \|_{C_t^0 \cH_x^{\beta-\kappa}([0,T] \times \bT^3)} \leq D K^{-\eta^\prime} \Big\}  \bigg)
\end{alignat*}
satisfies the estimate 
\begin{equation}\label{global:eq_gwp_probabilistic}
\inf_{M\geq K} \mup_M( \cE_K(B,D,T)) \geq 1 - T \zeta^{-1}  \exp(-\zeta A^\zeta).
\end{equation}
\end{proposition}

In the proof below, we need two modifications of the cubic stochastic object. We define
\begin{equation}\label{global:eq_modified_cubic_objects}
\<3DNtaup> \defe \Duh \big[ 1_{[0,\tau]}(t) \<3Np>\big]  \qquad \text{and} \qquad
 \<3DNMtaup> \defe \Duh \big[ 1_{[0,\tau]}(t)\big( \<3Np> - \<3Mp>\big) \big].
\end{equation}

\begin{proof}[Proof of Proposition \ref{global:prop_gwp}:] 
We encourage the reader to review the informal discussion of the argument in the introduction before diving into the details of this proof.\\
Let $\tau \in (0,1)$ be such that  $ 1/2\leq  A \tau^{b_+-b} \leq 1 $ and $J \defe T/\tau \in \mathbb{N}$. We let $B_j,D_j$, where $1\leq j \leq J$, be  increasing sequences which remain to be chosen.  We will prove below that our choice satisfies $B_j\leq B$ and $D_j\leq D$ for all $1\leq j \leq J$. We then have that 
\begin{alignat*}{3}
\cE_K(B_j,D_j,j\tau) &= \bigcap_{N\geq K}
 \bigg( && \Big\{ \, \purpledot \in \cH_x^{-1/2-\kappa}(\bT^3)\big| ~ w_N(t)= \Phi_N(t)\, \purpledot - \,  \<1p> - \, \<3DNp> \text{ satisfies } \\
&&& \|w_N\|_{\X{s_1}{b}([0,j\tau])}\leq B_j ~ \text{and} ~  \sum_{L_1\sim L_2} \| P_{L_1} \<1p> \cdot P_{L_2} w_N \|_{L_t^2 \cH_x^{-4\delta_1}([0,j\tau]\times \bT^3)} \leq B_j \Big\}  \\
&&&\medcap \Big\{ \, \purpledot \in \cH_x^{-1/2-\kappa}(\bT^3)\big|
 \| \Phi_N[t]\, \purpledot\,-  \Phi_K[t]\, \purpledot \|_{C_t^0 \cH_x^{\beta-\kappa}([0,j\tau] \times \bT^3)} \leq D_j K^{-\eta^\prime} \Big\}  \bigg). 
\end{alignat*}
We now claim for all $M\geq K$  that, under certain constraints on the sequences $B_j$ and $D_j$ detailed below, 
\begin{equation}\label{global:eq_gwp_p_base}
 \mup_M \Big( \cE_K(B_1,D_1,\tau)\Big) \geq 1- \zeta^{-1} \exp(\zeta A^\zeta) 
\end{equation}
and 
\begin{equation}\label{global:eq_gwp_p_step}
\mup_M \Big( \cE_K(B_j,D_j, j\tau)\Big) \geq \mup_M \Big( \cE_K(B_{j-1},D_{j-1},(j-1)\tau)\Big) -  \zeta^{-1} \exp(\zeta A^\zeta) . 
\end{equation}
We refer to \eqref{global:eq_gwp_p_base} as the base case and to \eqref{global:eq_gwp_p_step} as the induction step. We split the rest of the argument into several steps.

\emph{Step 1: The base case $\eqref{global:eq_gwp_p_base}$.} We set $B_1\defe A$ and $D_1\defe A$. If $\cL(A,\tau)$ is as in Proposition \ref{global:prop_lwp}, we obtain that $\cL(A,\tau)\subseteq \cE_K(B_1,D_1,\tau)$. This implies
\begin{equation*}
 \mup_M \Big( \cE_K(B_1,D_1,\tau)\Big) \geq  \mup_M \Big( \cL(A,\tau) \Big)  \geq 1- \zeta^{-1} \exp(\zeta A^\zeta).
\end{equation*}

\emph{Step 2: The induction step $\eqref{global:eq_gwp_p_step}$.} We first restrict to the event
\begin{equation}\label{global:eq_gwp_p_event}
\cS^{\gwp}(A,T,\tau) \defe \cL(A,\tau) \medcap \cL(A,2\tau) \medcap \cS^{\stime}(A,T,\tau) \medcap \cS^{\cub}(A,T,\tau) \medcap \cS^{\stab}(A,T,\tau). 
\end{equation} 
Using Proposition \ref{global:prop_lwp}, Proposition \ref{global:prop_structure_time}, Proposition \ref{global:prop_structure_cubic}, and Proposition \ref{global:prop_stability}, which also contain the definitions of the sets in \eqref{global:eq_gwp_p_event}, we obtain that 
\begin{equation*}
\mup_M( \cS^{\gwp}(A,T,\tau)) \geq 1 -\zeta^{-1} \exp(\zeta A^\zeta). 
\end{equation*}
Using the invariance of $\mup_M$ under $\Phi_M$, we also obtain that 
\begin{equation*}
\mup_M\Big( \Phi_M[\tau]^{-1} \cE_K\big(B_{j-1},D_{j-1},(j-1)\tau\big) \Big) = \mup_M\Big(\cE_K\big(B_{j-1},D_{j-1},(j-1)\tau\big) \Big).
\end{equation*}
In order to obtain the probabilistic estimate \eqref{global:eq_gwp_p_step}, it therefore suffices to prove the inclusion 
\begin{equation}
\cS^{\gwp}(A,T,\tau) \medcap  \Phi_M[\tau]^{-1} \cE_K\big(B_{j-1},D_{j-1},(j-1)\tau\big) \subseteq \cE_K(B_j,D_j,j\tau). 
\end{equation}
For the rest of this proof, we assume that  $\, \purpledot \in \cS^{\gwp}(A,T,\tau) \medcap  \Phi_M[\tau]^{-1} \cE_K\big(B_{j-1},D_{j-1},(j-1)\tau\big)$ and $N,M\geq K$. To clarify the structure of the proof, we divide our argument into further substeps. \\

\emph{Step 2.1: Time-translation.} We rephrase the condition $\,\greendot = \Phi_M[\tau]\, \purpledot \in \cE_K(B_{j-1},D_{j-1},(j-1)\tau)$ in terms of $\purpledot$. 

Since $\, \greendot \in \cE_K(B_{j-1},D_{j-1},(j-1)\tau)$, we obtain for all $t\in [\tau,j\tau]$ that 
\begin{equation*}
\Phi_N(t-\tau) \Phi_M[\tau] \, \purpledot = \Phi_N(t-\tau) \, \greendot  = \<1g>(t-\tau) + \<3DNg>(t-\tau) + w^\grn_{N,M}(t-\tau), 
\end{equation*}
where $w^{\grn}_{N,M}\colon [0,(j-1)\tau]\times \bT^3 \rightarrow \bR$ satisfies 
\begin{equation*}
\| w^{\grn}_{N,M}\|_{\X{s_1}{b}([0,(j-1)\tau])} \leq B_{j-1} \quad \text{and} \quad
 \sum_{L_1\sim L_2} \| P_{L_1} \<1g>  \cdot P_{L_2} w^{\grn}_{N,M} \|_{L_t^{2} H_x^{-4\delta_1}([0,(j-1)\tau]\times \bT^3)} \leq B_{j-1}. 
\end{equation*}
The superscript ``grn'' emphasizes that $w^{\grn}_{N,M}$ appears in the structure involving $\,\greendot\,$. Furthermore, we also have that 
\begin{equation}\label{global:eq_gwp_p_d1}
\| \Phi_N[t-\tau] \Phi_M[\tau] \, \purpledot - \Phi_K[t-\tau] \Phi_M[\tau] \, \purpledot \|_{C_t^0 \cH_x^{\beta-\kappa}([\tau,j\tau]\times \bT^3)} \leq D_{j-1} K^{-\eta^\prime}. 
\end{equation}
Since $\, \purpledot \in \cS^{\stime}(A,T,\tau)$ (as in Proposition \ref{global:prop_structure_time}), it follows for all $t\in [\tau,j\tau]$ that 
\begin{equation}\label{global:eq_gwp_p_mixed}
\Phi_N(t-\tau) \Phi_M[\tau] \, \purpledot =   \<1p>(t)+ \<3DNp>(t) - \<3DNMtaup>(t) + w_{N,M}(t),
\end{equation}
where $w_{N,M}\colon [\tau,j\tau] \times \bT^3 \rightarrow \bR$ satisfies
\begin{equation}\label{global:eq_gwp_p_wNM_bounds}
\| w_{N,M}\|_{\X{s_1}{b}([\tau,j\tau])}, ~
 \sum_{L_1\sim L_2} \| P_{L_1} \<1p>  \cdot P_{L_2} w_{N,M} \|_{L_t^{2} H_x^{-4\delta_1}([\tau,j\tau]\times \bT^3)} \leq T^\alpha A B_{j-1}. 
\end{equation}
Our next goal is to replace $\Phi_M[\tau]$ in \eqref{global:eq_gwp_p_mixed} by $\Phi_N[\tau]$, which is done in Step 2.2 and Step 2.3. \\

\emph{Step 2.2: The cubic stochastic object.} In this step, we correct the structure of  $\Phi_N(t-\tau) \Phi_M[\tau] \, \purpledot$, as stated in \eqref{global:eq_gwp_p_mixed}, by adding the ``partial'' cubic stochastic object. 

We define
$\widetilde{u}_N\colon [\tau,j\tau] \times \bT^3 \rightarrow \bR$ by 
\begin{equation}\label{global:eq_gwp_p_util}
\widetilde{u}_N(t) = \Phi_N(t-\tau) \Phi_M[\tau] \, \purpledot +  \<3DNMtaup>(t)  = \<1p>(t)+ \<3DNp>(t) + w_{N,M}(t).
\end{equation}
While $\widetilde{u}_N$ depends on $M$, this is not reflected in our notation. The reason is that, as will be shown below, $\widetilde{u}_N$ is a close approximation of $u_N(t) = \Phi_N(t)\, \purpledot$, which does not directly depend on $M$. In order to match the notation of $\widetilde{u}_N$, we also define $\widetilde{w}_N= w_{N,M}$, which leads to 
\begin{equation*}
\widetilde{u}_N(t) =  \<1p>(t)+ \<3DNp>(t) + \widetilde{w}_N(t).
\end{equation*}
Using $\,\purpledot \in \cS^{\cub}(A,T,\tau)$ (as in Proposition \ref{global:prop_structure_cubic}), it follows that there exist $H_N \in \LM([\tau,j\tau]) $ and $F_N\in \X{s_2-1}{b_+-1}([\tau,j\tau])$ satisfying the identity
\begin{equation}\label{global:eq_gwp_approx_1}
\begin{aligned}
&(-\partial_t^2 - 1 + \Delta) \util 
- P_{\leq N} \lcol \big( V \ast ( P_{\leq N} \util )^2  \big) P_{\leq N} \util \rcol \\
&= -P_{\leq N} \PCtrl( H_N, P_{\leq N} \<1p>) - F_N 
\end{aligned}
\end{equation}
and the estimate
\begin{equation}\label{global:eq_gwp_approx_2}
\| H_N  \|_{\LM([\tau,j\tau])}, \| F_N \|_{\X{s_2-1}{b_+-1}([\tau,j\tau])} \leq T^{4\alpha} A^4 B_{j-1}^3 K^{-\eta^\prime}. 
\end{equation}
Thus, $\widetilde{u}_N$ is an approximate solution to the nonlinear wave equation on $[\tau,j\tau]\times \bT^3$. Furthermore, it holds that 
\begin{equation}\label{global:eq_gwp_p_d2}
\| \widetilde{u}_N[t] - \Phi_N[t-\tau] \Phi_M[\tau] \, \purpledot \|_{C_t^0 \cH_x^{\beta-\kappa}([\tau,j\tau]\times \bT^3)} \leq  T^{4\alpha} A^4 B_{j-1}^3 K^{-\eta^\prime}.
\end{equation}
\emph{Step 2.3: Stability estimate.} In this step, we turn the approximate solution $\widetilde{u}_N$ into an honest solution and fully correct the initial data at $t=\tau$. 

We now verify the assumptions (A1)-(A5) in Proposition \ref{global:prop_stability}, where we replace $B$ by $T^\alpha A B_{j-1}$ and set $\theta = T^{4\alpha} A^4 B_{j-1}^3 K^{-\eta^\prime}$. The first assumption (A1) holds with $\widetilde{w}_N=w_{N,M}$ due to \eqref{global:eq_gwp_p_util}. The second assumption (A2) coincides with the bounds \eqref{global:eq_gwp_p_wNM_bounds}. The third assumption (A3) coincides with \eqref{global:eq_gwp_approx_1} and \eqref{global:eq_gwp_approx_2}. 

For the fourth assumption (A4), we rely on $\, \purpledot \in \cL(A,\tau)$ (as in Proposition \ref{global:prop_lwp}). First, we have that 
\begin{equation*}
\widetilde{u}_N[\tau] = \Phi_M[\tau]\, \purpledot +  \<3DNMtaup>[\tau] = \<1p>[\tau] + \<3DMp>[\tau] +  \<3DNMtaup>[\tau] + w_M[\tau] =  \<1p>[\tau] + \<3DNp>[\tau] + w_M[\tau]. 
\end{equation*}
Second, we have that 
\begin{equation*}
\Phi_N[\tau] \, \purpledot = \<1p>[\tau] + \<3DNp>[\tau] + w_N[\tau]. 
\end{equation*}
Using \eqref{lwp:item_P4_Gibbs} in Proposition \ref{global:prop_lwp}, this implies that $Z_N[\tau]\defe \Phi_N[\tau] - \widetilde{u}_N[\tau]$ satisfies 
\begin{equation}
\| Z_N[\tau] \|_{\mathscr{Z}([0,T], \scriptp;\tau,N,M)} \leq A T^{\alpha} K^{-\eta^\prime},
\end{equation}
which yields (A4). Finally, as long as $B_j \leq B$, the fifth assumption (A5) follows from the parameter condition \eqref{global:eq_gwp_parameter}. Thus, the assumptions (A1)-(A5) in Proposition \ref{global:prop_stability} hold. Since $\, \purpledot \in  \cS^{\stab}(A,T,\tau)$, we obtain for all $t\in [\tau,j\tau]$  that
\begin{equation}
\Phi_N(t) \, \purpledot =  \<1p>(t)+ \<3DNp>(t) + w_N(t),
\end{equation}
where the nonlinear component $w_N$  satisfies
\begin{equation}\label{global:eq_gwp_p_wN_bounds}
\| w_N \|_{\X{s_1}{b}([\tau,j\tau]\times \bT^3)}, ~  \sum_{L_1\sim L_2} \| P_{L_1} \<1p>  \cdot P_{L_2} w_{N,M} \|_{L_t^{2} H_x^{-4\delta_1}([\tau,j\tau]\times \bT^3)} \leq T^\alpha A B_{j-1} +1 \leq 2 T^{\alpha} A B_{j-1}. 
\end{equation}
Furthermore, 
\begin{equation}\label{global:eq_gwp_p_d3}
\| \Phi_N[t] \, \purpledot  - \widetilde{u}_{N}[t] \|_{C_t^0 \cH_x^{\beta-\kappa}([\tau,j\tau]\times \bT^3)} \leq C \exp( C(A+B_{j-1}+T)^C) K^{-\eta^\prime}. 
\end{equation}
By combining \eqref{global:eq_gwp_p_d1}, \eqref{global:eq_gwp_p_d2}, and \eqref{global:eq_gwp_p_d3}, we obtain
\begin{equation}\label{global:eq_gwp_p_d4}
\begin{aligned}
&\| \Phi_N[t] \, \purpledot  - \Phi_K[t-\tau] \Phi_M[\tau] \, \purpledot \|_{C_t^0 \cH_x^{\beta-\kappa}([\tau,j\tau]\times \bT^3)} \\
&\leq \big( D_{j-1} + T^{4\alpha} A^4 B_{j-1}^3 + C \exp( C(A+B_{j-1}+T)^C) \big)K^{-\eta^\prime}. 
\end{aligned}
\end{equation}
By combining the general case $N\geq K$ in \eqref{global:eq_gwp_p_d4} with the special case $N=K$, using the triangle inequality, and increasing $C$ if necessary, we also obtain that
\begin{equation}\label{global:eq_gwp_p_d5}
\begin{aligned}
&\| \Phi_N[t] \, \purpledot  - \Phi_K[t] \, \purpledot \|_{C_t^0 \cH_x^{\beta-\kappa}([\tau,j\tau]\times \bT^3)} \\
&\leq \big( 2 D_{j-1} +  C \exp( C(A+B_{j-1}+T)^C) \big)K^{-\eta^\prime}. 
\end{aligned}
\end{equation}

\emph{Step 2.4: Gluing.} In this step, we ``glue'' together our information on $[0,2\tau]$ (from local well-posedness) and $[\tau,j\tau]$ (from the previous step).

Since $\, \purpledot \in \cL(A,2\tau)$ (as in Proposition \ref{global:prop_lwp}), the function $w_N$ uniquely determined by
\begin{equation*}
\Phi_N(t) \, \purpledot =  \<1p>(t)+ \<3DNp>(t) + w_N(t)
\end{equation*}
satisfies 
\begin{equation*}
\| w_N \|_{\X{s_1}{b}([0,2\tau]\times \bT^3)}, ~  \sum_{L_1\sim L_2} \| P_{L_1} \<1p>  \cdot P_{L_2} w_{N} \|_{L_t^{2} H_x^{-4\delta_1}([0,2\tau]\times \bT^3)} \leq A. 
\end{equation*}
Furthermore, 
\begin{equation*}
\| \Phi_N[t] \, \purpledot  - \Phi_K[t] \, \purpledot \|_{C_t^0 \cH_x^{\beta-\kappa}([0,2\tau]\times \bT^3)} \leq A K^{-\eta^\prime}.
\end{equation*}
Together with \eqref{global:eq_gwp_p_wN_bounds}, \eqref{global:eq_gwp_p_d5}, and the gluing lemma (Lemma \ref{tools:lem_gluing}), which is only needed for the frequency-based $\X{s_1}{b}$-space, we obtain that 
\begin{equation}\label{global:eq_gwp_p_wN_bounds_final}
\| w_N \|_{\X{s_1}{b}([0,j\tau]\times \bT^3)}, ~  \sum_{L_1\sim L_2} \| P_{L_1} \<1p>  \cdot P_{L_2} w_{N} \|_{L_t^{2} H_x^{-4\delta_1}([0,j\tau]\times \bT^3)} \leq C \tau^{\frac{1}{2}-b} T^\alpha AB_{j-1}. 
\end{equation}
and 
\begin{equation}\label{global:eq_gwp_p_d6}
\| \Phi_N[t] \, \purpledot  - \Phi_K[t] \, \purpledot \|_{C_t^0 \cH_x^{\beta-\kappa}([0,j\tau]\times \bT^3)} \leq \big( 2 D_{j-1} +  C \exp( C(A+B_{j-1}+T)^C) \big)K^{-\eta^\prime}. 
\end{equation}
\emph{Step 2.5: Choosing $B_j$ and $D_j$.} 
Based on \eqref{global:eq_gwp_p_wN_bounds_final} and \eqref{global:eq_gwp_p_d6}, we now define 
\begin{equation*}
B_j \defe  C \tau^{\frac{1}{2}-b} T^\alpha AB_{j-1} \quad \text{and} \quad D_j \defe 2 D_{j-1} +  C \exp( C(A+B_{j-1}+T)^C). 
\end{equation*}

\emph{Step 3: Finishing up.} 
We recall that $ 1/2\leq  A \tau^{b_+-b} \leq 1 $, $J=T/\tau \sim TA^{\frac{1}{b_+-b}}$, $B_1=A$, and $D_1=A$. After increasing $C$ if necessary, we obtain that 
\begin{equation}
B_J \leq C \exp(C(A+T)^C)\leq B \quad \text{and} \quad D_J \leq C \exp( C(A+B_J+T)^C) \leq D. 
\end{equation}
This implies $\cE(B_J,D_J,J\tau) \subseteq \cE_K(B,D,T)$. By iterating \eqref{global:eq_gwp_p_step} and using the base case \eqref{global:eq_gwp_p_base}, we obtain (after decreasing $\zeta$) that
\begin{equation*}
\mup_M(\cE_K(B,D,T)) \geq \mup_M(\cE_K(B_J,D_J,J\tau)) \geq 1 - T \zeta^{-1}  \exp(-\zeta A^{\zeta}). 
\end{equation*}
This completes the proof. 
\end{proof}

In Proposition \ref{global:prop_gwp}, we obtained a quantitative global well-posedness result. In particular, we obtained (almost) explicit bounds on the growth of $w_N$, which are of independent interest. In the proof of Theorem \ref{theorem:gwp_invariance}, however, a softer statement is sufficient, which we isolate in Corollary \ref{global:cor_gwp} below. 

\begin{corollary}\label{global:cor_gwp}
Let $T\geq 1$, let $\theta>0$, and $K\geq 1$. Then, we define a closed subset of $\cH_x^{-1/2-\kappa}(\bT^3)$ by 
\begin{equation}\label{global:eq_SK}
\cS_K(T,\theta)\defe \Big\{ \, \purpledot \in \cH_x^{-1/2-\kappa}(\bT^3)\colon \sup_{N_1,N_2\geq K} \| \Phi_{N_1}[t] \, \purpledot- \Phi_{N_2}[t]\, \purpledot \|_{C_t^0 \cH_x^{\beta-\kappa}([-T,T]\times \bT^3)} \leq \theta \Big\}
\end{equation}
Furthermore, we define the event 
\begin{equation}\label{global:eq_S}
\cS \defe  \bigcap_{T\in \mathbb{N}} \bigcap_{\theta \in \mathbb{Q}_{>0}} \bigcup_{K\geq 1} \cS_K(T,\theta).
\end{equation}
Then, it holds that 
\begin{equation}\label{global:eq_S_prob}
\lim_{K,M\rightarrow \infty} \mup_M(\cS_K(T,\theta)) = 1 \qquad \text{and} \qquad \mup_\infty(\cS)=1. 
\end{equation}
\end{corollary}

\begin{proof}
We first prove the identity $\lim_{K,M\rightarrow \infty} \mup_M(\cS_K(T,\theta)) = 1 $. Using the time-reflection symmetry, it suffices to prove the statement with $\cS_K(T,\theta)$ replaced by 
\begin{equation*}
\cS_K^+(T,\theta)\defe \Big\{ \, \purpledot \in \cH_x^{-1/2-\kappa}(\bT^3)\colon \sup_{N_1,N_2\geq K} \| \Phi_{N_1}[t] \, \purpledot- \Phi_{N_2}[t]\, \purpledot \|_{C_t^0 \cH_x^{\beta-\kappa}([0,T]\times \bT^3)} \leq \theta \Big\}.
\end{equation*}
For any fixed $T,A,B,D\geq 1$ satisfying \eqref{global:eq_BD_condition} and $\theta>0$, we have for all sufficiently large $K,L\geq 1$ satisfying $K\geq L$ that
\begin{equation*}
\cS_K^+(T,\theta) \supseteq \cE_L(B,D,T), 
\end{equation*}
where $\cE_L(B,D,T)$ is as in Proposition \ref{global:prop_gwp}. Thus, 
\begin{equation*}
\lim_{K,M\rightarrow \infty} \mup_M(\cS_K(T,\theta))  \geq \liminf_{M\rightarrow \infty} \mup_M(\cE_L(B,D,T)) \geq 1 - \zeta^{-1} T \exp(\zeta A^\zeta). 
\end{equation*} 
After letting $A\rightarrow \infty$, this yields the first identity in \eqref{global:eq_S_prob}. \\
Using Theorem \ref{theorem:measures}, we have that a subsequence of $\mup_M$ converges weakly to $\mup_\infty$. Since $\cS_K(T,\theta)$ is closed, this implies
\begin{equation*}
1 = \lim_{K,M\rightarrow \infty} \mup_M(\cS_K(T,\theta)) \leq \liminf_{K\rightarrow \infty} \mup_{\infty}(\cS_K(T,\theta)) \leq \mup_\infty\Big( \bigcup_{K\geq 1} \cS_K(T,\theta)\Big). 
\end{equation*} 
This yields the second identity in \eqref{global:eq_S_prob}. 
\end{proof}

\subsection{Invariance}\label{section:invariance}

In this subsection, we complete the proof of Theorem \ref{theorem:gwp_invariance}. The global well-posedness follows from Corollary \ref{global:cor_gwp} and it remains to prove the invariance. Our argument closely resembles the proof of invariance for the one-dimensional  nonlinear Schr\"{o}dinger equation by Bourgain \cite{Bourgain94}. The only difference is that we work with the expectation of test functions instead of probabilities of sets, since they are more convenient for weakly convergent measures. 

\begin{proof}[Proof of Theorem \ref{theorem:gwp_invariance}:]
The global well-posedness follows directly from Corollary \ref{global:cor_gwp}. Thus, it remains to prove the invariance of the Gibbs measure $\mup_\infty$.

Let $t\in \bR$ be arbitrary. In order to prove that $\Phi_\infty[t]_\# \mup_\infty = \mup_\infty$, it suffices to prove for all bounded Lipschitz functions $f \colon \cH_x^{-1/2-\kappa}(\bT^3)\rightarrow \bR$ that 
 \begin{equation}\label{global:eq_inv_p1}
 \bE_{\mupscript_\infty}\big[ f( \Phi_\infty[t]\, \purpledot\,)\big] =  \bE_{\mupscript_\infty}\big[f(\, \purpledot\,) \big].
 \end{equation}
 We first rewrite the left-hand side of \eqref{global:eq_inv_p1}. Using the global well-posedness and dominated convergence, we have that 
 \begin{equation*}
  \bE_{\mupscript_\infty}\big[ f( \Phi_\infty[t]\, \purpledot\,)\big]  = \lim_{N\rightarrow \infty}   \bE_{\mupscript_\infty}\big[ f( \Phi_N[t]\, \purpledot\,)\big].
 \end{equation*}
 Using the weak convergence of $\mup_M$ to $\mup_\infty$ (from Theorem \ref{theorem:measures}) and the continuity of $\Phi_N[t]$ (for a fixed $N$), we have that 
 \begin{equation*}
  \lim_{N\rightarrow \infty}   \bE_{\mupscript_\infty}\big[ f( \Phi_N[t]\, \purpledot\,)\big] =   \lim_{N\rightarrow \infty}   \Big( \lim_{M\rightarrow \infty}\bE_{\mupscript_M}\big[ f( \Phi_N[t]\, \purpledot\,)\big] \Big). 
 \end{equation*}
 We now turn to the right-hand side of \eqref{global:eq_inv_p1}. Using the weak convergence of $\mup_M$ to $\mup_\infty$ and the invariance of $\mup_M$ under $\Phi_M[t]$, we obtain that 
 \begin{equation*}
 \bE_{\mupscript_\infty}\big[f(\, \purpledot\,) \big] = \lim_{M\rightarrow \infty} \bE_{\mupscript_M}\big[f(\, \purpledot\,) \big] = \lim_{M\rightarrow \infty } 
  \bE_{\mupscript_M}\big[ f( \Phi_M[t]\, \purpledot\,)\big]. 
 \end{equation*}
 Combining the last three identities, we can reduce \eqref{global:eq_inv_p1} to 
 \begin{equation}\label{global:eq_inv_p2}
 \limsup_{N,M\rightarrow \infty} \Big| \bE_{\mupscript_M}\big[ f( \Phi_N[t]\, \purpledot\,)\big] - \bE_{\mupscript_M}\big[ f( \Phi_M[t]\, \purpledot\,)\big] \Big| = 0 . 
 \end{equation}
 We now let $T\geq 1$ be such that $t\in [-T,T]$, let $\theta>0$, and let $K\geq 1$. We also let $\cS_K(T,\theta)$ be as in Corollary \ref{global:cor_gwp}. Then, we have that
 \begin{align*}
  &\limsup_{N,M\rightarrow \infty} \Big| \bE_{\mupscript_M}\big[ f( \Phi_N[t]\, \purpledot\,)\big] - \bE_{\mupscript_M}\big[ f( \Phi_M[t]\, \purpledot\,)\big] \Big| \\
  \leq&  \sup_{N,M\geq K} \Big| \bE_{\mupscript_M}\big[ f( \Phi_N[t]\, \purpledot\,)\big] - \bE_{\mupscript_M}\big[ f( \Phi_M[t]\, \purpledot\,)\big] \Big| \\
  \leq& 
  \sup_{N,M\geq K}  \bE_{\mupscript_M}\Big[
   1\big\{ \, \purpledot \in \cS_K(T,\theta)\big\} \Big|  f( \Phi_N[t]\, \purpledot\,)  -  f( \Phi_M[t]\, \purpledot\,)\Big|  \Big] \\
&+    \sup_{N,M\geq K}  \bE_{\mupscript_M}\Big[
      1\big\{ \, \purpledot \not \in \cS_K(T,\theta)\big\} \Big|  f( \Phi_N[t]\, \purpledot\,) -  f( \Phi_M[t]\, \purpledot\,)\Big| \Big] \\
\leq &\operatorname{Lip}(f) \cdot \theta + 2 \| f\|_{\infty} \sup_{M\geq K} \mup_M( \cH_x^{-1/2-\kappa}\backslash \cS_K(T,\theta)). 
 \end{align*}
 In the last line, $\operatorname{Lip}(f)$ is the Lipschitz-constant of $f$ and $\| f\|_\infty$ is the supremum of $f$. Using Corollary \ref{global:cor_gwp}, we obtain the estimate \eqref{global:eq_inv_p2} by first letting $K\rightarrow \infty$ and then letting $\theta\rightarrow 0$. 
\end{proof}

\subsection{Structure and stability theory}\label{section:stability}

In this subsection, we provide the ingredients used in the proof of global well-posedness (Proposition \ref{global:prop_gwp}). As described in the introduction, we will further split this subsection into four parts. 
\subsubsection{Structured local well-posedness}

In Proposition \ref{lwp:prop}, we obtained a structured local well-posedness result in terms of $\, \bluedot\, $ and  $\bP$. In Corollary \ref{lwp:cor_unstructured}, we already used Proposition \ref{lwp:prop} to prove the local existence of the limiting dynamics on the support of the Gibbs measure $\mup_\infty$, but did not obtain any structural information on the solution. We now remedy this defect, and obtain a structured local well-posedness result even on the support of the Gibbs measure. \\

The statement of the proposition differs slightly from the earlier Proposition \ref{lwp:prop} for two reasons: First, we formulate the result closer to the assumptions in the stability theory (Proposition \ref{local:prop_stability} and Proposition \ref{global:prop_stability}), which is useful in the globalization argument. Second, using the organization of this paper, it would be cumbersome to define the para-controlled component of $\Phi_N(t)\, \purpledot$ intrinsically through $\,\purpledot\,$, i.e., without relying on the ambient objects.

\begin{proposition}[Structured local well-posedness w.r.t. the Gibbs measure]\label{global:prop_lwp}
Let $A\geq 1$, let $0<\tau<1$, let $\alpha>0$ be a sufficiently large absolute constant, and let  $\zeta=\zeta(\epsilon,s_1,s_2,\kappa,\eta,\eta^\prime,b_+,b)>0$ be sufficiently small. We denote by $\purpledot\,$ a generic element of $\cH_x^{-1/2-\kappa}(\bT^3)$ and by $\LAtau$ the Borel subset of $\cH_x^{-1/2-\kappa}(\bT^3)$ defined by the following conditions:
\begin{enumerate}[(I)]
\item \label{lwp:item_P1_Gibbs} For any $N\geq 1$, the solution of \eqref{eq:nlw_N} with initial data $\purpledot$ exists on $[-\tau,\tau]$. 
\item \label{lwp:item_P2_Gibbs} For all  $N\geq 1$, there exist (a unique) $w_N\in \X{s_1}{b}([0,\tau])$ such that 
\begin{equation*}
\Phi_N(t)\,  \purpledot = \<1p>(t)+ \<3DNp>(t) + w_N(t). 
\end{equation*}
Furthermore, we have the bounds
\begin{align*}
   \| w_N \|_{\X{s_1}{b}([0,\tau])}\leq A \quad \text{and} \quad 
\sum_{L_1\sim L_2} \| P_{L_1} \, \<1p> \cdot  P_{L_2} w_N \|_{L_t^2 H_x^{-4\delta_1}([0,\tau]\times \bT^3)} \leq A. 
\end{align*}
\item\label{lwp:item_P3_Gibbs} It holds for all \( N,K \geq 1 \) that 
\begin{equation*}
\|  \Phi_N[t]\,  \purpledot - \Phi_K[t]\,  \purpledot\, \|_{C_t^0 \cH_x^{\beta-\kappa}([0,\tau]\times \bT^3)} \leq A \min(N,K)^{-\eta^\prime}. 
\end{equation*}
\item\label{lwp:item_P4_Gibbs}  It holds for all \( N,K \geq 1 \) and $T\geq 1$ that 
\begin{equation*}
\| w_K[\tau] \|_{\mathscr{Z}([0,T], \scriptp;\tau,N,K)} \leq A T^{\alpha}  \\
\end{equation*}
and 
\begin{equation*}
\| w_N[\tau] - w_K[\tau] \|_{\mathscr{Z}([0,T], \scriptp;\tau,N,K)} \leq A T^{\alpha} \min(N,K)^{-\eta^\prime}. 
\end{equation*}
\end{enumerate}
If  $A \tau^{b_+-b} \leq  1 $, then \( \LAtau \) has high probability under $\mup_M$ for all $M\geq 1$ and it holds that 
\begin{equation}\label{global:eq_lwp_probabilistic}
\mup_M(\LAtau) \geq 1 -  \zeta^{-1}\exp(-\zeta A^{\zeta}). 
\end{equation}
\end{proposition}

\begin{remark}\label{global:rem_organization_lwp}
Since we prove multilinear estimates for $\,\purpledot\,$ instead of $\,\bluedot\,$ in Section \ref{section:ftg}, a different incarnation of this paper may omit Proposition \ref{lwp:prop} and instead proof Proposition \ref{global:prop_lwp} directly. The author believes that our approach illustrates an interesting conceptual point: The singularity of the Gibbs measure does not enter heavily into the construction of the local limiting dynamics (see Corollary \ref{lwp:cor_unstructured}), but does affect the global theory. We believe, however, that this would be different for the cubic nonlinear wave equation. The reason is an additional renormalization in the construction of the $\Phi^4_3$-model (see e.g. \cite[Lemma 5: Step 3]{BG18}). 
\end{remark}

We recall that the $\Z$-norm appearing in (IV) is defined in Definition \ref{global:definition_structured}.

\begin{proof}[Proof of Proposition \ref{global:prop_lwp}:] 
By using  Theorem \ref{theorem:measures} and adjusting the value of $\zeta$, it suffices to prove the probabilistic estimate \eqref{global:eq_lwp_probabilistic} with the Gibbs measure $\mup_M$ replaced by the reference measure $\nup_M$. Using the representation of the reference measure from Theorem \ref{theorem:measures}, it holds that 
\begin{equation*}
\nup_M = \Law_{\bP} \big( \bluedot + \reddotM \big) . 
\end{equation*}
By applying this identity to the Borel set $\LAtau$, we obtain that 
\begin{equation*}
\nup_M(\LAtau) = \bP(  \bluedot + \reddotM \in \LAtau). 
\end{equation*}
Let $B=c A^{c}\leq A$, where $c= c(\epsilon,s_1,s_2,\kappa,\eta,\eta^\prime,b_+,b)>0$ is sufficiently small.  Let $\LdagMB \subseteq \Omega$ be as in Proposition \ref{lwp:prop}. We now show that 
\begin{equation}\label{global:eq_lwp_p1}
\bP\Big( \big\{ \, \bluedot + \reddotM \not \in \LAtau \big\} \medcap \LdagMB \Big) \leq \frac{1}{2} \zeta^{-1} \exp(-\zeta A^{\zeta}). 
\end{equation}
The property \eqref{lwp:item_P1} in Proposition \ref{lwp:prop} directly implies its counterpart. The main part of the argument lies in proving  \eqref{lwp:item_P2_Gibbs}. Instead of \eqref{lwp:item_P2_Gibbs}, we currently only have the property
\begin{itemize}
\item[(ii):] \label{lwp:item_global_p2} For all  $N\geq 1$, there exist $w_N^\prime \in \X{s_1}{b}([0,\tau])$, $H_N^\prime \in \modulation([0,\tau])$, and $Y_N^\prime \in \X{s_2}{b}([0,\tau])$, such that for all $t\in [0,\tau]$
\begin{equation*}
\Phi_N(t)\,  \purpledot = \<1b>(t) + \<3DN>(t) + w_N^\prime(t) ~~  \text{and}~~   w_N^\prime(t) = P_{\leq N} \Duh\big[ \Para(H_N^\prime ,  P_{\leq N} \<1b>)\big](t) + Y_N^\prime(t). 
\end{equation*}
Furthermore, we have the bounds
\begin{align*}
   \| w_N^\prime \|_{\X{s_1}{b}([0,\tau])}, \| H_N^\prime \|_{\modulation([0,\tau])}, \| Y_N^\prime \|_{\X{s_2}{b}([0,\tau])}  &\leq B \quad \text{and} \\
\sum_{L_1\sim L_2} \| P_{L_1} \, \<1b> \cdot P_{L_2} w_N^\prime \|_{L_t^2 H_x^{-4\delta_1}([0,\tau]\times \bT^3)} &\leq B. 
\end{align*}
\end{itemize}
Comparing \eqref{lwp:item_P2} and \eqref{lwp:item_P2_Gibbs}, this forces us to take
\begin{equation}\label{global:eq_lwp_p2}
w_N = \, \<1b>\, - \, \<1p> \, + \<3DNb> - \<3DNp> + w_N^\prime. 
\end{equation}
We now have to prove that the right-hand side of \eqref{global:eq_lwp_p2} satisfies the estimates in \eqref{lwp:item_P2_Gibbs}. Due to the decomposition $\, \purpledot = \, \bluedot + \reddotM$, we have that 
\begin{equation*}
\<1b> - \<1p> = - \<1r>. 
\end{equation*}
Using Theorem \ref{theorem:measures}, we have outside a set of probability $\lesssim \exp(-c B^{\frac{2}{k}}) $ under $\bP$ that 
\begin{equation*}
\big\| \, \<1r>\, \big\|_{\X{s_2}{b}([0,\tau])} \lesssim B. 
\end{equation*}

 Using Proposition \ref{ftg:prop_cubic}, we have outside an event with probability $\leq \zeta^{-1} \exp(-\zeta B^{\zeta})$ under $\bP$ that 
\begin{equation}\label{global:eq_temp1}
\<3DN> = \<3DNp> + I[\Para(H^{(3)}_N, P_{\leq N} \<1p>)] + Y_N^{(3)},
\end{equation}
where  $H_N^{(3)} \in \modulation([0,\tau])$ and $Y_N^{(3)} \in \X{s_2}{b}([0,\tau])$ satisfy
\begin{equation*}
\| H^{(3)}_N \|_{\modulation([0,\tau])} \leq B  \qquad \text{and} \qquad \| Y^{(3)}_N \|_{\X{s_2}{b}([0,\tau])} \leq B. 
\end{equation*}
To ease the reader's mind, we mention that the proof of Proposition \ref{ftg:prop_cubic} is based on the algebraic identity
\begin{align*}
 \<3DNp>  &=  \<3DN> +  2 \<3DNbrb> + \<3DNrrb> + \<3DNbbr> + 2 \<3DNbrr> + \<3DNrrr>,
\end{align*}
which uses mixed cubic stochastic objects. Finally, we have that 
\begin{align*}
w_N^\prime &= P_{\leq N} \Duh \big[ \PCtrl( H_N^\prime, P_{\leq N} \, \<1b>)\big] + Y_N^\prime \\
&= P_{\leq N} \Duh \big[ \PCtrl( H_N^\prime, P_{\leq N} \, \<1p>)\big] - P_{\leq N} \Duh \big[ \PCtrl( H_N^\prime, P_{\leq N} \, \<1r>)\big] + Y_N^\prime. 
\end{align*}
Using the inhomogeneous Strichartz estimate (Lemma \ref{tools:lem_inhomogeneous_strichartz}) and Lemma \ref{para:lem_basic},  we have that
\begin{align*}
 &\| P_{\leq N} \Duh \big[ \PCtrl( H_N^\prime, P_{\leq N} \<1r>) \big] \|_{\X{s_2}{b}([0,\tau])} 
 \lesssim  \| \PCtrl( H_N^\prime, P_{\leq N}  \<1r>\,)  \|_{L_t^\infty H_x^{s_2-1}([0,\tau] \times \bT^3)} \\ 
 &\lesssim \| H_N^\prime \|_{\LM([0,\tau])}  \| \, \<1r> \, \|_{L_t^\infty H_x^{s_2-1+8\epsilon}([0,\tau] \times \bT^3)} 
 \lesssim  B^2. 
\end{align*}
Thus, 
\begin{equation}\label{global:eq_temp2}
w_N = P_{\leq N} \Duh \big[ \PCtrl( H_N, P_{\leq N} \<1p>\,)\big] +Y_N,
\end{equation}
where 
\begin{equation*} 
H_N=H_N^\prime+ H_N^{(3)} \qquad \text{and} \qquad Y_N = Y_N^\prime -\, \<1r>\, + Y_N^{(3)} - P_{\leq N} \Duh \big[ \PCtrl(H_N^\prime, P_{\leq N} \<1r>\,) \big]
\end{equation*}
satisfy $\|H_N \|_{\LM([0,\tau])}, \|Y_N\|_{\X{s_2}{b}([0,\tau])} \lesssim B^2$. 
Using Lemma \ref{ftg:lem_type_conversion}, we also obtain that
\begin{equation*}
   \| w_N \|_{\X{s_1}{b}([0,\tau])} \lesssim B^5 \qquad \text{and} \qquad 
\sum_{L_1\sim L_2} \| P_{L_1} \, \<1p> \cdot  P_{L_2} w_N \|_{L_t^2 H_x^{-4\delta_1}([0,\tau]\times \bT^3)} \lesssim B^5. 
\end{equation*}
Inserting our choice of $B$, this completes the proof of \eqref{lwp:item_P2_Gibbs}. \\

The statement \eqref{lwp:item_P3_Gibbs} directly follows from \eqref{lwp:item_P3} in Proposition \ref{lwp:prop}. It now remains to prove \eqref{lwp:item_P4_Gibbs}. We emphasize that $T\geq 1$ is arbitrary, which will be useful in the stability theory below. We focus on the estimate for the difference, since the proof of the estimate for $w_K[\tau]$ is easier (but similar). Using Lemma \ref{ftg:lem_structured_perturbation}, we may restrict to $\, \bluedot \in \Theta_{\blue}^{\tsp}(B,T)$ and $\, \reddot \in \Theta_{\red}^{\tsp}(B,T)$. Then, we can replace the estimates in $\Z([0,T], \Zp; t_0,N,K)$ by estimates in $ \Z([0,T], \Zb; t_0,N,K)$. After rearranging \eqref{global:eq_temp2}, we have that 
\begin{equation*}
w_N = P_{\leq N} \Duh \big[ \PCtrl( H_N^\prime + H_N^{(3)}, P_{\leq N} \, \<1b>)\big] + Y_N^\prime - \<1r> + Y_N^{(3)} + P_{\leq N} \Duh \big[ \PCtrl(H_N^{(3)}, P_{\leq N} \<1r>)\big]. 
\end{equation*}
Thus, we obtain that 
\begin{equation*}
w_N[\tau]-w_K[\tau] = \Zbox_{N,K}[\tau] + \Zcirc_{N,K}[\tau],
\end{equation*}
where 
\begin{equation*}
\Zbox_{N,K}[\tau]
\defe P_{\leq N} \Duh \big[ \PCtrl( H_N^\prime + H_N^{(3)}, P_{\leq N} \, \<1b>)\big][\tau] 
- P_{\leq K} \Duh \big[ \PCtrl( H_K^\prime + H_K^{(3)}, P_{\leq K} \, \<1b>)\big][\tau] . 
\end{equation*}
and
\begin{align*}
\Zcirc_{N,K}[\tau] &\defe  Y_N^\prime - Y_K^\prime
+ Y_N^{(3)} - Y_K^{(3)} 
+ P_{\leq N} \Duh \big[ \PCtrl(H_N^{(3)}, P_{\leq N} \<1r>)\big]- P_{\leq K} \Duh \big[ \PCtrl(H_K^{(3)}, P_{\leq K} \<1r>)\big]. 
\end{align*}
The desired estimate then follows from the frequency-localized version of the multi-linear master estimate (Prop \ref{local:prop_master}), \eqref{lwp:item_P3} in Proposition \ref{lwp:prop}, and Proposition \ref{ftg:prop_cubic}. 
\end{proof}

\subsubsection{Structure and time-translation}

In the globalization argument, we use the invariance of the truncated Gibbs measures under the truncated flows to transform our bounds from the time-interval $[0,(j-1)\tau]$ to the time-interval $[\tau,j\tau]$. As the reader saw in the proof of Proposition \ref{global:prop_gwp}, however, the structural bounds are now phrased in terms of $\, \greendot = \Phi_M[\tau]\, \purpledot$. The next proposition translates the structural bounds back into $\purpledot$. 

\begin{proposition}[Structure and time-translation]\label{global:prop_structure_time}
Let $A\geq 1$, let $T\geq 1$, let $0<\tau\leq 1$, let $j \in \mathbb{N}$ satisfy $j\tau\leq T$, let $\alpha>0$ be a sufficiently large absolute constant, and let $\zeta=\zeta(\epsilon,s_1,s_2,\kappa,\eta,\eta^\prime,b_+,b)>0$ be sufficiently small. There exists a Borel set $\cS^{\stime}(A,T,\tau)\subseteq \cL(A,\tau)$ satisfying
\begin{equation}\label{global:eq_time_probabilistic_estimate}
\mup_M(\cS^{\stime}(A,T,\tau)) \geq 1 - \zeta^{-1} \exp( - \zeta A^\zeta) 
\end{equation}
for all $M\geq 1$ and such that the following holds for all $\, \purpledot \in \cS^{\stime}(A,T,\tau)$: 

Let $N, K \geq 1$, let $B\geq 1$, and define $\, \greendot = \Phi_K[\tau]\, \purpledot$. Let $w^{\grn}_{N,K}\in \X{s_1}{b}([0,(j-1)\tau])$ satisfy
\begin{itemize}
\item[(A1)] Global structured bounds in $\greendot$: 
\begin{equation*}
\| w^{\grn}_{N,K}\|_{\X{s_1}{b}([0,(j-1)\tau])} \leq B \quad \text{and} \quad
 \sum_{L_1\sim L_2} \| P_{L_1} \<1g>  \cdot P_{L_2} w^{\grn}_{N,K} \|_{L_t^{2} H_x^{-4\delta_1}([0,(j-1)\tau]\times \bT^3)} \leq B. 
\end{equation*}
\end{itemize} 
Define $w_{N,K}\colon [\tau,j\tau] \times \bT^3\rightarrow \bR$ through the identity
\begin{equation}\label{global:eq_green_to_purple}
\<1g>(t-\tau) + \<3DNg>(t-\tau) + w^{\grn}_{N,K}(t-\tau) =  \<1p>(t)+ \<3DNp>(t) - \<3DNKtaup>(t) + w_{N,K}(t). 
\end{equation}
Then, we obtain the following conclusion regarding $w_{N,K}$. 
\begin{itemize}
\item[(C1)] Incomplete structured global bounds in $\, \purpledot\,$: 
\begin{equation*}
\| w_{N,K}\|_{\X{s_1}{b}([\tau,j\tau])} \leq T^\alpha A B \quad \text{and} \quad
 \sum_{L_1\sim L_2} \| P_{L_1} \<1p>  \cdot P_{L_2} w_{N,K} \|_{L_t^{2} H_x^{-4\delta_1}([\tau,j\tau]\times \bT^3)} \leq T^\alpha A B. 
 \end{equation*}
\end{itemize}

\end{proposition}

\begin{remark}
The superscript ``$\grn$'' in $w_{N,K}^{\grn}$ stands for ``green'', which is motivated by the identity \eqref{global:eq_green_to_purple}. We refer in the conclusion to ``incomplete structured global bounds" since the right-hand side in \eqref{global:eq_green_to_purple} does not yet have the desired form. The partial cubic stochastic object
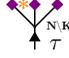
\begin{equation*}
\<3DNKtaup>
\end{equation*}
is subtracted from it and hence we regard the structure as incomplete. 
\end{remark}

\begin{proof}
Before we turn to the analytical and probabilistic estimates, we discuss the definition and Borel measurability of $\cS^{\stime}(A,T,\tau)$. 
We let $\cS^{\stime}(A,T,\tau)$ be the intersection of $\cL(A,\tau)$ with the set of $\purpledot\in \cH^{-1/2-\kappa}_x$ satisfying the implication (A1)$\rightarrow$(C1) for all $N,K,B$, and $w^{\grn}_{N,K}$. For fixed parameters and  a fixed function $w^{\grn}_{N,K}$, the set of $\purpledot \in \cH_x^{-1/2-\kappa}$ satisfying (A1) and/or (C1) is closed and hence Borel measurable. Using a separability argument, it suffices to require the implication (A1)$\rightarrow$(C1) for countably many $w^{\grn}_{N,K}$, which yields the measurability of $\cS^{\stime}(A,T,\tau)$.

We now turn to the analytical and probabilistic estimates. If $\, \purpledot \in \cL(A,\tau)$, it follows from \eqref{lwp:item_P2_Gibbs} and \eqref{lwp:item_P4_Gibbs} from Proposition \ref{global:prop_lwp} that 
\begin{equation*}
\greendot = \<1p>[\tau] + \<3DKp>[\tau] + Z_K[\tau],
\end{equation*}
where the remainder $Z_K[\tau]$ satisfies
\begin{equation*}
\| Z_K[\tau]\|_{\Z([0,T],\scriptp;\tau,N,K)} \leq A T^\alpha. 
\end{equation*}
By applying the linear propagator to $\,\greendot$, we obtain for all $t\geq \tau$ that 
\begin{equation}\label{global:eq_time_linear}
\<1g>(t-\tau) = \<1p>(t) + \<3DKptau>(t)+Z_K(t),
\end{equation}
where we recall from \eqref{global:eq_modified_cubic_objects} that 
\begin{equation*}
 \<3DKptau>(t) = \Duh\Big[ 1_{[0,\tau]} \<3Kp>\Big](t)
\end{equation*}
Regarding the cubic stochastic object, we have that 
\begin{equation}\label{global:eq_time_cubic}
\begin{aligned}
\<3DNg>(t-\tau)&= \Duh \Big[ 1_{[\tau,\infty)} \<3Ng>(\cdot -\tau)\Big] (t) \\
			&= \Duh \Big[ 1_{[\tau,\infty)} \<3Np>\Big](t)  + \Duh \Big[ 1_{[\tau,\infty)} \Big(\<3Ng>(\cdot -\tau)- \<3Np>(\cdot)\Big)\Big] (t)
\end{aligned}
\end{equation}
Combining the algebraic identity 
\begin{equation*}
\Duh\Big[ 1_{[0,\tau]} \<3Kp>\Big](t) +  \Duh \Big[ 1_{[\tau,\infty)} \<3Np>\Big](t)   = \<3DNp>(t) - \<3DNKtaup>(t)
\end{equation*}
with \eqref{global:eq_time_linear} and \eqref{global:eq_time_cubic}, it follows that 
\begin{equation}\label{global:eq_time_w}
\begin{aligned}
w_{N,K}(t) &= w^{\grn}_{N,K}(t) + Z_K(t)  + \Duh \Big[ 1_{[\tau,\infty)} \Big(\<3Ng>(\cdot -\tau)- \<3Np>(\cdot)\Big)\Big] (t). 
\end{aligned}
\end{equation}
Equipped with the identity \eqref{global:eq_time_w} for $w_{N,K}$, it remains to prove the conclusion (C1) on an event satisfying \eqref{global:eq_time_probabilistic_estimate}. The second and third summand in \eqref{global:eq_time_w} can be treated using Lemma \ref{ftg:lem_type_conversion}, Proposition \ref{ftg:prop_master} (combined with \eqref{global:eq_time_linear}), and Lemma \ref{ftg:lem_multilinear_Z}.
 Thus, it remains to prove (C1) for the first summand in \eqref{global:eq_time_w}. Using \eqref{global:eq_time_linear}, we have that
 \begin{align}
  &\sum_{L_1\sim L_2} \| P_{L_1} \<1p>(t) \cdot P_{L_2} w^{\grn}_{N,K}(t-\tau) \|_{L_t^{2} H_x^{-4\delta_1}([\tau,j\tau]\times \bT^3)} \notag \\
  &\leq \sum_{L_1\sim L_2} \| P_{L_1} \<1g>(t) \cdot P_{L_2} w^{\grn}_{N,K}(t) \|_{L_t^{2} H_x^{-4\delta_1}([0,(j-1)\tau]\times \bT^3)} 
  \label{global:eq_time_p1}\\
  &+ \sum_{L_1\sim L_2} \| P_{L_1}  \<3DKptau>(t) \cdot P_{L_2} w^{\grn}_{N,K}(t-\tau) \|_{L_t^{2} H_x^{-4\delta_1}([\tau,j\tau]\times \bT^3)} 
  \label{global:eq_time_p2}\\
  &+ \sum_{L_1\sim L_2} \| P_{L_1}  Z_K(t) \cdot P_{L_2} w^{\grn}_{N,K}(t-\tau) \|_{L_t^{2} H_x^{-4\delta_1}([\tau,j\tau]\times \bT^3)}. 
  \label{global:eq_time_p3}
 \end{align}
 The first term \eqref{global:eq_time_p1} can be bounded using assumption (A1). The second term \eqref{global:eq_time_p2} is bounded by Corollary \ref{ftg:cor_cubic}, and the third term \eqref{global:eq_time_p3} is bounded by Lemma \ref{phy:lem_bilinear_tool}. 
\end{proof}

\subsubsection{Structure and the cubic stochastic object}

In Proposition \ref{global:prop_structure_time} above, the right-hand side of  \eqref{global:eq_green_to_purple} does not have the desired structure. In the next proposition, we will show that adding the ``partial'' cubic stochastic object $\<3DNKtaup>$ only leads to a small error in the nonlinear wave equation.  

\begin{proposition}[Structure and the cubic stochastic object:]\label{global:prop_structure_cubic}
Let $T\geq 1$, let $A\geq 1$, let $0<\tau<1$, let $\alpha>0$ be a sufficiently large absolute constant, and let $\zeta=\zeta(\epsilon,s_1,s_2,\kappa,\eta,\eta^\prime,b_+,b)>0$ be sufficiently small. Then, there exists a Borel set $\cS^{\cub}(A,T,\tau)\subseteq \cH_x^{-1/2-\kappa}(\bT^3)$ satisfying 
\begin{equation*}
\mup_M(\cS^{\cub}(A,T,\tau)) \geq 1 - \zeta^{-1} \exp(-\zeta A^{\zeta}) 
\end{equation*}
for all $M\geq 1$ and such that the following holds for all $\, \purpledot\, \in \cS^{\cub}(A,T,\tau)$:\\
Let $N,K\geq1$, let $B\geq 1$, let $j\in \mathbb{N}$, let  $\cJ=[\tau,j\tau]\subseteq [0,T]$, and let $u_{N,K}\colon \cJ \times \bT^3 \rightarrow \bR$. Furthermore, we make the following assumptions:
\begin{itemize}
\item[(A1)] Incomplete structure: There exists a $w_{N,K}(t)\in \X{s_1}{b}(\cJ)$ satisfying all $t\in \cJ$ the identity
\begin{equation*}
u_{N,K}(t)=\<1p>(t)+ \<3DNp>(t) - \<3DNKtaup>(t) + w_{N,K}(t).
\end{equation*}
\item[(A2)] Incomplete structured global bounds: 
\begin{equation*}
\| w_{N,K} \|_{\X{s_1}{b}(\cJ)} \leq B  \quad \text{and} \quad \sum_{L_1\sim L_2} \| P_{L_1} \<1p>\, \cdot P_{L_2} w_{N,K} \|_{L_t^2 H_x^{-4\delta_1}(\cJ \times \bT^3)} \leq B. 
\end{equation*}
\end{itemize}
We define a function $\util \colon \cJ \times \bT^3 \rightarrow \bR$ by
\begin{equation*}
\util(t) = u_{N,K}(t) +  \<3DNKtaup>(t). 
\end{equation*}
Then, $\util$ satisfies the following three properties.
\begin{itemize}
\item[(C1)] \label{global:item_cubic_C1} Structure: For all $t\in \cJ$,  it holds that 
\begin{equation*}
\util(t)=\<1p>(t)+ \<3DNp>(t)+ \widetilde{w}_N(t),
\end{equation*}
where $\widetilde{w}_N = w_{N,K}$. 
\item[(C2)] Approximate solution: There exist $H_N\in \LM(\cJ)$ and $F_N \in \X{s_2-1}{b_+-1}(\cJ)$ satisfying
\begin{align*}
&(-\partial_t^2 - 1 + \Delta) \util 
- P_{\leq N} \lcol \big( V \ast ( P_{\leq N} \util )^2  \big) P_{\leq N} \util \rcol \\
&=(-\partial_t^2 - 1 + \Delta) u_{N,K}
- P_{\leq N} \lcol \big( V \ast ( P_{\leq N} u_{N,K} )^2  \big) P_{\leq N} u_{N,K} \rcol \\
&-P_{\leq N} \PCtrl( H_N, P_{\leq N} \<1p>) - F_N 
\end{align*}
and 
\begin{equation*}
\| H_N \|_{\LM(\cJ)}, \| F_N \|_{\X{s_2-1}{b_+-1}(\cJ)} < T^\alpha A B^3 \min(N,K)^{-\eta^\prime}. 
\end{equation*}
\item[(C3)] Closeness: It holds that 
\begin{equation*}
\| \widetilde{u}_N[t] - u_{N,K}[t]  \|_{C_t^0 \cH_x^{\beta-\kappa}(\cJ \times \bT^3)} < T^\alpha A B^3 \min(N,K)^{-\eta^\prime}. 
\end{equation*}
\end{itemize}
\end{proposition}

\begin{proof} 
We simply choose $\cS^{\cub}(A,T,\tau)$ as the set of all $\purpledot\in \cH_x^{-1/2-\kappa}$ where the implication (A1),(A2) $\rightarrow$(C1),(C2),(C3) holds for all $N,K,B,j$, and $w_{N,K}$. Similar as in the proof of Proposition \ref{global:prop_structure_time}, a separability argument yields the Borel measurability of $\cS^{\cub}(A,T,\tau)$. \\
We now show that $\cS^{\cub}(A,T,\tau)$ satisfies the desired probabilistic estimate. The first conclusion (C1) follows directly from the definition of $\widetilde{u}_N$. We now turn to the second conclusion, which is the main part of the argument. First, we recall that  $\<3DNKtaup>$ solves the linear wave equation on  $\cJ = [\tau,j\tau]$. Together with the definition of $\util$, this implies 
\begin{align*}
&(-\partial_t^2 - 1 + \Delta) \util 
- P_{\leq N} \lcol \big( V \ast ( P_{\leq N} \util )^2  \big) P_{\leq N} \util \rcol \\
&- \Big( (-\partial_t^2 - 1 + \Delta) u_{N,K}
- P_{\leq N} \lcol \big( V \ast ( P_{\leq N} u_{N,K} )^2  \big) P_{\leq N} u_{N,K} \rcol  \Big)\\
&=  P_{\leq N} \lcol \Big( V \ast \Big( P_{\leq N} u_{N,K} + P_{\leq N} \<3DNKtaup> \Big)^2  \Big) P_{\leq N} \Big(u_{N,K} + \<3DNKtaup>\Big) \rcol   \\
&- P_{\leq N} \lcol \big( V \ast ( P_{\leq N} u_{N,K} )^2  \big) P_{\leq N} u_{N,K} \rcol. \\
\end{align*}
We emphasize that in the cubic stochastic object $\<3DNKtaup>$, the linear evolution $\<1p>$ enters at a frequency $\gtrsim \min(N,K)$ in at least one of the arguments. Using the frequency-localized version of the multi-linear master estimate for Gibbsian initial data (Proposition \ref{ftg:prop_master}), we obtain the conclusion (C2). 

Finally, (C3) directly follows from the frequency-localized version of Proposition \ref{ftg:prop_cubic}.
\end{proof}

\subsubsection{Stability theory}

The last ingredient for the globalization argument is a stability estimate. The proof will rely on our previous stability estimate for Gaussian random data from Proposition \ref{local:prop_stability}. As a result, the argument closely resembles a similar step in the local theory, where we proved Proposition \ref{global:prop_lwp} through Proposition \ref{lwp:prop}. 

\begin{proposition}[Stability estimate]\label{global:prop_stability}
Let $T\geq 1$, let $A\geq 1$, let $0<\tau\leq 1$, and let  $\zeta=\zeta(\epsilon,s_1,s_2,\kappa,\eta,\eta^\prime,b_+,b)>0$ be sufficiently small. There exists a constant $C=C(\epsilon,s_1,s_2,b_+,b_-)$ and a Borel set $\cS^{\stab}(A,T,\tau) \subseteq \cH_x^{-1/2-\kappa}(\bT^3)$ satisfying
\begin{equation}\label{global:eq_stability_probabilistic}
\mup_M(\, \purpledot \in \cS^{\stab}_{\blue}(A,T,\tau)) \geq 1 - \zeta^{-1} \exp(-\zeta A^{\zeta}) 
\end{equation}
such that the following holds for all  $\, \purpledot \in \cS^{\stab}(A,T,\tau)$: \\
Let $N\geq 1$, $B\geq 1$, $0<\theta<1$, and let $\cJ = [t_0,t_1]\subseteq [0,T]$, where $t_0 ,t_1 \in \tau \bZ$.  Let  $\util\colon J \times \bT^3 \rightarrow \bR$ be an approximate solution of \eqref{eq:nlw_N} satisfying the following assumptions.
\begin{itemize}
\item[(A1)] Structure: We have the decomposition
\begin{equation*}
\util = \<1p> + \<3DNp> + \wtil. 
\end{equation*}
\item[(A2)] Global bounds: It holds that
\begin{equation*}
\| \wtil \|_{\X{s_1}{b}(\cJ)} \leq B \quad \text{and} \quad \sum_{L_1\sim L_2} \|  P_{L_1} \<1p>\, \cdot  P_{L_2} \wtil \|_{L_t^2 H_x^{-4\delta_1}(\cJ\times \bT^3)} \leq B. 
\end{equation*}
\item[(A3)] Approximate solution: There exists $H_N\in \LM(\cJ)$ and $F_N \in \X{s_2-1}{b_+-1}(\cJ)$ satisfying the identity
\begin{align*}
(-\partial_t^2 - 1 + \Delta) \util  = P_{\leq N} \lcol \big( V \ast ( P_{\leq N} \util )^2  \big) P_{\leq N} \util \rcol - P_{\leq N} \PCtrl(H_N , P_{\leq N} \<1p>\,) - F_N 
\end{align*}
and the estimates
\begin{equation*}
\| H_N \|_{\LM(\cJ)} < \theta \qquad \text{and} \qquad \| F_N \|_{\X{s_2-1}{b_+-1}(\cJ)} < \theta.  
\end{equation*}
\end{itemize}
Furthermore, let $Z_N[t_0] \in H_x^{s_1}(\bT^3)$ be a perturbation satisfying the following assumption.
\begin{itemize}
\item[(A4)] Structured perturbation: There exists a $K\geq 1$ such that
\begin{equation*}
\| Z[t_0] \|_{\Z(\cJ,\scriptp;t_0,N,K)}\leq \theta. 
\end{equation*}
\end{itemize}
Finally, assume that 
\begin{itemize}
\item[(A5)] Parameter condition: $C \exp\Big( C (A+B+T)^C \Big) \theta \leq 1$. 
\end{itemize}
Then, there exists a solution $u_N\colon J \times \bT^3\rightarrow \bR$ of \eqref{eq:nlw_N} satisfying the initial value condition $u_N[t_0]= \util[t_0] + Z_N[t_0]$ and  the following conclusions. 
\begin{itemize}
\item[(C1)] Preserved structure: We have the decomposition 
\begin{equation*}
u_N= \<1p> + \<3DNp> + w_N. 
\end{equation*}
\item[(C2)] Closeness: The difference $u_N-\util=w_N-\wtil$ satisfies 
\begin{align*}
\| u_N - \util \|_{\X{s_1}{b}(\cJ)} &\leq C \exp\big( C (A+B+T)^C \big) \theta, \\
 \sum_{L_1\sim L_2} \|  P_{L_1} \<1p>\, \cdot P_{L_2} (u_N-\util) \|_{L_t^2 H_x^{-4\delta_1}(\cJ\times \bT^3)} &\leq  C \exp\big( C (A+B+T)^C \big) \theta.
\end{align*}
\item[(C3)] Preserved global bounds: It holds that
\begin{equation*}
\| w_N \|_{\X{s_1}{b}(\cJ)} \leq B_\theta \quad \text{and} \quad \sum_{L_1\sim L_2} \| P_{L_1} \<1p> \cdot  P_{L_2} w_N \|_{L_t^2 H_x^{-4\delta_1}(\cJ\times \bT^3)} \leq B_\theta,
\end{equation*}
where $ B_\theta \defe B+ C \exp\big( C (A+B+T)^C \big) \theta$. 
\end{itemize}
\end{proposition}
\begin{proof} Similar as in the proof of Proposition \ref{global:prop_structure_cubic}, we can define $\cS^{\stab}(A,T,\tau)$ through the implications (A1)-(A5) $\rightarrow$ (C1)-(C3) and prove its measurability using a separability argument. \\

It remains to prove the probabilistic estimate \eqref{global:eq_stability_probabilistic}. Using Theorem \ref{theorem:measures}, it suffices to prove that 
\begin{equation*}
\bP\big( \, \bluedot + \reddotM \in \cS^{\stab}(A,T,\tau)\big) \geq 1 - \zeta^{-1} \exp(\zeta A^{\zeta}). 
\end{equation*}
Using Lemma \ref{local:lem_type_conversion}, Corollary \ref{ftg:cor_cubic}, Proposition \ref{local:prop_stability}, Lemma \ref{local:lem_type_conversion}, and Lemma \ref{ftg:lem_structured_perturbation}, which also contain the definitions of the sites below, we may restrict to the event
\begin{equation}
\begin{aligned}
&\Big\{ \bluedot \in  \Theta_{\blue}^{\type}(A,T) \medcap \Theta^{\stab}_{\blue}(A,T) \medcap \Theta_{\blue}^{\cub}(A,T) \medcap \Theta_{\blue}^{\tsp}(A,T) \Big\} \\
&\medcap \Big\{ \reddotM \in \Theta_{\red}^{\type}(A,T) \medcap  \Theta_{\red}^{\tsp}(A,T) \Big\}. 
\end{aligned}
\end{equation}

Our goal is to use Proposition \ref{local:prop_stability} (with slightly adjusted parameters). To this end, we need to convert the assumptions (A1)-(A5) involving $\purpledot$ into similar statements based on $\bluedot$. We let $D>0$ be  a large implicit (but absolute) constant, which may change its value between different lines. We now let $N,B,\theta,\cJ,\widetilde{u}_N,\widetilde{w}_N, H_N, F_N$, and $Z_N[t_0]$ be as in (A1)-(A5). We then define $w_N[\purpledot\rightarrow \bluedot]$ by 
\begin{equation*}
\<1p> + \<3DNp>= \<1b> + \<3DNb> + w_N[\purpledot\rightarrow \bluedot],
\end{equation*}
which implies
\begin{equation*}
\util = \<1b> + \<3DNb> +  w_N[\purpledot\rightarrow \bluedot]+\wtil. 
\end{equation*}
Using Corollary \ref{ftg:cor_cubic} and Lemma \ref{ftg:lem_type_conversion_blue_purple}, we obtain that 
\begin{equation*}
\| \wtil \|_{\X{s_1}{b}(\cJ)} \leq B  \quad \text{and} \quad \sum_{L_1\sim L_2} \|  P_{L_1} \<1b>\, \cdot  P_{L_2} \wtil \|_{L_t^2 H_x^{-4\delta_1}(\cJ\times \bT^3)} \leq T^\alpha A^D B 
\end{equation*}
as well as  
\begin{equation*}
\| w_N[\purpledot\rightarrow \bluedot] \|_{\X{s_1}{b}(\cJ)} \leq T^\alpha A^D \quad \text{and} \quad \sum_{L_1\sim L_2} \|  P_{L_1} \<1b>\, \cdot  P_{L_2}w_N[\purpledot\rightarrow \bluedot] \|_{L_t^2 H_x^{-4\delta_1}(\cJ\times \bT^3)} \leq T^\alpha A^D. 
\end{equation*}
Thus, (A2) in Proposition \ref{local:prop_stability} is satisfied with $B^\prime = 2 T^\alpha A^D B$. A similar argument based on Lemma \ref{ftg:lem_type_conversion_blue_purple} and Lemma \ref{ftg:lem_structured_perturbation} also yields (A3) and (A4) in Proposition \ref{local:prop_stability} with $\theta^\prime= 2 T^\alpha A^D B$. Furthermore, the stronger assumption (A5) in this proposition implies (as long as $C$ is sufficiently large) that 
\begin{equation*}
C \exp\Big( C (A+B^\prime)^{\frac{2}{b_+-b}} T^{\frac{40}{b_+-b}} \Big) \theta^\prime \leq 1. 
\end{equation*}
Thus, Proposition \ref{local:prop_stability} implies that 
\begin{align*}
\| u_N - \util \|_{\X{s_1}{b}(\cJ)} &\leq C \exp\big( C (A+B^\prime)^{\frac{2}{b_+-b}} T^{\frac{40}{b_+-b}} \big) \theta^\prime, \\
\sum_{L_1\sim L_2} \|  P_{L_1} \<1b>\, \cdot  P_{L_2} (u_N-\util) \|_{L_t^2 H_x^{-4\delta_1}(\cJ\times \bT^3)} &\leq  C \exp\big( C (A+B^\prime)^{\frac{2}{b_+-b}} T^{\frac{40}{b_+-b}} \big) \theta^\prime.
\end{align*}
Arguing similarly as above to replace $\<1b>$ by $\<1p>$, this proves the conclusion (C2). The conclusion (C3) then follows from the triangle inequality and assumption (A2). 
\end{proof}

\section{Ingredients, tools, and methods}\label{section:tools}

In this section we provide tools that will be used throughout the rest of this paper. In order to make this section accessible to readers with a primary background in either dispersive or stochastic partial differential equations, our exposition will be detailed. We encourage the reader to skip sections covering areas of his or her expertise. 

In Section \ref{section:bourgain}, we cover $\X{s}{b}$-spaces, which are also called Bourgain spaces. The $\X{s}{b}$-spaces will allow us to utilize multi-linear dispersive effects. In Section \ref{section:continuity}, we present a continuity argument. In Section \ref{section:sin_cancellation}, we prove an oscillatory sum estimate for a series involving the $\operatorname{sine}$-function. While the proof is standard, its relevance to dispersive equations is surprising and the cancellation was first used by Gubinelli, Koch, and Oh in \cite{GKO18a}. In Section \ref{section:counting}, we state several counting estimates related to the dispersive symbol of the wave equation. The counting estimate play an important role in the estimates of our stochastic objects. In Section \ref{section:gaussian_processes}, we recall elementary properties of Gaussian processes, which have been heavily used in the first part of the series \cite{BB20a}. In Section \ref{section:multiple_integrals}, we provide background regarding multiple stochastic integrals. This section has an algebraic flavor and the multiple stochastic integrals will be used to separate the non-resonant and resonant components of our stochastic object. In Section \ref{section:hypercontractivity}, we discuss Gaussian hypercontractivity and its implications for random matrices. 

\subsection{Bourgain spaces and transference principles}\label{section:bourgain}
In this subsection, we recall the definitions and elementary properties of $\X{s}{b}$-spaces, which are often also called Bourgain spaces. Heuristically, $\X{s}{b}$-spaces contain space-time functions $u$ which behave like solutions to the linear wave equation. This principle will be made more precise through the transference principles below. We refer the reader to \cite[Section 2.6]{Tao06} and \cite[Section 3.3]{ET16} for a more detailed introduction. 

\begin{definition}[$\X{s}{b}$-spaces]\label{tools:def_xsb}
For any $s,b\in \bR$ and $u\colon \bR \times \bT^3 \rightarrow \bR$, we define the $\X{s}{b}$-norm by 
\begin{equation}\label{tools:eq_xsb}
\| u \|_{\X{s}{b}} \defe \| \langle n \rangle^s \langle |\lambda|- \langle n \rangle \rangle^b \widehat{u}(\lambda,n) \|_{L_\lambda^2 \ell_n^2(\bR\times \bZ^3)}.
\end{equation}
If $\cJ \subseteq \bR$ is any interval, we define the restricted norm by 
\begin{equation}
\| u \|_{\X{s}{b}(\cJ)} \defe \inf \{ \| v \|_{\X{s}{b}} \colon v(t,x)|_\cJ = u \}. 
\end{equation}
We denote the corresponding function spaces by $\X{s}{b}$ and $\X{s}{b}(\cJ)$, respectively. 
\end{definition}
In \eqref{tools:eq_xsb}, we could have used the symbol $\langle |\lambda|- |n| \rangle$ instead of $\langle |\lambda| - \langle n \rangle \rangle$. Since $\langle n \rangle = |n| + \mathcal{O}(1)$, this would yield an equivalent definition. Our first lemma shows the connection between the $\X{s}{b}$-spaces and the half-wave operators. 

\begin{lemma}[\protect{Characterization of $\X{s}{b}$}]\label{tools:lem_xsb_characterization}
Let $s,b\in \bR$ and let $u\colon \bR \times \bT^3 \rightarrow \bR$. Then, it holds that 
\begin{equation}\label{tools:eq_xsb_characterization_1}
\| u \|_{\X{s}{b}(\bR)} \lesssim \min_{\pm} \| \langle \nabla \rangle^{s} \exp( \mp i t \langle \nabla \rangle) u \|_{L_x^2 H_t^b(\bT^3 \times \bR)}.
\end{equation}
Furthermore, we have the equivalence
\begin{equation}\label{tools:eq_xsb_characterization_2}
\| u \|_{\X{s}{b}(\bR)} \sim  \min_{\substack{u_+,u_- \in \X{s}{b}(\bR)\colon \\ u = u_+ + u_-}} \max_{\pm} \| \langle \nabla \rangle^{s} \exp( \mp i t \langle \nabla \rangle) u_\pm \|_{L_x^2 H_t^b(\bT^3 \times \bR)}.
\end{equation}
\end{lemma}

\begin{proof}
Using Plancherel's identity, it holds that 
\begin{equation*}
\| \langle \nabla \rangle^{s} \exp( \mp i t \langle \nabla \rangle) u \|_{L_x^2 H_t^b(\bT^3 \times \bR)} =
 \| \langle n \rangle^s \langle \pm \lambda - \langle n \rangle \rangle^b \widehat{u}(\lambda,n) \|_{L_\lambda^2 \ell_n^2(\bR \times \bZ^3)}. 
\end{equation*}
The first estimate \eqref{tools:eq_xsb_characterization_1} then follows from $||\lambda|-\langle n \rangle| \leq |\pm \lambda - \langle n \rangle|$. The inequality ``$\lesssim$" in the identity \eqref{tools:eq_xsb_characterization_2} follows from the triangle inequality and \eqref{tools:eq_xsb_characterization_1}. The inequality ``$\gtrsim$" follows by defining $u_\pm$ as
\begin{equation*}
\widehat{u}(\lambda,n) = 1\big\{ \pm \lambda \geq 0  \big\} \cdot \widehat{u}(\lambda,n). 
\end{equation*}
\end{proof}

Our next lemma plays an important role in the local theory. It yields the required smallness of the nonlinearity on a small time-interval. 

\begin{lemma}[Time-localization lemma]\label{tools:lem_localization}
Let $-1/2 < b_1 \leq b_2 < 1/2$ and let $1/2< b <1$. Let $\psi \in \mathcal{S}(\bR)$ be a Schwartz-function and let $0 < \tau \leq 1$. Then, it holds for all $F\in \X{s}{b_2}(\bR)$ that 
\begin{equation}\label{tools:eq_local_1}
\big\| \psi( t/\tau) F \big\|_{\X{s}{b_1}(\bR)} \lesssim \tau^{b_2 -b_1} \| F \|_{\X{s}{b_2}(\bR)} \qquad \text{and} \qquad \| F \|_{\X{s}{b_1}([0,\tau])} \lesssim \tau^{b_2 -b_1} \| F \|_{\X{s}{b_2}([0,\tau])}.
\end{equation}
Furthermore, we have for all $u \in \X{s}{b}(\bR)$ that 
\begin{equation}\label{tools:eq_local_2}
\|  \psi( t/\tau) u \|_{\X{s}{b}(\bR)} \lesssim \tau^{\frac{1}{2}-b} \| u \|_{\X{s}{b}(\bR)}. 
\end{equation}
\end{lemma}

A proof of Lemma \ref{tools:lem_localization} or a similar result can be found in many textbooks on dispersive PDE, such as  \cite[Section 2.6]{Tao06} or \cite[Section 3.3]{ET16}. Since the second estimate \eqref{tools:eq_local_2} is not usually found in the literature, we present a self-contained proof. 

\begin{proof}
By using duality and a composition, we may assume that $0 \leq b_1 \leq b_2 < 1/2$. Let $F_+,F_- \in \X{s}{b_2}(\bR)$ satisfying $F=F_++ F_-$. Using Lemma \ref{tools:lem_xsb_characterization}, we obtain that
\begin{equation}\label{tools:eq_local_p1}
\| \psi( t/\tau) F \|_{\X{s}{b_1}(\bR)} \lesssim    \max_{\pm} \|  \psi( t/\tau) \langle \nabla \rangle^{s} \exp( \mp i t \langle \nabla \rangle) F_\pm \|_{L_x^2 H_t^{b_1}(\bT^3 \times \bR)}.
\end{equation}
Using interpolation between the $b_1=0$ and $b_1=b_2$ as well as the fractional product rule (or a simple para-product estimate), one has for all $f\in H_t^{b_2}(\bR)$ the estimate 
\begin{equation}\label{tools:eq_local_p2}
\| \psi( t/\tau)  f \|_{H_t^{b_1}(\bR)} \lesssim \tau^{b_2 -b_1}  \| f \|_{H_t^{b_2}(\bR)}. 
\end{equation}
Combining \eqref{tools:eq_local_p1} and \eqref{tools:eq_local_p2} yields the first estimate in \eqref{tools:eq_local_1}. The second estimate in \eqref{tools:eq_local_p1} then follows from the first estimate and the definition of the restricted norms. Finally, the second estimate \eqref{tools:eq_local_p2} follows from the same argument, except that \eqref{tools:eq_local_p2} is replaced by 
\begin{equation}\label{tools:eq_local_p3}
\| \psi( t/\tau)  f \|_{H_t^{b}(\bR)} \lesssim \| \psi( t/\tau)  \|_{H_t^{b}(\bR)}  \| f \|_{H_t^{b}(\bR)} \lesssim  \tau^{\frac{1}{2}-b}  \| f \|_{H_t^{b}(\bR)},
\end{equation}
which follows from the algebra property of $H_t^b(\bR)$. 
\end{proof}

\begin{lemma}[Restricted norms and continuity]\label{tools:lem_restricted_continuity}
Let $s\in \bR$ and let $-1/2< b^\prime <1/2$. Then, we have for any interval $\cJ \subseteq \bR$ and any $F\in \X{s}{b^\prime}(\bR)$ that 
\begin{equation}\label{tools:eq_rc_1}
\| 1_{\cJ} F \|_{\X{s}{b^\prime}(\bR)} \lesssim \|  F \|_{\X{s}{b^\prime}(\bR)}.
\end{equation}
Furthermore, if $G\in \X{s}{b^\prime}(\cJ)$, then
\begin{equation}\label{tools:eq_rc2}
\| G \|_{\X{s}{b^\prime}(\cJ)} \sim \| 1_{\cJ} G \|_{\X{s}{b^\prime}(\bR)} 
\end{equation}
Finally, if $t_0\defe \inf \cJ$, then the map
\begin{equation}\label{tools:eq_rc3}
t \in \cJ \mapsto \| 1_{[t_0,t]} G \|_{\X{s}{b^\prime}(\bR)} 
\end{equation}
is continuous. 
\end{lemma}

\begin{proof}
We begin with the proof of \eqref{tools:eq_rc_1}. By using a similar reduction as in the proof of Lemma \ref{tools:lem_localization}, it suffices to prove that 
\begin{equation}\label{tools:eq_gluing_rc_p1}
\| 1_{\cJ} g(t) \|_{\cH_t^{b^\prime}(\bR)} \lesssim \| g(t) \|_{\cH_t^{b^\prime}(\bR)}. 
\end{equation} 
 By writing $1_\cJ$ as a superposition of different indicator functions, it suffices to prove the estimate for $(-\infty,a)$ and $(a,\infty)$, where $a\in \bR$, instead of $\cJ$. Using the time-reflection and time-translation symmetry of $H_t^{b^\prime}(\bR)$, it suffices to prove the estimate for $\cJ$ replaced by $(0,\infty)$. Thus, it remains to prove 
 \begin{equation}\label{tools:eq_gluing_rc_p2}
\| 1_{(0,\infty)} g(t) \|_{\cH_t^{b^\prime}(\bR)} \lesssim \| g(t) \|_{\cH_t^{b^\prime}(\bR)}. 
\end{equation} 
This follows from (a modification of) the fractional product rule or a simple paraproduct estimate. 
We now turn to the proof of \eqref{tools:eq_rc2}. By the definition of the restricted norms, we clearly have the upper-bound $\| G \|_{\X{s}{b^\prime}(\cJ)} \lesssim \| 1_{\cJ} G \|_{\X{s}{b^\prime}(\bR)} $. Now, let $\widetilde{G}\in \X{s}{b}(\bR)$ satisfies $\widetilde{G}|_{\cJ} = G$. Using \eqref{tools:eq_rc_1}, we obtain that 
\begin{equation*}
\| 1_{\cJ} G \|_{\X{s}{b^\prime}(\bR)} = \| 1_{\cJ} \widetilde{G} \|_{\X{s}{b^\prime}(\bR)} \lesssim  \|  \widetilde{G} \|_{\X{s}{b^\prime}(\bR)}.
\end{equation*}
After taking the infimum in $\widetilde{G}$, this yields the other lower-bound in \eqref{tools:eq_rc2}. \\
Finally, we prove the continuity of \eqref{tools:eq_rc3}. By a density argument, it suffices to take $G\in \X{s}{1/2}(\bR)$. For any $0<\delta<1/2-b$ and any $t_1,t_2\in \cJ$, we obtain from Lemma \ref{tools:lem_localization} that 
\begin{equation*}
\Big| \| 1_{[t_0,t_1]} G \|_{\X{s}{b^\prime}(\bR)} -\| 1_{[t_0,t_2]} G \|_{\X{s}{b^\prime}(\bR)} \Big| \leq \| 1_{(t_1,t_2]} G \|_{\X{s}{b^\prime}(\bR)}
\lesssim |t_1-t_2|^\delta \| G\|_{\X{s}{1/2}(\bR)}. 
\end{equation*}
This implies the H\"{o}lder-continuity. 
\end{proof}

The next gluing lemma will be used to combine $\X{s}{b}$-bounds on different intervals. While such a result is trivial for purely physical function spaces, such as $L_t^q L_x^p$, it is slightly more complicated for the $\X{s}{b}$-spaces, since they rely on the time-frequency variable. 

\begin{lemma}[Gluing lemma]\label{tools:lem_gluing}
Let  $s\in \bR$, let $-1/2 < b^\prime <1/2$, let $1/2 < b <1$, and  let $\cJ,\cJ_1,\cJ_2 $ be  bounded intervals satisfying $\cJ_1 \medcap \cJ_2 \not = \emptyset$. Then, we have for
 all $F\colon (\cJ_1 \medcup \cJ_2) \times \bT^3 \rightarrow \bR $ that
\begin{equation}\label{tools:eq_gluing_1}
\| F \|_{\X{s}{b^\prime}(\cJ_1\scriptmedcup\cJ_2)} \lesssim \| F \|_{\X{s}{b^\prime}(\cJ_1)} +  \| F \|_{\X{s}{b^\prime}(\cJ_2)} . 
\end{equation}
Furthermore, let $\tau \defe | \cJ_1 \medcap \cJ_2|$. Then, it holds for all 
$ u \colon (\cJ_1 \medcup \cJ_2) \times \bT^3 \rightarrow \bR $
that 
\begin{equation}\label{tools:eq_gluing_2}
\| u \|_{\X{s}{b}(\cJ_1\scriptmedcup\cJ_2)} \lesssim \tau^{\frac{1}{2}-b} \big( \| u \|_{\X{s}{b}(\cJ_1)} +  \| u \|_{\X{s}{b}(\cJ_2)} \big). 
\end{equation}
\end{lemma}

\begin{proof}
We begin with the proof of  \eqref{tools:eq_gluing_1}. Using Lemma \ref{tools:lem_restricted_continuity}, we have that 
\begin{align*}
\| F \|_{\X{s}{b^\prime}(\cJ_1\scriptmedcup\cJ_2)} &\lesssim \| 1_{\cJ_1\scriptmedcup\cJ_2}  F \|_{\X{s}{b^\prime}(\bR)} \\
&\lesssim   \| 1_{\cJ_1}  F \|_{\X{s}{b^\prime}(\bR)} + \| 1_{\cJ_2\backslash \cJ_1}  F \|_{\X{s}{b^\prime}(\bR)} \\
&\lesssim \| F \|_{\X{s}{b^\prime}(\cJ_1)} + \| F \|_{\X{s}{b^\prime}(\cJ_2\backslash \cJ_1)} \\
&\lesssim  \| F \|_{\X{s}{b^\prime}(\cJ_1)} +  \| F \|_{\X{s}{b^\prime}(\cJ_2)}. 
\end{align*}

The proof of the second estimate \eqref{tools:eq_gluing_2} is similar. Instead of working with an actual indicator function, we use a smooth cut-off function on the spatial scale $\sim \tau$ and a variant of \eqref{tools:eq_local_p3} instead of \eqref{tools:eq_gluing_rc_p2}. 
\end{proof}

Our last two lemmas where concerned with the behavior of $\X{s}{b}$-spaces over small or overlapping time-intervals. In this respect, the $\X{s}{b}$-spaces are more complicated than purely physical function spaces. We now turn to transference principles, which do not have a direct analog in purely physical function spaces.

\begin{lemma}[{Linear transference principle (cf. \cite[Lemma 2.9]{Tao06})}]\label{tools:lem_linear_transference}
Let $b>1/2$, let $s\in \bR$, and assume that the norm $\|\cdot\|_Y$ satisfies
\begin{equation}
\| e^{i\alpha t} e^{\pm it \langle \nabla \rangle} u_0 \|_Y \leq C \| u_0 \|_{H_x^s}
\end{equation}
for all $\alpha \in \bR$ and all $ u_0  \in H_x^s$. Then, it holds for all $u \in \X{s}{b}$ that 
\begin{equation}
\| u \|_Y \lesssim C \| u \|_{\X{s}{b}}. 
\end{equation}
\end{lemma}

The linear transference principle allows us to reduce linear estimates for functions in $\X{s}{b}$-spaces to estimates for the half-wave operators. 

\begin{corollary}\label{tools:cor_strichartz}
For any $b>1/2$, $s\in \bR$, any $4\leq p \leq \infty$, any compact interval $J\subseteq \bR$, and any $u\colon J \times \bT^3 \rightarrow \bC$,  we have that 
\begin{align}
\| u[t] \|_{C_t^0  \cH_x^s(J\times \bT^3)} &\lesssim \| u \|_{\X{s}{b}(J)}, \label{tools:eq_strichartz_1} \\
\| \langle \nabla \rangle^{s+\frac{4}{p}-\frac{3}{2}} u(t)\|_{L_t^p L_x^p(J\times \bT^3)} &\lesssim (1+|J|)^{1/p} \| u \|_{\X{s}{b}(J)},\label{tools:eq_strichartz_2}  \\
\| \langle \nabla \rangle^{s-1-} u(t)\|_{L_t^2 L_x^\infty(J\times \bT^3)} &\lesssim (1+|J|)^{1/2} \| u \|_{\X{s}{b}(J)}.\label{tools:eq_strichartz_3} 
\end{align}
\end{corollary}

The corollary follows directly from the linear transference principle (Lemma \ref{tools:lem_linear_transference}) and the Strichartz estimates for the linear wave equation.

The next lemma is the most basic ingredient for any contraction argument based on $\X{s}{b}$-spaces.

\begin{lemma}[{Energy-estimate (cf. \cite[Lemma 2.12]{Tao06} and \cite[Lemma 3.2]{ET16})}]\label{tools:lem_energy}
Let $1/2<b<1$, let $s\in \bR$, let $\cJ\subseteq \bR$ be a compact interval, let $t_0 \in \cJ$, and let 
\begin{equation}
(-\partial_t^2-1+\Delta) u = F.
\end{equation}
Then, it holds that 
\begin{equation}\label{tools:eq_energy}
 \|  u \|_{\X{s}{b}(\cJ)} \lesssim \big(1+|\cJ|\big)^2  \big( \| u[t_0] \|_{\chs}+ \| F\|_{\X{s-1}{b-1}(\cJ)} \big). 
\end{equation}
\end{lemma}
The statement of Lemma \ref{tools:lem_energy} in \cite{ET16,Tao06} only includes intervals of size $\sim 1$. The more general version follows by using the triangle inequality, iterating the bound on unit intervals, and \eqref{tools:eq_strichartz_1}. The square in the pre-factor can likely be improved but is inessential in our argument, since the stability theory already loses exponential factors in the final time $T$. 

The most important terms in the nonlinearity can only be estimated through multi-linear dispersive effects and hence require a direct analysis of the $\X{s-1}{b-1}$-norm. However, several more minor terms can be estimated more easily through physical methods. In order to pass back from the frequency-based $\X{s-1}{b-1}$-space into purely physical spaces, we provide the following inhomogeneous Strichartz estimate. 

\begin{lemma}[Inhomogeneous Strichartz estimate in $\X{s}{b}$-spaces]\label{tools:lem_inhomogeneous_strichartz}
Let $1/2<b<1$, let $s\in \bR$, let $\cJ\subseteq \bR$ be a compact interval, and let $F\colon \cJ \times \bT^3 \rightarrow \bR$. 
Then, we have the two estimates 
\begin{align}
 \|  F \|_{\X{s-1}{b-1}(\cJ)} &\lesssim  \| F\|_{L_t^{2b} H_x^{s-1}(\cJ\times \bT^3)}, \label{tools:eq_inhomogeneous_strichartz_1}\\
 \| F \|_{\X{s-1}{b-1}(\cJ)} &\lesssim   (1+|\cJ|) \| \langle \nabla \rangle^{s-\frac{1}{2}+ \frac{2b-1}{b} s} F\|_{L_t^{4/3} L_x^{4/3}(\cJ\times \bT^3)}. \label{tools:eq_inhomogeneous_strichartz_2}
\end{align}
\end{lemma}

\begin{remark}
 For $0\leq s \leq 1$, we will often simplify the right-hand side of \eqref{tools:eq_inhomogeneous_strichartz_1} by using that
\begin{equation*}
\frac{2b-1}{b} s \leq 4 (b-1/2). 
\end{equation*}
\end{remark}

\begin{proof}
We first prove \eqref{tools:eq_inhomogeneous_strichartz_1}. Using \eqref{tools:eq_strichartz_1} and duality, we have that 
\begin{equation*}
\| F \|_{\X{s-1}{-b}(\cJ)} \lesssim \| F \|_{L_t^1 H_x^{s-1}(\cJ\times \bT^3)}. 
\end{equation*}
By Plancherel, we also have that 
\begin{equation*}
\| F \|_{\X{s-1}{0}(\cJ)} \lesssim \| F \|_{L_t^2 H_x^{s-1}(\cJ\times \bT^3)}. 
\end{equation*}
Using interpolation, this implies \eqref{tools:eq_inhomogeneous_strichartz_1}. The proof of the second estimate \eqref{tools:eq_inhomogeneous_strichartz_2} is similar and relies on duality, \eqref{tools:eq_strichartz_2}, Plancherel, and interpolation.
\end{proof}

When utilizing multilinear dispersive effects, we will often use the following lemma to estimate the $\X{s-1}{b_--1}$-norm. 

\begin{lemma}
Let $s\in \bR$ and let $T\geq 1$. Let $\cA$ be a finite index set and let $(n_\alpha)_{\alpha \in \cA} \subseteq \bZ^3$, 
$(\theta_\alpha)_{\alpha \in \cA} \subseteq \bR$, and $(c_\alpha)_{\alpha \in \cA} \subseteq \bC$. Define
\begin{equation}
F(t,x) \defe \sum_{\alpha \in \cA} c_\alpha \exp(i \langle n_\alpha, x \rangle + i t \theta_\alpha). 
\end{equation}
Then, it holds that 
\begin{equation}
\begin{aligned}
&\| F \|_{\X{s-1}{b_--1}([0,T])} \\
&\lesssim T \max_{\pm} 
\Big\| \langle \lambda \rangle^{b_--1} \langle n \rangle^{s-1} \sum_{\alpha \in \cA}  1\big\{ n=n_\alpha\big\} \, c_\alpha 
\widehat{\chi}\big( T (\lambda \mp \langle n \rangle - \theta_\alpha) \big) \Big\|_{L_\lambda^2 \ell_n^2(\bR \times \bZ^3)}. 
\end{aligned}
\end{equation}
\end{lemma}

\begin{proof}
For any $G \colon \bR \times \bT^3 \rightarrow \C$, we have that 
\begin{align*}
\| G \|_{ \X{s-1}{b_--1}(\bR)} &=  \| \langle |\lambda|- \langle n \rangle \rangle^{b_--1} \langle n \rangle^{s-1}  \widehat{G}(\lambda,n) \|_{L_\lambda^2 \ell_n^2(\bR\times \bZ^3)} \\
&\lesssim \max_{\pm} \|  \langle \lambda \pm  \langle n \rangle \rangle^{b_--1} \langle n \rangle^{s-1} \widehat{G}(\lambda,n) \|_{L_\lambda^2 \ell_n^2(\bR\times \bZ^3)} \\
&= \max_{\pm} \|  \langle \lambda  \rangle^{b_--1} \langle n \rangle^{s-1} \widehat{G}(\lambda \mp \langle n\rangle,n) \|_{L_\lambda^2 \ell_n^2(\bR\times \bZ^3)} .
\end{align*}
We then apply this inequality to $G(t,x)= \chi(t/T) F(t,x)$. 
\end{proof}

Finally, we present an estimate for the Fourier-transform of a (localized) time-integral. 

\begin{lemma}\label{tools:lem_estimate_fouriertransform}
Let $T\geq 1$ and let $\lambda,\lambda_1,\lambda_2 \in \bR$. Then, it holds that 
\begin{equation}\label{tools:eq_estimate_fouriertransform_1}
\Big|\mathcal{F}_t\Big( \chi(t/T) \exp(i\lambda_1 t) \int_0^t \exp\big( i \lambda_2 t^\prime) \dtprime \Big)(\lambda) \Big| \lesssim T^2 
\Big( \langle \lambda - \lambda_1 - \lambda_2\rangle^{-10}  + \langle \lambda - \lambda_1  \rangle^{-10} \Big) \langle \lambda_2 \rangle^{-1}. 
\end{equation}
Furthermore, if $\cJ\subseteq [0,T]$ is an interval, then
\begin{equation}\label{tools:eq_estimate_fouriertransform_2}
\Big|\mathcal{F}_t\Big( \chi(t/T) \exp(i\lambda_1 t) \int_0^t 1_{\cJ}(t^\prime)  \exp\big( i \lambda_2 t^\prime) \dtprime  \Big)(\lambda) \Big| \lesssim T^2 
\Big( \langle \lambda - \lambda_1 - \lambda_2\rangle^{-1}  + \langle \lambda - \lambda_1  \rangle^{-1} \Big) \langle \lambda_2 \rangle^{-1}. 
\end{equation}
\end{lemma}

\begin{proof}
We first prove \eqref{tools:eq_estimate_fouriertransform_1}. A direct calculation yields
\begin{equation}
\mathcal{F}_t\Big( \chi(t/T) \exp(i\lambda_1 t) \int_0^t \exp\big( i \lambda_2 t^\prime) \dtprime \Big )(\lambda)
= \frac{T}{i\lambda_2} \Big( \widehat{\chi}\big(T (\lambda - \lambda_1 - \lambda_2)\big)- \widehat{\chi}\big(T (\lambda - \lambda_1)\big) \Big). 
\end{equation}
For $|\lambda_2|\gtrsim 1$, the estimate follows from the decay of $\widehat{\chi}$. For $|\lambda_2|\lesssim 1$, the estimate follows from the fundamental theorem of calculus and the decay of $\widehat{\chi}^\prime$. We also used $T\geq 1$, which implies that $\langle T \cdot \rangle^{-10} \lesssim \langle \cdot \rangle^{-10}$. 

We now turn to \eqref{tools:eq_estimate_fouriertransform_2}. Since the restriction to $\cJ$ only appears in the integral, we can replace $\cJ$ by its closure. We now let $\cJ=[t_-,t_+]\subseteq [0,T]$. By integrating the exponential, we have that 
\begin{equation*}
 \int_0^t 1_{\cJ}(t^\prime)  \exp\big( i \lambda_2 t^\prime) \dtprime  = \frac{1}{i\lambda_2} 
 \big( \exp( i \lambda_2 (t\wedge t_+)) - \exp( i \lambda_2 (t\wedge t_-)) \big) ,
\end{equation*}
where $x \wedge y$ denotes the minimum of $x$ and $y$. This implies
\begin{align*}
&\mathcal{F}_t\Big( \chi(t/T) \exp(i\lambda_1 t) \int_0^t 1_{\cJ}(t^\prime)  \exp\big( i \lambda_2 t^\prime) \dtprime\Big)(\lambda)\\
&= \frac{1}{i\lambda_2} \int_{\bR} \chi(t/T) \exp( i (\lambda + \lambda_1 ) t ) \big( \exp( i \lambda_2 (t\wedge t_+)) - \exp( i \lambda_2 (t\wedge t_-)) \big) \mathrm{d}t. 
\end{align*}
The estimate then follows by distinguishing the cases $|\lambda_1|\lesssim 1$, $|\lambda_1|\gg 1 \gtrsim |\lambda_2|$, and $|\lambda_1|,|\lambda_2|\gg 1$, together with the triangle inequality and a simple integration by parts. 
\end{proof}

\subsection{Continuity argument}\label{section:continuity}
In this short subsection, we present a modification of the standard continuity argument. The modification is a result of the possibile discontinuity of $t\in [0,T] \mapsto \| u \|_{\X{s}{b}([0,t])}$, where $u \in \X{s}{b}([0,T])$ and $b>1/2$. As a replacement, we will rely on the continuity statement in Lemma \ref{tools:lem_restricted_continuity}. A different approach to this problem was obtained in \cite[Theorem 3]{Tao01}, which yields the quasi-continuity, and may even yield the continuity (see the discussion in \cite[Section 12]{Tao01}). 

\begin{lemma}[Continuity argument]\label{tools:lem_continuity}
Let $\cJ =[t_0,t_1)$, let $f\colon \cJ \rightarrow [0,\infty)$ be a nonnegative function, and let $g\colon \cJ \rightarrow [0,\infty)$ be a continuous, nonnegative function. Let $A\geq 1$, $0<\theta,\delta<1$, and assume that 
\begin{equation}\label{tools:eq_continuity_inequality}
f(t) \leq g(t) \leq g(t_0) + \delta (A^2+f(t)^2) (f(t)+\theta)
\end{equation}
for all $t \in [t_0,t_1)$. Furthermore, assume that 
\begin{equation}\label{tools:eq_continuity_parameter}
g(t_0)+ \delta^2 A \theta \leq 1 \qquad \text{and} \qquad \delta (A^2+6) \leq 1/4. 
\end{equation}
Then, it holds that 
\begin{equation*}
f(t) \leq g(t) \leq 2 (g(t_0)+\delta A^2 \theta) 
\end{equation*}
for all $t \in [t_0,t_1)$. 
\end{lemma}

\begin{proof}
The estimate \eqref{tools:eq_continuity_inequality} implies that 
\begin{equation*}
g(t) \leq g(t_0) + \delta (A^2+g(t)^2) (g(t)+\theta)
\end{equation*}
for all $t \in [t_0,t_1)$. Using the condition \eqref{tools:eq_continuity_parameter}, we also have that 
\begin{equation*}
 g(t_0) + \delta (A^2+4  (g(t_0)+\delta A^2 \theta)^2)  ( g(t_0)+\delta A^2 \theta +\theta) \leq \frac{3}{2}  (g(t_0)+\delta A^2 \theta) . 
\end{equation*}
Using the standard continuity method (see e.g. \cite[Section 1.3]{Tao06}), this implies 
\begin{equation*}
g(t) \leq 2  (g(t_0)+\delta A^2 \theta) 
\end{equation*}
for all $t \in [t_0,t_1)$. 
\end{proof}

\subsection{Sine-cancellation lemma}\label{section:sin_cancellation}

In this subsection, we prove an oscillatory sum estimate which critically relies on the fact that the $\operatorname{sine}$-function is odd. The same cancellation was exploited in earlier work of Gubinelli-Koch-Oh \cite[Section 4]{GKO18a} and we present a slight generalization of their argument.

\begin{lemma}\label{tools:lem_sin_cancellation}
Let $f\colon \bR\times \bR \times \bZ^3 \rightarrow \bC$, $ a \in \bZ^3$, $T \geq 1$, let $\cJ\subseteq [0,T]$ be an interval, and  let $A,N\geq 1$. Assume that $|a|\lesssim A \ll N$. Furthermore, assume that $f$ satisfies for all $|t|,|t^\prime|\leq T$  that 
\begin{equation*}
|f(t,t^\prime,n)|\leq A \langle n \rangle^{-3}, \quad |f(t,t^\prime,n)-f(t,t^\prime,-n)|\leq A \langle n \rangle^{-4}, \quad \text{and}\quad |\partial_{t^\prime}f(t,t^\prime,n)|\leq A \langle n \rangle^{-4}.
\end{equation*}
Then, it holds that 
\begin{equation}
\begin{aligned}
&\sup_{\lambda \in \bR} \sup_{|t|\leq T} \Big| \sum_{n\in \bZ^3} \chi_N(n) \int_0^t 1_{\cJ}(t^\prime) \sin((t-t^\prime) \langle a + n \rangle) \cos((t-t^\prime) \langle n \rangle) \exp(i\lambda t^\prime) f(t,t^\prime,n) \dtprime\Big| \\
&\lesssim T^2 A^3 \log(2+N) N^{-1}. 
\end{aligned}
\end{equation}
\end{lemma}
The dependence on $A$ is not essential and can likely be improved. In all our applications of this lemma, $A$ is negligible compared to $N$. We emphasize that the estimate fails if we only assume that $|f(t,t^\prime,n)|\leq A \langle n \rangle^{-3}$. Indeed, after removing the truncation $\chi_N$, the corresponding sum could diverge logarithmically. 

\begin{proof}
Using trigonometric identities, we have that 
\begin{align}
&2  \sum_{n\in \bZ^3} \chi_N(n) \int_0^t  1_{\cJ}(t^\prime) \sin((t-t^\prime) \langle a + n \rangle) \cos((t-t^\prime) \langle n \rangle) \exp(i\lambda t^\prime) f(t,t^\prime,n) \dtprime \notag \\
& =  \sum_{n\in \bZ^3} \chi_N(n) \int_0^t  1_{\cJ}(t^\prime)\sin\big((t-t^\prime) (\langle a + n \rangle - \langle n \rangle) \big) \exp(i\lambda t^\prime) f(t,t^\prime,n) \dtprime \label{tools:eq_sin_p1} \\
& +  \sum_{n\in \bZ^3} \chi_N(n) \int_0^t  1_{\cJ}(t^\prime)\sin\big((t-t^\prime) (\langle a + n \rangle  + \langle n \rangle) \big) \exp(i\lambda t^\prime) f(t,t^\prime,n) \dtprime. \label{tools:eq_sin_p2}
\end{align}
We estimate the terms \eqref{tools:eq_sin_p1} and \eqref{tools:eq_sin_p2} separately. We begin with \eqref{tools:eq_sin_p1}, which is the more difficult term. Since $|\langle a + n \rangle - \langle n \rangle|\lesssim A$, we do not expect to gain in $N$ through the integration in $t^\prime$. Instead, we utilize a pointwise cancellation. By using the symmetry $n \leftrightarrow -n$ in the summation, we obtain 
\begin{align*}
& 2 \Big| \sum_{n\in \bZ^3} \chi_N(n)  \sin\big((t-t^\prime) (\langle a + n \rangle - \langle n \rangle) \big)  f(t,t^\prime,n) \Big| \\
&= \Big| \sum_{n\in \bZ^3} \chi_N(n)  
\Big( \sin\big((t-t^\prime) (\langle n+ a \rangle - \langle n \rangle) \big)  f(t,t^\prime,n) +
\sin\big((t-t^\prime) (\langle n- a \rangle - \langle n \rangle) \big)  f(t,t^\prime,-n)
\Big) \Big| \\
&\lesssim  \sum_{n\in \bZ^3} \chi_N(n)  
\Big| \sin\big((t-t^\prime) (\langle n+ a \rangle - \langle n \rangle) \big) + \sin\big((t-t^\prime) (\langle n - a  \rangle - \langle n \rangle) \big)  \Big| \cdot |f(t,t^\prime,n)| \\
&+  \sum_{n\in \bZ^3} \chi_N(n)  |f(t,t^\prime,n)-f(t,t^\prime,-n)|. 
\end{align*}
Using the assumptions on $f$, the second summand is easily bounded by $A N^{-1}$.  We now concentrate on the first summand. Using a Taylor expansion, we have that
\begin{equation}
\langle n \pm a \rangle - \langle n \rangle = \pm \frac{n \cdot a}{\langle n \rangle} + \mathcal{O}\Big( A^2 N^{-1}\Big). 
\end{equation}
Using that the $\operatorname{sine}$-function is odd, we obtain that
\begin{align*}
&\Big|  \sin\big((t-t^\prime) (\langle n+ a \rangle - \langle n \rangle) \big) + \sin\big((t-t^\prime) (\langle n - a  \rangle - \langle n \rangle)\Big| \\
&=\Big|  \sin\big((t-t^\prime) (\langle n+ a \rangle - \langle n \rangle) \big) - \sin\big(-(t-t^\prime) (\langle n - a  \rangle - \langle n \rangle)\Big| \\
&\leq T \Big| \langle n+ a \rangle - \langle n \rangle + \langle n-  a \rangle - \langle n \rangle\Big| \\
&\lesssim T A^2 N^{-1}. 
\end{align*}
Putting both estimates together and integrating in $t^\prime$, we see that the first term \eqref{tools:eq_sin_p1} is bounded by $T^2 A^3 N^{-1}$, which is acceptable.\\

We now turn to the estimate of \eqref{tools:eq_sin_p2}. Since $\langle n + a \rangle + \langle n\rangle \gtrsim N$, we expect to gain a factor of $N$ through integration by parts. We have that 
\begin{align*}
&\Big| \sum_{n\in \bZ^3} \chi_N(n) \int_0^t   1_{\cJ}(t^\prime) \sin\big((t-t^\prime) (\langle a + n \rangle  + \langle n \rangle) \big) \exp(i\lambda t^\prime) f(t,t^\prime,n) \dtprime\Big|\\
&\lesssim \max_{\pm} \Big|   \sum_{n\in \bZ^3} \chi_N(n) \int_0^t  1_{\cJ}(t^\prime)  \exp\Big( i\lambda t^\prime\pm i t^\prime (\langle a + n \rangle  + \langle n \rangle)\Big) f(t,t^\prime,n) \dtprime\Big| \\
&\lesssim  \max_{\pm}   \sum_{n\in \bZ^3} \chi_N(n) \frac{1}{1+ | \langle a + n \rangle  + \langle n \rangle \pm \lambda |} 
\Big( \sup_{0 \leq t^\prime\leq  t} \big| f(t,t^\prime,n) \big| + T \sup_{0 \leq t^\prime\leq  t} \big| \partial_{t^\prime} f(t,t^\prime,n) \big| \Big) \\
&\lesssim  T A N^{-3}  \max_{\pm}   \sum_{n\in \bZ^3} \chi_N(n) \frac{1}{1+ | \langle a + n \rangle  + \langle n \rangle \pm \lambda |}. 
\end{align*}
In order to finish the estimate, it only remains to prove that 
\begin{equation*}
 \sum_{n\in \bZ^3} \chi_N(n) \frac{1}{1+ | \langle a + n \rangle  + \langle n \rangle \pm \lambda |} \lesssim \log(2+N) N^2 
\end{equation*}
Since the function $x \mapsto \langle x\rangle$ is $1$-Lipschitz, we can estimate the sum by an integral and obtain that
\begin{equation*}
 \sum_{n\in \bZ^3} \chi_N(n) \frac{1}{1+ | \langle a + n \rangle  + \langle n \rangle \pm \lambda |} \lesssim 
\int_{\bR^3}  1\{ |\xi|\sim N \}  \frac{1}{1+ | \langle \xi + a  \rangle  + \langle \xi \rangle \pm \lambda |} \mathrm{d}\xi. 
\end{equation*}
Due to the rotation invariance of the Lebesgue measure, we can then reduce to $a=(0,0,|a|)$. To estimate the integral, we first switch into polar coordinates $(r,\theta,\varphi)$. Since $A\ll N$, we have for fixed angles $\theta$ and $\varphi$ that $r \mapsto \langle \xi + a  \rangle  + \langle \xi \rangle $ is bi-Lipschitz on $r \sim N$. After a further change of variables, this yields
\begin{equation*}
\int_{\bR^3}  1\{ |\xi|\sim N \}  \frac{1}{1+ | \langle \xi + a  \rangle  + \langle \xi \rangle \pm \lambda |} \mathrm{d}\xi 
\lesssim N^2 \int_0^\infty 1 \{  r\sim N \} \frac{1}{1+|r \pm \lambda|} \mathrm{d}r \lesssim N^2 \log(2+N). 
\end{equation*}

\end{proof}

\subsection{Counting estimates}\label{section:counting}

In this subsection, we record several counting estimates. The counting estimates are the most technical part of our treatment of $\So$, $\CPara$, and $\RMT$. Fortunately, they can be used as a black-box, and we encourage the reader to only skim this section during first reading. 

Before we state our counting estimates, we discuss the main ingredients and the differences between the nonlinear wave and Schr\"{o}dinger equations. In contrast to the counting estimates for the nonlinear Schr\"{o}dinger equation, the counting estimates for the wave equation require no analytic number theory. The reason is that the mapping $n \mapsto \langle n \rangle$ is globally $1$-Lipschitz, whereas the Lipschitz constant of $n \mapsto |n|^2$ grows linearly. This allows us to reduce all (discrete) counting estimates to estimates of the volume of (continuous) sets. More specifically, we will use that the intersection of (most) thin annuli has a smaller volume than the individual annuli. 

Another difference between the wave and Schr\"{o}dinger equation is related to the symmetries of the equation. The Schr\"{o}dinger equation enjoys the Galilean symmetry, which is useful in obtaining  ``shifted" versions of several estimates. For instance, it yields that frequency-localized Strichartz estimates for the Schr\"{o}dinger equation are the same for cubes centered either at or away from the origin. On the frequency-side, it is related to the Galilean transform
\begin{equation*}
(n,\lambda) \mapsto (n-a, \lambda - 2 a \cdot n + |a|^2),
\end{equation*}
which preserves the discrete paraboloid and plays an important role in decoupling theory (cf. \cite[Section 4]{Demeter20}). It often allows us to replace conditions such as $|n|\sim N$ in counting estimates by the more general restriction $|n-a|\sim N$ for some fixed $a\in \bZ^3$. In contrast, the Lorentzian symmetry of the wave equation on Euclidean space does not even preserve the periodicity of $u\colon \bR \times \bT^3 \rightarrow \bR$. As illustrated by the Klainerman-Tataru-Strichartz estimates (cf. \cite{KT99} and Lemma \ref{phy:lem_KT}), the frequency-shifted Strichartz estimates are more complicated for the wave equation than for the Schr\"{o}dinger equation. As will be clear from this section, similar difficulties arise in the counting estimates.\\
The last difference between the Schr\"{o}dinger  and wave equation we mention here is a result of the multiplier $\langle \nabla \rangle^{-1}$ in the Duhamel integral for the wave equation. Together with multilinear dispersive effects, we therefore obtain two separate smoothing effects in the nonlinear wave equation, which are related to the elliptic symbol $\langle n \rangle$ and the dispersive symbol $\langle |\lambda| - |n| \rangle$. In contrast, the Schr\"{o}dinger only exhibits a single smoothing effect related to the dispersive symbol $\lambda - |n|^2$. In most situations, we expect that the combined smoothing effects in the wave equation are stronger than the single smoothing effect in the Schr\"{o}dinger equation. However, it may be more difficult to capture the combined smoothing effect in a single proposition, as has been done in  \cite[Proposition 4.9]{DNY20} for the Schr\"{o}dinger equation. 

In Section \ref{section:basic_counting}, we prove basic counting estimates which form the foundation of the rest of this section. In Section \ref{section:counting_cubic}-\ref{section:counting_septic}, we state several cubic, quartic, quintic, and septic counting estimates. In order to not interrupt the flow of the main argument, we placed their (standard) proofs in the appendix. In Section \ref{section:tensor}, we present estimates for the operator norm of (deterministic) tensors. The tensor estimates are not (yet) standard in the literature on random dispersive equations, so we  include their proofs in the body of the paper. 

\subsubsection{Basic counting estimates}\label{section:basic_counting}

\begin{lemma}[Basic counting lemma]\label{tools:lem_basic_counting}
Let $a\in \bZ^3$, let $A,N\geq 1$, and assume that $|a|\sim A$. Then, it holds that 
\begin{equation}\label{tools:eq_basic_counting}
\sup_{m\in  \bZ}  \# \big\{ n \in \bZ^3\colon |n|\sim N, | \langle a + n \rangle \pm \langle n \rangle -m |\lesssim 1 \big\} \lesssim \min(A,N)^{-1} N^3. 
\end{equation}
\end{lemma}

We emphasize that the upper bound in \eqref{tools:eq_basic_counting} cannot be improved to $N^2$. The reason is that $|\langle a + n \rangle - \langle n \rangle|\lesssim A$, which implies that 
\begin{equation*}
\sup_{m\in  \bZ}  \# \big\{ n \in \bZ^3\colon |n|\sim N, | \langle a + n \rangle - \langle n \rangle -m |\lesssim 1 \big\} \gtrsim A^{-1} N^3. 
\end{equation*}
As already mentioned above, the main step in the proof converts the discrete estimate \eqref{tools:eq_basic_counting} into a continuous analogue. After this reduction, the estimate boils down to multi-variable calculus.

\begin{proof}
Since $\langle \xi \rangle = |\xi|+ \mathcal{O}(1)$, we may replace $\langle \cdot \rangle$ in \eqref{tools:eq_basic_counting} by $|\cdot|$ after increasing the implicit constant. Furthermore, since $\xi \mapsto \langle \xi + a \rangle \pm \langle \xi \rangle$ is globally Lipschitz, we see that the $1$-neighborhood of the set on the left-hand side of \eqref{tools:eq_basic_counting} is contained in 
\begin{equation*}
 \big\{ \xi \in \bR^3\colon |\xi|\sim N, | | a + \xi | \pm |\xi  |  -m |\lesssim 1 \big\}. 
\end{equation*}
Since the integer vectors are $1$-separated, it follows that 
\begin{equation*}
\# \big\{ n \in \bZ^3\colon |n|\sim N, |  | a + n  | \pm  | n |-m |\lesssim 1 \big\} \lesssim 
 \operatorname{Leb}\Big(\big\{ \xi \in \bR^3\colon |\xi|\sim N, |  | a + \xi  | \pm  | \xi | -m |\lesssim 1 \big\}\Big). 
\end{equation*}
We now decompose 
\begin{align*}
& \operatorname{Leb}\Big(\big\{ \xi \in \bR^3\colon |\xi|\sim N, |  |a + \xi  |\pm | \xi  | -m |\lesssim 1 \big\}\Big)\\
&\lesssim \sum_{\substack{m_1,m_2\in \bZ\colon \\ | m_1 \pm m_2 - m |\lesssim 1 }} 
\operatorname{Leb}\Big(\big\{ \xi \in \bR^3\colon |\xi|\sim N,   |a + \xi | = m_1 + \mathcal{O}(1) , |\xi| = m_2 + \mathcal{O}(1) \big\}\Big) \\
&\lesssim N \sup_{m_1,m_2\in \bZ} \operatorname{Leb}\Big(\big\{ \xi \in \bR^3\colon |\xi|\sim N,   |a + \xi | = m_1 + \mathcal{O}(1) , |\xi| = m_2 + \mathcal{O}(1) \big\}\Big) . 
\end{align*}
In the last line, we used that there are at most $\sim N$ non-trivial choices of $m_2$. Once $m_2$ is fixed, the condition $ | m_1 \pm m_2 - m |\lesssim 1$ implies that there are at most $\sim 1 $ non-trivial choices for $m_1$. Thus, it remains to prove for $|m_1|\lesssim \max(A,N)$ and $|m_2|\sim N$ that 
\begin{equation}
 \operatorname{Leb}\Big(\big\{ \xi \in \bR^3\colon |\xi|\sim N,   |a + \xi | = m_1 + \mathcal{O}(1) , |\xi| = m_2 + \mathcal{O}(1) \big\}\Big) \lesssim \min(A,N)^{-1} N^2. 
\end{equation}
Using the rotation invariance of the Lebesgue measure, we may assume that $a=|a| e_3$, i.e., $a$ points in the direction of the $z$-axis. By switching into polar coordinates, we obtain that 
\begin{align*}
& \operatorname{Leb}\Big(\big\{ \xi \in \bR^3\colon |\xi|\sim N,   |a + \xi | = m_1 + \mathcal{O}(1) , |\xi| = m_2 + \mathcal{O}(1) \big\}\Big)\\
&\lesssim  N^2 \int_0^\infty \int_0^\pi 1\big\{ r=m_2 + \mathcal{O}(1) \big\} 1 \big\{ \sqrt{|a|^2 + 2 r |a| \cos(\theta) + r^2} = m_1 + \mathcal{O}(1) \big\} \sin(\theta) \mathrm{d}\theta\, \mathrm{d}r.   
\end{align*}
The condition $\sqrt{|a|^2 + 2 r |a| \cos(\theta) + r^2} = m_1 + \mathcal{O}(1) $ together with $|m_1|\lesssim \max(A,N)$ implies that 
\begin{equation}
\cos(\theta) = 1 - \frac{(|a|+r)^2}{2|a|r} + \frac{m_1^2}{2|a|r} + \mathcal{O}\big( \max(A,N) A^{-1} N^{-1}\big). 
\end{equation}
For a fixed $r$, this shows that $\cos(\theta)$ is contained in an interval of size $\sim \min(A,N)^{-1}$. After a change of variables from $\theta$ to $\cos(\theta)$, this yields 
\begin{equation}\label{tools:eq_basic_counting_p1}
\begin{aligned}
& N^2 \int_0^\infty \int_0^\pi 1\big\{ r=m_2 + \mathcal{O}(1) \big\} 1 \big\{ \sqrt{|a|^2 + 2 r |a| \cos(\theta) + r^2} = m_1 + \mathcal{O}(1) \big\} \sin(\theta) \mathrm{d}\theta\, \mathrm{d}r\\
&\lesssim  \min(A,N)^{-1} N^2  \int_0^\infty 1\big\{ r=m_2 + \mathcal{O}(1) \big\} \mathrm{d}r \\
&\lesssim \min(A,N)^{-1} N^2. 
\end{aligned}
\end{equation}
\end{proof}

\begin{remark}\label{tools:rem_2d_counting}
Our proof of the basic counting lemma (Lemma \ref{tools:lem_basic_counting}) easily generalizes to spatial dimensions $d\geq 3$. In two spatial dimensions, however, only weaker estimates are available. The reason lies in the absence of the $\operatorname{sine}$-function in the area element for polar coordinates, which breaks \eqref{tools:eq_basic_counting_p1}. From a PDE perspective, the parallel interactions in two-dimensional wave equations are stronger than the planar interactions in three-dimensional wave equations. Ultimately, this requires a modification in the probabilistic scaling heuristic and we encourage the reader to compare \cite[Section 1.3.2]{DNY19} and \cite[Proposition 1.5]{OO19}. 
\end{remark}

We now present a minor modification of the basic counting lemma (Lemma \ref{tools:lem_basic_counting}). The condition $|n|\sim N$ is augmented by $|n+a|\sim B$. We emphasize that the vector $a\in \bZ^3$ in this constraint is the same vector as in the dispersive symbol. 

\begin{lemma}[``Two-ball" basic counting lemma]\label{tools:lem_two_balls}
Let $N,A,B\geq 1$. Let $a\in \bZ^3$ satisfy $|a|\sim A$. Then, it holds that 
\begin{equation}\label{tools:eq_two_balls}
\sup_{m\in  \bZ}  \# \big\{ n \in \bZ^3\colon |n|\sim N, |n+a|\sim B, | \langle a + n \rangle \pm \langle n \rangle -m |\lesssim 1 \big\} \lesssim \min(A,B,N)^{-1} \min(B,N)^3. 
\end{equation}
\end{lemma}
\begin{proof}
Using the basic counting lemma (Lemma \ref{tools:lem_basic_counting}), we have that 
\begin{align*}
&\sup_{m\in  \bZ}  \# \big\{ n \in \bZ^3\colon |n|\sim N, |n+a|\sim B, | \langle a + n \rangle \pm \langle n \rangle -m |\lesssim 1 \big\} \\
&\leq \sup_{m\in  \bZ}  \# \big\{ n \in \bZ^3\colon |n|\sim N, | \langle a + n \rangle \pm \langle n \rangle -m |\lesssim 1 \big\} \\
&\lesssim \min(A,N)^{-1} N^3. 
\end{align*}
After using a change of variables $b\defe n+a$, we obtain similarly that 
\begin{equation*}
\sup_{m\in  \bZ}  \# \big\{ n \in \bZ^3\colon |n|\sim N, |n+a|\sim B, | \langle a + n \rangle \pm \langle n \rangle -m |\lesssim 1 \big\} \lesssim \min(A,B)^{-1} B^3. 
\end{equation*}
By combining both estimates we obtain \eqref{tools:eq_two_balls}.
\end{proof}

\subsubsection{Cubic counting estimate}\label{section:counting_cubic}

As mentioned in the beginning of this section, we only discuss and state the remaining counting estimates, but  postpone the proofs until the appendix. 

The cubic counting estimates play an important role in our analysis of the nonlinearity $\<3N>$. In the following, we use $\max$, $\med$, and $\min$ for the maximum, median, and minimum of three frequency-scales. 

\begin{proposition}[Main cubic counting estimate]\label{tools:prop_cubic_counting}
Let $\pm_{123},\pm_1,\pm_2,\pm_3 \in \{ +,-\}$ and define the phase 
\begin{equation*}
\varphi(n_1,n_2,n_3)\defe \pm_{123} \langle n_{123} \rangle \pm_1 \langle n_1 \rangle  \pm_2 \langle n_2 \rangle  \pm_3 \langle n_3 \rangle.
\end{equation*}
Let $N_1,N_2,N_3,N_{12},N_{123}\geq 1$ and let $m\in \bZ$. Then, we have the following counting estimates:
\begin{enumerate}[(i)]
\item \label{tools:item_cubic_1} In the variables $n_1,n_2$, and $n_3$, we have that 
\begin{align*}
&\# \{ (n_1,n_2,n_3) \colon |n_1|\sim N_1,|n_2|\sim N_2, |n_3|\sim N_3, |\varphi-m|\leq 1\} \\
&\lesssim \med(N_1,N_2,N_3)^{-1} (N_1 N_2 N_3)^3,  \notag
\end{align*}
\item \label{tools:item_cubic_2} In the variables $n_{123},n_1$, and $n_2$, we have that 
\begin{align*}
&\# \{ (n_{123},n_1,n_2) \colon |n_{123}|\sim N_{123},|n_1|\sim N_1, |n_2|\sim N_2, |\varphi-m|\leq 1\} \\
&\lesssim \med(N_{123},N_1,N_2)^{-1} (N_{123} N_1 N_2)^3.  \notag
\end{align*}
\item \label{tools:item_cubic_3} In the variables $n_{123},n_{12}$, and $n_1$, we have that 
\begin{align*}
&\# \{ (n_{123},n_{12},n_1) \colon |n_{123}|\sim N_{123},|n_{12}|\sim N_{12}, |n_1|\sim N_1, |\varphi-m|\leq 1\} \\
&\lesssim \min\big(N_{12},\max(N_{123},N_1)\big)^{-1} (N_{123} N_{12} N_1)^3.  \notag
\end{align*}
\item \label{tools:item_cubic_4} In the variables $n_{12},n_1$, and $n_3$, we have that 
\begin{align*}
&\# \{ (n_{12},n_1,n_3) \colon |n_{12}|\sim N_{12},|n_1|\sim N_1, |n_3|\sim N_3, |\varphi-m|\leq 1\} \\
&\lesssim  \min\big(N_{12},\max(N_{1},N_3)\big)^{-1} ( N_{12} N_1 N_3)^3. \notag
\end{align*}
\end{enumerate}
\end{proposition}

\begin{remark}
The four estimates in Proposition \ref{tools:prop_cubic_counting} are sharp. In our analysis of the cubic nonlinearity, the frequencies $n_1,n_2$, and $n_3$ represent the frequencies of the three individual factors. The frequency $n_{12}$ appears through the convolution with the interaction potential $V$. Finally, the frequency $n_{123}$, which is the frequency of the full nonlinearity, appears through the multiplier $\langle \nabla \rangle^{-1}$ in the Duhamel integral and in estimates of the $H_x^s$ and $\X{s}{b}$-norms. 
\end{remark}

Since we postpone the proof, let us ease the reader's mind with the heuristic argument behind \eqref{tools:item_cubic_1}. Without the restriction due to the phase $\varphi$, the combined frequency variables $(n_1,n_2,n_3)$ live in a set of cardinality $(N_1 N_2 N_3)^3$. As long as the level sets of $\varphi$ have comparable cardinalities, we expect to gain a factor corresponding to the possible values of $\varphi$ on the set $\{ (n_1,n_2,n_3) \colon |n_1|\sim N_1,|n_2|\sim N_2, |n_3|\sim N_3\}$. Since $\varphi$ is globally Lipschitz, one may ideally hope for a gain of the form $\max(N_1,N_2,N_3)$. Unfortunately, since 
\begin{equation}\label{tools:eq_heuristic_hilolo}
| \langle n_{123} \rangle - \langle n_1 \rangle + \langle n_2 \rangle + \langle n_3 \rangle | \lesssim \max(N_2,N_3), 
\end{equation}
the high$\times$low$\times$low-interactions rule out a gain in $\max(N_1,N_2,N_3)$. As it turns out, however, our basic counting estimates allows us to obtain a gain of the form $\med(N_1,N_2,N_3)$, which is consistent with \eqref{tools:eq_heuristic_hilolo}.

\begin{proposition}[Cubic sum estimate]\label{tools:prop_cubic_sum}
Let $0<s \leq 1/2$, $0\leq \gamma < s + 1/2$, and let $N_1,N_2,N_3 \geq 1$. Let the signs $\pm_{123},\pm_1,\pm_2,\pm_3\in \{+,-\}$ be given and define the phase 
\begin{equation}\label{tools:eq_phase_varphi}
\varphi(n_1,n_2,n_3)\defe \pm_{123} \langle n_{123} \rangle \pm_1 \langle n_1 \rangle  \pm_2 \langle n_2 \rangle  \pm_3 \langle n_3 \rangle.
\end{equation}
Then, it holds that 
\begin{equation}
\begin{aligned}
&\sup_{m\in \bZ}  \sum_{n_1,n_2,n_3\in \bZ^3} \bigg[ \Big( \prod_{j=1}^{3} \chi_{N_j}(n_j) \Big) \langle n_{123} \rangle^{2(s-1)} \langle n_{12} \rangle^{-2\gamma} \Big( \prod_{j=1}^{3} \langle n_j \rangle^{-2} \Big) 1\big\{ |\varphi-m|\leq 1\big\}\bigg] \\
&\lesssim \max(N_1,N_2,N_3)^{2(s-\gamma)} + \max(N_1,N_2)^{1-2\gamma} \max(N_1,N_2,N_3)^{2s-1}. 
\end{aligned}
\end{equation}
\end{proposition}

\begin{remark}
Proposition \ref{tools:prop_cubic_sum} plays an essential role in proving that $\<3DN>$ has regularity $\beta-$. In that argument, we will simply set $\gamma=\beta$. 
\end{remark}

\subsubsection{Cubic sup-counting estimates} 

We now present cubic counting estimates involving suprema, which will be used in the proof of the tensor estimates in Section \ref{section:tensor}. In turn, the tensor estimates will then be used to prove the  random matrix estimates in Section \ref{section:rmt}. 
\begin{lemma}[Cubic sup-counting estimates]\label{tools:lem_sup}
Let $N_{123},N_1,N_2,N_3\geq 1$ and $m\in\bZ$.  Let the signs $\pm_{123},\pm_1,\pm_2,\pm_3\in \{+,-\}$ be given and define the phase
\begin{equation*}
\varphi(n_1,n_2,n_3)\defe \pm_{123} \langle n_{123} \rangle \pm_1 \langle n_1 \rangle  \pm_2 \langle n_2 \rangle  \pm_3 \langle n_3 \rangle.
\end{equation*}
Then, the following estimates hold:
\begin{enumerate}[(i)]
\item \label{tools:item_sup_1} Taking the  supremum in $n$ and counting $n_1,n_2,n_3$, we have 
\begin{align*}
&\sup_{n\in \bZ^3} \# \Big\{ (n_1,n_2,n_3)\colon |n_1|\sim N_1, |n_2|\sim N_2, |n_3|\sim N_3, n=n_{123}, |\varphi-m|\leq 1\Big\}\\
&\lesssim \med(N_1,N_2,N_3)^3  \min(N_1,N_2,N_3)^2. 
\end{align*}
\item \label{tools:item_sup_2} Taking the  supremum in $n_2$ and counting $n,n_1,n_3$, we have 
\begin{align*}
&\sup_{n_2\in \bZ^3} \# \Big\{ (n,n_1,n_3)\colon |n|\sim N_{123}, |n_1|\sim N_1, |n_3|\sim N_3, n=n_{123}, |\varphi-m|\leq 1\Big\}\\
&\lesssim \med(N_{123},N_1,N_3)^3  \min(N_{123},N_1,N_3)^2. 
\end{align*}
\item \label{tools:item_sup_3} Taking the  supremum in $n$ and counting $n_1,n_{12},n_3$, we have 
\begin{align*}
&\sup_{n\in \bZ^3} \# \Big\{ (n_{12},n_2,n_3)\colon |n_{12}|\sim N_{12}, |n_2|\sim N_2, |n_3|\sim N_3, n=n_{123}, |\varphi-m|\leq 1\Big\}\\
&\lesssim \min(N_{12},N_1)^{-1} (N_{12} N_2)^3. 
\end{align*}
\item \label{tools:item_sup_4} Taking the  supremum in $n_3$ and counting $n,n_{12},n_2$, we have 
\begin{align*}
&\sup_{n\in \bZ^3}  \# \Big\{ (n,n_{12},n_2)\colon |n|\sim N_{123}, |n_{12}|\sim N_{12}, |n_2|\sim N_2, n=n_{123}, |\varphi-m|\leq 1\Big\}\\
&\lesssim \min(N_{12},N_1)^{-1} (N_{12} N_2)^3. 
\end{align*}
\end{enumerate}
\end{lemma}

\subsubsection{Para-controlled cubic counting estimate}

We now present our final cubic counting estimate. It will be used to control 
\begin{equation*}
 \nparaboxld \hspace{-0.8ex} \lcol  \Big( V \ast \Big( P_{\leq N} \<1b> \cdot P_{\leq N}   X_N  \Big)  P_{\leq N} \<1b> \Big) \rcol ~, 
\end{equation*}
which appears in $\CPara$. 

\begin{lemma}[Para-controlled cubic sum estimate]\label{tools:lem_paracontrolled_counting}
Let $N_{123},N_1,N_2,N_3\geq 1$ and $m\in\bZ$.  Let the signs $\pm_{123},\pm_1,\pm_2,\pm_3\in \{+,-\}$ be given and define the phase 
\begin{equation*}
\varphi(n_1,n_2,n_3)\defe \pm_{123} \langle n_{123} \rangle \pm_1 \langle n_1 \rangle  \pm_2 \langle n_2 \rangle  \pm_3 \langle n_3 \rangle.
\end{equation*}
Then, it holds that for all $0<\gamma<\beta$ that 
\begin{equation}
\begin{aligned}
& \sup_{\substack{n_2 \in \bZ^3 \colon \\ |n_2|\sim N_2}}
\sum_{n_1,n_3\in \bZ^3} \Big( \prod_{j=1,3} 1\big\{ |n_j| \sim N_j \big\} \Big) \langle n_{123} \rangle^{2(s_2-1)} \langle n_{12} \rangle^{-2\beta} \langle n_1 \rangle^{-2} \langle n_3 \rangle^{-2} \, 1\big\{ |\varphi-m|\leq 1\big\} \\
&\lesssim \max(N_1,N_2,N_3)^{2\delta_2} N_1^{-2\gamma} N_2^{2\gamma}. 
\end{aligned}
\end{equation}
\end{lemma}
\subsubsection{Quartic counting estimates}

Our expansion of the solution $u_N$ and  $\So$ only contain cubic, quintic, and septic stochastic object. The quartic counting estimates will be used to control products such as 
\begin{equation*}
P_{\leq N} \<1b> \cdot  P_{\leq N}  \<3DN>,
\end{equation*}
which occur as factors in the physical term $\Phy$. We present two estimates which control the non-resonant (Lemma \ref{tools:lemma_nonresonant_quartic}) and resonant portions (Lemma \ref{tools:lemma_resonant_quartic}) of the product, respectively. On our way to the resonant estimate, we also prove the basic resonance estimate (Lemma \ref{tools:lem_basic_resonance}). 

\begin{lemma}[Non-resonant quartic sum estimate]\label{tools:lemma_nonresonant_quartic}
Let  $s<-1/2-\eta$ and let  $N_1,N_2,N_3,N_4\geq 1$.  Let the signs $\pm_{123},\pm_1,\pm_2,\pm_3\in \{+,-\}$ be given and define 
\begin{equation*}
\varphi(n_1,n_2,n_3)\defe \pm_{123} \langle n_{123} \rangle \pm_1 \langle n_1 \rangle  \pm_2 \langle n_2 \rangle  \pm_3 \langle n_3 \rangle.
\end{equation*}
Then, it holds that 
\begin{align*}
&\sup_{m\in \bZ} \sum_{n_1,n_2,n_3,n_4\in \bZ^3} \Big( \prod_{j=1}^{4}1\big\{ |n_j|\sim N_j \big\} \Big) \langle n_{1234} \rangle^{2s} \langle n_{123} \rangle^{-2} 
|\widehat{V}_S(n_1,n_2,n_3)|^2 
\Big(\prod_{j=1}^4 \langle n_j \rangle^{-2}\Big) 1\big\{ |\varphi-m|\leq 1\big\} \\
&\lesssim \max(N_1,N_2,N_3)^{-2\beta+2\eta} N_4^{-2\eta}. 
\end{align*}
\end{lemma}

\begin{lemma}[Basic resonance estimate]\label{tools:lem_basic_resonance}
Let $n_1,n_2 \in \bZ^3$ be arbitrary, let $N_3\geq 1$, let the signs $\pm_{123},\pm_1,\pm_2,\pm_3\in \{+,-\}$ be given, and define 
\begin{equation*}
\varphi(n_1,n_2,n_3)\defe \pm_{123} \langle n_{123} \rangle \pm_1 \langle n_1 \rangle  \pm_2 \langle n_2 \rangle  \pm_3 \langle n_3 \rangle.
\end{equation*}
 Then, it holds that 
 \begin{equation}\label{tools:eq_basic_resonance}
  \sum_{m\in \bZ} \sum_{n_3 \in \bZ^3} \langle m \rangle^{-1} 1\big\{ |n_3|\sim N_3\big\} 
 \langle n_{123} \rangle^{-1} \langle n_{3} \rangle^{-2} 1\big\{ |\varphi-m| \leq 1 \big\} \lesssim \log(2+N_3) \langle n_{12} \rangle^{-1}. 
 \end{equation}
\end{lemma}

\begin{lemma}[Resonant quartic sum estimate]\label{tools:lemma_resonant_quartic}
Let $N_1,N_2,N_3\geq 1$ and let $-1/2 <s < 0 $. Let the signs $\pm_{123},\pm_1,\pm_2,\pm_3\in \{+,-\}$ be given and define 
\begin{equation*}
\varphi(n_1,n_2,n_3)\defe \pm_{123} \langle n_{123} \rangle \pm_1 \langle n_1 \rangle  \pm_2 \langle n_2 \rangle  \pm_3 \langle n_3 \rangle.
\end{equation*}
 Then, it holds that 
\begin{align*}
&\sum_{n_1,n_2 \in \bZ^3} \bigg[  \Big( \prod_{j=1}^2 1 \big\{ |n_j| \sim N_j \big\} \langle n_{12} \rangle^{2s} \langle n_1 \rangle^{-2} \langle n_2 \rangle^{-2} \\
&\times \Big( \sum_{m\in \bZ} \sum_{n_3 \in \bZ^3} \langle m \rangle^{-1} 1\big\{ |n_3|\sim N_3\big\} 
 \langle n_{123} \rangle^{-1} \langle n_{3} \rangle^{-2} 1\big\{ |\varphi-m| \leq 1 \big\} \Big)^2  \bigg] \\
&\lesssim  \log(2+ N_3)^2 \max(N_1,N_2)^{2s}. 
\end{align*}
\end{lemma}

\subsubsection{Quintic counting estimates}

In order to estimate  the quintic stochastic objects
\begin{equation*}
 \<113N> \qquad \text{and} \qquad \<131N>,
\end{equation*}
we require quintic sum estimates. Even at the quintic level, we need to make full use of dispersive effects. This is in contrast to the septic counting effects, which only rely on dispersive effects for cubic sub-objects but do not require dispersive effects at the full septic level. \\

We present three separate quintic sum estimates, which correspond to zero, one, or two probabilistic resonances.

\begin{lemma}[Non-resonant quintic sum estimate]\label{tools:lem_counting_quintic}
Let  $s\leq 1/2-2\eta$ and $N_1,N_2,N_3,N_4,N_5\geq1$. Furthermore, we define three phase-functions by 
\begin{align*}
\psi(n_3,n_4,n_5) &\defe \pm_{345} \langle n_{345} \rangle \pm_3 \langle n_3 \rangle \pm_4 \langle n_4 \rangle \pm_5 \langle n_5 \rangle, \\ 
\varphi(n_1,\hdots,n_5)
 &\defe \pm_{12345} \langle n_{12345}\rangle \pm_{345} \langle n_{345} \rangle \pm_1 \langle n_1 \rangle \pm_2 \langle n_2 \rangle,\\
\widetilde{\varphi}(n_1,\hdots,n_5) &\defe \pm_{12345} \langle n_{12345} \rangle \mp_{345} \langle n_{345} \rangle +\sum_{j=1}^5 (\pm_j ) \langle n_j \rangle . 
\end{align*} 
Then, it holds that 
\begin{align*}
&\sup_{m,m^\prime\in\bZ} \sum_{n_1,\hdots,n_5\in \bZ^3} \bigg[ \Big( \prod_{j=1}^{5} 1\big\{|n_j|\sim N_j\} \big) \langle n_{12345} \rangle^{2(s-1)} \langle n_{1345} \rangle^{-2\beta} 
\langle n_{345} \rangle^{-2} \langle n_{34} \rangle^{-2\beta} \Big( \prod_{j=1}^5 \langle n_j \rangle^{-2} \Big) \\
&\times 1\big\{ |\psi-m|\leq 1 \big\} \cdot \Big( 1\big\{ |\varphi-m^\prime|\leq 1\} +  1\big\{ |\widetilde{\varphi}-m^\prime|\leq 1\} \Big) \bigg] \\
&\lesssim \max(N_1,N_3,N_4,N_5)^{-2\beta+4\eta} N_2^{-2\eta}. 
\end{align*}
\end{lemma}

\begin{lemma}[Single-resonance quintic sum estimate]\label{tools:lem_counting_single_resonance_quintic}
Let $n_4,n_5\in \bZ^3$, $N_{45}\geq 1$, and $|n_{45}|\sim N_{45}$. Furthermore, let $\pm_3 \in \{+,-\}$ . Then, it holds that 
\begin{align*}
&\sup_{m\in \bZ^3} \sum_{n_3\in \bZ^3} \bigg[ 1\big\{ |n_3|\sim N_3\big\} \langle n_{345} \rangle^{-1} \langle n_3 \rangle^{-2}
1\big\{ \langle n_{345} \rangle \pm_3 \langle n_3 \rangle \in [m,m+1)\big\} \bigg] \\
&\lesssim N_{45}^{-1}. 
\end{align*}
\end{lemma}

After renaming the variables, Lemma \ref{tools:lem_counting_single_resonance_quintic} is essentially the same as Lemma \ref{tools:lem_basic_resonance}. Our reason for restating Lemma \ref{tools:lem_counting_single_resonance_quintic} is to make it easier for the reader to refer back to this section.

\begin{lemma}[Double-resonance quintic sum estimate]\label{tools:lem_counting_double_resonance_quintic}
Let $N_3,N_4,N_5\geq 1$ and let $\pm_3,\pm_4,\pm_5\in \{+,-\}\). Then, it holds that  
\begin{equation}
\begin{aligned}
& \sup_{m\in \bZ^3} \sup_{|n_5|\sim N_5} \sum_{n_3,n_4\in \bZ^3} \bigg[ \Big( \prod_{j=3}^4 1\big\{ |n_j|\sim N_j \big \}\Big) \langle n_{345} \rangle^{-1} \langle n_{45} \rangle^{-\beta} \langle n_3 \rangle^{-2} \langle n_4 \rangle^{-2} \\
&\times  1\big\{ \langle n_{345} \rangle \pm_3 \langle n_3 \rangle \pm_4 \langle n_4 \rangle \pm_5 \langle n_5 \rangle \in [m,m+1) \big\} \bigg] \\
&\lesssim \max(N_4,N_5)^{-\beta+\eta}.  
\end{aligned}
\end{equation}
\end{lemma}

\subsubsection{Septic counting estimates}\label{section:counting_septic}

In order to state our septic counting estimates, we need to introduce pairings, where our definition is motivated by a similar notion in \cite[Section 1.9]{DNY19}. The pairings are designed to capture the resonances in the septic stochastic objects 
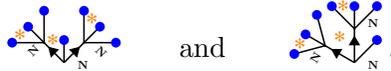
\begin{equation*}
\<313N> \qquad \text{and} \qquad \<331N>. 
\end{equation*}

\begin{definition}[Pairings]\label{tools:def_pairings}
Let $J\geq 1$. We call a relation $\scrP \subseteq \{ 1,\hdots,J\}^2 $ a pairing if 
\begin{enumerate}[(i)]
\item $\scrP$ is anti-reflexive, i.e, $(j,j)\not \in \scrP$ for all $1\leq j\leq J$, 
\item $\scrP$ is symmetric, i.e., $(i,j)\in \scrP$ if and only if $(j,i)\in \scrP$, 
\item $\scrP$ is univalent, i.e., for each $1\leq i \leq J$, $(i,j) \in \scrP$ for at most one $1\leq j \leq J$. 
\end{enumerate}
If $(i,j)\in \scrP$, the tuple $(i,j)$ is called a pair (or $\scrP$-pair). If $1\leq j \leq J$ is contained in a pair, we call $j$ paired (or $\scrP$-paired). With a slight abuse of notation, we also write $j\in \scrP$ if $j$ is paired. If $j$ is not paired, we also say that $j$ is unpaired and write $j\not \in \scrP$. 

Furthermore, let $\mathcal{A}=(\mathcal{A}_l)_{l=1,\hdots,L}$ be a partition of $\{1,\hdots,J\}$. We say that $\scrP$ respects $\mathcal{A}$ if $i,j\in \mathcal{A}_l$ for some $1\leq l \leq L$ implies that $(i,j)\not \in \scrP$. In other words, $\scrP$ does not pair elements of the same set inside the partition.

Finally, we call a vector $(n_1,\hdots,n_J)\in (\bZ^3)^J$ of frequencies admissible (or $\scrP$-admissible) if $(i,j)\in \scrP$ implies that $n_i=-n_j$.  
\end{definition}

Using Definition \ref{tools:def_pairings}, we can now state the septic sum estimate.

\begin{lemma}[Septic sum estimate]\label{tools:lem_counting_septic}
Let $1/2<s<1$ and let $N_{1234567},N_{1234},N_{567},N_4 \geq 1$. For any $\pm_1,\pm_2,\pm_3\in \{ +,-\}$, we define the phase
\begin{equation*}
\varphi(n_j,\pm_j\colon 1 \leq j \leq 3) \defe \langle n_{123} \rangle \pm_1 \langle n_1 \rangle \pm_2 \langle n_2 \rangle \pm_3 \langle n_3 \rangle. 
\end{equation*}
Furthermore, we define
\begin{equation*}
\Phi(n_1,n_2,n_3) = \sum_{\pm_1,\pm_2,\pm_3} \sum_{m\in \bZ} \langle m \rangle^{-1} |\widehat{V}_S(n_1,n_2,n_3)| \langle n_{123} \rangle^{-1} \Big( \prod_{j=1}^3 \langle n_j \rangle^{-1} \Big) 1\big\{ |\varphi-m|\leq 1 \big\}. 
\end{equation*}
Finally, let $\scrP$ be a pairing of $\{1,\hdots,7\}$ which respects the partition $\{1,2,3\},\{4\},\{5,6,7\}$ and define the non-resonant frequency $n_{\text{nr}}\in \bZ^3$ by 
\begin{equation*}
n_{\text{nr}} \defe \sum_{j\not \in \scrP} n_j. 
\end{equation*}
Then, it holds that 
\begin{align*}
&\sum_{(n_j)_{j\not \in \scrP}} \hspace{-1ex} \langle n_{\text{nr}} \rangle^{2(s-1)}
 \bigg(  \hspace{-0.2ex} \sum_{\substack{(n_j)_{j \in \scrP}}}^{\ast}  \hspace{-1ex}1\big\{ |n_{1234567}|\sim N_{1234567}\big\} 1\big\{ |n_{1234}|\sim N_{1234}\big\} 1\big\{|n_{567}| \sim N_{567} \big\} 1\big\{|n_4|\sim N_4\big\}  \\
 &\times |\widehat{V}(n_{1234})| \Phi(n_1,n_2,n_3) \langle n_4 \rangle^{-1} \Phi(n_5,n_6,n_7) \bigg)^2 \\
 &\lesssim \log(2+N_4)^2 \Big( N_{1234567}^{2(s-\frac{1}{2})}   N_{567}^{-2(\beta-\eta)} + N_{1234567}^{-2 (1-s-\eta)} \Big) N_{1234}^{-2\beta},
\end{align*}
where $ \sum_{\substack{(n_j)_{j \in \scrP}}}^{\ast}$ denotes the sum over admissible frequencies. 
\end{lemma}

While the septic sum estimate (Lemma \ref{tools:lem_counting_septic}) may appear complicated, its proof is much easier than the cubic sum estimate (Lemma \ref{tools:prop_cubic_sum}) or the quintic sum estimate (Lemma \ref{tools:lem_counting_quintic}). The reason is that we do not rely on dispersive effects at the (full) septic level, and only use the dispersive effects in the cubic stochastic sub-objects.

\subsubsection{Tensor estimates}\label{section:tensor}

The counting estimates from Section \ref{section:counting_cubic}-\ref{section:counting_septic} will be combined with Wiener chaos estimates to control stochastic objects such as $\<3N>$. The estimates of the random matrix terms will follow a similar spirit. However, the Wiener chaos estimates will be replaced by the moment method (see Proposition \ref{tools:prop_moment_method}) and the counting estimates will be replaced by deterministic tensor estimates. The tensor estimates, which partially rely on the counting estimates, are the main goal of this subsection. 

We first recall the tensor notation from \cite[Section 2.1]{DNY20}. 

\begin{definition}[Tensors and tensor norms]\label{tool:def_tensor}
Let $\cJ \subseteq \mathbb{N}_0$ be a finite set. A tensor $h=h_{n_\cJ}$ is a function from $(\bZ^3)^{|\cJ|}$ into $\bC$, where the input variables are given by $n_{\cJ}$. A partition of $\cJ$ is a pair of sets $(\mathcal{A},\mathcal{B})$ such that $\mathcal{A} \medcup \mathcal{B}=\cJ$ and $\mathcal{A} \medcap \mathcal{B} = \emptyset$. For any partition $(\mathcal{A},\mathcal{B})$, we define the tensor norm
\begin{equation}
\| h\|_{n_{\mathcal{A}}\rightarrow n_{\mathcal{B}}}^2 = \sup \Big\{ \sum_{n_\mathcal{B}} \Big| \sum_{n_{\mathcal{A}}} h_{n_\cJ} z_{n_{\mathcal{A}}} \Big|^2 \colon \sum_{n_{\mathcal{A}}} | z_{n_{\mathcal{A}}}|^2 = 1 \Big\}. 
\end{equation}
\end{definition}

For example, if $h=h_{nn_1n_2n_3}$, then
\begin{equation*}
\| h \|_{n_1 n_2 n_3 \rightarrow n}^2 = \sup \Big\{ \sum_{n \in \bZ^3} \Big| \sum_{n_1,n_2,n_3\in \bZ^3} h_{n n_1 n_2 n_3} z_{n_1 n_2 n_3} \Big|^2 \colon \sum_{n_1,n_2,n_3\in \bZ^3} |z_{n_1 n_2 n_3}|^2 =1 \Big\}. 
\end{equation*}

\begin{lemma}[First deterministic tensor estimate]\label{tools:lem_first_tensor}
Let $s<1/2+\beta-2\delta_1-6\eta$, $N_1,N_2,N_3,N_{12},N_{123}\geq 1$, $m\in \bZ$, and $\pm_1,\pm_2,\pm_3,\pm_{123}\in \{ +,-\}$. Define the phase-function $\varphi$ by
\begin{equation*}
\varphi(n_1,n_2,n_3)\defe \pm_{123} \langle n_{123} \rangle \pm_1 \langle n_1 \rangle  \pm_2 \langle n_2 \rangle  \pm_3 \langle n_3 \rangle.
\end{equation*}
and the truncated tensor $h$ by 
\begin{equation}\label{tools:eq_first_tensor}
\begin{aligned}
h_{nn_1n_2n_3} \defe& \chi_{N_{123}}(N_{123}) \chi_{N_{12}}(n_{12}) \Big( \prod_{j=1}^3 \rho_{\leq N}(n_j)\chi_{N_j}(n_j) \Big)  \\
&1\big\{n=n_{123} \big\} 1\{ |\varphi-m|\leq 1\} \langle n \rangle^{s-1} \widehat{V}(n_{12}) \langle n_1 \rangle^{-1} \langle n_2 \rangle^{-1} \langle n_3 \rangle^{-s_1}. 
\end{aligned}
\end{equation}
Then, we have the estimate
\begin{equation}
\max\big( \| h \|_{n_1n_2n_3\rightarrow n},\| h \|_{n_3 \rightarrow nn_1 n_2},\| h \|_{n_1 n_3\rightarrow nn_2} , \| h \|_{n_2n_3\rightarrow nn_1} \big) 
\lesssim \max(N_1,N_2,N_3)^{-\eta}. 
\end{equation}
\end{lemma}

\begin{remark}
The first deterministic tensor estimate (Lemma \ref{tools:lem_first_tensor}) is the main ingredient in the estimate of 
\begin{equation*}
w_N \mapsto \Big( V \ast \<2N> \Big) P_{\leq N} w_N,
\end{equation*}
which is the first term in $\RMT$. In contrast to the second tensor estimate below, we only impose $s<1/2+\beta$ instead of $s<1/2$ (up to small corrections). The reason is that both instances of $\<1b>$ are part of the convolution with $V$. 
\end{remark}
\begin{proof}
The main ingredients are Schur's test and the sup-counting estimate (Lemma \ref{tools:lem_sup}). 

\emph{Step 1: $\| h\|_{n_1 n_2 n_3 \rightarrow n}$}. Due to the symmetry $n_1 \leftrightarrow n_2$, we may assume that $N_1 \geq N_2$. Using Schur's test, we have that 
\begin{align*}
& \| h \|_{n_1 n_2  n_3 \rightarrow n}^2 \\
&\lesssim N_{123}^{2(s-1)} N_{12}^{-2\beta} N_1^{-2} N_2^{-2} N_3^{-2s_1} \\
&\times \sup_{n\in \bZ^3} \sum_{n_1,n_2,n_3 \in \bZ^3}  \Big( \prod_{j=1}^3 1 \big\{ |n_j| \sim N_k \big \} \Big)  1\big\{ |n_{12}| \sim N_{12}| \big\} 1\big\{ |n| \sim N_{123} \big\} 1 \big \{ n=n_{123} \big\} 1 \big \{ |\varphi-m|\leq 1 \big\} \\
&\times \sup_{n_1,n_2,n_3\in \bZ^3} \sum_{n  \in \bZ^3}  \Big( \prod_{j=1}^3 1 \big\{ |n_j| \sim N_k \big \} \Big)  1\big\{ |n_{12}| \sim N_{12}| \big\} 1\big\{ |n| \sim N_{123} \big\} 1 \big \{ n=n_{123} \big\} 1 \big \{ |\varphi-m|\leq 1 \big\}
\end{align*}
Since $n$ is uniquely determined by $n_1,n_2$, and $n_3$, the last factor can easily be bounded by one. By using \eqref{tools:item_sup_3} in Lemma \ref{tools:lem_sup} and $\max(N_{12},N_2)\lesssim\max(N_1,N_2) = N_1$, we obtain that 
\begin{align*}
\| h \|_{n_1 n_2  n_3 \rightarrow n}^2 
&\lesssim N_{123}^{2(s-1)} N_{12}^{-2\beta} N_1^{-2} N_2^{-2} N_3^{-2s_1}  \max(N_{12},N_2) N_{12}^2 N_2^2 \\
&\lesssim N_{123}^{2(s-1)} N_{12}^{2-2\beta} N_1^{-1} N_3^{-2s_1}  \\
&\lesssim N_{123}^{2(s-1)} N_{12}^{1-2\beta+2\eta} N_1^{-2\eta} N_3^{-2s_1}.
\end{align*}
Furthermore, we have that $N_{12} \lesssim \max(N_{123},N_3) \lesssim N_{123} \cdot N_3$. Inserting this into the last inequality yields
\begin{equation*}
\| h \|_{n_1 n_2  n_3 \rightarrow n}^2   \lesssim N_{123}^{2s-1-2\beta+2\eta} N_1^{-2\eta}N_3^{1-2s_1-2\beta +2 \eta} \lesssim (N_1 N_3)^{-2\eta}. 
\end{equation*}

\emph{Step 2: $\| h\|_{n_3 \rightarrow n_1 n_2 n}$}. The argument follow Step 1 nearly verbatim, except that we use \eqref{tools:item_sup_4} in Lemma \ref{tools:lem_sup} instead of \eqref{tools:item_sup_3}. \\

\emph{Step 3: $\| h \|_{n_1 n_3\rightarrow n_2 n}$}. In this step, we ignore the dispersive effects, i.e., we simply bound
\begin{equation*}
 1 \big \{ |\varphi-m|\leq 1 \big\} \leq 1. 
\end{equation*}
By increasing $s$ if necessary, we may assume $s \geq 1/2$. Using Schur's test and a simple volume argument, we have that 
\begin{align*}
& \| h \|_{n_1 n_3 \rightarrow n_2 n }^2 \\
&\lesssim N_{123}^{2(s-1)} N_{12}^{-2\beta} N_1^{-2} N_2^{-2} N_3^{-2s_1} \allowdisplaybreaks[4]\\
&\times \sup_{n,n_2\in \bZ^3} \sum_{n_1,n_3 \in \bZ^3}  \Big( \prod_{j=1}^3 1 \big\{ |n_j| \sim N_k \big \} \Big)  1\big\{ |n_{12}| \sim N_{12}| \big\} 1\big\{ |n| \sim N_{123} \big\} 1 \big \{ n=n_{123} \big\}  \allowdisplaybreaks[4]\\
&\times \sup_{n_1,n_3\in \bZ^3} \sum_{n_2,n \in \bZ^3}  \Big( \prod_{j=1}^3 1 \big\{ |n_j| \sim N_k \big \} \Big)  1\big\{ |n_{12}| \sim N_{12}| \big\} 1\big\{ |n| \sim N_{123} \big\} 1 \big \{ n=n_{123} \big\}  \allowdisplaybreaks[4]\\
&\lesssim  N_{123}^{2(s-1)} N_{12}^{-2\beta} N_1^{-2} N_2^{-2} N_3^{-2s_1} \min(N_1,N_{12},N_3)^3 \min(N_2,N_{12},N_{123})^3 \allowdisplaybreaks[4]\\
&\lesssim   N_{123}^{2(s-1)} N_{12}^{-2\beta} N_1^{-2} N_2^{-2} N_3^{-2s_1} N_1^{2-2\eta} N_{12}^{1+4\eta-2s_1} N_3^{2s_1-2\eta} N_2^{2-2\eta} N_{12}^{2s-1+2\eta} N_{123}^{2(s-1)} \\
&\lesssim N_{12}^{2s-1-2\beta+2\delta_1 + 6\eta} (N_1 N_2 N_3)^{-2\eta}. 
\end{align*}
In the second last inequality, we used $s \geq 1/2$. Since $2s-1-2\beta+2\delta_1 + 6\eta\leq 0$, this is acceptable. \\

\emph{Step 4: $\| h \|_{n_2 n_3\rightarrow n_1 n}$}. Due to the symmetry $n_1 \leftrightarrow n_2$, the estimate follows from Step 3. 
\end{proof}

We now turn to the second tensor estimate. 

\begin{lemma}[Second deterministic tensor estimate]\label{tools:lem_second_tensor}
Let $s<1/2-\eta$, $N_1,N_2,N_3,N_{12},N_{123}\geq 1$, $m\in \bZ$, and $\pm_1,\pm_2,\pm_3,\pm_{123}\in \{ +,-\}$. Define the phase-function $\varphi$ by
\begin{equation*}
\varphi(n_1,n_2,n_3)\defe \pm_{123} \langle n_{123} \rangle \pm_1 \langle n_1 \rangle  \pm_2 \langle n_2 \rangle  \pm_3 \langle n_3 \rangle
\end{equation*}
and the truncated tensor $h$ by 
\begin{equation}
\begin{aligned}
h_{nn_1n_2n_3} \defe& \chi_{N_{123}}(N_{123}) \chi_{N_{12}}(n_{12}) \Big( \prod_{j=1}^3 \rho_{\leq N}(n_j)\chi_{N_j}(n_j) \Big)  \\
&1\big\{n=n_{123} \big\} 1\{ |\varphi-m|\leq 1\} \langle n \rangle^{s-1} \widehat{V}(n_{12}) \langle n_1 \rangle^{-1} \langle n_2 \rangle^{-s_2} \langle n_3 \rangle^{-1}. 
\end{aligned}
\end{equation}
Then, we have the estimate
\begin{equation}
\max\big( \| h \|_{n_1n_2n_3\rightarrow n},\| h \|_{n_2\rightarrow nn_1 n_3}, \| h \|_{n_2n_3\rightarrow nn_1}, \| h \|_{n_1 n_2\rightarrow nn_3} \big) 
\lesssim N_{12}^{-\beta} \max(N_1,N_2,N_3)^{-\eta}. 
\end{equation}
\end{lemma}

\begin{remark}
Lemma \ref{tools:lem_second_tensor} is the main ingredient in the estimate of
\begin{equation*}
Y_N \mapsto \lcol  V \ast \Big( P_{\leq N} \<1b> \cdot P_{\leq N} \big(Y_N \big) \Big)  \nparald P_{\leq N} \<1b> \Big)\rcol ~ ,
\end{equation*}
which is the second term in $\RMT$. 
\end{remark}

\begin{proof}
The argument is similar to the proof of Lemma \ref{tools:lem_first_tensor}.\\

\emph{Step 1: $\| h\|_{n_1 n_2 n_3 \rightarrow n}$}. Using Schur's test, we have that 
\begin{equation*}
\begin{aligned}
\| h\|_{n_1 n_2 n_3 \rightarrow n}^2  &\lesssim N_{123}^{2 (s-1)} N_{12}^{-2\beta} N_1^{-2} N_2^{-2 s_2} N_3^{-2} \\
&\times \sup_{|n|\sim N_{123}} \sum_{n_1,n_2,n_3 \in \bZ^3} \Big( \prod_{j=1}^3 1 \big\{ |n_j| \sim N_k \big \} \Big) 1 \big \{ n=n_{123} \big\} 1 \big \{ |\varphi-m|\leq 1 \big\} \\
&\times \sup_{n_1,n_2,n_3 \in \bZ^3} \sum_{n\in \bZ^3} \Big( \prod_{j=1}^3 1 \big\{ |n_j| \sim N_k \big \} \Big) 1 \big \{ n=n_{123} \big\} 1 \big \{ |\varphi-m|\leq 1 \big\}. 
\end{aligned}
\end{equation*}
The last factor  is easily bounded by one, since $n$ is uniquely determined by $n_1,n_2$, and $n_3$. By using \eqref{tools:item_sup_1} in Lemma \ref{tools:lem_sup} and $s_2 \leq 1$, we obtain that 
\begin{align*}
\| h\|_{n_1 n_2 n_3 \rightarrow n}^2
&\lesssim   N_{123}^{2 (s-1)} N_{12}^{-2\beta}  \med(N_1,N_2,N_3)^3 \min(N_1,N_2,N_3)^2 N_1^{-2} N_2^{-2 s_2} N_3^{-2} \\
&\lesssim  N_{123}^{2 (s-1)} N_{12}^{-2\beta} \max(N_1,N_2,N_3)^{-2s_2}   \med(N_1,N_2,N_3)^{3-2} \min(N_1,N_2,N_3)^{2-2} \\
&\lesssim N_{123}^{2 (s-1)} N_{12}^{-2\beta} \max(N_1,N_2,N_3)^{1-2s_2}. 
\end{align*}
This is acceptable since $s\leq 1$ and $\eta \ll \delta_2$.

\emph{Step 2: $\| h\|_{n_2\rightarrow n_1 n_3 n}$}. This argument is similar to Step 1, but the roles of $n_2$ and $n$ are reversed. Using Schur's test, we obtain that 
\begin{equation*}
\begin{aligned}
\| h\|_{n_2 \rightarrow n_1 n_3 n}^2  &\lesssim N_{123}^{2 (s-1)} N_{12}^{-2\beta} N_1^{-2} N_2^{-2 s_2} N_3^{-2} \\
&\times \sup_{|n_2|\sim N_{2}} \sum_{n_1,n_3,n \in \bZ^3} \Big( \prod_{j=1}^3 1 \big\{ |n_j| \sim N_k \big \} \Big)  1\big\{ |n| \sim N_{123} \big\} 1 \big \{ n=n_{123} \big\} 1 \big \{ |\varphi-m|\leq 1 \big\} \\
&\times \sup_{n_1,n_3,n \in \bZ^3} \sum_{n_2\in \bZ^3} \Big( \prod_{j=1}^3 1 \big\{ |n_j| \sim N_k \big \} \Big) 1\big\{ |n| \sim N_{123} \big\}1 \big \{ n=n_{123} \big\} 1 \big \{ |\varphi-m|\leq 1 \big\}. 
\end{aligned}
\end{equation*}
As before, the last factor is easily bounded by one. By using \eqref{tools:item_sup_2} in Lemma \ref{tools:lem_sup} and $2(s-1)\geq -2$, we obtain that 
\begin{align*}
\| h\|_{n_2 \rightarrow n_1 n_3 n}^2  &\lesssim N_{123}^{2 (s-1)} N_{12}^{-2\beta} \med(N_{123},N_1,N_3)^3 \min(N_{123},N_1,N_3)^2 N_1^{-2} N_2^{-2 s_2} N_3^{-2} \\
&\lesssim N_{12}^{-2\beta} N_2^{-2s_2} \max(N_{123},N_1,N_3)^{2s-1} \\
&\lesssim N_{12}^{-2\beta} N_2^{-2s_2} \max(N_1,N_2,N_3)^{-2\eta}.
\end{align*}
In the last line, we used that $s<1/2-\eta$.

\emph{Step 3: $\| h\|_{ n_1 n_2\rightarrow n_3 n}$}. In this step, we ignore the dispersive effects, i.e., we simply bound
\begin{equation*}
 1 \big \{ |\varphi-m|\leq 1 \big\} \leq 1. 
\end{equation*}
Using Schur's test and a simple volume bound, we obtain that 
\begin{align*}
\| h\|_{ n_1 n_2\rightarrow n_3 n}^2 &\lesssim  N_{123}^{2 (s-1)} N_{12}^{-2\beta} N_1^{-2} N_2^{-2 s_2} N_3^{-2}\\
&\times \sup_{n_3,n\in \bZ^3} \sum_{n_1,n_2\in \bZ^3}  \Big( \prod_{j=1}^3 1 \big\{ |n_j| \sim N_k \big \} \Big)  1\big\{ |n| \sim N_{123} \big\} 1 \big \{ n=n_{123} \big\} \\
&\times \sup_{n_1,n_2\in \bZ^3} \sum_{n_3,n\in \bZ^3}  \Big( \prod_{j=1}^3 1 \big\{ |n_j| \sim N_k \big \} \Big)  1\big\{ |n| \sim N_{123} \big\} 1 \big \{ n=n_{123} \big\} \\
&\lesssim  N_{123}^{2 (s-1)} N_{12}^{-2\beta} N_1^{-2} N_2^{-2 s_2} N_3^{-2} \min(N_1,N_2)^3 \min(N_3,N_{123})^3 \\
&\lesssim  N_{123}^{2 (s-1)} N_{12}^{-2\beta} N_1^{-2} N_2^{-2 s_2} N_3^{-2} N_1^{2-2\eta} N_2^{1+2\eta} N_3^{2-2\eta} N_{123}^{1+2\eta}  \\
&\lesssim N_{12}^{-2\beta} \max(N_1,N_2,N_3)^{-2\eta}. 
\end{align*}

\emph{Step 4: $\| h\|_{ n_2 n_3\rightarrow n_1 n}$} Arguing exactly as in Step 3, we obtain that 
\begin{align*}
\| h\|_{ n_2 n_3\rightarrow n_1 n}^2 &\lesssim   N_{123}^{2 (s-1)} N_{12}^{-2\beta} N_1^{-2} N_2^{-2 s_2} N_3^{-2} \min(N_2,N_3)^3 \min(N_1,N_{123})^3 \\
&\lesssim N_{12}^{-2\beta} \max(N_1,N_2,N_3)^{-2\eta}. 
\end{align*}
\end{proof}

\subsection{Gaussian processes}\label{section:gaussian_processes}
We briefly review the notation from the stochastic control perspective of the first paper in this series \cite{BB20a}, which was used in the proof of Theorem \ref{theorem:measures}. In comparison with the first part of this series, however, we change the notation for the stochastic time variable.  We use $\cs$, which is a calligraphic ``s", to denote the time-variable in the stochastic control perspective. While the chosen font in $\cs$ may be slightly unusual, we hope that this prevents any confusion with the time-variable $t$ in the nonlinear wave equation.

We let \( (B_\cs^n)_{n\in \bZ^3\backslash\{0\}} \)  be a sequence of standard complex Brownian motions such that \( B_\cs^{-n} = \overline{B_\cs^n} \) and \( B_\cs^n, B_\cs^m \) are independent for \( n \neq \pm m \). We let \( B_\cs^0 \) be a standard real-valued Brownian motion independent of  \( (B_\cs^n)_{n\in \bZ^3\backslash\{0\}} \). Furthermore, we let \( B_\cs(\cdot) \) be the Gaussian process with Fourier coefficients \( (B_\cs^n)_{n \in \bZ^3} \), i.e., 
\begin{equation*}
B_\cs(x) \defe \sum_{n\in \bZ^3} e^{i\langle n,x \rangle} B_\cs^n. 
\end{equation*}
For every \( \cs\geq 0 \), the Gaussian process formally satisfies \( \bE[ B_\cs (x)B_\cs(y) ] = \cs \cdot \delta(x-y)  \) and hence \( B_\cs(\cdot)\) is a scalar multiple of spatial white noise.   We also let \( (\mathcal{F}_\cs)_{\cs\geq 0} \) be the filtration corresponding to the family of Gaussian processes \( (B_\cs^n)_{\cs\geq 0} \).

The Gaussian free field \( \cg \), however, has covariance \( (1-\Delta)^{-1} \). To this end, we now introduce the Gaussian process \( W_\cs(x) \). We let \( \sigma_\cs(\xi) = \big( \frac{\mathrm{d}}{\mathrm{d}\cs} \rho_\cs^2(\xi)\big)^{1/2} \), where $\rho_\cs$ is the frequency-truncation from Section \ref{section:notation}. For any $n \in \bZ^3$, we then define
\begin{equation}\label{tools:eq_wtn}
W_\cs^{n} \defe \int_0^\cs \frac{\sigma_{\cs^\prime}(n)}{\langle n\rangle} \,  \mathrm{d} B_{\cs^\prime}^n ~. 
\end{equation}
We note that \( W_\cs^n \) is a complex Gaussian random variable with variance \( \rho_\cs^2(n)/\langle n \rangle^2\). We finally set 
\begin{equation}
W_\cs(x) \defe \sum_{n\in \bZ^3} e^{i \langle n , x \rangle} W_\cs^n. 
\end{equation} 
Since the Gaussian random data $\,\bluedot \in \cH_x^{-1/2-\kappa}(\bT^3)$ in Theorem \ref{theorem:measures} is a tuple of the initial data and initial velocity, we now let $(B^{\cos},W^{\cos})$ and $(B^{\sin},W^{\sin})$ be two independent copies of $(B,W)$. Using this notation, we then take
\begin{equation}\label{tools:eq_bluedot}
\bluedot = \begin{pmatrix} W_\infty^{\cos}(x) , \langle \nabla \rangle W_\infty^{\sin}(x) \end{pmatrix}. 
\end{equation}
Using \eqref{tools:eq_bluedot}, we can represent the linear evolution as 
\begin{equation*}
\<1b>(t) = \cos(t\langle \nabla \rangle) W_\infty^{\cos}+ \sin(t\langle \nabla \rangle) W_\infty^{\sin},
\end{equation*}
which also motivates our notation. 

\subsection{Multiple stochastic integrals}\label{section:multiple_integrals}

In this section, we recall several definitions and results related to multiple stochastic integrals. A similar but shorter section already appeared in the appendix of the first paper of this series \cite{BB20a}. More detailed introductions can be found in the excellent textbook \cite{Nualart06} and lecture notes \cite{M14}. The usefulness of this section is best illustrated by Proposition   \ref{tools:prop_stochastic_representation} below.

We define a Borel measure $\lambda$ on $\bR_{\geq 0} \times \bZ^3$ by 
\begin{equation*}
\mathrm{d}\lambda(\cs,n) = \frac{\sigma_\cs^2(n)}{\langle n \rangle^2} \ds \mathrm{d}n,
\end{equation*}
where $\ds$ is the Lebesgue measure and $\dn$ is the counting measure on $\bZ^3$. We define the corresponding inner product by 
\begin{equation}
\langle f , g \rangle = \sum_{n\in \bZ^3} \int_0^\infty f(\cs,n) \overline{g(\cs,n)} \frac{\sigma_\cs^2(n)}{\langle n \rangle^2} \ds. 
\end{equation}
For any $f\in L^2(\bR_{\geq 0} \times \bZ^3, \mathrm{d}\lambda)$, we define
\begin{equation*}
W[f]= \sum_{n\in \bZ^3} \int_0^\infty f(\cs,n) \dWs{n}.
\end{equation*}

The inner integral can be understood as an It\^{o}-integral. Then, we can identify $W$ with the family of complex-valued Gaussian random variables 
\begin{equation*}
W= \{ W[f] \colon f \in L^2(\bR_{\geq 0} \times \bZ^3, \mathrm{d}\lambda)\}. 
\end{equation*}
For any $f \in L^2(\bR_{\geq 0} \times \bZ^3, \mathrm{d}\lambda)$, we define the reflection operator $\mathcal{R}$ by
\begin{equation*}
\mathcal{R}f(\cs,n) \defe \overline{f(\cs,-n)}. 
\end{equation*}
Clearly, $\mathcal{R}$ is a real-linear isometry. Using It\^{o}'s isometry, a short calculation yields that 
\begin{equation*}
\bE \big[ W[f] \overline{W[g]}\big] = \langle f , g \rangle \quad \text{and} \quad \bE \big[ W[f] {W[g]}\big] = \langle f, \mathcal{R} g \rangle. 
\end{equation*}
Since this will be important below, we note that the second identity reads
\begin{equation}
 \bE \big[ W[f] {W[g]}\big] =  \sum_{n\in \bZ^3} \int_0^\infty f(\cs,n) g(\cs,-n) \frac{\sigma_\cs^2(n)}{\langle n \rangle^2} \ds. 
\end{equation}
To emphasize the integral character of $W[f]$, we now write 
\begin{equation*}
\cI_1[f] \defe W[f]. 
\end{equation*}
In this notation, it becomes evident that we have been working with single-variable stochastic calculus. In order to express the resonances in our stochastic objects, it is more natural to work with multi-variable stochastic calculus. For $k\geq 1$, we define the measure $\lambda_k$ on $(\bR_{\geq 0} \times \bZ^3)^k$ by 
\begin{equation*}
\lambda_k \defe \lambda \otimes \hdots \otimes \lambda. 
\end{equation*}
To simplify the notation, we set $\cHk \defe L^2((\bR\times \bZ^3)^k, \dlambda_k)$. For any $f\in \cHk$, the multiple stochastic integral $\cI_k[f]$ can then be constructed as in \cite[Section 1.1.2]{Nualart06}. We only recall the basic ingredients and refer to \cite{Nualart06} for more details. 

We denote by $\mathcal{E}_k$ the set of elementary functions of the form
\begin{equation*}
f(\cs_1,n_1,\hdots,\cs_k,n_k) = \sum_{\substack{l_1,\hdots, l_k \in \\  \{ \pm 1 , \hdots, \pm L \} } } a_{l_1,\hdots,l_k} 1_{A_{l_1}\times \hdots \times A_{l_k}}(\cs_1,n_1,\hdots,\cs_k,n_k). 
\end{equation*}
Here, $\{A_1, A_{-1}, \hdots, A_L, A_{-L}\}$ is a regular system (cf. \cite[Chapter 4]{M14}), i.e., satisfies 
\begin{equation*}
A_{-l} = \big\{ (\cs, -n) \colon (\cs, n) \in A_l \big\} 
\end{equation*}
for all $1\leq l \leq L$ and $A_{l_1} \medcap A_{l_2} = \emptyset$ for all $l_1 \neq l_2$. Furthermore, we impose that $a_{l_1,\hdots, l_k}=0$ vanishes if $l_{k_1} = \pm l_{k_2}$ for some $k_1 \neq k_2$. 
For an elementary function, we define the multiple stochastic integral by 
\begin{equation}
\cI_k[f] \defe  \sum_{\substack{l_1,\hdots, l_k \in \\  \{ \pm 1 , \hdots, \pm L \} } }  a_{l_1,\hdots,l_k} \prod_{j=1}^k W[A_{l_j}]. 
\end{equation}
Furthermore, we define the symmetrization of $f$ by
\begin{equation}
\widetilde{f}(\cs_1,n_1,\hdots,\cs_k,n_k) = \frac{1}{k!} \sum_{\pi \in S_k} f(\cs_{\pi(1)},n_{\pi(1)},\hdots,\cs_{\pi(k)},n_{\pi(k)}).
\end{equation}

\begin{lemma}[Basic properties]\label{tools:lem_basic_integral}
For any $k,l\geq 1$, $f \in \mathcal{E}_k$, and $g\in \mathcal{E}_l$, it holds that:
\begin{enumerate}[(i)]
\item $\cI_k$ is linear.
\item The integral is invariant under symmetrization, i.e.,  $\cI_k[f]= \cI_k[\widetilde{f}]$.
\item \label{tools:item_integral_3} We have the It\^{o}-isometry formula
\begin{equation*}
\bE\big[ \cI_k[f] \cdot \overline{\cI_l[g]}\big] =\delta_{kl} k! \int \widetilde{f}\,  \overline{\widetilde{g}}~\dlambda_k.
\end{equation*}
\item \label{tools:item_integral_4} We have the formula for the expectation
\begin{equation*}
\begin{aligned}
&\bE\big[ \cI_k[f] \cdot \cI_l[g]\big] \\
&=\delta_{kl} k! \sum_{n_1,\hdots,n_k} \int_0^\infty \hdots \int_0^\infty  
\widetilde{f}(\cs_1,n_1,\hdots,\cs_k,n_k) \cdot   \widetilde{g}(\cs_1,-n_1,\hdots,\cs_k,-n_k) 
\Big( \prod_{j=1}^k  \frac{\sigma_{\cs_j}^2(n_j)}{\langle n_j \rangle^2} \Big) \ds_k \hdots \ds_1.
\end{aligned}
\end{equation*}
\end{enumerate}
\end{lemma}

\begin{proof}
Up to minor modifications, the proof can be found in \cite[p.9]{Nualart06} or \cite[Chapter 4]{M14}.
\end{proof}

Using a density argument (see e.g. \cite[p.10]{Nualart06} or \cite[Lemma 4.1]{M14}), we can extend $\cI_k$ from elementary functions to $\cHk$. In particular, for any fixed $m_1,\hdots,m_k\in \bZ^3$, we have that 
\begin{equation*}
\prod_{j=1}^k \delta_{n_j=m_j} \in \cHk
\end{equation*}
and we can write 
\begin{equation}\label{tools:eq_individual_multiple_integral}
\int_{[0,\infty)^k}  \dW{m}{k} \hdots  \dW{m}{1} \defe \cI_k\Big[ \prod_{j=1}^k \delta_{n_j=m_j}  \Big]. 
\end{equation}
We vehemently emphasize that the stochastic integral \eqref{tools:eq_individual_multiple_integral} does not coincide with the product $\prod_{j=1}^{k} W_\infty^{m_j}$. Instead, as will be clear from the product formula (Lemma \ref{tools:lem_product}) below, the stochastic integral \eqref{tools:eq_individual_multiple_integral} only contains the non-resonant portion of this product. \\

If $f=f(n_1,\hdots,n_k)$ does not depend on the stochastic-time variables $\cs_1,\hdots,\cs_k$, the linearity of the multiple stochastic integral $\cI_k$ and \eqref{tools:eq_individual_multiple_integral} naturally imply that 
\begin{equation}
\cI_k[f] = \sum_{n_1,\hdots,n_k\in \bZ^3} f(n_1,\hdots,n_k)\int_{[0,\infty)^k}  \dW{n}{k} \hdots  \dW{n}{1} .
\end{equation}
Using \eqref{tools:item_integral_3} in Lemma \ref{tools:lem_basic_integral}, it follows that 
\begin{align*}
&\bE\Big[ \Big( \int_{[0,\infty)^k}  \dW{n}{k} \hdots  \dW{n}{1} \Big) \cdot \overline{ \Big( \int_{[0,\infty)^k}  \dW{m}{k} \hdots  \dW{m}{1} \Big)} \Big] \\
&= \frac{1}{k!} \Big( \sum_{\pi,\pi^\prime \in S_k} 1\big\{ n_{\pi(j)} = m_{\pi^\prime(j)} \, \text{for all } 1\leq j \leq k\big\} \Big)   \int_{[0,\infty)^k}  \prod_{j=1}^m \frac{\sigma_{\cs_j}^2(n_j)}{\langle n_j \rangle^2} \ds_k \hdots \ds_1 \\
&=  \Big( \sum_{\pi \in S_k} 1\big\{ n_{\pi(j)} = m_{j} \, \text{for all } 1\leq j \leq k\big\} \Big)  \prod_{j=1}^m \langle n_j \rangle^{-2}.
\end{align*}
Up to permutations, the family of multiple stochastic integrals \eqref{tools:eq_individual_multiple_integral} is therefore orthogonal. 
Naturally, a similar formula holds without the complex conjugate. More generally, if $f$ depends on the stochastic time-variables $\cs_1,\hdots,\cs_k$, we have that 
\begin{equation}
\cI_k[f] = \sum_{n_1,\hdots,n_k\in \bZ^3} \int_{[0,\infty)^k} f(\cs_1,n_1,\hdots,\cs_k,n_k) \dW{n}{k} \hdots  \dW{n}{1} .
\end{equation}
Here, the summands on the right-hand side are understood as multiple stochastic integrals with fixed $n_1,\hdots,n_k$ (by inserting an indicator as in \eqref{tools:eq_individual_multiple_integral}). As is shown in the next lemma, this notation is consistent with iterated It\^{o}-integrals. 

\begin{lemma}\label{tools:lem_iterated_integrals}
Let \( k\geq 1 \) and let \( f \in \cHk \) be symmetric. Then, it holds that 
\begin{equation}
\cI_k[f] = k! \sum_{n_1,\hdots,n_k \in \bZ^3} \int_0^\infty \int_0^{\cs_1} \hdots \int_0^{\cs_{k-1}}f(\cs_1,n_1,\hdots,\cs_k,n_k) \dW{n}{k} \hdots \dW{n}{1},
\end{equation}
where the right-hand side is understood as an iterated Itô integral. 
\end{lemma}

This follows from the discussion of \cite[(1.27)]{Nualart06}. As a consequence of this lemma, we could also work with iterated It\^{o}-integrals instead of multiple stochastic integrals. While the iterated It\^{o}-integrals are more natural whenever martingale properties are utilized, the multiple stochastic integrals have a much simpler product formula, which simplifies many of our computations.

Before we can state the product formula, we need to define the contraction.

\begin{definition}[Contraction]
Let $k,l\geq 1$, let $f\in \cHk$, and let $g\in \cHl$. For any $0\leq r \leq \min(k,l)$, we define the contraction of $r$ indices by
\begin{align*}
&(f \otimes_r g )(\cs_1,n_1,\hdots, \cs_{k+l-2r}, n_{k+l-2r}) \\
\defe&\sum_{m_1,\hdots,m_r \in \bZ^3} \int_0^\infty \hdots \int_0^\infty \bigg[ f(\cs_1,n_1,\hdots,\cs_{k-r},n_{k-r},\mathscr{r}_1,m_1,\hdots,\mathscr{r}_r,m_r)  \\
&\times g(\cs_{k+1-r},n_{k+1-r},\hdots,\cs_{k+l-2r},n_{k+l-2r},\mathscr{r}_1,-m_1,\hdots,\mathscr{r}_r,-m_r) \prod_{j=1}^k \frac{\sigma_{\mathscr{r}_j}^2(m_j)}{\langle m_j \rangle^2} \bigg]\mathrm{d} \mathscr{r}_r \hdots \mathrm{d} \mathscr{r}_1. 
\end{align*}
\end{definition}
We note that even if $f\in \cHk$ and let $g\in \cHl$  are both symmetric, the contraction $f \otimes_r g$ may not be symmetric. 
The reader should note the similarity of the contraction with the formula for the expectation in \eqref{tools:item_integral_4} of Lemma \ref{tools:lem_basic_integral}, which is no coincidence. If $f,g \in \cH_1$, then 
\begin{equation}
\bE\Big[ \cI_1[f] \cdot \cI_1[g] \Big] = f \otimes_1 g. 
\end{equation}
Thus, $f \otimes_1 g$ describes the (full) resonance portion of the product $ \cI_1[f] \cdot \cI_1[g]$. The product formula is a (major) generalization of this simple fact.

\begin{lemma}[{Product formula for multiple stochastic integrals (cf. \cite[Prop 1.1.3]{Nualart06})}]\label{tools:lem_product}
Let $k,l\geq 1$ and let $f\in \cHk$ and $g\in \cHl$ be symmetric. Then, it holds that 
\begin{equation}
\cI_k[f] \cdot \cI_l[g] = \sum_{r=0}^{\min(k,l)} r! {k \choose r} {l \choose r} \cI_{k+l-2r}[f \otimes_r g]. 
\end{equation}
\end{lemma}

Using the product formula (Lemma \ref{tools:lem_product}), we can compute the non-resonant, partially resonant, and fully resonant portions of products such as 
\begin{equation*}
(P_{\leq N}\<1b>)(t,x) \cdot (P_{\leq N}\<1b>)(t,x) \quad \text{and} \quad \<2N>(t,x) \cdot \<3N>(t,x). 
\end{equation*}
Once the Duhamel operator occurs in the expression, however, we also need to consider two different physical times $t$ and $t^\prime$. For instance, in our estimate of the quintic stochastic object 
\begin{equation*}
\<113N>,
\end{equation*} 
we need to control
\begin{equation*}
 \Big( V\ast \<2N>(t,x) \Big)  \cdot \Big( P_{\leq N} \sin((t-t^\prime)\langle \nabla \rangle) \langle \nabla \rangle^{-1} \Big)\Big( \<3N>(t^\prime,x) \Big)
\end{equation*}
In order to consider two different physical times $t$ and $t^\prime$, we need to consider multiple stochastic integrals with respect to two different (correlated) Gaussian processes, which we abstractly denote by $W^{a}$ and $W^{b}$. We will assume that $\Law_\bP(W^{a})=\Law_\bP(W^{b})=\Law_\bP(W)$. Regarding the relationship between the different Gaussian processes $W^{a}$ and $W^{b}$, we assume that $W^{a,n}$ and $W^{b,m}$ are independent for $m \neq \pm n$. Furthermore, let $\mathfrak{C}\colon \bZ^3 \rightarrow [-1,1]$ be an even function. We assume that 

\begin{equation}
\bE\Big[ \Was{n}{1} \Wbs{m}{2}\Big] = \delta_{n=-m} \, \mathfrak{C}(n) \int_0^{\cs_1\wedge \cs_2} \frac{\sigma_\cs^2(n)}{\langle n \rangle^2} \ds  
\end{equation}
and 
\begin{equation}
\bE\Big[ \Was{n}{1} \overline{\Wbs{m}{2}} \Big] = \delta_{n=m} \, \mathfrak{C}(n) \int_0^{\cs_1\wedge \cs_2} \frac{\sigma_\cs^2(n)}{\langle n \rangle^2} \ds  
\end{equation}
Thus, $\mathfrak{C}$ is the (appropriately normalized) correlation of  $W^{a}$ and $W^{b}$. We can then set up the theory of multiple stochastic integrals with respect to a mixture of  $W^{a}$ and $W^{b}$ as before. In order to fit this theory into the same framework as in \cite{Nualart06}, one only has to replace $\bR\times \bZ^3$ by $\bR\times \bZ^3\times \{a,b\}$. A short calculation shows for any bounded and compactly supported $f,g\colon \bR\times \bZ^3\times \{a,b\} \rightarrow \bC$ that 
\begin{equation}
\begin{aligned}
&\bE \Big[ \Big( \sum_{\iota=a,b} \sum_{n\in \bZ^3} \int_0^\infty f(\cs,n,\iota) \mathrm{d}W^{(\iota),n}_{\cs} \Big) \Big( \sum_{\iota=a,b} \sum_{n\in \bZ^3} \int_0^\infty g(\cs,n,\iota) \mathrm{d}W^{(\iota),n}_{\cs} \Big) \Big]\\
&=\sum_{\iota,\iota^\prime=a,b} \sum_{n \in \bZ^3} \Big( 1\big\{ \iota=\iota^\prime\big\} +  \mathfrak{C}(n) 1\big\{ \iota \neq \iota^\prime\big\} \Big)\int_0^\infty f(\cs,n,\iota)\cdot g(\cs,-n,\iota^\prime) \frac{\sigma_\cs^2(n)}{\langle n \rangle^2} \mathrm{d}\cs.
\end{aligned}
\end{equation}
and 
\begin{equation}\label{tools:eq_expectation_mixed_processes}
\begin{aligned}
&\bE \Big[ \Big( \sum_{\iota=a,b} \sum_{n\in \bZ^3} \int_0^\infty f(\cs,n,\iota) \mathrm{d}W^{(\iota),n}_{\cs} \Big) 
\Big( \overline{\sum_{\iota=a,b} \sum_{n\in \bZ^3} \int_0^\infty g(\cs,n,\iota) \mathrm{d}W^{(\iota),n}_{\cs}} \Big) \Big]\\
&=\sum_{\iota,\iota^\prime=a,b} \sum_{n \in \bZ^3} \Big( 1\big\{ \iota=\iota^\prime\big\} +  \mathfrak{C}(n) 1\big\{ \iota \neq \iota^\prime\big\} \Big)
\int_0^\infty f(\cs,n,\iota)\cdot \overline{g(\cs,n,\iota^\prime)} \frac{\sigma_\cs^2(n)}{\langle n \rangle^2} \mathrm{d}\cs.
\end{aligned}
\end{equation}
The sesquilinear form in \eqref{tools:eq_expectation_mixed_processes}, viewed as a function in $f$ and $g$, is no longer positive definite. For instance, if $W^{(a)}=-W^{(b)}$, and hence $\mathfrak{C}=-1$, $f=g$, and $f(\cs,n,a)=f(\cs,n,b)$ for all $\cs \in \bR_{\geq 0}$ and $n \in \bZ^3$, it vanishes identically. Nevertheless, due to the condition $|\mathfrak{C}|\leq 1$ imposed on the correlation function $\mathfrak{C}$, it is bounded by (a scalar multiple of) the inner product 
\begin{equation*}
\sum_{\iota=a,b} \sum_{n \in \bZ^3} 
\int_0^\infty f(\cs,n,\iota)\cdot \overline{g(\cs,n,\iota)} \frac{\sigma_\cs^2(n)}{\langle n \rangle^2} \mathrm{d}\cs.
\end{equation*}
After defining a measure $\widetilde{\lambda}$ on $\bR\times \bZ^3\times \{a,b\}$ by $\mathrm{d}\widetilde{\lambda}= \mathrm{d} \lambda  \mathrm{d}\iota $, where $\mathrm{d}\iota$ is the integration with respect to the counting measure on $\{a,b\}$, this allows us to construct multiple stochastic integrals for functions in
\begin{equation*}
L^2((\bR\times \bZ^3\times \{a,b\})^k, \widetilde{\lambda}_k). 
\end{equation*}
Similar as in \eqref{tools:eq_individual_multiple_integral}, this allows us to define mixed multiple stochastic integrals such as 
\begin{equation}
\int_{[0,\infty)^3} \mathrm{d}\Was{n_3}{3} \mathrm{d}\Was{n_2}{2} \mathrm{d}\Wbs{n_1}{1}. 
\end{equation}
Unfortunately, the general theory now becomes notationally cumbersome. We therefore decided to only state the much simpler special case of the product formula needed in this paper. 
\begin{lemma}[Quadratic-Cubic product formula]\label{tools:lemma_23product}
Let $f\colon (\bZ^3)^2 \rightarrow \bC$ and let $g\colon (\bZ^3)^3 \rightarrow \bC$. We assume that $g$ is symmetric but do not require any symmetry of $f$. Then, it holds that 
\begin{align*}
&\Big( \sum_{n_1,n_2\in \bZ^3}  \hspace{-2ex}f(n_1,n_2) \int_{[0,\infty)^2} \hspace{-2ex} \mathrm{d}\Was{n_2}{2} \mathrm{d}\Was{n_1}{1} \Big)
\times \Big( \sum_{n_3,n_4,n_5\in \bZ^3} \hspace{-2ex} g(n_3,n_4,n_5) \int_{[0,\infty)^3}  \hspace{-2ex} \mathrm{d}\Wbs{n_5}{5} \mathrm{d}\Wbs{n_4}{4}\mathrm{d}\Wbs{n_3}{3} \Big) \\
&=\sum_{n_1,n_2,n_3,n_4,n_5\in \bZ^3} f(n_1,n_2) g(n_3,n_4,n_5)
 \int_{[0,\infty)^5} \mathrm{d}\Wbs{n_5}{5} \mathrm{d}\Wbs{n_4}{4}\mathrm{d}\Wbs{n_3}{3}  \mathrm{d}\Was{n_2}{2} \mathrm{d}\Was{n_1}{1}\\
&+ 3 \sum_{n_2,n_4,n_5\in \bZ^3} \Big( \sum_{n_1\in \bZ^3} f(n_1,n_2) g(-n_1,n_4,n_5) \frac{\mathfrak{C}(n_1)}{\langle n_1 \rangle^2} \Big) 
 \int_{[0,\infty)^3} \mathrm{d}\Wbs{n_5}{5} \mathrm{d}\Wbs{n_4}{4} \mathrm{d}\Was{n_2}{2} \\
&+ 3 \sum_{n_1,n_4,n_5\in \bZ^3} \Big( \sum_{n_2\in \bZ^3} f(n_1,n_2) g(-n_2,n_4,n_5) \frac{\mathfrak{C}(n_2)}{\langle n_2 \rangle^2} \Big) 
 \int_{[0,\infty)^3} \mathrm{d}\Wbs{n_5}{5} \mathrm{d}\Wbs{n_4}{4} \mathrm{d}\Was{n_1}{1} \\
&+6 \sum_{n_5 \in \bZ^3} \Big( \sum_{n_1,n_2\in \bZ^3} f(n_1,n_2) g(-n_1,-n_2,n_5)  \frac{\mathfrak{C}(n_1) \mathfrak{C}(n_2)}{\langle n_1 \rangle^2\langle n_2 \rangle^2} \Big)
\int_0^\infty  \mathrm{d}\Wbs{n_5}{5}. 
\end{align*}
\end{lemma}

\begin{remark}
Instead of working with the product $ f(n_1,n_2) g(n_3,n_4,n_5)$, the formula has a natural extension to functions $h(n_1,\hdots,n_5)$ which are symmetric in $n_3,n_4$, and $n_5$. To this end, one only has to decompose
\begin{equation*}
h(n_1,n_2,n_3,n_4,n_5) = \sum_{m_1,m_2 \in \bZ^3} 1\big\{ (n_1,n_2)=(m_1,m_2)\big\} \cdot h(m_1,m_2,n_3,n_4,n_5). 
\end{equation*}
We can then apply Lemma \ref{tools:lemma_23product} to the individual summands. 
\end{remark}

\begin{remark}\label{tools:rem_why_multiple}
While the formula in Lemma \ref{tools:lemma_23product} is complicated, it is still an order of magnitude easier than working with products of Gaussians directly. If the reader is not convinced, we encourage him to work out (by hand) the corresponding resonant/non-resonant decomposition of 
\begin{equation*}
\begin{aligned}
&\Big( \sum_{n_1,n_2\in \bZ^3}  \hspace{-2ex}f(n_1,n_2)   \Big(G_{n_1}^{(a)}\cdot G_{n_2}^{(a)} - \frac{\delta_{n_{12}=0}}{\langle n_1\rangle^2}\Big)) \Big) \\
&\times \Big( \sum_{n_3,n_4,n_5\in \bZ^3} \hspace{-2ex} g(n_3,n_4,n_5)
\Big(G_{n_3}^{(b)}\cdot G_{n_4}^{(b)}\cdot G_{n_5}^{(b)} 
- \frac{\delta_{n_{34}=0}}{\langle n_3\rangle^2} G_{n_5}^{(b)}
- \frac{\delta_{n_{35}=0}}{\langle n_3\rangle^2} G_{n_4}^{(b)}
- \frac{\delta_{n_{45}=0}}{\langle n_4\rangle^2} G_{n_3}^{(b)}
\Big)  \Big),
\end{aligned}
\end{equation*}
where $G^{(\iota)}=W^{(\iota)}_{\infty}$ for $\iota=a,b$ are (correlated) families of Gaussian random variables. 
\end{remark}

After establishing the important definitions and properties of multiple stochastic integrals, it only remains to connect them with our stochastic objects. Let $\Wcos{}{\cs}$ and $\Wsin{}{\cs}$ be the Gaussian processes defined in Section \ref{section:gaussian_processes}. We recall that the linear evolution of the random initial data $\bluedot$ is given by
\begin{equation}
\begin{aligned}
\<1b>(t) &= \sum_{n\in \bZ^3} \big( \cos(t \langle  n\rangle) \Wcos{}{\infty} + \sin(t \langle n \rangle) \Wsin{}{\infty} \big) \exp(i \langle n , x \rangle) \\ 
	&= \sum_{n\in \bZ^3} 
	\Big( \int_0^\infty \hspace{-1ex} \mathrm{d}\big(  \cos(t \langle  n\rangle) \Wcos{}{\cs} + \sin(t \langle n \rangle) \Wsin{}{\cs} \big) \Big)
	\exp(i \langle n , x \rangle).
	\end{aligned}
\end{equation}

In order to obtain a similar expression for the stochastic objects $\<2N>$ and $\<3N>$, we define for any $k\geq 1$ and $n_1,\hdots,n_k \in \bZ^3$ the multiple stochastic integral
\begin{equation}\label{tools:eq_Ik}
\begin{aligned}
\cI_k[t,n_1,\hdots,n_k] 
\defe \int_{[0,\infty)^k}  &\mathrm{d}\big(  \cos(t \langle  n_k\rangle) \Wcos{k}{\cs_k} + \sin(t \langle n \rangle) \Wsin{k}{\cs_k} \big) \hdots  \\
&\mathrm{d}\big(  \cos(t \langle  n_1\rangle) \Wcos{1}{\cs_1} + \sin(t \langle n_1 \rangle) \Wsin{1}{\cs_1} \big). 
\end{aligned}
\end{equation}

In the proof of multi-linear dispersive estimates, it is essential to separate the time-variable $t$ from the randomness. To this end, we define the Gaussian processes
\begin{equation}
\Wpm{}{\cs} \defe \Wcos{}{\cs} \pm \Wsin{}{\cs}. 
\end{equation}
Similar as in \eqref{tools:eq_Ik}, we define for any $k\geq 1$, any $\pm_1,\hdots,\pm_k \in \{ +,- \}$, and any $n_1,\hdots,n_k \in \bZ^3$ the multiple stochastic integral
\begin{equation}
\cI_k[n_j;\pm_j \colon 1 \leq j \leq k] \defe \int_{[0,\infty)^k} \mathrm{d} \Wpm{k}{\cs_k}\hdots\mathrm{d} \Wpm{1}{\cs_1}. 
\end{equation}
It then follows that there exists coefficients $c \colon \{ +,-\}^k \rightarrow \bC$  depending only on the signs such that 
\begin{equation}
\cI_k[t,n_1,\hdots,n_k] = \sum_{\pm_1,\hdots,\pm_k} c(\pm_1,\hdots,\pm_k) \, \Big( \prod_{j=1}^k \exp( \pm_j i t \langle n_j \rangle) \Big) 
\cI_k[n_j;\pm_j \colon 1 \leq j \leq k]. 
\end{equation} 
For convenience, we also define the normalized multiple stochastic integrals by 
\begin{equation}\label{tools:eq_normalized_integrals}
\widetilde{\cI}_k[n_j;\pm_j \colon 1 \leq j \leq k] = \Big( \prod_{j=1}^k \langle n_j \rangle \Big) \cdot\cI_k[n_j;\pm_j \colon 1 \leq j \leq k] 
\end{equation}
We close this subsection with the following stochastic representation, which expresses the quadratic and cubic stochastic objects through multiple stochastic integrals.

\begin{proposition}\label{tools:prop_stochastic_representation}
Let $t\in \bR$ and $N\geq 1$. Then, we have for all $n_1,n_2 \in \bZ^3$ that
\begin{equation}\label{tools:eq_stochastic_representation_1}
\widehat{ \<1b>}\, (t,n_1) \cdot \widehat{ \<1b>}\, (t,n_2)  - \frac{1}{\langle n_{12} \rangle^2} \delta_{n_{12}=0} = \cI_2[t,n_1,n_2]. 
\end{equation}
Furthermore, it holds that 
\begin{align}
\<2N>(t,x) &= \sum_{n_1,n_2\in \bZ^3} \Big( \prod_{j=1}^2 \rho_N(n_j) \Big) \, \cI_2[t,n_1,n_2], \\
\<3N>(t,x) &= \sum_{n_1,n_2,n_3\in \bZ^3} \Big( \prod_{j=1}^3 \rho_N(n_j) \Big) \, \widehat{V}(n_{12}) \,  \cI_3[t,n_1,n_2,n_3]. 
\end{align}
\end{proposition}

\begin{proof}
This follows from \cite[Lemma 2.5 and Proposition 2.9]{BB20a}, Lemma \ref{tools:lem_iterated_integrals}, and that the distribution of 
\begin{equation*}
(\cs,n) \mapsto \cos(t\langle n \rangle) \Wcos{}{\cs} + \sin(t \langle n \rangle) \Wsin{}{\cs}
\end{equation*}
is the same for all $t\in \bR$. 
\end{proof}

\subsection{Gaussian hypercontractivity and the moment method}\label{section:hypercontractivity}

In this section, we first review Gaussian hypercontractivity and its consequences. To help the reader with a primary background in dispersive equations, let us first illustrate this phenomenon through a basic example. Let $Z_\sigma$ be a Gaussian random variable with mean zero and variance $\sigma^2$. Using the exact formula for the moments of a Gaussian, we have for all $m\geq 1$ that 
\begin{equation*}
\bE\big[ Z_\sigma^2\big] = \sigma^2 \qquad \text{and} \qquad \bE \big[ Z_\sigma^{2m} \big] = \frac{(2m)!}{2^m m!} \cdot \sigma^{2m}. 
\end{equation*}
A simple estimate now yields that 
\begin{equation*}
\Big( \bE \big[ Z_\sigma^{2m} \big] \Big)^{\frac{1}{2m}} \leq \Big( \frac{(2m)^{2m}}{2^m (m/e)^{m}}\Big)^{\frac{1}{2m}}  \cdot \sigma
= \sqrt{2em} \Big( \bE \big[ Z_\sigma^{2} \big] \Big)^{\frac{1}{2}}.
\end{equation*}
Using H\"{o}lder's inequality, we obtain for all $p \geq 2$ that 
\begin{equation}\label{tools:eq_Gaussian_moments}
\| Z_\sigma \|_{L^p_\omega} \lesssim \sqrt{p} \| Z_\sigma \|_{L^2_\omega}. 
\end{equation}
Thus, higher $L^p_\omega$-norms of Gaussians can be controlled through the lower $L^2_\omega$-norm. The \emph{``hyper''} in Gaussian hypercontractivity refers exactly to this gain of integrability. While \eqref{tools:eq_Gaussian_moments} is not too interesting by itself, its significance lies in its generalizations to polynomials in infinitely many Gaussians! Furthermore, Gaussian hypercontractivity has connections to many different inequalities in analysis and probability theory, such as logarithmic Sobolev inequalities. \\

Our first proposition is also known as a Wiener chaos estimate. A version of this proposition can be found in \cite[Theorem I.22]{Simon74} or \cite[Theorem 1.4.1]{Nualart06}. 

\begin{proposition}[Gaussian hypercontractivity] \label{tools:prop_Gaussian_hypercontractivity}
Let $k\geq 1$, let $\pm_1,\hdots,\pm_k\in \{ +,-\}$, and let $a\colon (\bZ^3)^k \rightarrow \bC$ be a discrete function with finite support. Define the $k$-th order Gaussian chaos $\cG_k$ by 
\begin{equation}
\cG_k \defe  \sum_{n_1,\hdots,n_k \in \bZ^3} a(n_1,\hdots,n_k) \cI_k[\pm_j,n_j \colon 1 \leq j \leq k ]. 
\end{equation}
Then, it holds for all $p\geq 2$ that 
\begin{equation}
\| \cG_k \|_{L^p_\omega(\bP)} \lesssim p^{\frac{k}{2}} \| \cG_k \|_{L^2_\omega(\bP)}. 
\end{equation}
\end{proposition}

Proposition \ref{tools:prop_Gaussian_hypercontractivity} will play an important role in the estimates of stochastic objects such as $\<3N>$. While Proposition \ref{tools:prop_Gaussian_hypercontractivity} bounds the moments of the Gaussian chaos, the reader may prefer or be more familiar with a bound on probabilistic tails. As the next lemma shows, the two viewpoints are equivalent.

\begin{lemma}[Moments and tails]\label{tools:lem_moments_and_tails}
Let $Z$ be a random variable and let $\gamma>0$. Then, the following properties are equivalent, where the parameter $K_1,K_2>0$ appearing below differ from each other by at most a constant factor depending only on $\gamma$. 
\begin{enumerate}
\item The tails of $Z$ satisfy for all $\lambda\geq 0$ the inequality
\begin{equation*}
\bP(|Z|\geq \lambda) \leq 2 \exp\big( - (\lambda/K_1)^{\gamma}\big).
\end{equation*}
\item The moments of $Z$ satisfy for all $p\geq 2$ the inequality
\begin{equation*}
\| Z \|_{L^p}\leq K_2 p^{\frac{1}{\gamma}}. 
\end{equation*}
\end{enumerate}
\end{lemma}
The lemma is an easy generalization of \cite[Proposition 2.5.2 or Proposition 2.7.1]{Vershynin18}. As we have seen above, a Gaussian random variable corresponds to $\gamma=2$. 
It is convenient to capture the size of $K_2$ in Lemma \ref{tools:lem_moments_and_tails} (and hence $K_1$) through a norm.

\begin{definition}
Let $\gamma>0$ and let $Z$ be a random variable. We define
\begin{equation*}
\| Z \|_{\Psi_\gamma} = \sup_{p\geq 2} p^{-\frac{1}{\gamma}} \| Z \|_{L^p_\omega}. 
\end{equation*}
\end{definition}
For more information regarding the $\Psi_\gamma$-norms, we refer the reader to the excellent textbook \cite{Vershynin18}. The next lemma shows that the $\Psi_\gamma$-norm is well-behaved under taking maxima of several random variables.  
\begin{lemma}[Maxima and the $\Psi_\gamma$-norm]\label{tools:eq_moments_sups}
Let $\gamma>0$, let $J \in \mathbb{N}$, and let $Z_1,\hdots,Z_J$ be random variables on the same probability space. Then, it holds that 
\begin{equation*}
\| \max(Z_1,\hdots,Z_J) \|_{\Psi_\gamma} \leq e \log(2+J)^{\frac{1}{\gamma}} \max_{j=1,\hdots,J} \| Z_j \|_{\Psi_\gamma}. 
\end{equation*}
\end{lemma}

While this is only a minor generalization of \cite[Exercise 2.5.10]{Vershynin18}, we include the short proof.

\begin{proof}
Let $p\geq 2$. For any $r\geq p$, it follows from the embedding $\ell^r_j \hookrightarrow \ell^\infty_j$ and H\"{o}lder's inequality that 
\begin{gather*}
\| \max(Z_1,\hdots,Z_J) \|_{L^p_\omega} \leq \| Z_j \|_{L^p_\omega \ell_j^\infty} \leq \| Z_j \|_{L^p_\omega \ell_j^r}  \leq \| Z_j \|_{L^r_\omega \ell_j^r} \leq J^{\frac{1}{r}} r^{\frac{1}{\gamma}} \max_{j=1,\hdots,J} \| Z_j \|_{\Psi_\gamma}. 
\end{gather*}
Then, we choose $r=\log(2+J) p$, which yields the desired estimate. 
\end{proof}

We now turn to a combination of Gaussian hypercontractivity and the moment method, which will be essential to our treatment of the random matrix terms $\RMT$. The following proposition, which is easy-to-use, general, and essentially sharp, was recently obtained by Deng, Nahmod, and Yue in \cite[Proposition 2.8]{DNY20}. Before we state the estimate, we need the following definition, which relies on the tensor notation from Definition \ref{tool:def_tensor}. 

\begin{definition}[Contracted random tensor]\label{tools:def_random_tensor}
Let $\cJ\subseteq \mathbb{N}_0$, let $(\pm_j)_{j\in \cJ}$ be given, and let $N_{\max}\geq 1$. Let $h=h_{n_\cJ}$ be a tensor and assume that all vectors in the support of $h$ satisfy $\|n_\cJ\|\leq N_{\max}$. Let $\cS \subseteq \cJ$ and define $k\defe \# \cS$. We then define the contracted random tensor $h_c = (h_c)_{n_{\cJ\backslash \cS}}$ by
\begin{equation}
h_c (n_i \colon i  \not \in \cS) \defe \sum_{\substack{ (n_j)_{j \in \cS}}} h(n_{\cJ}) \cdot
\widetilde{\cI}_k[\pm_j,n_j\colon j   \in \cS] ,
\end{equation}
where the normalized multiple stochastic integrals are as in \eqref{tools:eq_normalized_integrals}.  
\end{definition}

In the next proposition, we use the tensor norms from Definition \ref{tool:def_tensor}. 

\begin{proposition}[{\cite[Proposition 2.8, Proposition 4.14]{DNY20}}]\label{tools:prop_moment_method}
Let $\cJ,\cS,N_{\max},h,h_c$, and $k$ be as in Definition \ref{tools:def_random_tensor}. Let $\mathcal{A},\mathcal{B}$ be a partition of $\{1,\hdots,J\} \backslash \cS$. Then, we have for all $p \geq 2$ and $\theta>0$ that 
\begin{equation}
\| \| h_c \|_{n_\mathcal{A} \rightarrow n_\mathcal{B}} \|_{L^p_\omega(\bP)} \lesssim_\theta N_{\max}^\theta   \Big( \max_{\mathcal{X},\mathcal{Y}} \| h \|_{n_\mathcal{X} \rightarrow n_\mathcal{Y}} \Big) p^{\frac{k}{2}},
\end{equation}
where the maximum is taken over all sets $\mathcal{X},\mathcal{Y}$ which satisfy $\mathcal{A}\subseteq \mathcal{X}$, $\mathcal{B}\subseteq \mathcal{Y}$, and form a partition of $\cJ$. 
\end{proposition}

In \cite{DNY20}, the proposition is stated in terms of non-resonant products of Gaussians instead of multiple stochastic integrals. Furthermore, the probabilistic estimate is stated in terms of the tail-behavior instead of the moment growth. Both of these modifications can be obtained easily by replacing the large deviation estimate \cite[Lemma 4.4]{DNY20} in the proof by Proposition \ref{tools:prop_Gaussian_hypercontractivity}. \\
We often simply refer to Proposition \ref{tools:prop_moment_method} as the moment method, since it is the main ingredient of the proof (cf. \cite{DNY20}).  While the full generality of Proposition \ref{tools:prop_moment_method} is needed in \cite{DNY20}, we will only rely on the following special case. 

\begin{example}
Let $\pm_1,\pm_2 \in \{ +,-\}$, let $h=h(n,n_1,n_2,n_3)$ be a tensor and assume that $\| (n,n_1,n_2,n_3)\| \lesssim N_{\max}$ on the support of $h$. Define the contracted random tensor $h_c$ by
\begin{equation}
h_c(n,n_3) \defe \sum_{n_1,n_2 \in \bZ^3} h(n,n_1,n_2,n_3) \cdot \cI_2[\pm_j,n_j \colon j=1,2]. 
\end{equation}
Then, we have for all $p \geq 2$ and $\theta>0$ that 
\begin{equation*}
\Big\| \| h_c \|_{n_3 \rightarrow n} \Big\|_{L^p_\omega} \lesssim_\theta N_{\max}^{\theta} \max\big( \| h \|_{n_1n_2n_3\rightarrow n},\| h \|_{n_3 \rightarrow nn_1 n_2},\| h \|_{n_1 n_3\rightarrow nn_2} , \| h \|_{n_2n_3\rightarrow nn_1} \big) \cdot p. 
\end{equation*}
\end{example}

\section{Explicit stochastic objects}\label{section:stochastic_object}

In this section, we estimate the stochastic objects appearing in the expansion of $u_N$ and in the evolution equations for $X_N$ and $Y_N$. The analysis of explicit stochastic objects is necessary for both dispersive and parabolic equations. We refer the interested reader to the treatment of the cubic stochastic heat equation in \cite{CC18,Hairer16} and the quadratic stochastic wave equation in \cite{GKO18a} for illustrative examples. While the algebraic aspects are similar in dispersive and parabolic settings, the analytic aspects are quite different. In the parabolic setting, the regularity of stochastic objects can be determined through simple  ``power-counting''. In contrast, the optimal estimates in the dispersive setting require more complicated multi-linear dispersive estimates. We remind the reader that, as explained in Remark \ref{intro:rem_restriction_beta}, we restrict ourselves to $0<\beta<1/2$.

\subsection{Cubic stochastic objects}

In this subsection, we analyze the cubic stochastic object $\<3N>$ and the corresponding solution to the forced wave equation $\<3DN>$. Ignoring the smoother component $\reddotM$ of the initial data, they correspond to the first Picard iterate of \eqref{eq:nlw_N}. 

\begin{proposition}[Cubic stochastic objects]\label{so3:prop}
Let $T\geq 1$ and let $s<\beta-\eta$. Then, it holds that 
\begin{equation}\label{so3:eq_estimate_1}
\Big\| \sup_{N \geq 1} \| \<3N> \|_{\X{s-1}{b_+-1}([0,T])} \Big\|_{L^p_\omega(\bP)} \lesssim T^2  p^{\frac{3}{2}}. 
\end{equation}
Furthermore, we have that 
\begin{equation}\label{so3:eq_estimate_2}
\Big\| \sup_{N \geq 1} \| \<3DN> \|_{C_t^0 \cC_x^s([0,T]\times \bT^3)} \Big\|_{L^p_\omega(\bP)} \lesssim T^2 p^{\frac{3}{2}}. 
\end{equation}
In the frequency-localized version of \eqref{so3:eq_estimate_1} and \eqref{so3:eq_estimate_2}, which is detailed in the proof, we gain an $\eta^\prime$-power of the maximal frequency-scale. Furthermore, we can replace $\<3DN>$ by $\<3DNtau>= \Duh \big[ 1_{[0,\tau]} \<3N>\big]$. 
\end{proposition}

\begin{remark}
We recall that the parameter $T$ is important for the globalization argument, but does not enter into the local well-posedness theory. In order to achieve smallness on a short interval, we will instead use the time-localization lemma (Lemma \ref{tools:lem_localization}) and $b_+> b$. 
\end{remark}

\begin{proof}
We first prove \eqref{so3:eq_estimate_1}, which forms the main part of the argument. In the end, we follow a standard and short argument to show that \eqref{so3:eq_estimate_1}, Gaussian hypercontractivity, and translation invariance imply \eqref{so3:eq_estimate_2}. 
To simplify the notation, we set $N_{\max}=\max(N_1,N_2,N_3)$. In this argument, we rely on multiple stochastic integrals. Recalling the multiple stochastic integrals from \eqref{tools:eq_Ik} and the stochastic representation formula (Proposition \ref{tools:prop_stochastic_representation}), we have that 
\begin{equation*}
\begin{aligned}
\<3N> (t,x)
&= \sum_{n_1,n_2,n_3 \in \bZ^3} \rho_N(n_{123}) \Big( \prod_{j=1}^{3} \rho_N(n_j) \Big) \widehat{V}(n_{12}) \exp\big( i \langle n_{123}, x \rangle\big) \cI_3[t,n_1,n_2,n_3] \\
&=\sum_{\pm_1,\pm_2,\pm_3} \sum_{n_1,n_2,n_3 \in \bZ^3}  \bigg[ c(\pm_j\colon 1\leq j \leq 3) \Big( \prod_{j=1}^{3} \rho_N(n_j) \Big) \widehat{V}(n_{12}) \exp\big( i \langle n_{123}, x \rangle\big) \\
&\quad \times\Big( \prod_{j=1}^3 \exp(\pm_j i t \langle n_j \rangle) \Big)  \cI_3[\pm_j,n_j\colon 1 \leq j \leq 3] \bigg],
\end{aligned}
\end{equation*}
where $c(\pm_j\colon 1\leq j \leq 3)$ are deterministic coefficients. Using a Littlewood-Paley decomposition, we obtain that 
\begin{align*}
\<3N> &= \sum_{\pm_1,\pm_2,\pm_3} \sum_{N_1,N_2,N_3 \geq 1} c(\pm_j\colon 1\leq j \leq 3) \<3N> [\pm_j,N_j \colon 1\leq j \leq 3],
\end{align*}
where
\begin{align*}
\<3N> [\pm_j,N_j \colon 1\leq j \leq 3](t,x) & \defe  \sum_{n_1,n_2,n_3 \in \bZ^3}  \bigg[  \rho_N(n_{123}) \Big( \prod_{j=1}^{3} \rho_N(n_j) \chi_{N_j}(n_j) \Big) \widehat{V}(n_{12})  \\
&\quad \times \exp\big( i \langle n_{123}, x \rangle\big) \Big( \prod_{j=1}^3 \exp(\pm_j i t \langle n_j \rangle) \Big)  \cI_3[\pm_j,n_j\colon 1 \leq j \leq 3] \bigg].
\end{align*}
We estimate each dyadic block separately. We first prove the desired estimate for $b_-$ instead of $b_+$ and then later upgrade the estimate.
Using Minkowski's integral inequality and Gaussian hypercontractivity (Proposition \ref{tools:prop_Gaussian_hypercontractivity}), we obtain that 
\begin{align}
& \Big\|    \Big\|  \<3N> [\pm_j,N_j \colon 1\leq j \leq 3]  \Big\|_{\X{s-1}{b_--1}([0,T])}  \Big\|_{L^p_\omega}  \notag\\
&\lesssim \max_{\pm_{123}}
 \Big\| \mathcal{F}_{t,x}\Big( \chi(t/T) \<3N> [\pm_j,N_j \colon 1\leq j \leq 3](t,x) \Big)(\lambda \mp_{123} \langle n \rangle, n) \|_{L^p_\omega L_\lambda^2 \ell_n^2 ( \Omega \times \bR \times \bT^3)} \notag \\
 &\lesssim p^{\frac{3}{2}} \max_{\pm_{123}}
 \Big\| \mathcal{F}_{t,x}\Big( \chi(t/T) \<3N> [\pm_j,N_j \colon 1\leq j \leq 3](t,x) \Big)(\lambda \mp_{123} \langle n \rangle, n) \|_{L^2_\omega L_\lambda^2 \ell_n^2 ( \Omega \times \bR \times \bT^3)} . \label{so3:eq_p6}
\end{align}

 For a fixed sign $\pm_{123}$, we define the phase $\varphi$ by 
\begin{equation*}
\varphi(n_1,n_2,n_3)\defe \pm_{123} \langle n_{123} \rangle \pm_1 \langle n_1 \rangle  \pm_2 \langle n_2 \rangle  \pm_3 \langle n_3 \rangle.
\end{equation*}
Using the definition of $\varphi$, we can write the space-time Fourier transform of a dyadic piece in the cubic stochastic object $\chi(t/T) \<3N>$ as 
\begin{equation}\label{so3:eq_p5}
\begin{aligned}
&\mathcal{F}_{t,x}\Big( \chi(t/T) \<3N> [\pm_j,N_j \colon 1\leq j \leq 3](t,x) \Big)(\lambda \mp_{123} \langle n \rangle, n) \\
&= T \sum_{n_1,n_2,n_3 \in \bZ^3} \bigg[ 1\big\{ n= n_{123} \big\} \rho_N(n_{123}) \Big( \prod_{j=1}^{3} \rho_N(n_j) \chi_{N_j}(n_j) \Big) \widehat{V}(n_{12})      \\ 
& \times \widehat{\chi}\big( T (\lambda-\varphi(n_1,n_2,n_3))\big)\cI_3[\pm_j,n_j\colon 1 \leq j \leq 3] \bigg] . 
\end{aligned}
\end{equation}
Using the orthogonality of the multiple stochastic integrals and the decay of $\widehat{\chi}$, we obtain that 
\begin{align*}
 & \Big\| \mathcal{F}_{t,x}\Big( \chi(t/T) \<3N> [\pm_j,N_j \colon 1\leq j \leq 3](t,x) \Big)(\lambda \mp_{123} \langle n \rangle, n) \Big\|_{L^2_\omega L_\lambda^2 \ell_n^2 ( \Omega \times \bR \times \bT^3)}^2 \\
 &\lesssim T^2 N_1^{-2} N_2^{-2} N_3^{-2}  \sum_{n_1,n_2,n_3 \in \bZ^3}  \bigg[  \Big( \prod_{j=1}^{3} \chi_{N_j}(n_j) \Big) \langle n_{123}\rangle^{2(s-1)} |\widehat{V}(n_{12})|^2  \\
 &\times 
  \int_{\bR} \dlambda \, \langle \lambda \rangle^{2(b_--1)}  |\widehat{\chi}\big( T (\lambda-\varphi(n_1,n_2,n_3))\big)|^2 \bigg] \\
  &\lesssim T^2 N_1^{-2} N_2^{-2} N_3^{-2}  \sum_{n_1,n_2,n_3 \in \bZ^3}   \Big( \prod_{j=1}^{3} \chi_{N_j}(n_j) \Big) \langle n_{123}\rangle^{2(s-1)} |\widehat{V}(n_{12})|^2  \langle \varphi(n_1,n_2,n_3) \rangle^{2(b_--1)}  \\
  &\lesssim  T^2 N_1^{-2} N_2^{-2} N_3^{-2} \sup_{m\in \bZ^3}  \sum_{n_1,n_2,n_3 \in \bZ^3}   \Big( \prod_{j=1}^{3} \chi_{N_j}(n_j) \Big) \langle n_{123}\rangle^{2(s-1)} |\widehat{V}(n_{12})|^2 1\big\{ |\varphi-m|\leq 1\big\}. 
\end{align*}
Combining this with \eqref{so3:eq_p6} and using the cubic sum estimate (Proposition \ref{tools:prop_cubic_sum}), we obtain that 
\begin{equation*}
 \Big\|    \Big\|  \<3N> [\pm_j,N_j \colon 1\leq j \leq 3]  \Big\|_{\X{s-1}{b_--1}([0,T])}  \Big\|_{L^p_\omega} \lesssim T p^{\frac{3}{2}} N_{\max}^{s-\beta}.  
\end{equation*}

Since there are at most $\lesssim \log(10+N_{\max})$ non-trivial choices for $N$, we obtain from Lemma \ref{tools:eq_moments_sups} that 
\begin{equation}\label{so3:eq_p1}
\begin{aligned}
& \Big\| \sup_{N\geq 1}   \Big\|  \<3N> [\pm_j,N_j \colon 1\leq j \leq 3]  \Big\|_{\X{s-1}{b_--1}([0,T])}  \Big\|_{L^p_\omega} \\
&\lesssim T \log\log(10+N_{\max})^2  N_{\max}^{s-\beta} p^{\frac{3}{2}}. 
\end{aligned}
\end{equation}
After summing over the dyadic scales, \eqref{so3:eq_p1} almost implies  \eqref{so3:eq_estimate_1} except that $b_-$ needs to be replaced by $b_+$. To achieve this, we utilize the room of the estimate \eqref{so3:eq_p1} in the maximal frequency scale. Using Plancherel's theorem, Minkowski's integral inequality, and Gaussian hypercontractivity, we have that 
\begin{align*}
&\Big\| \sup_{N\geq 1}   \Big\|  \<3N> [\pm_j,N_j \colon 1\leq j \leq 3]  \Big\|_{\X{0}{0}([0,T])}  \Big\|_{L^p_\omega} \\
&\lesssim  \log\log(10+N_{\max})^2  \sup_{N} \Big\|  1\big\{ 0 \leq t \leq T\big\}  \<3N> [\pm_j,N_j \colon 1\leq j \leq 3] \|_{L^p_\omega L_t^2 L_x^2} \\
&\lesssim  T^{\frac{1}{2}}  \log\log(10+N_{\max})^2  p^{\frac{3}{2}} \Big( \sum_{n_1,n_2,n_3 \in \bZ^3}    \prod_{j=1}^{3} \big(  \chi_{N_j}(n_j) \langle n_j \rangle^{-2} \big) \Big)^{\frac{1}{2}} \\
&\lesssim T^{\frac{1}{2}}  \log\log(10+N_{\max})^2  N_{\max}^{\frac{3}{2}} p^{\frac{3}{2}}. 
\end{align*}
By interpolating this estimate with \eqref{so3:eq_p1}, we obtain that 
\begin{equation}\label{so3:eq_p2}
\begin{aligned}
& \Big\| \sup_{N\geq 1}   \Big\|  \<3N> [\pm_j,N_j \colon 1\leq j \leq 3]  \Big\|_{\X{s-1}{b_+-1}([0,T])}  \Big\|_{L^p_\omega} \\
&\lesssim T  \log\log(10+N_{\max})^2   N_{\max}^{s-\beta+4(b_+-b_-)} p^{\frac{3}{2}} \\
&\lesssim T N_{\max}^{s-\beta+5(b_+-b_-)} p^{\frac{3}{2}}. 
\end{aligned}
\end{equation}
After summing over the dyadic scales, this finally yields \eqref{so3:eq_estimate_1}. We prove the second estimate \eqref{so3:eq_estimate_2} using the (frequency-localized version of the) first estimate. We present the details of the (standard) argument, but skip similar steps in subsequent proofs. Using the energy estimate (Lemma \ref{tools:lem_energy}) and the (frequency-localized version of the) first estimate \eqref{so3:eq_estimate_1}, we obtain that
\begin{equation}\label{so3:eq_p3}
 \Big\| \sup_{N\geq 1}   \Big\|  \<3DN> [\pm_j,N_j \colon 1\leq j \leq 3]  \Big\|_{L_t^\infty H_x^s}  \Big\|_{L^p_\omega} 
\lesssim (1+T)    N_{\max}^{s-\beta+5(b_+-b_-)} p^{\frac{3}{2}}. 
\end{equation}
For any $2\leq q \leq p$, we have from Sobolev embedding (in space-time), Minkowski's integral inequality, and Gaussian hypercontractivity that 
\begin{align}
& \Big\|  \<3DN> [\pm_j,N_j \colon 1\leq j \leq 3]  \Big\|_{L^p_\omega L_t^\infty \cC_x^s} \notag \\
&\lesssim   N_{\max}^{\frac{4}{q}}  \Big\|  \langle \nabla \rangle^s \<3DN> [\pm_j,N_j \colon 1\leq j \leq 3]  \Big\|_{L^p_\omega L_t^q L_x^q} \notag \\
&\lesssim   N_{\max}^{\frac{4}{q}}  \Big\|  \langle \nabla \rangle^s \<3DN> [\pm_j,N_j \colon 1\leq j \leq 3]  \Big\|_{L_t^q L_x^q L^p_\omega} \notag \\
&\lesssim   N_{\max}^{\frac{4}{q}}  p^{\frac{3}{2}} \Big\|  \langle \nabla \rangle^s \<3DN> [\pm_j,N_j \colon 1\leq j \leq 3]  \Big\|_{L_t^q L_x^q L^2_\omega}. \label{so3:eq_p4}
\end{align}
For a fixed $t\in\bR$, the distribution of $ \langle \nabla \rangle^s \<3DN> [\pm_j,N_j \colon 1\leq j \leq 3](t,x)$ is translation invariant. Thus, we can replace the $L_x^q$-norm in \eqref{so3:eq_p4} by the $L_x^2$-norm. Using Minkowski's integral inequality and \eqref{so3:eq_p3} then yields 
\begin{align*}
 &\Big\|  \<3DN> [\pm_j,N_j \colon 1\leq j \leq 3]  \Big\|_{L^p_\omega L_t^\infty \cC_x^s} 
 \lesssim   N_{\max}^{\frac{4}{q}}  p^{\frac{3}{2}} \Big\|  \langle \nabla \rangle^s \<3DN> [\pm_j,N_j \colon 1\leq j \leq 3]  \Big\|_{L_\omega^2 L_t^q L_x^2 } \\
& \lesssim T^{1+\frac{1}{q}}  N_{\max}^{s-\beta+5(b_+-b_-)+\frac{4}{q}}  p^{\frac{3}{2}} . 
\end{align*}
By choosing $q=q(b_+,b_-)$ sufficiently large and then summing over dyadic scales, this proves \eqref{so3:eq_estimate_2} for $p \gtrsim_{b_+,b_-} 1$. The smaller values of $p$ can be handled by using H\"{o}lder's inequality in $\omega$. \\

Finally, the statement for $\<3DN>$ replaced by $\<3DNtau>$ follows from the boundedness of $1_{[0,\tau]}(t)$ on $\X{s_2-1}{b_+-1}$, which was proven in Lemma \ref{tools:lem_restricted_continuity}.
\end{proof}

\subsection{Quartic stochastic objects}\label{section:quartic} 
The expansion  $u_N= \<1b>+ \<3DN> +w_N$ or the explicit stochastic objects in $\So$ only contain linear, cubic, quintic, or septic stochastic objects. However, the physical terms $\Phy$ contain terms such as 
\begin{equation*}
   V \ast \Big( P_{\leq N} \<1b> \cdot  P_{\leq N}  \<3DN>\Big)  P_{\leq N} w_N 
\qquad \text{or} \qquad 
  V \ast \Big( P_{\leq N} \<1b> \paraneq  P_{\leq N} w_N\Big)  P_{\leq N} \<3DN>. 
\end{equation*}
Since we treat $w_N \in \X{s_1}{b}$ using deterministic methods, they can be viewed as quartic expressions in the random initial data $\bluedot\,$. Furthermore, due to the convolution with the interaction potential $V$ in the second term, we also have to understand the product of $\<1b>$ and $\<3DN>$ at two different spatial points. 

\begin{proposition}\label{so4:prop}
Let $N_{123},N_4\geq 1$. Then, we have for all $s< -1/2-\eta$ and all $T\geq 1$ that 
\begin{equation}\label{so4:eq_estimate_1}
\begin{aligned}
&\Big\| \sup_{N\geq 1} \sup_{y \in \bT^3} \Big\| \Big( P_{N_{123}} P_{\leq N} \<3DN>(t,x-y) \Big) \cdot P_{N_4} P_{\leq N} \<1b>(t,x) \Big\|_{C_t^0 \cC_x^{s}([0,T]\times \bT^3)} \Big\|_{L^p_\omega(\bP)} \\
& \lesssim T^3 p^2 \max(N_{123},N_4)^{-\frac{\eta}{2}} N_4^\kappa. 
\end{aligned}
\end{equation}
If $N_{123}\sim N_4$, we have for all $s< -1/2+\beta-2\eta$ that 
\begin{equation}\label{so4:eq_estimate_2}
\begin{aligned}
&\Big\|   \sup_{N\geq 1} \sup_{y \in \bT^3} \Big\| \Big( P_{N_{123}} P_{\leq N} \<3DN>(t,x-y) \Big) \cdot P_{N_4} P_{\leq N} \<1b>(t,x) \Big\|_{C_t^0 \cC_x^{s}([0,T]\times \bT^3)} \Big\|_{L^p_\omega(\bP)} \\
&\lesssim T^3 p^2  N_4^{\kappa}. 
\end{aligned}
\end{equation}
Finally, without the shift in $y\in \bT^3$, we have for $s<-1/2-\eta$ that
\begin{equation}\label{so4:eq_estimate_3}
\begin{aligned}
&\Big\|  \sup_{N\geq 1}  \Big\| \Big( P_{N_{123}}  P_{\leq N} \<3DN>(t,x) \Big) \cdot P_{N_4} P_{\leq N} \<1b>(t,x) \Big\|_{C_t^0 \cC_x^{s}([0,T]\times \bT^3)} \Big\|_{L^p_\omega(\bP)} \\
& \lesssim T^3 p^2 \max(N_{123},N_4)^{-\frac{\eta}{10}}. 
\end{aligned}
\end{equation}
\end{proposition}

\begin{remark}\label{so4:rem}
In the fully frequency-localized version of Proposition \ref{so4:prop}, which is detailed in the proof, we gain an $\eta^\prime$-power of the maximal frequency-scale. As in Proposition \ref{so3:prop}, we may also replace $\<3DN>$ by  $\<3DNtau>= \Duh \big[ 1_{[0,\tau]} \<3N>\big]$.
\end{remark}

\begin{remark}
We recall that $\eta$ is much smaller than $\kappa$ and hence the right-hand sides of \eqref{so4:eq_estimate_1} and \eqref{so4:eq_estimate_2} diverge as $N_4\rightarrow \infty$. The third estimate \eqref{so4:eq_estimate_3} is quite delicate and requires the $\operatorname{sine}$-cancellation lemma. A similar estimate is not available for the partially shifted process and it is likely that at least a logarithmic loss is necessary in \eqref{so4:eq_estimate_1} and \eqref{so4:eq_estimate_2} as $N_4$ tends to infinity. 
\end{remark}

\begin{proof}
We prove \eqref{so4:eq_estimate_1} and \eqref{so4:eq_estimate_2} simultaneously. The third estimate \eqref{so4:eq_estimate_3} will mainly utilize the same estimates, but also requires the $\operatorname{sine}$-cancellation lemma (Lemma \ref{tools:lem_sin_cancellation}). Using the representation based on multiple stochastic integrals (Proposition \ref{tools:prop_stochastic_representation}), we have that 
\begin{align*}
 &\Big( P_{N_{123}} P_{\leq N} \<3DN>(t,x-y) \Big) \cdot P_{N_4} P_{\leq N} \<1b>(t,x) \\
&= \sum_{N_1,N_2,N_3 \geq 1} \sum_{n_1,n_2,n_3,n_4\in \bZ^3} \rho_{N}^2(n_{123}) \chi_{N_{123}}(n_{123}) \Big( \prod_{j=1}^{4} \rho_{ N}(n_j)  \chi_{N_j}(n_j) \Big) \widehat{V}_S(n_1,n_2,n_3)   \\ 
&\times \langle n_{123} \rangle^{-1} \exp\Big( i \langle n_{1234}, x \rangle - i \langle n_{123} , y \rangle\Big) 
 \Big( \int_0^t \sin((t-t^\prime) \langle n_{123} \rangle) ~ \cI_3[t^\prime;n_1,n_2,n_3] \cdot \cI_1[t;n_4] \, \mathrm{d}t^\prime \Big).  
\end{align*}
Using the product formula for multiple stochastic integrals, we obtain that 
\begin{align*}
 &\Big( P_{N_{123}} P_{\leq N} \<3DN>(t,x-y) \Big) \cdot P_{N_4} P_{\leq N} \<1b>(t,x) \\
&= \sum_{N_1,N_2,N_3 \geq 1} \cG^{(4)}(t,x,y;N_\ast) +  \sum_{N_1,N_2,N_3 \geq 1} \cG^{(2)}(t,x,y;N_\ast),
\end{align*}
where the dependence on $N_{123},N_1,N_2,N_3,N_4$ is indicated by $N_\ast$ and the quartic and quadratic Gaussian chaoses are given by 
\begin{align*}
&\cG^{(4)}(t,x,y;N_\ast) \\
&=  \sum_{\pm_1,\pm_2,\pm_3,\pm_4} \sum_{n_1,n_2,n_3,n_4\in \bZ^3} \bigg[ c(\pm_j\colon 1 \leq j \leq 4) \rho_{N}^2(n_{123}) \rho_N(n_4) \chi_{N_{123}}(n_{123}) \\
&\times  \big( \prod_{j=1}^{4} \rho_{\leq N}(n_j)  \chi_{N_j}(n_j) \big) \widehat{V}_S(n_1,n_2,n_3)   
\langle n_{123} \rangle^{-1} \exp\Big( i \langle n_{1234}, x \rangle - i \langle n_{123} , y \rangle\Big) \\
& \times \exp(\pm_4 i t \langle n_4 \rangle)  \Big( \int_0^t \sin((t-t^\prime) \langle n_{123} \rangle) \big( \prod_{j=1}^3 \exp(\pm_j i t^\prime \langle n_j \rangle) \big)  \, \mathrm{d}t^\prime \Big) 
\cI_4(\pm_j, n_j \colon 1 \leq j \leq 4 )\bigg]
\end{align*}
and 
\begin{align*}
&\cG^{(2)}(t,x,y;N_\ast) \\
&= 3 \sum_{\pm_1,\pm_2} \sum_{n_1,n_2\in \bZ^3}  \bigg[   c(\pm_1,\pm_2) \Big( \prod_{j=1}^{2} \rho_{N}(n_j)\chi_{N_j}(n_j) \Big) \exp\Big(  i \langle n_{12}, x \rangle\Big)   \\ 
&\times \bigg( \sum_{n_3 \in \bZ^3} \Big[ \rho_{N}^2(n_{123}) \rho_N^2(n_3) \chi_{N_{123}}(n_{123}) \chi_{N_3}(n_3) \chi_{N_4}(n_3) \langle n_{123} \rangle^{-1}  \langle n_3 \rangle^{-2} \widehat{V}_S(n_1,n_2,n_3)  \\
 &\times  \exp\big(- i \langle n_{123} , y \rangle\big) \int_0^t \sin((t-t^\prime) \langle n_{123} \rangle) \cos( (t-t^\prime) \langle n_3 \rangle) \prod_{j=1}^2 \exp( \pm_j  i t^\prime \langle n_ j \rangle \rangle) \, \mathrm{d}t^\prime \Big] \bigg) \\
&\times \cI_{2}(\pm_j,n_j\colon j=1,2) \bigg]. 
\end{align*}
The quartic Gaussian chaos $\cG^{(4)}$  and quadratic Gaussian chaoses $\cG^{(2)}$ contain the resonant and non-resonant terms of the product, respectively. We estimate both terms separately. \\

\emph{The non-resonant term $\cG^{(4)}$:} 
We first let $ s < -1/2 - \eta$. Using Gaussian hypercontractivity and standard reductions  (see e.g. the proof of Proposition \ref{so3:prop}), it suffices to estimate the $L_t^\infty L_\omega^2 H_x^s$-norm instead of the $L^p_\omega L_t^\infty \cC_x^s$-norm. Let the phase-function $\varphi$ be as in \eqref{tools:eq_phase_varphi}. Using the orthogonality of the multiple stochastic integrals, we have for a fixed $t\in [0,T]$ that 
\begin{align*}
&\big\| \cG^{(4)}(t,x,y;N_\ast)) \big\|_{L_\omega^2 H_x^s}^2 \\
&\lesssim \sum_{\pm_1,\pm_2,\pm_3} \sum_{n_1,n_2,n_3,n_4 \in \bZ^3}\bigg[ \chi_{N_{123}}(n_{123})   \Big( \prod_{j=1}^4 \chi_{N_j}(n_j) \Big) |\widehat{V}_S(n_1,n_2,n_3)|^2 \langle n_{1234} \rangle^{2s} \langle n_{123} \rangle^{-2} \Big( \prod_{j=1}^4 \langle n_j \rangle^{-2} \Big) \allowdisplaybreaks[3] \\
&\times \Big|  \int_0^t \sin((t-t^\prime) \langle n_{123} \rangle) \big( \prod_{j=1}^3 \exp(\pm_j i t^\prime \langle n_j \rangle) \big)  \, \mathrm{d}t^\prime \Big|^2 \bigg] \\
&\lesssim (1+T)^2  \sum_{\pm_1,\pm_2,\pm_3} \sum_{n_1,n_2,n_3,n_4 \in \bZ^3} \sum_{m\in \bZ}  \bigg[ \langle m \rangle^{-2} \chi_{N_{123}}(n_{123})   \Big( \prod_{j=1}^4 \chi_{N_j}(n_j) \Big) |\widehat{V}_S(n_1,n_2,n_3)|^2 \langle n_{1234} \rangle^{2s} \\
&\times   \langle n_{123} \rangle^{-2}\Big( \prod_{j=1}^4 \langle n_j \rangle^{-2} \Big) 1\big\{ |\varphi-m|\leq 1 \big\} \bigg] \allowdisplaybreaks[3] \\
&\lesssim T^2 \sup_{m\in \bZ}   \sum_{\pm_1,\pm_2,\pm_3} \sum_{n_1,n_2,n_3,n_4 \in \bZ^3}   \bigg[  \chi_{N_{123}}(n_{123})   \Big( \prod_{j=1}^4 \chi_{N_j}(n_j) \Big) |\widehat{V}_S(n_1,n_2,n_3)|^2
\langle n_{1234} \rangle^{2s} \\
&\times   \langle n_{123} \rangle^{-2}  \Big( \prod_{j=1}^4 \langle n_j \rangle^{-2} \Big) 1\big\{ |\varphi-m|\leq 1 \big\} \bigg]. 
\end{align*}
Using the non-resonant quartic sum estimate (Lemma \ref{tools:lemma_nonresonant_quartic}), it follows that 
\begin{equation*}
\big\| \cG^{(4)}(t,x,y;N_{123},N_1,N_2,N_3,N_4)) \big\|_{L_\omega^2 H_x^s}^2 \\
\lesssim T^2 \max(N_1,N_2,N_3)^{-2\beta+2\eta} N_4^{-2\eta}. 
\end{equation*}
This yields \eqref{so4:eq_estimate_1} for the non-resonant component. If $N_{123} \sim N_4$, then $\max(N_1,N_2,N_3) \gtrsim N_4$, and hence we can raise the value of $s$ by $\beta - \eta$. Thus, we also obtain \eqref{so4:eq_estimate_2} for the non-resonant component. Even when $y \neq 0$, our estimate for the non-resonant component does not exhibit any growth in $N_4$, and hence it also yields \eqref{so4:eq_estimate_3} for the non-resonant component. \\

\emph{The resonant term $\cG^{(2)}$:} This term exhibits a higher spatial regularity and we let $-1/2<s<0$. Using Gaussian hypercontractivity and standard reductions  (see e.g. the proof of Proposition \ref{so3:prop}), it suffices to estimate the $L_t^\infty L_\omega^2 H_x^s$-norm instead of the $L^p_\omega L_t^\infty \cC_x^s$-norm. Using the orthogonality of the multiple stochastic integrals, we have that 
\begin{equation}\label{so4:eq_resonant_p1}
\begin{aligned}
&\big\| \cG^{(2)}(t,x,y;N_\ast) \big\|_{L_\omega^2 H_x^s}^2 \\
&\lesssim \sum_{\pm_1,\pm_2} \sum_{n_1,n_2 \in \bZ^3} \bigg[  \Big( \prod_{j=1}^2 \chi_{N_j}(n_j) \Big) \langle n_{12} \rangle^{2s} \langle n_1 \rangle^{-2} \langle n_2 \rangle^{-2} \\
&\times \bigg| \sum_{n_3 \in \bZ^3} \Big[ \rho_{N}^2(n_{123}) \rho_N^2(n_3) \chi_{N_{123}}(n_{123}) \chi_{N_3}(n_3) \chi_{N_4}(n_3) \langle n_{123} \rangle^{-1}  \langle n_3 \rangle^{-2} \widehat{V}_S(n_1,n_2,n_3)  \\
 &\times  \exp\big(- i \langle n_{123} , y \rangle\big) \int_0^t \sin((t-t^\prime) \langle n_{123} \rangle) \cos( (t-t^\prime) \langle n_3 \rangle) \prod_{j=1}^2 \exp( \pm_j  i t^\prime \langle n_ j \rangle \rangle) \, \mathrm{d}t^\prime \Big] \bigg|^2 \bigg]. 
\end{aligned}
\end{equation}
We now present two estimates of \eqref{so4:eq_resonant_p1}. The first estimate will yield \eqref{so4:eq_estimate_1} and \eqref{so4:eq_estimate_2}. The second estimate is restricted to the case $y=0$ and yields, combined with the first estimate, \eqref{so4:eq_estimate_3}. After computing the integral in $t^\prime$ and decomposing according to the dispersive symbol, we obtain from Cauchy-Schwarz that 
\begin{align*}
\eqref{so4:eq_resonant_p1} &\lesssim
T^2 1\big\{ N_3 \sim N_4 \big\} \sum_{n_1,n_2\in \bZ^3}  \bigg[  \Big( \prod_{j=1}^2 \chi_{N_j}(n_j) \Big) \langle n_{12} \rangle^{2s} \langle n_1 \rangle^{-2} \langle n_2 \rangle^{-2} \\
&\times \bigg( \sum_{m\in \bZ} \sum_{n_3\in \bZ^3} \langle m\rangle^{-1} \chi_{N_3}(n_3) |\widehat{V}(n_1,n_2,n_3)| \langle n_{123} \rangle^{-1} \langle n_3 \rangle^{-2} 1\big\{ |\varphi-m|\leq 1\big\} \bigg)^2 \bigg]. 
\end{align*}
Using the resonant quartic sum estimate (Lemma \ref{tools:lemma_resonant_quartic}), this implies that
\begin{equation*}
\eqref{so4:eq_resonant_p1} \lesssim   T^2 1 \big\{ N_3 \sim N_4 \big\} \log(2+N_4)^2 \max(N_1,N_2)^{2s}. 
\end{equation*}
This clearly implies \eqref{so4:eq_estimate_1} and \eqref{so4:eq_estimate_2}. Except for the logarithmic divergence in $N_4$ (and hence $N_3$), it also implies \eqref{so4:eq_estimate_3}. We now need to restrict to $y=0$ and we may assume that $N_1,N_2 \ll N_3$. For fixed $n_1,n_2 \in \bZ^3$, we can apply the $\operatorname{sine}$-cancellation lemma
(Lemma \ref{tools:lem_sin_cancellation}) with $A=\max(N_1,N_2)$ and 
\begin{equation*}
\begin{aligned}
&f(t,t^\prime,n_3) \\
&\defe \rho_{N}^2(n_{123}) \rho_N^2(n_3) \chi_{N_{123}}(n_{123}) \chi_{N_3}(n_3) \chi_{N_4}(n_3) \langle n_{123} \rangle^{-1}  \langle n_3 \rangle^{-2} \widehat{V}_S(n_1,n_2,n_3) 
 \prod_{j=1}^2 \exp( \pm_j  i t^\prime \langle n_ j \rangle \rangle). 
\end{aligned}
\end{equation*} 
This yields
\begin{equation*}
\begin{aligned}
&\eqref{so4:eq_resonant_p1}\big|_{y=0} \\
& \lesssim T^4 1\big\{ N_3 \sim N_4 \big\} \max(N_1,N_2)^8 N_3^{-2} \sum_{n_1,n_2 \in \bZ^3} \langle n_{12} \rangle^{2s} 
\Big( \prod_{j=1}^2 1\big\{ |n_j| \sim N_j \big\}  \langle n_j \rangle^{-2} \Big)  \\
&\lesssim T^4 \max(N_1,N_2)^{10} N_3^{-2}. 
\end{aligned}
\end{equation*}
By combining our two estimates of $\eqref{so4:eq_resonant_p1}\big|_{y=0}$ we arrive at \eqref{so4:eq_estimate_3}. 

\end{proof}

\begin{remark}\label{quartic:rem_resonant_regularity}
As we have seen in the proof of Proposition \ref{so4:prop}, the (probabilistic) resonant portion of $P_{\leq N} \<3DN>  \cdot  P_{\leq N} \<1b>$ has spatial regularity $0-$, which is better than the sum of the individual spatial regularities. As a result, the probabilistic resonances between linear and cubic stochastic objects in Section \ref{section:septic} are relatively harmless. 
\end{remark}

\subsection{Quintic stochastic objects}

In this subsection, we control the quintic stochastic objects in $\So$, i.e., 
\begin{equation*}
  \nparaboxld \<131N> \qquad \text{and} \qquad \<113N>. 
\end{equation*}
Since $\So$ is part of the evolution equation for the smoother nonlinear remainder $Y_N$, the quintic stochastic objects have to be controlled at regularity $s_2-1$. 

\begin{proposition}[First quintic stochastic object]\label{so5i:prop}
For any $T \geq 1$ and any $p \geq 2$, it holds that 
\begin{equation}
\Big\| \sup_{N \geq 1} \Big\| \nparaboxld \<131N> \Big\|_{\X{s_2-1}{b_+-1}([0,T])} \Big\|_{L^p_\omega(\Omega)} \lesssim T^2 p^{\frac{5}{2}}. 
\end{equation}
\end{proposition}

\begin{proposition}[Second quintic stochastic object]\label{so5ii:prop}
For any $T \geq 1$ and any $p \geq 2$, it holds that 
\begin{equation}
\Big\| \sup_{N \geq 1} \Big\| \<113N> \Big\|_{\X{s_2-1}{b_+-1}([0,T])} \Big\|_{L^p_\omega(\Omega)} \lesssim T^2 p^{\frac{5}{2}}. 
\end{equation}
\end{proposition}

\begin{remark}
In the frequency-localized versions of Proposition \ref{so5i:prop} and Proposition \ref{so5ii:prop}, which are detailed in the proof, we gain an $\eta^\prime$-power of the maximal frequency-scale. As in Proposition \ref{so3:prop}, we may also replace $\<3DN>$ by  $\<3DNtau>= \Duh \big[ 1_{[0,\tau]} \<3N>\big]$. We will not further comment on these minor modifications.
\end{remark}

\begin{proof}[Proof of Proposition \ref{so5i:prop}:]
Throughout the proof, we ignore the supremum in $N\geq 1$ and only prove a uniform estimate for a fixed $N$. Using the frequency-localized estimates below and the same argument as in the proof of Proposition \ref{so3:prop}, we can insert the supremum in $N$ at the end of the proof. \\ 

We first obtain a representation of the quintic stochastic object using multiple stochastic integral. Using \eqref{local:eq_so_term2} and Proposition \ref{tools:prop_stochastic_representation}, we have that 
\begin{align*}
 &\nparaboxld \<131N> (t,x) \\
 &= \sum_{\substack{N_{345},N_1,\hdots,N_5 \colon \\ \max(N_1,N_{345}) > N_2^\epsilon} } \sum_{n_1,\hdots,n_5 \in \bZ^3}
  \bigg[ \rho_N^2(n_{345}) \chi_{N_{345}}(n_{345}) \Big( \prod_{j=1}^5 \rho_N(n_j) \chi_{N_j}(n_j) \Big) 
  \widehat{V}(n_{1345}) \widehat{V}_S(n_3,n_4,n_5) \\
  &\times \langle n_{345} \rangle^{-1} \exp\big( i \langle n_{12345} , x \rangle\big)  
 \cI_2[t,n_1,n_2] \Big( \int_0^t \sin\big( (t-t^\prime) \langle n_{123} \rangle \big) \cI_3[t^\prime,n_3,n_4,n_5] \dtprime \Big) \bigg]. 
\end{align*}
Using the product formula for mixed multiple stochastic integrals (Proposition \ref{tools:prop_stochastic_representation} and Lemma \ref{tools:lemma_23product}), we obtain that 
\begin{equation}\label{so5i:eq_chaos_decomposition}
 \nparaboxld \<131N> (t,x) 
 = \sum_{\substack{N_{345},N_1,\hdots,N_5 \colon \\  \max(N_1,N_{345}) > N_2^\epsilon} } \Big( \cG^{(5)} + \cG^{(3)} + \widetilde{\cG}^{(3)} + \cG^{(1)} \Big)(t,x;N_\ast),
\end{equation}
where the dependence on $N_{345},N_1,\hdots,N_5$ is indicated by $N_\ast$ and the quintic, cubic, and linear Gaussian chaoses are defined as follows. The quintic chaos is given by 
\begin{align*}
 &\cG^{(5)}(t,x;N_\ast) \\
 &\defe \sum_{\pm_1,\hdots,\pm_5} c(\pm_j \colon 1 \leq j \leq 5) \sum_{n_1,\hdots,n_5 \in \bZ^3}
  \bigg[ \rho_N^2(n_{345}) \chi_{N_{345}}(n_{345}) \Big( \prod_{j=1}^5 \rho_N(n_j) \chi_{N_j}(n_j) \Big) 
  \widehat{V}(n_{1345})  \\
  &\times \widehat{V}_S(n_3,n_4,n_5)  \langle n_{345} \rangle^{-1} \exp\big( i \langle n_{12345} , x \rangle\big)  
\Big( \prod_{j=1}^2 \exp\big( \pm_j i t \langle n_j \rangle \big) \Big)  \\
& \times \Big( \int_0^t \sin\big( (t-t^\prime) \langle n_{123} \rangle \big)   \prod_{j=3}^5 \exp\big( \pm_j i t^\prime \langle n_j \rangle \big)  \dtprime \Big)  \cI_5[ \pm_j , n_j \colon 1 \leq  j \leq 5] \bigg].
\end{align*}
The two cubic Gaussian chaoses are given by 
\begin{align*} 
&\cG^{(3)}(t,x;N_\ast) \allowdisplaybreaks[3]\\
&\defe \sum_{\pm_2,\pm_4,\pm_5} c(\pm_2,\pm_4,\pm_5) \sum_{n_2,n_4,n_5 \in \mathbb{Z}^3} \bigg[  \Big( \prod_{j=2,4,5} \rho_N(n_j) \chi_{N_j}(n_j) \Big) 
\widehat{V}(n_{45})  \exp\big( i \langle n_{245} , x \rangle \big) \\
& \times \sum_{n_3 \in \bZ^3} \bigg( \rho_{N}^2(n_3) \rho_N^2(n_{345}) \chi_{N_{345}}(n_{345}) \chi_{N_1}(n_3) \chi_{N_3}(n_3)
 \widehat{V}_S(n_3,n_4,n_5) \langle n_{345} \rangle^{-1} \langle n_3 \rangle^{-2} \exp\big( \pm_2 i t \langle n_2 \rangle\big)  \\
&\times \int_0^t \sin\big((t-t^\prime) \langle n_{345} \rangle) \cos\big((t-t^\prime) \langle n_{3} \rangle) \prod_{j=4,5} \exp\big( \pm_j i t^\prime \langle n_j \rangle \big) \dtprime \bigg) \cI_{3}[\pm_j,n_j \colon j=2,4,5] \bigg] 
\end{align*}
and 
\begin{align*} 
&\widetilde{\cG}^{(3)}(t,x;N_\ast) \allowdisplaybreaks[3] \\
&\defe \sum_{\pm_1,\pm_4,\pm_5} c(\pm_1,\pm_4,\pm_5) \sum_{n_1,n_4,n_5 \in \mathbb{Z}^3} \bigg[  \Big( \prod_{j=1,4,5} \rho_N(n_j) \chi_{N_j}(n_j) \Big) 
  \exp\big( i \langle n_{145} , x \rangle \big) \sum_{n_3 \in \bZ^3} \bigg( \rho_{N}^2(n_3) \\
& \times  \rho_N^2(n_{345}) \chi_{N_{345}}(n_{345}) \chi_{N_1}(n_3) \chi_{N_3}(n_3)\widehat{V}_S(n_3,n_4,n_5) \widehat{V}(n_{1345})   \langle n_{345} \rangle^{-1} \langle n_3 \rangle^{-2} \exp\big( \pm_1 i t \langle n_1 \rangle\big)  \\
&\times \int_0^t \sin\big((t-t^\prime) \langle n_{345} \rangle) \cos\big((t-t^\prime) \langle n_{3} \rangle) \prod_{j=4,5} \exp\big( \pm_j i t^\prime \langle n_j \rangle \big) \dtprime \bigg) \cI_{3}[\pm_j,n_j \colon j=1,4,5] \bigg]. 
\end{align*}
Finally, the linear Gaussian chaos (or simply Gaussian) is given by
\begin{align*}
&\cG^{(1)}(t,x;N_\ast) \allowdisplaybreaks[3]\\
&\defe \sum_{\pm_5} c(\pm_5)  \sum_{n_5\in \bZ^3} \rho_N(n_5) \chi_{N_5}(n_5) \exp\big( i \langle n_5, x \rangle \big)  \sum_{n_3,n_4 \in \bZ^4}  \bigg[ \rho_N^2(n_{345})  \rho_N^2(n_3) \rho_N^2(n_4)\chi_{N_{345}}(n_{345})  \\
&\times \chi_{N_1}(n_3) \chi_{N_3}(n_3) \chi_{N_2}(n_4) \chi_{N_4}(n_4) \widehat{V}_S(n_3,n_4,n_5) \widehat{V}(n_{45}) 
 \langle n_{345} \rangle^{-1}\langle n_3 \rangle^{-2} \langle n_4 \rangle^{-2} \\
 &\times \int_0^t \sin\big( (t-t^\prime) \langle n_{345} \rangle\big) \cos\big( (t-t^\prime) \langle n_3 \rangle\big)  \cos\big( (t-t^\prime) \langle n_4 \rangle\big) \exp\big( \pm_5 i t^\prime \langle n_5 \rangle\big) \dtprime \bigg] \cI_{1}[\pm_5,n_5] . 
\end{align*}

Each of the frequency-localized Gaussian chaoses in \eqref{so5i:eq_chaos_decomposition} is now estimated separately. We encourage the reader to concentrate on the estimates for $\cG^{(5)}$ and $\cG^{(1)}$, which already contain all ideas and ingredients. \\

\emph{The non-resonant term $\cG^{(5)}$:} \\
Let $s=1/2-\eta$. We will first estimate the $\X{s-1}{b_--1}$-norm of a dyadic piece and then use the condition $\max(N_1,N_{345})>N_2^\epsilon$ to increase the value of $s$.   Using Gaussian hypercontractivity (Proposition \ref{tools:prop_Gaussian_hypercontractivity}), the orthogonality of multiple stochastic integrals, and Lemma \ref{tools:lem_estimate_fouriertransform}, we obtain that
\begin{align}
&\Big\| \| \cG^{(5)}(t,x;N_\ast) \|_{\X{s-1}{b_--1}([0,T])} \Big\|_{L^p_\omega}^2 \notag \\
&\lesssim \max_{\pm_{12345}} \Big\| \| \langle \lambda \rangle^{b_--1} \langle n \rangle^{s-1}  \mathcal{F}_{t,x} \big( \chi(t/T) \cG^{(5)} (t,x;N_\ast) \big)(\lambda \mp_{12345} \langle n \rangle, n) \|_{L_\lambda^2 \ell_n^2(\bR\times \bZ^3)} \Big\|_{L^p_\omega}^2  \notag \\
&\lesssim   p^5   \max_{\pm_{12345}} \Big\| \| \langle \lambda \rangle^{b_--1} \langle n \rangle^{s-1} \mathcal{F}_{t,x} \big( \chi(t/T) \cG^{(5)} (t,x;N_\ast) \big)(\lambda \mp_{12345} \langle n \rangle, n) \|_{L_\lambda^2 \ell_n^2(\bR\times \bZ^3)} \Big\|_{L^2_\omega}^2   \notag \\
&\lesssim T^2 p^5   \max_{\substack{\pm_{12345},\pm_{345},\\ \pm_1, \hdots,\pm_5} } \sum_{n_1,\hdots,n_5 \in \bZ^3} \bigg[   \chi_{N_{345}}(n_{345})  \Big( \prod_{j=1}^5  \chi_{N_j}(n_j) \Big) \langle n_{12345} \rangle^{2(s-1)}   \langle n_{345} \rangle^{-2}  \notag \\
&  \times  |\widehat{V}(n_{1345})|^{2}    |\widehat{V}_S(n_3,n_4,n_5)|^2  \Big( \prod_{j=1}^5 \langle n_j \rangle^{-2} \Big) 
\big(1+|\pm_{345} \langle n_{345} \rangle \pm_3 \langle n_3 \rangle \pm_4 \langle n_4 \rangle \pm_5 \langle n_5 \rangle| \big)^{-2} 
 \notag \\ 
& \times 
\int_\bR \langle \lambda \rangle^{2(b_--1)} \bigg(    1+ \min\Big( \Big|\lambda - (\pm_{12345} \langle n_{12345}\rangle \pm_{345} \langle n_{345} \rangle \pm_1 \langle n_1 \rangle \pm_2 \langle n_2 \rangle  )\Big|,  \notag \\
& \hspace{17ex} \Big|\lambda - (\pm_{12345} \langle n_{12345} \rangle \mp_{345} \langle n_{345} \rangle +\sum_{j=1}^5 (\pm_j ) \langle n_j \rangle  \Big|\Big) \bigg)^{-2} 
\dlambda \bigg]. \label{so5i:eq_p1}
\end{align}
To break down this long formula, we define the phase-functions
\begin{align*}
\psi(n_3,n_4,n_5) &\defe \pm_{345} \langle n_{345} \rangle \pm_3 \langle n_3 \rangle \pm_4 \langle n_4 \rangle \pm_5 \langle n_5 \rangle, \\ 
\varphi(n_1,\hdots,n_5)
 &\defe \pm_{12345} \langle n_{12345}\rangle \pm_{345} \langle n_{345} \rangle \pm_1 \langle n_1 \rangle \pm_2 \langle n_2 \rangle,\\
\widetilde{\varphi}(n_1,\hdots,n_5) &\defe \pm_{12345} \langle n_{12345} \rangle \mp_{345} \langle n_{345} \rangle +\sum_{j=1}^5 (\pm_j ) \langle n_j \rangle . 
\end{align*}
Integrating in $\lambda$ and decomposing according to the value of the phases, we obtain that 
\begin{align*}
\eqref{so5i:eq_p1}   
&\lesssim  T^2 p^5 \log(2+\max(N_1,\hdots,N_5))     \max_{\substack{\pm_{12345},\pm_{345},\\ \pm_1, \hdots,\pm_5} } \sup_{m,m^\prime \in \bZ} \sum_{n_1,\hdots,n_5 \in \bZ^3} \bigg[   \chi_{N_{345}}(n_{345})   \\
&  \times \Big( \prod_{j=1}^5  \chi_{N_j}(n_j) \Big) \langle n_{12345} \rangle^{2(s-1)} \langle n_{345} \rangle^{-2} |\widehat{V}(n_{1345})|^{2} 
  |\widehat{V}_S(n_3,n_4,n_5)|^2  \Big( \prod_{j=1}^5 \langle n_j \rangle^{-2} \Big) \\
  &\times \, 1\big\{ |\psi-m|\leq 1\big\} \, \Big( 1\big\{ |\varphi-m^\prime|\leq1 \big\} + 1\big\{ |\widetilde{\varphi}-m^\prime|\leq1 \big\}\Big)  \bigg]. 
\end{align*}
Using the non-resonant quintic sum estimate (Lemma  \ref{tools:lem_counting_quintic}),  we finally obtain that
\begin{equation}\label{so5i:eq_p2}
\Big\| \| \cG^{(5)}(t,x;N_\ast) \|_{\X{s-1}{b_--1}([0,T])} \Big\|_{L^p_\omega} 
\lesssim T p^{\frac{5}{2}} \max(N_1,N_3,N_4,N_5)^{-\beta+\eta} N_2^{-\eta}. 
\end{equation}
Due to the operator $\nparaboxld$, we have that 
\begin{equation*}
 \max(N_1,N_3,N_4,N_5) \gtrsim \max(N_1,N_2,N_3,N_4,N_5)^\epsilon. 
\end{equation*}
Thus, \eqref{so5i:eq_p2} implies 
\begin{equation*}
\Big\| \| \cG^{(5)}(t,x;N_\ast) \|_{\X{s_2-1}{b_--1}([0,T])} \Big\|_{L^p_\omega} 
\lesssim T p^{\frac{5}{2}} \max(N_1,N_2,N_3,N_4,N_5)^{\delta_2 + 3 \eta - \epsilon \beta},
\end{equation*}
which is acceptable. \\

\emph{Single-resonance term $\cG^{(3)}$:}

This term only yields a non-trivial contribution if $N_1\sim N_3$. In particular, $\max(N_1,N_{345}) > N_2^\epsilon$ implies that $\max(N_3,N_4,N_5) \gtrsim N_2^\epsilon$. 
Using the inhomogeneous Strichartz estimate (Lemma \ref{tools:lem_inhomogeneous_strichartz}) and Gaussian hypercontractivity, we have that 
\begin{equation}\label{so5i:eq_p6}
\begin{aligned}
\Big\| \| \cG^{(3)}(t,x;N_\ast) \|_{\X{s_2-1}{b_--1}([0,T])} \Big\|_{L^p_\omega} 
&\lesssim \Big\|   \| \cG^{(3)}(t,x;N_\ast) \|_{L_t^{2b_+} H_x^{s_2-1}([0,T]\times \bT^3)}  \Big\|_{L^p_\omega}  \\
&\lesssim T^{\frac{1}{2}} \Big\|  \| \cG^{(3)}(t,x;N_\ast) \|_{L_t^{2} H_x^{s_2-1}([0,T]\times \bT^3)} \Big\|_{L^p_\omega}  \\
&\lesssim T p^{\frac{3}{2}} \, \sup_{t\in [0,T]}\Big\| \| \cG^{(3)}(t,x;N_\ast) \|_{H_x^{s_2-1}(\bT^3)} \Big\|_{L^2_\omega}. 
\end{aligned}
\end{equation}
Using the orthogonality of the multiple stochastic integrals, we have that 
\begin{equation}\label{so5i:eq_p7}
\begin{aligned}
&\sup_{t\in [0,T]}\Big\| \| \cG^{(3)}(t,x;N_\ast) \|_{H_x^{s_2-1}(\bT^3)} \Big\|_{L^2_\omega}^2 \\
&\lesssim N_{45}^{-2\beta} N_2^{-2} N_4^{-2} N_5^{-2} \hspace{-1ex} \sum_{n_2,n_4,n_5\in \bZ^3} \hspace{-1ex} \chi_{N_{45}}(n_{45}) \Big( \prod_{j=2,4,5} \chi_{N_j}(n_j)\Big)   \langle n_{245} \rangle^{2(s_2-1)} 
\mathscr{S}(n_2,n_4,n_5;t,N_\ast)^2,
\end{aligned}
\end{equation}
where
\begin{align*}
&\mathscr{S}(n_2,n_4,n_5;t,N_\ast) \\
&\defe \bigg|\sum_{n_3\in \bZ^3}   \bigg[  \rho_{N}^2(n_3) \rho_N^2(n_{345}) \chi_{N_{345}}(n_{345}) 
\chi_{N_1}(n_3) \chi_{N_3}(n_3)  \widehat{V}_S(n_3,n_4,n_5) \langle n_{345} \rangle^{-1} \langle n_3 \rangle^{-2}  \\
 &\times \exp\big( \pm_2 i t \langle n_2 \rangle\big)  \int_0^t \sin\big((t-t^\prime) \langle n_{345} \rangle) \cos\big((t-t^\prime) \langle n_{3} \rangle) \prod_{j=4,5} \exp\big( \pm_j i t^\prime \langle n_j \rangle \big) \dtprime \bigg] \bigg|. 
\end{align*}

Define the phase-function $\varphi$ by
\begin{equation}
\varphi(n_3,n_4,n_5) \defe \langle n_{345} \rangle \pm_3 \langle n_3 \rangle \pm_4 \langle n_4 \rangle \pm_5 \langle n_5 \rangle. 
\end{equation}
By performing the integral, using the triangle-inequality, expanding the square, and using Lemma \ref{tools:lem_basic_resonance}, we obtain that 
\begin{align*}
&\mathscr{S}(n_2,n_4,n_5;t,N_\ast)^2 \\
&\lesssim T^2 \max_{\pm_3,\pm_4,\pm_5} \bigg(\sum_{m\in \bZ} \sum_{n_3\in \bZ^3} \langle m \rangle^{-1} \chi_{N_{3}}(n_3) \langle n_{345} \rangle^{-1} \langle n_3 \rangle^{-2} 1\big\{ |\varphi-m| \leq 1 \big\} \bigg)^2  \\
&\lesssim  T^2 \log(2+\max(N_3,N_4,N_5)) \bigg( \max_{\pm_3,\pm_4,\pm_5} \sup_{m\in \bZ} \sum_{n_3\in \bZ^3}  \chi_{N_{3}}(n_3) \langle n_{345} \rangle^{-1} \langle n_3 \rangle^{-2} 1\big\{ |\varphi-m| \leq 1 \big\} \bigg) \\
&\times \bigg( \max_{\pm_3,\pm_4,\pm_5} \sum_{m\in \bZ} \sum_{n_3\in \bZ^3} \langle m \rangle^{-1} \chi_{N_{3}}(n_3) \langle n_{345} \rangle^{-1} \langle n_3 \rangle^{-2} 1\big\{ |\varphi-m| \leq 1 \big\} \bigg) \\
&\lesssim T^2 \log(2+\max(N_3,N_4,N_5))  \langle n_{45} \rangle^{-1} \\
&\times \max_{\pm_3,\pm_4,\pm_5} \sum_{m\in \bZ} \sum_{n_3\in \bZ^3} \langle m \rangle^{-1} \chi_{N_{3}}(n_3) \langle n_{345} \rangle^{-1} \langle n_3 \rangle^{-2} 1\big\{ |\varphi-m| \leq 1 \big\}. 
\end{align*}
By inserting this into \eqref{so5i:eq_p7} and summing in $n_2\in \bZ^3$ first, we obtain 
\begin{align*}
&\sup_{t\in [0,T]}\Big\| \| \cG^{(3)}(t,x;N_\ast) \|_{H_x^{s_2-1}(\bT^3)} \Big\|_{L^2_\omega}^2\\
&\lesssim T^2  \log(2+\max(N_3,N_4,N_5)) \Big( \prod_{j=2}^5 N_j^{-2} \Big)    \\
&\times \max_{\pm_3,\pm_4,\pm_5} \sum_{m\in \bZ} \sum_{n_2,n_3,n_4,n_5\in \bZ^3} \bigg[ \langle m \rangle^{-1} \Big( \prod_{j=2}^5 \chi_{N_j}(n_j)\Big) \langle n_{245} \rangle^{2(s_2-1)} \langle n_{345} \rangle^{-1} \langle n_{45} \rangle^{-1-2\beta} 1\big\{ |\varphi-m|\leq 1\big\} \bigg] \\
& \lesssim T^2  \log(2+\max(N_3,N_4,N_5))  N_2^{2s_2-1} \Big( \prod_{j=3}^5 N_j^{-2} \Big)    \\
 &\times \max_{\pm_3,\pm_4,\pm_5} \sum_{m\in \bZ} \sum_{n_3,n_4,n_5\in \bZ^3} \bigg[ \langle m \rangle^{-1} \Big( \prod_{j=3}^5 \chi_{N_j}(n_j)\Big) \langle n_{345} \rangle^{-1} \langle n_{45} \rangle^{-1-2\beta} 1\big\{ |\varphi-m|\leq 1\big\} \bigg] \\
 &\lesssim  T^2  \log(2+\max(N_3,N_4,N_5))  N_2^{2s_2-1} \max(N_4,N_5)^{-2\beta}.
\end{align*}
In the last line, we have used the cubic sum estimate (Proposition \ref{tools:prop_cubic_sum}). In total, this yields
\begin{equation}\label{so5i:eq_p8}
\sup_{t\in [0,T]}\Big\| \| \cG^{(3)}(t,x;N_\ast) \|_{H_x^{s_2-1}(\bT^3)} \Big\|_{L^2_\omega} \\
\lesssim T \log(2+\max(N_3,N_4,N_5)) N_2^{s_2-\frac{1}{2}} \max(N_4,N_5)^{-\beta}. 
\end{equation} 
Recalling that $\max(N_3,N_4,N_5)>N_2^\epsilon$, we are only missing decay in $N_3$. By using the $\operatorname{sine}$-cancellation lemma
 (Lemma \ref{tools:lem_sin_cancellation}) to estimate $\mathscr{S}(n_2,n_4,n_5;t,N_\ast)$, we easily obtain that 
 \begin{equation}\label{so5i:eq_p9}
\sup_{t\in [0,T]}\Big\| \| \cG^{(3)}(t,x;N_\ast) \|_{H_x^{s_2-1}(\bT^3)} \Big\|_{L^2_\omega} \\
\lesssim T^2 N_2^{s_2-\frac{1}{2}} \max(N_4,N_5)^5 N_3^{-1}.  
\end{equation} 
After combining \eqref{so5i:eq_p8}, \eqref{so5i:eq_p9}, and the condition $\max(N_3,N_4,N_5)>N_2^\epsilon$, we obtain an acceptable estimate. \\

\emph{Single-resonance term $\widetilde{\cG}^{(3)}$:} This term can be controlled through similar (or simpler) arguments than $\cG^{(3)}$ and we omit the details.\\

\emph{Double-resonance term $\cG^{(1)}$:} \\
This term only yields a non-trivial contribution when $N_1 \sim N_3$ and $N_2 \sim N_4$. We note that the sum in $n_3 \in \bZ^3$ may appear to diverge logarithmically (once the dyadic localization is removed). However, the $\operatorname{sine}$-function in the Duhamel integral yields additional cancellation, which was first observed by Gubinelli, Koch, and Oh in \cite{GKO18a} and generalized slightly in Lemma \ref{tools:lem_sin_cancellation}.   \\

Using the inhomogeneous Strichartz estimate (Lemma \ref{tools:lem_inhomogeneous_strichartz}), it follows that 
\begin{align*}
\| \cG^{(1)}(t,x;N_\ast) \|_{\X{s_2-1}{b_--1}([0,T])} &\lesssim \| \cG^{(1)}(t,x;N_\ast) \|_{L_t^{2b_+} H_x^{s_2-1}([0,T]\times \bT^3)} \\
&\lesssim T^{\frac{1}{2}} \| \cG^{(1)}(t,x;N_\ast) \|_{L_t^{2} H_x^{s_2-1}([0,T]\times \bT^3)}. 
\end{align*}
Using Gaussian hypercontractivity (Proposition \ref{tools:prop_Gaussian_hypercontractivity}) and the orthogonality of multiple stochastic integrals, we obtain that 
\begin{equation}\label{so5i:eq_p5}
\begin{aligned}
&T \, \Big\| \| \cG^{(1)}(t,x;N_\ast) \|_{L_t^{2} H_x^{s_2-1}([0,T]\times \bT^3)} \Big\|_{L^p_\omega}^2 \\
&\lesssim T p \, \Big\| \| \cG^{(1)}(t,x;N_\ast) \|_{L_t^{2} H_x^{s_2-1}([0,T]\times \bT^3)} \Big\|_{L^2_\omega}^2 \\
&\lesssim T^2 p \sup_{t\in [0,T]} \sum_{n_5\in \bZ^3}   \chi_{N_5}(n_5) \langle n_5 \rangle^{2(s_2-1)-2} \mathscr{S}(n_5;t,N_\ast)^2
\end{aligned}
\end{equation}
where
\begin{align*}
\mathscr{S}(n_5;t,N_\ast) 
&\defe \bigg| \sum_{n_3,n_4 \in \bZ^4}  \bigg[ \rho_N^2(n_{345})  \rho_N^2(n_3) \rho_N^2(n_4)\chi_{N_{345}}(n_{345}) \chi_{N_1}(n_3) \chi_{N_3}(n_3) \chi_{N_2}(n_4) \chi_{N_4}(n_4)  \\
&\times  \widehat{V}_S(n_3,n_4,n_5) \widehat{V}(n_{45}) 
 \langle n_{345} \rangle^{-1}\langle n_3 \rangle^{-2} \langle n_4 \rangle^{-2} \\
 &\times \int_0^t \sin\big( (t-t^\prime) \langle n_{345} \rangle\big) \cos\big( (t-t^\prime) \langle n_3 \rangle\big)  \cos\big( (t-t^\prime) \langle n_4 \rangle\big) \exp\big( \pm_5 i t^\prime \langle n_5 \rangle\big) \dtprime \bigg] \bigg|. 
\end{align*}
We now present two different estimates of $\mathscr{S}(n_5;t,N_\ast)$. The first (and main) estimates almost yields control over $\cG^{(1)}$, but exhibits a logarithmic divergence in $N_3$. The second estimates exhibits polynomial growth in $N_4$ and $N_5$, but yields the desired decay in $N_3$. \\

Using that $|\widehat{V}(n_{45})|\lesssim \langle n_{45} \rangle^{-\beta}$ and the crude estimate $|\widehat{V}_S(n_3,n_4,n_5)|\lesssim 1$, we obtain that 
\begin{align*}
\mathscr{S}(n_5;t,N_\ast) 
&\lesssim N_{345}^{-1} N_3^{-2} N_4^{-2} \sum_{n_3,n_4\in \bZ^3} \bigg[ 1 \big\{ |n_3|\sim N_3,\, |n_4|\sim N_4,\, |n_{345}|\sim N_{345} \big\}\langle n_{45} \rangle^{-\beta} \\
&\times \Big| \int_0^t \sin\big((t-t^\prime) \langle n_{345} \rangle\big) \cos\big((t-t^\prime) \langle n_{3} \rangle\big) 
\cos\big((t-t^\prime) \langle n_{4} \rangle\big)  \exp\big( \pm_5 i t^\prime \langle n_5 \rangle\big) \dtprime \Big| \bigg] \\
&\lesssim T \log(2+\max(N_3,N_4,N_5)) N_{345}^{-1} N_3^{-2} N_4^{-2} \\
&\times \max_{\pm_3,\pm_4,\pm_5} \sup_{m\in \bZ}  \sum_{n_3,n_4\in \bZ^3} \bigg[ 1 \big\{ |n_3|\sim N_3,\, |n_4|\sim N_4,\, |n_{345}|\sim N_{345} \big\}\langle n_{45} \rangle^{-\beta} 1\big\{ |\varphi-m|\leq 1\big\}\bigg], 
\end{align*}
where the phase-function $\varphi$ is given by
\begin{equation*}
\varphi(n_3,n_4,n_5) \defe \langle n_{345} \rangle \pm_3 \langle n_3 \rangle \pm_4 \langle n_4 \rangle \pm_5 \langle n_5 \rangle.  
\end{equation*}
Using the counting estimate from Lemma \ref{tools:lem_counting_double_resonance_quintic}, it follows that
\begin{equation}\label{so5i:eq_p3}
\mathscr{S}(n_5;t,N_\ast) \lesssim T \log(2+\max(N_3,N_4,N_5))  \max(N_4,N_5)^{-\beta+\eta}. 
\end{equation}

Alternatively, it follows from the $\operatorname{sine}$-cancellation lemma (Lemma \ref{tools:lem_sin_cancellation}) with $A=N_4^2 N_5^2$, say, that 
\begin{equation}\label{so5i:eq_p4}
\mathscr{S}(n_5;t,N_\ast) \lesssim T^2   N_3^{-1} N_4^5 N_5^2. 
\end{equation}
By combining \eqref{so5i:eq_p5}, \eqref{so5i:eq_p3}, and \eqref{so5i:eq_p4}, it follows that 
\begin{align*}
&T^{\frac{1}{2}} \, \Big\| \| \cG^{(1)}(t,x;N_\ast) \|_{L_t^{2} H_x^{s_2-1}([0,T]\times \bT^3)} \Big\|_{L^p_\omega}  \\
&\lesssim T^3 p^{\frac{1}{2}}  \log(2+\max(N_3,N_4,N_5))  N_5^{s_2-\frac{1}{2}}
\min\Big( N_4^{-\beta}, N_5^{-\beta}, N_3^{-1} N_4^{5} N_5^5\Big) \\
&\lesssim  T^3 p^{\frac{1}{2}} N_5^{s_2-\frac{1}{2}-\beta+20\eta} \max(N_3,N_4,N_5)^{-\eta}\\
&\lesssim  T^3 p^{\frac{1}{2}} \max(N_3,N_4,N_5)^{-\eta}. 
\end{align*}
This contribution is acceptable.

\end{proof}

\begin{proof}[Proof of Proposition \ref{so5ii:prop}:]
This estimate is similar (but easier) than Proposition \ref{so5i:prop} and we therefore omit the details. Instead of gaining additional regularity through the para-differential operator as in Proposition \ref{so5ii:prop}, we simply use interaction potential $V$ and the crude inequality 
\begin{equation*}
\langle n_{12} \rangle^{-2\beta} \lesssim \langle n_{12} \rangle^{-2\gamma} \lesssim \langle n_{12345} \rangle^{-2\gamma} \langle n_{345} \rangle^{2\gamma}
\end{equation*}
for $0<\gamma<\beta$. 
\end{proof}

\subsection{Septic stochastic objects}\label{section:septic}

The next proposition controls the third and fourth term in $\So$, i.e., in \eqref{lwp:eq_so}. 

\begin{proposition}[Septic stochastic objects]\label{so7:prop}
Let $T\geq 1$ and $p \geq 1$. Then, it holds that 
\begin{align}
\bigg\| \sup_{N\geq 1}  \bigg\| \, \<313N> \, \bigg\|_{\X{s_2-1}{b_+-1}([0,T])}    \bigg\|_{L^p_\omega(\bP)}
&\lesssim T^4 p^{7/2}, \label{so7:eq_1} \\
\bigg\| \sup_{N\geq 1}  \bigg\| \,  \nparald \<331N>\, \bigg\|_{\X{s_2-1}{b_+-1}([0,T])}    \bigg\|_{L^p_\omega(\bP)}
&\lesssim T^4 p^{7/2}. \label{so7:eq_2} 
\end{align}
\end{proposition}

\begin{remark}
In the frequency-localized version of Proposition \ref{so7:prop},  we gain an $\eta^\prime$-power of the maximal frequency-scale. As in Proposition \ref{so3:prop}, we may also replace $\<3DN>$ by  $\<3DNtau>= \Duh \big[ 1_{[0,\tau]} \<3N>\big]$. We will not further comment on these minor modifications.
\end{remark}

\begin{proof}
We only prove \eqref{so7:eq_1}. The second estimate \eqref{so7:eq_2} follows from similar (but slightly simpler) arguments. To simplify the notation, we formally set $N=\infty$. The same argument also yields the estimate for the supremum over $N$. Using the inhomogeneous Strichartz estimate (Lemma \ref{tools:lem_inhomogeneous_strichartz}) and Gaussian hypercontractivity (Proposition \ref{tools:prop_Gaussian_hypercontractivity}), it suffices to prove that 
\begin{equation}\label{so7:eq_p1}
 \sup_{t\in [0,T]} \bigg\| \bigg\| \, \<313> \, \bigg\|_{H_x^{s_2-1}(\bT^3)}    \bigg\|_{L^2_\omega(\bP)} \lesssim T^3. 
\end{equation}
Using a Littlewood-Paley decomposition, we write 
\begin{equation*}
\<313> = \sum_{N_{1234567}, N_{1234},N_4,N_{567}} \<313>[N_{1234567}, N_{1234},N_4,N_{567}], 
\end{equation*}
where 
\begin{equation}\label{so7:eq_p2}
\<313>[N_{1234567}, N_{1234},N_4,N_{567}] 
\defe P_{N_{1234567}} \bigg[ (P_{N_{1234}} \widehat{V}) \ast \Big( \<3D> \cdot P_{N_4} \<1b> \Big) \, P_{N_{567}} \<3D>\bigg]. 
\end{equation}
We now present two separate estimates of \eqref{so7:eq_p2}. The first estimate, which is the main part of the argument, almost yields \eqref{so7:eq_p1}, but contains a logarithmic divergence in $N_4$. The second (short) estimate exhibits polynomial decay in $N_4$, and is only used to remove this logarithmic divergence. \\

\emph{Main estimate:} Using the stochastic representation of the cubic nonlinearity (Proposition \ref{tools:prop_stochastic_representation}) and \eqref{tools:eq_normalized_integrals}, we obtain that 
\begin{equation}\label{so7:eq_p3}
\begin{aligned}
&\, \<313>[N_{1234567}, N_{1234},N_4,N_{567}] \\
&= \sum_{n_1,\hdots,n_7\in \bZ^3} \sum_{\pm_1,\hdots,\pm_7} \bigg[ \chi_{N_{1234567}}(n_{1234567}) \chi_{N_{1234}}(n_{1234}) \chi_{N_4}(n_4)  \chi_{N_{567}}(n_{567})  \widehat{V}(n_{1234}) \\
& \times \Phi(t,n_j, \pm_j\colon 1 \leq j \leq 3) e^{\pm i t \langle n_4 \rangle} \frac{1}{\langle n_4 \rangle} \Phi(t,n_j, \pm_j\colon 5 \leq j \leq 7) \exp\big({i \langle n_{1234567} , x \rangle}\big) \\
&\times \widetilde{\cI}_{3}[n_j,\pm_j\colon 1\leq j \leq 3] \widetilde{\cI}_{1}[n_4,\pm_4] \widetilde{\cI}_{3}[n_j,\pm_j\colon 5\leq j \leq 7] \bigg]. 
\end{aligned}
\end{equation}
Here, the amplitude $\Phi$ is given by 
\begin{align*}
 &\Phi(t,n_j, \pm_j\colon 1 \leq j \leq 3)  \\
 &\defe \langle n_{123}  \rangle^{-1} \widehat{V}_S(n_1,n_2,n_3)  \Big( \prod_{j=1}^3 \langle n_j \rangle^{-1} \Big) \Big( \int_0^t \sin\big((t-t^\prime) \langle n_{123} \rangle \prod_{j=1}^3 \exp\big( \pm_j i t^\prime \langle n_j \rangle\big) \dtprime \Big). 
\end{align*}
Comparing with $\Phi(n_1,n_2,n_3)$ as in Lemma \ref{tools:lem_counting_septic}, we have that
\begin{equation}\label{so7:eq_p4}
\sup_{t\in [0,T]} |\Phi(t,n_j, \pm_j\colon 1 \leq j \leq 3)| \lesssim T \Phi(n_1,n_2,n_3). 
\end{equation}
We now rely on the notation from Definition \ref{tools:def_pairings} and Lemma \ref{tools:lem_counting_septic}. Using the product formula for multiple stochastic integrals twice (Lemma \ref{tools:lem_product}), the orthogonality of multiple stochastic integrals, and \eqref{so7:eq_p4}, we obtain that 
\begin{align*}
 &\sup_{t\in [0,T]} \bigg\| \, \<313>[N_{1234567}, N_{1234},N_4,N_{567}]  \, \bigg\|_{L_\omega^2 H_x^{s_2-1}(\Omega \times \bT^3)}^2 \\
&\lesssim T^4 \sum_{\scrP} \sum_{(n_j)_{j\not \in \scrP}} \hspace{-1ex} \langle n_{\text{nr}} \rangle^{2(s_2-1)}
 \bigg(  \hspace{-0.2ex} \sum_{\substack{(n_j)_{j \in \scrP}}}^{\ast}  \hspace{-1ex}1\big\{ |n_{1234567}|\sim N_{1234567}\big\} 1\big\{ |n_{1234}|\sim N_{1234}\big\} 1\big\{|n_{567}| \sim N_{567} \big\}  \\ 
  &\times 1\big\{|n_4|\sim N_4\big\} |\widehat{V}(n_{1234})| \Phi(n_1,n_2,n_3) \langle n_4 \rangle^{-1} \Phi(n_5,n_6,n_7) \bigg)^2. \\
\end{align*}
The sum in $\scrP$ is taken over all pairings  which respect the partition $\{1,2,3\},\{4\},\{5,6,7\}$. For a similar argument, we refer the reader to \cite[Lemma 4.1]{DNY19}. Using Lemma \ref{tools:lem_counting_septic}, it follows that 
\begin{equation}
\begin{aligned}
 &\sup_{t\in [0,T]} \bigg\| \, \<313>[N_{1234567}, N_{1234},N_4,N_{567}]  \, \bigg\|_{L_\omega^2 H_x^{s_2-1}(\Omega \times \bT^3)} \\
 &\lesssim T^2  \log(2+N_4)  \Big( N_{1234567}^{(s_2-\frac{1}{2})}   N_{567}^{-(\beta-\eta)} + N_{1234567}^{- (1-s_2-\eta)} \Big) N_{1234}^{-\beta}. 
 \end{aligned}
 \end{equation}
 Since $N_{1234567}\lesssim \max(N_{1234},N_{567})$ and $N_{1234567} \sim N_{567}$ if $N_{1234}\ll N_{567}$, we obtain that 
 \begin{equation}\label{so7:eq_p5}
\begin{aligned}
 &\sup_{t\in [0,T]} \bigg\| \, \<313>[N_{1234567}, N_{1234},N_4,N_{567}]  \, \bigg\|_{L_\omega^2 H_x^{s_2-1}(\Omega \times \bT^3)} \\
 &\lesssim T^2 \log(2+N_4) \max(N_{1234567},N_{1234},N_{567})^{-(\beta-\eta-\delta_2)}. 
\end{aligned}
\end{equation}

\emph{Removing the logarithmic divergence in $N_4$:} 
Using Proposition \ref{so3:prop} and \eqref{so4:eq_estimate_3} from Proposition \ref{so4:prop}, we obtain that 

\begin{equation}\label{so7:eq_p6}
\begin{aligned}
 &\sup_{t\in [0,T]} \bigg\| \, \<313>[N_{1234567}, N_{1234},N_4,N_{567}]  \, \bigg\|_{L_\omega^2 H_x^{s_2-1}(\Omega \times \bT^3)} \\
 &\lesssim \Big\| P_{N_{1234}} \Big[ \<3D> \cdot P_{N_4} \<1b> \Big]  \Big\|_{L^4_\omega L_t^\infty L_x^2(\Omega \times [0,T] \times \bT^3)}  \Big\| P_{N_{567}} \<3D> \Big\|_{L_\omega^4 L_t^\infty L_x^\infty(\Omega\times [0,T]\times \bT^3)} \\
 &\lesssim T^5 N_{1234} N_4^{-\frac{\eta}{10}}. 
 \end{aligned}
\end{equation}
By combing \eqref{so7:eq_p5} and \eqref{so7:eq_p6}, we obtain that 
 \begin{equation}\label{so7:eq_p7}
\begin{aligned}
 &\sup_{t\in [0,T]} \bigg\| \, \<313>[N_{1234567}, N_{1234},N_4,N_{567}]  \, \bigg\|_{L_\omega^2 H_x^{s_2-1}(\Omega \times \bT^3)} \\
 &\lesssim T^3  N_{4}^{-\eta^2} \max(N_{1234567},N_{1234},N_{567})^{-(\beta-\eta-2\delta_2)}. 
 \end{aligned}
 \end{equation}
 After summing over the dyadic scales, this yields \eqref{so7:eq_1}. 
\end{proof}

\section{Random matrix theory estimates}\label{section:rmt}

In this section, we control the random matrix terms $\RMT$. 
Techniques from random matrix theory, such as the moment method, were first applied to dispersive equations in Bourgain's seminal paper \cite{Bourgain96}. Over the last decade, they have become an indispensable tool in the study of dispersive PDE and we refer the interested reader to \cite{Bourgain97,CG19,DH19,DNY19,FOSW19,GKO18a,Richards16}. Very recently, Deng, Nahmod, and Yue \cite[Proposition 2.8]{DNY20} obtained an easy-to-use, general, and essentially sharp random matrix estimate, which is proved using the moment method. We have previously recalled their estimate in Proposition \ref{tools:prop_moment_method}. The proofs of Proposition \ref{rmt:prop1} and Proposition \ref{rmt:prop2} combine their random matrix estimate with the counting estimates in Section \ref{section:counting}. 

\begin{proposition}[First RMT estimate]\label{rmt:prop1}
Let $T\geq 1$ and let $p\geq 1$. Then, it holds that 
\begin{equation}\label{rmt:eq_rmt1}
\Big\| \sup_{N\geq 1} \sup_{\cJ\subseteq [0,T]}  \sup_{\| w \|_{\X{s_1}{b}(\cJ)}\leq1} \Big \| 
(V\ast \<2N>) \cdot  P_{\leq N}w 
\Big\|_{\X{s_2-1}{b_+-1}(\cJ)} \Big\|_{L^p_\omega(\bP)} \lesssim Tp. 
\end{equation}
\end{proposition}

\begin{remark}
This proposition controls the first term in $\RMT$, i.e., in \eqref{lwp:eq_RMT}. In the frequency-localized version of \eqref{rmt:eq_rmt1}, which is detailed in the proof, we gain an $\eta^\prime$-power in the maximal frequency-scale. 
\end{remark}

\begin{proof} The arguments splits into two steps: First, we bring \eqref{rmt:eq_rmt1} into a random matrix form. Then, we prove a random matrix estimate using the moment method (Proposition \ref{tools:prop_moment_method}).  \\

\emph{Step 1: The random matrix form.} 
By definition of the restricted norms, it holds that 
\begin{equation}\label{rmt:eq_rmt2_p1}
\begin{aligned}
&\sup_{\cJ\subseteq [0,T]}  \sup_{\| w \|_{\X{s_1}{b}(\cJ)}\leq1} 
\Big \|  (V\ast \<2N>) \cdot  P_{\leq N}w \Big\|_{\X{s_2-1}{b_+-1}(\cJ)} \\
&\leq   \sup_{\| w \|_{\X{s_1}{b}(\bR)}\leq1} 
\Big \| \chi(t/T) (V\ast \<2N>) \cdot  P_{\leq N}w \Big\|_{\X{s_2-1}{b_+-1}(\bR)}. 
\end{aligned}
\end{equation}
We bound the right-hand side of \eqref{rmt:eq_rmt2_p1} with $b_+$ replaced by $b_-$. Using the frequency-localized estimate in the arguments below and a similar reduction as in the proof of Proposition \ref{so3:prop}, we can then upgrade the value from $b_-$ to $b_+$. Let $w\in \X{s_1}{b}(\bR)$ satisfy 
$\| w \|_{\X{s_1}{b}(\bR)}\leq 1 $. We define $w_{\pm}\in \X{s_1}{b}(\bR)$ by
\begin{equation*}
\widehat{w}_{\pm}(\lambda,n) \defe 1 \big \{ \pm \lambda \geq 0 \big\} \widehat{w}(\lambda,n). 
\end{equation*} 
Then, it holds that $w=w_+ + w_-$ and 
\begin{equation*}
\| w \|_{\X{s_1}{b}(\bR)} \sim \max_{\pm} \| \langle n \rangle^{s_1} \langle \lambda \rangle^{b} \widehat{w}_{\pm}(\lambda \pm \langle n \rangle, n) \|_{L_\lambda^2 \ell_n^2(\bR \times \bT^3)}. 
\end{equation*}
Using this decomposition of $w$ and the stochastic representation of the renormalized square, we  obtain that the nonlinearity is given by
\begin{align*}
&(V\ast \<2N>) \cdot  P_{\leq N}w  \\
&= \sum_{\pm_1,\pm_2,\pm_3} \sum_{N_1,N_2,N_3} \sum_{n_1,n_2,n_3 \in \bZ^3} \bigg[c(\pm_1,\pm_2) \Big( \prod_{j=1}^3 \rho_N(n_j) \chi_{N_j}(n_j) \Big) 
\widehat{V}(n_{12}) \cI_2[ \pm_j, n_j \colon j=1,2] \\
&\times \Big( \prod_{j=1}^2 \exp\big(\pm_j i t \langle n_j \rangle\big) \Big) \widehat{w}_{\pm_3}(t,n_3) \exp\big(i \langle n_{123}, x \rangle \big) \bigg] \allowdisplaybreaks[3] \\
&=  \sum_{\pm_1,\pm_2,\pm_3} \sum_{N_1,N_2,N_3}
 \int_{\bR} \dlambda_3 \sum_{n_1,n_2,n_3 \in \bZ^3} \bigg[c(\pm_1,\pm_2) \Big( \prod_{j=1}^3 \rho_N(n_j) \chi_{N_j}(n_j) \Big) 
\widehat{V}(n_{12}) \cI_2[ \pm_j, n_j \colon j=1,2] \\
&\times  \exp\big( it \lambda_3 \big) \Big( \prod_{j=1}^3 \exp\big( \pm_j  i t  \langle n_j \rangle \big) \Big) \widehat{w}_{\pm_3}(\lambda_3\pm_3 \langle n_3 \rangle,n_3) \exp\big(i \langle n_{123}, x \rangle \big) \bigg].
\end{align*}
To simplify the notation, we define the phase-function $\varphi\colon (\bZ^3)^3 \rightarrow \bR$  by 
\begin{equation}
\varphi(n_1,n_2,n_3) = \pm_{123} \langle n_{123} \rangle  \pm_1 \langle n_1 \rangle \pm_2 \langle n_2 \rangle \pm_3 \langle n_3 \rangle.  
\end{equation}
The space-time Fourier transform of the time-truncated nonlinearity is therefore given by
\begin{equation}\label{rmt:eq_rmt2_p4}
\begin{aligned}
&\mathcal{F}\Big( \chi(\cdot/T) (V\ast \<2N>) \cdot  P_{\leq N}w\Big)(\lambda\pm_{123} \langle n \rangle, n)  \\
&=  T \sum_{\pm_1,\pm_2,\pm_3} \sum_{N_1,N_2,N_3}
 \int_{\bR} \dlambda_3 \sum_{n_1,n_2,n_3 \in \bZ^3} \bigg[c(\pm_1,\pm_2)  1 \big\{ n=n_{123} \big\} 
\widehat{\chi}\big( T ( \lambda - \lambda_3  - \varphi(n_1,n_2,n_3)) \big) \\ 
&\times \Big( \prod_{j=1}^3 \rho_N(n_j) \chi_{N_j}(n_j) \Big)  \widehat{V}(n_{12}) \cI_2[ \pm_j, n_j \colon j=1,2]   \widehat{w}_{\pm_3}(\lambda_3\pm_3 \langle n_3 \rangle,n_3) \bigg]. 
\end{aligned}
\end{equation}
To simplify the following notation, we emphasize the dependence on the frequency-scales $N_1,N_2,N_3$ by writing $N_\ast$ and omit the dependence on $\pm_{123},\pm_1,\pm_2,\pm_3$, and $T$ from our notation. We define the tensor $h(n,n_1,n_2,n_3;\lambda,\lambda_3,N_\ast)$ by 
\begin{equation}\label{rmt:eq_rmt2_p5}
\begin{aligned}
h(n,n_1,n_2,n_3;\lambda,\lambda_3,N_\ast) \defe&
T c(\pm_1,\pm_2)  1 \big\{ n=n_{123} \big\} 
\widehat{\chi}\big( T ( \lambda - \lambda_3 - \varphi(n_1,n_2,n_3)) \big) \\
&\times \Big( \prod_{j=1}^3 \rho_N(n_j) \chi_{N_j}(n_j) \Big)  \widehat{V}(n_{12}) \langle n \rangle^{s_2-1} \langle n_1 \rangle^{-1} \langle n_2 \rangle^{-1} \langle n_3 \rangle^{-s_1}. 
\end{aligned}
\end{equation}
Furthermore, we define the contracted random tensor $h_c(n,n_3;\lambda,\lambda_3)$ by 
\begin{equation}\label{rmt:eq_rmt2_p3}
h_c(n,n_3;\lambda,\lambda_3,N_\ast) = \sum_{n_1,n_2 \in \bZ^3} h(n,n_1,n_2,n_3;\lambda,\lambda_3,N_\ast) \cdot \widetilde{\cI}_2[\pm_j, n_j \colon j=1,2]. 
\end{equation}
By combining our previous expression of the nonlinearity \eqref{rmt:eq_rmt2_p4} with the definition \eqref{rmt:eq_rmt2_p3}, we obtain that 
\begin{equation*}
\begin{aligned}
&\mathcal{F}\Big( \chi(\cdot/T) (V\ast \<2N>) \cdot  P_{\leq N}w\Big)(\lambda\pm_{123} \langle n \rangle, n)  \\
&= \langle n \rangle^{-(s_2-1)} \sum_{\pm_1,\pm_2,\pm_3} \sum_{N_1,N_2,N_3} \int_{\bR} \dlambda_3  \sum_{n_3\in \bZ^3} h_c(n,n_3;\lambda,\lambda_3,N_\ast) \langle n_3 \rangle^{s_1} \widehat{w}_{\pm_3}(\lambda_3\pm_3 \langle n_3 \rangle,n_3). 
\end{aligned}
\end{equation*}
We estimate each combination of signs and each dyadic block separately. Using the tensor norms from Definition \ref{tool:def_tensor}, the contribution to the $\X{s_2-1}{b_--1}$-norm is bounded by 
\begin{equation*}
\begin{aligned}
& \Big\|  \langle \lambda \rangle^{b_- -1}  \int_{\bR} \dlambda_3  \sum_{n_3\in \bZ^3} h_c(n,n_3;\lambda,\lambda_3,N_\ast) \langle n_3 \rangle^{s_1} \widehat{w}_{\pm_3}(\lambda_3\pm_3 \langle n_3 \rangle,n_3) \Big\|_{L_\lambda^2 \ell_n^2(\bR \times \bT^3)} \\
& \lesssim
 \Big\|   \langle \lambda \rangle^{b_- -1}  \langle \lambda_3 \rangle^{-b}  \|  h_c(n,n_3;\lambda,\lambda_3,N_\ast)\|_{n_3 \rightarrow n} \Big \|_{L_\lambda^2 L_{\lambda_3}^2(\bR\times \bR)}
\cdot \| w \|_{\X{s_1}{b}(\bR)}. 
\end{aligned}
\end{equation*}
In order to control the operator norm in \eqref{rmt:eq_rmt2_p1}, it therefore remains to prove that 
\begin{equation}\label{rmt:eq_rmt2_p2}
\Big \| \Big\|   \langle \lambda \rangle^{b_- -1}  \langle \lambda_3 \rangle^{-b}  \|  h_c(n,n_3;\lambda,\lambda_3,N_\ast)\|_{n_3 \rightarrow n} \Big\|_{L_\lambda^2 L_{\lambda_3}^2(\bR\times \bR)} \Big\|_{L^p_\omega(\bP)} \lesssim T \max(N_1,N_2,N_3)^{-\frac{\eta}{2}} p. 
\end{equation}

\emph{Step 2: Proof of the random matrix estimate \eqref{rmt:eq_rmt2_p2}.} 
Using Minkowski's integral inequality, we have that 
\begin{align*}
&\Big \| \Big\|   \langle \lambda \rangle^{b_- -1}  \langle \lambda_3 \rangle^{-b}  \|  h_c(n,n_3;\lambda,\lambda_3,N_\ast)\|_{n_3 \rightarrow n} \Big\|_{L_\lambda^2 L_{\lambda_3}^2(\bR\times \bR)} \Big\|_{L^p_\omega(\bP)} \\
&\leq \Big\|   \langle \lambda \rangle^{b_- -1}  \langle \lambda_3 \rangle^{-b} \Big \|  \|  h_c(n,n_3;\lambda,\lambda_3,N_\ast)\|_{n_3 \rightarrow n}   \Big\|_{L^p_\omega(\bP)} \Big\|_{L_\lambda^2 L_{\lambda_3}^2(\bR\times \bR)} \\
&\leq \Big\|   \langle \lambda \rangle^{b_- -1}  \langle \lambda_3 \rangle^{-b}  \Big\|_{L_\lambda^2 L_{\lambda_3}^2(\bR\times \bR)} 
\cdot \sup_{\lambda,\lambda_3 \in\bR} \Big \|  \|  h_c(n,n_3;\lambda,\lambda_3,N_\ast)\|_{n_3 \rightarrow n}   \Big\|_{L^p_\omega(\bP)} \\
&\lesssim  \sup_{\lambda,\lambda_3 \in\bR} \Big \|  \|  h_c(n,n_3;\lambda,\lambda_3,N_\ast)\|_{n_3 \rightarrow n}   \Big\|_{L^p_\omega(\bP)}.
\end{align*}
We emphasize that the supremum over $\lambda,\lambda_3 \in \bR$ is outside of the $L^p_\omega(\bP)$-norm. Using the moment method (Proposition \ref{tools:prop_moment_method}), it holds that 
\begin{align*}
&\sup_{\lambda,\lambda_3 \in\bR} \Big \|  \|  h_c(n,n_3;\lambda,\lambda_3,N_\ast)\|_{n_3 \rightarrow n}   \Big\|_{L^p_\omega(\bP)} \\
&\lesssim \max(N_1,N_2,N_3)^{\frac{\eta}{2}}  \sup_{\lambda,\lambda_3 \in \bR} 
\max\Big( \| h(\cdot;\lambda,\lambda_3,N_\ast)\|_{n_1n_2n_3\rightarrow n}, \| h(\cdot;\lambda,\lambda_3,N_\ast) \|_{n_3 \rightarrow nn_1 n_2}, \\
&\qquad\|h(\cdot;\lambda,\lambda_3,N_\ast) \|_{n_1 n_3\rightarrow nn_2} , \| h(\cdot;\lambda,\lambda_3,N_\ast) \|_{n_2n_3\rightarrow nn_1} \Big) p. 
\end{align*}
In order to estimate the tensor norms of $h(\cdot;\lambda,\lambda_3,N_\ast)$, we further decompose it according to the value of the phase-function $\varphi$. For any $m \in \bZ$, we define
\begin{equation*}
\begin{aligned}
\widetilde{h}(n,n_1,n_2,n_3;m,N_\ast) \defe&
T   1 \big\{ n=n_{123} \big\} 
1\big\{ |\varphi(n_1,n_2,n_3)) -m | \leq 1 \big\} \Big( \prod_{j=1}^3 \rho_N(n_j) \chi_{N_j}(n_j) \Big)  \\
&\times  |\widehat{V}(n_{12})| \langle n \rangle^{s_2-1} \langle n_1 \rangle^{-1} \langle n_2 \rangle^{-1} \langle n_3 \rangle^{-s_1}. 
\end{aligned}
\end{equation*}
Using the definition of $h$ in \eqref{rmt:eq_rmt2_p5} and the decay of $\widehat{\chi}$, we obtain that 
\begin{align*}
&|h(n,n_1,n_2,n_3;\lambda,\lambda_3,N_\ast) | \\
&\lesssim
 \sum_{m\in \bZ} | \widehat{\chi}\big( T (\lambda-\lambda_3 - \varphi(n_1,n_2,n_3) ) \big)| \, 1\big\{ | \varphi(n_1,n_2,n_3)-m| \leq 1\big\} \, \widetilde{h}(n,n_1,n_2,n_3;m,N_\ast) \\
 &\lesssim \sum_{m\in \bZ} \langle \lambda_3 - \lambda - m\rangle^{-2} \, \widetilde{h}(n,n_1,n_2,n_3;m,N_\ast). 
\end{align*}
Using the triangle inequality for the tensor norms and the first deterministic tensor estimate (Lemma \ref{tools:lem_first_tensor}), it follows that 
\begin{align*}
 &\max(N_1,N_2,N_3)^{\frac{\eta}{2}}  \sup_{\lambda,\lambda_3 \in \bR} 
\max\Big( \| h(\cdot;\lambda,\lambda_3,N_\ast)\|_{n_1n_2n_3\rightarrow n}, \| h(\cdot;\lambda,\lambda_3,N_\ast) \|_{n_3 \rightarrow nn_1 n_2}, \\
&\quad\|h(\cdot;\lambda,\lambda_3,N_\ast) \|_{n_1 n_3\rightarrow nn_2} , \| h(\cdot;\lambda,\lambda_3,N_\ast) \|_{n_2n_3\rightarrow nn_1} \Big) \\
&\lesssim \max(N_1,N_2,N_3)^{\frac{\eta}{2}}  \sup_{m \in \bZ} 
\max\Big( \| \widetilde{h}(\cdot;m,N_\ast) \|_{n_1n_2n_3\rightarrow n}, \| \widetilde{h}(\cdot;m,N_\ast)\|_{n_3 \rightarrow nn_1 n_2}, \\
&\quad\| \widetilde{h}(\cdot;m,N_\ast)\|_{n_1 n_3\rightarrow nn_2} , \| \widetilde{h}(\cdot;m,N_\ast) \|_{n_2n_3\rightarrow nn_1} \Big) \\
&\lesssim T \max(N_1,N_2,N_3)^{-\frac{\eta}{2}}. 
\end{align*}

\end{proof}

\begin{proposition}[Second RMT estimate]\label{rmt:prop2}
Let $T\geq 1$ and let $p\geq 1$. Then, it holds that 
\begin{equation}\label{rmt:eq_rmt2}
\Big\| \sup_{N\geq 1}  \sup_{\cJ\subseteq [0,T]} \sup_{\| Y \|_{\X{s_2}{b}(\cJ)}\leq1} \Big \| \lcol 
V \ast \big( P_{\leq N} \<1b> \cdot P_{\leq N} Y \big) \nparald P_{\leq N} \<1b> \rcol 
\Big\|_{\X{s_2-1}{b_+-1}(\cJ)} \Big\|_{L^p_\omega(\bP)} \lesssim Tp. 
\end{equation}
\end{proposition}

\begin{remark}
This proposition controls the second term in $\RMT$, i.e., in \eqref{lwp:eq_RMT}. In the frequency-localized version of \eqref{rmt:eq_rmt2}, which is detailed in the proof, we gain an $\eta^\prime$-power in the maximal frequency-scale. 
\end{remark}

\begin{proof}
Due to the operator $\nparald$, the renormalization $\MN P_{\leq N} Y$ does not just cancel the probabilistic resonances between the two factors of $\,\<1b>\,$ in 
\begin{equation*}
V \ast \big( P_{\leq N} \<1b> \cdot P_{\leq N} Y \big) \nparald P_{\leq N} \<1b>.
\end{equation*}
As a result, we need to decompose $\MN =\MNparald + \MNnparald$, where the symbols corresponding to the multipliers are given by
\begin{align*}
\mNparald(n) &\defe \sum_{\substack{L,K\colon L \leq K^\epsilon}} \frac{\widehat{V}(n+k)}{\langle k \rangle^2} \chi_L(n+k) \chi_K(k) \rho_N^2(k), \\
\mNnparald(n) &\defe \sum_{\substack{L,K\colon L > K^\epsilon}} \frac{\widehat{V}(n+k)}{\langle k \rangle^2} \chi_L(n+k) \chi_K(k) \rho_N^2(k). 
\end{align*}
The random operator 
\begin{equation*}
V \ast \big( P_{\leq N} \<1b> \cdot P_{\leq N} Y \big) \nparald P_{\leq N} \<1b> - \MNnparald P_{\leq N} Y
\end{equation*}
can then be controlled using the same argument as in the proof of Proposition \ref{rmt:prop1}, except that we use Lemma \ref{tools:lem_second_tensor} instead of Lemma \ref{tools:lem_first_tensor}. Thus, it only remains to show that 
\begin{equation}\label{rmt:eq_second_p1}
\| \MNparald P_{\leq N} Y \|_{\X{s_2-1}{b_+-1}(\cJ)} \lesssim  T\| Y \|_{\X{s_2}{b}(\cJ)}. 
\end{equation}
The estimate \eqref{rmt:eq_second_p1} has a lot of room and can be established through the following simple argument. On the support of the summand in the definition of $\mNparald$, it holds that $|n+k| \lesssim |k|^\epsilon$. Using only that $\widehat{V}$ is bounded, this implies that 
\begin{equation*}
|\mNparald(n)| \lesssim \sum_{K\geq 1} \sum_{k \in \bZ^3} K^{-2} 1\big\{ |n+k| \lesssim K^\epsilon \big\} \lesssim \sum_{K\geq 1} K^{-2+3\epsilon} \lesssim 1. 
\end{equation*}
Thus, the symbol $\mNparald(n)$ is uniformly bounded and hence the corresponding multiplier $\MNparald$ is bounded on each Sobolev-space $H_x^s(\bT^3)$. Using the Strichartz estimates (Corollary  \ref{tools:cor_strichartz} and Lemma \ref{tools:lem_inhomogeneous_strichartz}), we obtain that 
\begin{gather*}
\| \MNparald P_{\leq N} Y \|_{\X{s_2-1}{b_+-1}(\cJ)}  \lesssim \| \MNparald P_{\leq N} Y \|_{L_t^{2b_+} H_x^{s_2-1}(\cJ \times \bT^3)} \\
\lesssim (1+|\cJ|)  \| Y \|_{L_t^{\infty} H_x^{s_2-1}(\cJ \times \bT^3)} 
\lesssim  (1+|\cJ|) \| Y \|_{\X{s_2}{b}(\cJ)}. 
\end{gather*}
\end{proof}

\section{Para-controlled estimates}\label{section:paracontrolled}

The main goal of this section is to estimate the terms in $\CPara$. We remind the reader that the para-controlled approach to stochastic partial differential equations was introduced in the seminal paper of Gubinelli, Imkeller, and Perkowski \cite{GIP15} and first applied to dispersive equations by Gubinelli, Koch, and Oh in \cite{GKO18a}. \\

The following definitions of the low-frequency modulation space $\modulation$ and the para-controlled structure $\Para$ are following similar ideas as the framework in \cite{GKO18a}. 

\begin{definition}[Low-frequency modulation space]\label{para:def_LM}
Let $H= \{ H(t,x;K)\}_{K\geq 1}\}$ be a family of space-time functions from $\bR \times \bT^3$ into $\bC$ satisfying
\begin{equation}\label{para:eq_fourier_support}
\supp( \widehat{H}(t,x;K) ) \subseteq \{ k \in \bZ^3 \colon |k| \leq 8 K^\epsilon \}.
\end{equation}
We define the low-frequency modulation norm by 
\begin{equation}\label{para:eq_modulation_norm}
\| H \|_{\modulation(\bR)} \defe \sup_{K\geq 1} K^{-4\epsilon} \| \widehat{H}(\lambda,k;K) \|_{\ell_k^\infty L_\lambda^1(\bZ^3 \times \bR)} .
\end{equation}
We define the corresponding low-frequency modulation space $\modulation(\bR)$ by 
\begin{equation}\label{para:eq_modulation_space}
\modulation(\bR) = \big\{ H \colon \| H \|_{\modulation(\bR)} < \infty \big\}. 
\end{equation}
Furthermore, let $\cJ \subseteq \bR$ be a time-interval and let $H= \{ H(t,x;K)\}_{K\geq 1}\}$ be a family of space-time functions from $\cJ \times \bT^3$ into $\bR$ satisfying \eqref{para:eq_fourier_support}. Similar as in the definition of $\X{s}{b}$-spaces, we define the restricted norm by 
\begin{equation}
\| H \|_{\modulation(\cJ)} = \inf \big \{ \| H^\prime \|_{\modulation(\bR)} \colon H^\prime(t) = H(t) \quad \text{for all } t \in \cJ \big\}. 
\end{equation}
The corresponding time-restricted low-frequency modulation space $\modulation(\cJ)$ can then be defined as in \eqref{para:eq_modulation_norm} after replacing the norm. 
\end{definition}

\begin{definition}[Para-controlled]
Let $\cJ \subseteq \bR$ be an interval, let $\phi\colon \cJ \times \bT^3 \rightarrow  \bC$ be a distribution, and let $H$ be as in Definition \ref{para:def_LM}. Then, we define 
\begin{equation}\label{para:eq_para}
\Para(H,\phi)(t,x) = \sum_{K\geq 1} H(t,x;K) (P_K \phi)(t,x). 
\end{equation}
\end{definition}
If $H \in \modulation(\bR)$, we have that 
\begin{equation}\label{para:eq_para_expansion}
\begin{aligned}
&\Para(H,\phi)(t,x) \\
&= \sum_{K\geq 1} \sum_{k_1 \in \bZ^3} \int_\bR \dlambda_1 \widehat{H}(\lambda_1,k_1;K) 
\Big( \exp\big( i \lambda_1 t  \big) \sum_{k_2 \in \bZ^3} \chi_K(k_2) \widehat{\phi}(t,k_2) \exp\big( i \langle k_{12}, x \rangle\big) \Big). 
\end{aligned}
\end{equation}
The expression \eqref{para:eq_para_expansion} will be used in all of our estimates involving $\Para$. The sum in $k_1$, the integral in $\lambda_1$, and the pre-factor $ \widehat{H}(\lambda_1,k_1;K) $ will be inessential. The main step will consist of estimates for 
\begin{equation*}
 \exp\big( i \lambda_1 t  \big) \sum_{k_2 \in \bZ^3} \chi_K(k_2) \widehat{\phi}(t,k_2) \exp\big( i \langle k_{12}, x \rangle\big) , 
\end{equation*}
which essentially behaves like $P_K \phi(t,x)$. For most purposes, the reader may simply think of $\Para(H,\phi)$ as $\phi$.

\begin{lemma}[Basic mapping properties of $\Para$]\label{para:lem_basic}
For any $s\in \bR$, any interval $\cJ\subseteq \bR$, any $\phi \in L_t^\infty H_x^s(\cJ \times \bT^3)$, and any $H \in \modulation(\cJ)$, we have
\begin{equation}
\| \Para(H,\phi) \|_{L_t^\infty H_x^{s-8\epsilon}(\cJ\times \bT^3)} \lesssim \| H \|_{\LM(\cJ)} \| \phi \|_{L_t^\infty H_x^{s}(\cJ\times \bT^3)} . 
\end{equation}
\end{lemma}

\begin{proof}
We treat each dyadic piece in $\PCtrl$ separately.  Using the Fourier support condition \eqref{para:eq_fourier_support}, we have that 
\begin{align*}
\|  H(t,x;K) (P_K \phi)(t,x) \|_{H_x^{s-8\epsilon}(\bT^3)} 
&= \Big\| \sum_{k_1,k_2 \in \bZ^3} \chi_K(k_2) \widehat{H}(t,k_1;K) \widehat{\phi}(t,k_2) \exp\big( i \langle k_{12}, x \rangle\big) \Big \|_{H_x^{s-8\epsilon}(\bT^3)} \\
&\lesssim \sum_{k_1 \in \bZ^3} | \widehat{H}(t,k_1;K)| 
 \Big\| \sum_{k_2 \in \bZ^3} \chi_K(k_2) \widehat{\phi}(t,k_2) \exp\big( i \langle k_{12}, x \rangle\big) \Big \|_{H_x^{s-8\epsilon}(\bT^3)} \\
 &\lesssim  K^{-8\epsilon} \Big( \sum_{k_1 \in \bZ^3} | \widehat{H}(t,k_1;K)| \Big) \| \phi(t) \|_{H_x^s(\bT^3)} \\
 &\lesssim K^{-\epsilon} \| H \|_{\LM(\cJ)}  \| \phi(t) \|_{H_x^s(\bT^3)} . 
\end{align*}
The desired estimate follows after summing in $K$. 
\end{proof}

In the next two lemmas, we show that the terms appearing in the evolution equation \eqref{lwp:eq_XN} for $X_N$ fit into our para-controlled framework. 

\begin{lemma}\label{para:lemma_obj_1}
Let $\cJ \subseteq \bR$ be an interval and let $f,g \in \X{-1}{b}(\cJ)$. Then, there exists a (canonical) $H \in \LM(\cJ)$ satisfying
\begin{equation}
\paraboxld \Big( V \ast (f \, g) \phi \Big) = \PCtrl(H,\phi)
\end{equation}
for all space-time distributions $\phi \colon \cJ \times \bT^3 \rightarrow \C$. Furthermore, it holds that 
\begin{equation}\label{para:lemma_obj_1_est}
\| H \|_{\LM(\cJ)} \lesssim \| f \|_{\X{-1}{b}(\cJ)} \cdot \| g \|_{\X{-1}{b}(\cJ)} . 
\end{equation}
\end{lemma}

\begin{remark}
Due to the overlaps in the support of the Littlewood-Paley multipliers $\chi_K$, the low-frequency modulation $H \in \LM(\cJ)$ is not quite unique. As will be clear from the proof, however, there is a canonical choice. This canonical choice is also bilinear in $f$ and $g$. 
\end{remark}

\begin{proof}
Using the definition of the restricted norms, it suffices to treat the case $\cJ=\bR$. We have that 
\begin{align*}
&\paraboxld \Big( V \ast (f \, g) \phi \Big)(t,x) \\
&= \sum_{\substack{N_1,N_2,K\colon \\ N_1,N_2 \leq K^\epsilon}} \sum_{n_1,n_2,k\in \bZ^3} \chi_{N_1}(n_1) \chi_{N_2}(n_2) \chi_K(k) \widehat{V}(n_{12}) \widehat{f}(t,n_1) \widehat{g}(t,n_2) \widehat{\phi}(t,k) \exp\big( i \langle n_{12} + k , x \rangle\big) \\
&= \PCtrl(H,\phi)(t,x),
\end{align*}
where
\begin{equation}\label{para:lemma_obj_1_p1}
\widehat{H}(t,k_1;K) 
= \sum_{\substack{N_1,N_2\colon \\ N_1,N_2 \leq K^\epsilon}} \sum_{\substack{n_1,n_2\in \bZ^3\colon\\ n_{12} =k_1 }}
\chi_{N_1}(n_1) \chi_{N_2}(n_2)\widehat{V}(n_{12}) \widehat{f}(t,n_1) \widehat{g}(t,n_2)
\end{equation}
It therefore remains to show  $H\in \LM(\bR)$ and the estimate \eqref{para:lemma_obj_1_est}. The Fourier support condition \eqref{para:eq_fourier_support} is a consequence of the multiplier $\chi_{N_1}(n_1) \chi_{N_2}(n_2)$ in \eqref{para:lemma_obj_1_p1}. 
To see the estimate \eqref{para:lemma_obj_1_est}, we first note that 
\begin{equation*}
\widehat{H}(\lambda,k_1;K) 
= \sum_{\substack{N_1,N_2\colon \\ N_1,N_2 \leq K^\epsilon}} \sum_{\substack{n_1,n_2\in \bZ^3\colon\\ n_{12} =k_1 }}
\chi_{N_1}(n_1) \chi_{N_2}(n_2)\widehat{V}(n_{12}) \Big( \widehat{f}(\cdot,n_1) \ast \widehat{g}(\cdot,n_2)\Big)(\lambda).
\end{equation*}
Using Young's convolution inequality and Cauchy-Schwarz, we obtain that 
\begin{align*}
&\| \widehat{H}(\lambda,k_1;K)  \|_{L_\lambda^1(\bR)} \\
&\lesssim  \sum_{\substack{N_1,N_2\colon \\ N_1,N_2 \leq K^\epsilon}} \sum_{\substack{n_1,n_2\in \bZ^3\colon\\ n_{12} =k_1 }}
\chi_{N_1}(n_1) \chi_{N_2}(n_2) |\widehat{V}(n_{12})| 
\|  \widehat{f}(\lambda,n_1) \|_{L_\lambda^1(\bR)}
\|  \widehat{g}(\lambda,n_2) \|_{L_\lambda^1(\bR)} \\
&\lesssim \sum_{\substack{n_1,n_2\in \bZ^3\colon\\ n_{12} =k_1 }} 1\big\{ |n_1|,|n_2| \lesssim K^\epsilon \big\} 
\|  \langle |\lambda| - \langle n_1 \rangle \rangle^b \widehat{f}(\lambda,n_1) \|_{L_\lambda^2(\bR)}
\|  \langle |\lambda| - \langle n_2 \rangle \rangle^b \widehat{g}(\lambda,n_2) \|_{L_\lambda^2(\bR)} \\
&\lesssim \Big( \sum_{n_1 \in \bZ^3}  1\big\{ |n_1| \lesssim K^\epsilon \big\} \|  \langle |\lambda| - \langle n_1 \rangle \rangle^b \widehat{f}(\lambda,n_1) \|_{L_\lambda^2(\bR)}^2 \Big)^{\frac{1}{2}} \\
&\quad \times \Big( \sum_{n_2 \in \bZ^3}  1\big\{ |n_2| \lesssim K^\epsilon \big\} \|  \langle |\lambda| - \langle n_2 \rangle \rangle^b \widehat{g}(\lambda,n_2) \|_{L_\lambda^2(\bR)}^2 \Big)^{\frac{1}{2}} \\
&\lesssim K^{-2\epsilon} \| f \|_{\X{-1}{b}(\cJ)} \cdot \| g \|_{\X{-1}{b}(\cJ)}. 
\end{align*}
The desired estimate \eqref{para:lemma_obj_1_est} now follows after taking the supremum in $K\geq1$ and $k_1\in \bZ^3$. 
\end{proof}

\begin{lemma}\label{para:lemma_obj_2}
Let $\cJ \subseteq \bR$ be an interval, let $s\in [-1,1]$, let $f\in \X{-s}{b}(\cJ)$, and let $g\in \X{s}{b}$. Then, there exists a (canonical) $H\in \LM(\cJ)$ satisfying 
\begin{equation}\label{para:eq_obj_2}
 V \ast (f \, g) \parald \phi  = \PCtrl(H,\phi)
\end{equation}
for all space-time distributions $\phi\colon J \times \bT^3 \rightarrow \bC$. Furthermore, it holds that 
\begin{equation}\label{para:lemma_obj_2_est}
\| H \|_{\LM(\cJ)} \lesssim \| f \|_{\X{-s}{b}(\cJ)} \cdot \| g \|_{\X{s}{b}(\cJ)} . 
\end{equation}
\end{lemma}

\begin{remark}
We emphasize that Lemma \ref{para:lemma_obj_2} fails if we replace the assumptions by $f,g\in \X{-1}{b}(\cJ)$ as in Lemma \ref{para:lemma_obj_1}. The reason is that the product $f\cdot g$ inside the convolution with the interaction potential $V$ is not even well-defined. 
\end{remark}

\begin{proof}
The argument is similar to the proof of Lemma \ref{para:lemma_obj_1}. As before, it suffices to treat the case $\cJ=\bR$. A direct calculation yields the identity \eqref{para:eq_obj_2} with
\begin{equation}
H(t,k_1;K) 
=\sum_{K_1 \leq K^\epsilon} \chi_{K_1}(k_1) \widehat{V}(k_1)  \sum_{\substack{n_1,n_2\in \bZ^3\colon\\ n_{12} =k_1 }}
\widehat{f}(t,n_1) \widehat{g}(t,n_2). 
\end{equation}
Using Young's convolution inequality and Cauchy-Schwarz, we obtain that 
\begin{align*}
&\| \widehat{H}(\lambda,k_1;K) \|_{L_\lambda^1(\bR)} \\
&\lesssim \sum_{\substack{n_1,n_2\in \bZ^3\colon\\ n_{12} =k_1 }}
 \| \widehat{f}(\lambda,n_1) \|_{L_\lambda^1(\bR)} \| \widehat{g}(\lambda,n_2) \|_{L_\lambda^1(\bR)} \\
&\lesssim 
\Big( \sum_{n_1 \in \bZ^3} \langle n_1 \rangle^{-2s}  \| \langle |\lambda| - \langle n_1 \rangle \rangle^{b} \widehat{f}(\lambda,n_1) \|_{L_\lambda^2(\bR)}^2 \Big)^{\frac{1}{2}} 
\Big( \sum_{n_2 \in \bZ^3} \langle n_2 - k_1 \rangle^{2s}  \| \langle |\lambda| - \langle n_2 \rangle \rangle^{b} \widehat{g}(\lambda,n_2) \|_{L_\lambda^2(\bR)}^2 \Big)^{\frac{1}{2}}. 
\end{align*}
Using that $\langle n_2 - k_1 \rangle \lesssim \langle k_1 \rangle + \langle n_2 \rangle \lesssim K^\epsilon \langle n_2 \rangle$, we obtain the estimate \eqref{para:lemma_obj_2_est}.
\end{proof}

\subsection{Quadratic para-controlled estimate}

In this subsection, we show that $P_{\leq N} X_N \paraeq P_{\leq N} \<1b>$ is well-defined uniformly in $N$ even though the sum of the individual spatial regularities is negative. Together with Lemma \ref{phy:lem_bilinear_tool}, this will control the second and third term in $\Phy$, i.e., 
\begin{equation*}
V \ast \Big( P_{\leq N} X_N \paraeq P_{\leq N} \<1b> \Big) \cdot P_{\leq N} \<3DN> \quad \text{and} \quad 
V \ast \Big( P_{\leq N} X_N \paraeq P_{\leq N} \<1b> \Big) \cdot P_{\leq N} w_N. 
\end{equation*}

\begin{proposition}[Quadratic para-controlled object]\label{para:prop_quadratic_object}
Let  $T\geq 1$. For any $s<-2\eta - 10 \epsilon$ and $p \geq 2$, we have that 
\begin{align*}
&\sum_{L_1 \sim L_2} L_1^{2\eta} \Big\|  \sup_{N\geq 1} \sup_{\cJ \subseteq [0,T]} \sup_{\| H \|_{\modulation(\cJ)}\leq 1} \Big\|  \big( P_{L_1} P_{\leq N} \Duh \big) \Big[   1_{\cJ} \Para(H, P_{\leq N} \<1b>)  \Big]  \cdot P_{L_2} \<1b> \Big\|_{L_t^\infty \cC_x^s([0,T]\times \bT^3)} \Big\|_{L_\omega^p(\bP)} \\ &\lesssim T^3 p,
\end{align*}
where the supremum in $\cJ$ is taken only over intervals.
\end{proposition}

\begin{proof}
The supremum in $N$ can be handled through the decay in the frequency-localized version below and we omit it throughout the proof. Using the definition of the $\LM(\cJ)$-norm, we may take the supremum over  $H \in \LM(\bR)$ with norm bounded by one. By inserting the expansion \eqref{para:eq_para_expansion}, we obtain that 
\begin{align*}
&  \big( P_{L_1} P_{\leq N} \Duh \big) \Big[   1_{\cJ} \Para(H, P_{\leq N} \<1b>)  \Big](t,x) \cdot P_{L_2} \<1b>(t,x) \\
&= \sum_{N_1} \sum_{k_1\in \bZ^3} \int_{\bR} \dlambda_1 \widehat{H}(\lambda_1,k_1;N_1)  
\sum_{n_1,n_2 \in \bZ^3} \bigg[\rho_N(k_1+n_1) \chi_{L_1}(n_1+k_1) \rho_N(n_1) \chi_{N_1}(n_1) \chi_{L_2}(n_2)  \\
&\times  \widehat{\<1b>}(t,n_2)
\bigg(\int_0^t 1_{\cJ}(t^\prime) \frac{\sin((t-t^\prime)\langle k_1 + n_1 \rangle)}{\langle k_1 + n_1 \rangle} \exp\big( i \lambda_1 t^\prime \big) \widehat{\<1b>}(t^\prime,n_1) \dtprime \bigg) \exp\big(i\langle n_{12} + k_1 , x \rangle\big) \bigg]. 
\end{align*}
Due to the definition of $\LM$, we only obtain a non-trivial contribution if $N_1 \sim L_1 \sim L_2$. Using the triangle inequality, it follows that
\begin{align*}
 &\sup_{\cJ \subseteq [0,T]} \sup_{\| H \|_{\modulation(\bR)}\leq 1} \Big\|    \big( P_{L_1} P_{\leq N} \Duh \big) \Big[   1_{\cJ} \Para(H, P_{\leq N} \<1b>)  \Big] \cdot P_{L_2} \<1b> \,\Big\|_{L_t^\infty \cC_x^s([0,T]\times \bT^3)} \\
 &\lesssim \sum_{N_1}  N_1^{7 \epsilon} \sup_{\cJ \subseteq [0,T]} \sup_{\substack{k_1 \in \bZ^3\colon \\ |k_1|\leq 8 N_1^\epsilon}} \sup_{\lambda_1 \in \bR} \bigg\| \sum_{n_1,n_2 \in \bZ^3} \bigg[ \rho_N(k_1+n_1) \chi_{L_1}(n_1+k_1) \rho_N(n_1) \chi_{N_1}(n_1) \chi_{L_2}(n_2)   \\
&\times  \exp\big(i\langle n_{12} + k_1 , x \rangle\big) \widehat{\<1b>}(t,n_2)
\bigg(\int_0^t  1_{\cJ}(t^\prime) \frac{\sin((t-t^\prime)\langle k_1 + n_1 \rangle)}{\langle k_1 + n_1 \rangle} \exp\big( i \lambda_1 t^\prime \big) \widehat{\<1b>}(t^\prime,n_1) \dtprime \bigg)  \bigg] \bigg\|_{L_t^\infty \cC_x^s([0,T]\times \bT^3)}.
\end{align*}
To obtain the desired estimate, it suffices to prove for all $N_1\sim L_1 \sim L_2$ that 
\begin{equation}\label{para:eq_quadratic_p1}
\begin{aligned}
 &\bigg\| \sup_{\cJ \subseteq [0,T]} \sup_{\substack{k_1 \in \bZ^3\colon \\ |k_1|\leq 8 N_1^\epsilon}} \sup_{\lambda_1 \in \bR} \bigg\| \sum_{n_1,n_2 \in \bZ^3} \bigg[ \rho_N(k_1+n_1)  \chi_{L_1}(n_1+k_1)      \\
&\times  
 \rho_N(n_1) \chi_{N_1}(n_1) \chi_{L_2}(n_2)  \exp\big(i\langle n_{12} + k_1 , x \rangle\big)  \widehat{\<1b>}(t,n_2) \\
 &\times     \bigg(\int_0^t   1_{\cJ}(t^\prime)\frac{\sin((t-t^\prime)\langle k_1 + n_1 \rangle)}{\langle k_1 + n_1 \rangle}  \exp\big( i \lambda_1 t^\prime \big) \widehat{\<1b>}(t^\prime,n_1) \dtprime \bigg)  \bigg] \bigg\|_{L_t^\infty \cC_x^s([0,T]\times \bT^3)} 
 \bigg\|_{L^p_\omega(\Omega)} \allowdisplaybreaks[3] \\ 
 &\lesssim T^3 N_1^{-2\eta-9\epsilon}. 
\end{aligned}
\end{equation}
We claim that instead of \eqref{para:eq_quadratic_p1}, it suffices to prove the simpler estimate
\begin{equation}\label{para:eq_quadratic_p2}
\begin{aligned}
 & \sup_{t\in [0,T]}  \sup_{\cJ \subseteq [0,T]} \sup_{\substack{k_1 \in \bZ^3\colon \\ |k_1|\leq 8 N_1^\epsilon}} \sup_{\lambda_1 \in \bR} \bigg\|   \sum_{n_1,n_2 \in \bZ^3} \bigg[ \rho_N(k_1+n_1) \chi_{L_1}(n_1+k_1) \\ 
&\times  
 \rho_N(n_1) \chi_{N_1}(n_1) \chi_{L_2}(n_2)  \exp\big(i\langle n_{12} + k_1 , x \rangle\big) \widehat{\<1b>}(t,n_2)   \\
 &\times \bigg(\int_0^t  1_{\cJ}(t^\prime) \frac{\sin((t-t^\prime)\langle k_1 + n_1 \rangle)}{\langle k_1 + n_1 \rangle}  \exp\big( i \lambda_1 t^\prime \big) \widehat{\<1b>}(t^\prime,n_1) \dtprime \bigg)  \bigg] \bigg\|_{ L^2_\omega H_x^s} \allowdisplaybreaks[3]
 \\
 &\lesssim T^2 N_1^{-2\eta-10\epsilon} p. 
\end{aligned}
\end{equation}
The reduction of \eqref{para:eq_quadratic_p1} to \eqref{para:eq_quadratic_p2} is standard and we only sketch the argument. The supremum in $k_1$ can easily be moved outside the moment by using Lemma \ref{tools:eq_moments_sups} and accepting a logarithmic loss in $N_1$. To deal with the supremum in $\lambda_1 \in \bR$, we treat two separate cases. Using the Lipschitz estimate $|\exp(i \lambda_1 t^\prime) - \exp(i\widetilde{\lambda}_1 t^\prime)| \lesssim |t^\prime| |\lambda_1 - \widetilde{\lambda}_1|$, the supremum over $|\lambda_1|\lesssim N_1^{10}$ can easily be replaced by the supremum over a grid on $[-N_1^{10},N_1^{10}]$ with mesh size $\sim N_1^{-10}$. The discrete supremum can then be moved outside the probabilistic moment using Lemma \ref{tools:eq_moments_sups}. For $|\lambda_1| \gtrsim N_1^{10}$, a simple integration by parts gains a factor of $|\lambda_1|^{-1}$ and we can proceed using crude estimates. The supremum over $t\in [0,T]$ and $\cJ\subseteq [0,T]$, which is parametrized by its two endpoints, can be moved outside of the probabilistic moment using the first part of the argument for $\lambda_1$.  Finally, Gaussian hypercontractivity allows us to replace $L^p_\omega \cC_x^s$ by $L^2_\omega H_x^s$. \\

We now turn to the proof of the simpler estimate \eqref{para:eq_quadratic_p2}. Using the product formula for multiple stochastic integrals, we have that 
\begin{align*}
& \sum_{n_1,n_2 \in \bZ^3} \bigg[ \rho_N(k_1+n_1) \chi_{L_1}(n_1+k_1) \rho_N(n_1) \chi_{N_1}(n_1) \chi_{L_2}(n_2)   \exp\big(i\langle n_{12} + k_1 , x \rangle\big)  \\
\times  
&\widehat{\<1b>}(t,n_2)   \bigg(\int_0^t  1_{\cJ}(t^\prime) \frac{\sin((t-t^\prime)\langle k_1 + n_1 \rangle)}{\langle k_1 + n_1 \rangle}  \exp\big( i \lambda_1 t^\prime \big) \widehat{\<1b>}(t^\prime,n_1) \dtprime \bigg) \bigg] \\
&= \cG^{(2)}(t,x) + \cG^{(0)}(t,x), 
\end{align*}
where the Gaussian chaoses $\cG^{(2)}$ and $\cG^{(0)}$  are given by 
\begin{align*}
 \cG^{(2)}(t,x) \defe& \sum_{\pm_1,\pm_2} \sum_{n_1,n_2 \in \bZ^3} \bigg[ c(\pm_1,\pm_2)\rho_N(k_1+n_1) \chi_{L_1}(n_1+k_1) \rho_N(n_1) \chi_{N_1}(n_1) \chi_{L_2}(n_2)   \\
 & \times \bigg( \int_0^t   1_{\cJ}(t^\prime)\frac{\sin((t-t^\prime)\langle k_1 + n_1 \rangle)}{\langle k_1 + n_1 \rangle}  \exp\big( i \lambda_1 t^\prime \pm_1 i t^\prime \langle n_1 \rangle \pm_2 i t \langle n_2 \rangle \big)  \dtprime \bigg)  \\
 &\exp\big(i\langle n_{12} + k_1 , x \rangle\big)   \cI_2[\pm_j,n_j\colon j=1,2] \bigg], \\
  \cG^{(0)}(t,x) \defe& \exp(i \langle k_1 , x \rangle) \sum_{n_1\in \bZ^3} \bigg[ 
  \rho_N(k_1+n_1) \chi_{L_1}(n_1+k_1) \rho_N(n_1) \chi_{N_1}(n_1) \chi_{L_2}(n_1) 
   \frac{1}{\langle n_1 + k_1 \rangle \langle n_1 \rangle^2} \\
   &\times  \bigg( \int_0^t  1_{\cJ}(t^\prime) \sin((t-t^\prime)\langle k_1 + n_1 \rangle)
   \cos((t-t^\prime)\langle n_1 \rangle) \exp\big( i\lambda_1 t^\prime\big) \dtprime \bigg) \bigg] .  
\end{align*}
The quadratic Gaussian chaos $\cG^{(2)}$ is the non-resonant part and the constant ``Gaussian chaos'' $\cG^{(0)}$ is the resonant part. We now treat both components separately. \\

\emph{Contribution of the quadratic Gaussian chaos  $\cG^{(2)}$:} Using the orthogonality of the multiple stochastic integrals and taking absolute values inside the $t^\prime$-integral, we have that 
\begin{align*}
&\|  \cG^{(2)}(t,x) \|_{L^2_\omega H_x^s(\Omega\times\bT^3)}^2 \\
&\lesssim T^2 \sum_{n_1,n_2\in \bZ^3} \chi_{N_1}(n_1) \chi_{L_2}(n_2) \langle k_1 + n_{12} \rangle^{2s} \langle k_1 + n_1 \rangle^{-2} \langle n_1 \rangle^{-2} \langle n_2 \rangle^{-2} \\
&\lesssim T^2 N_1^{-6} \sum_{n_1\in \bZ^3}   \chi_{N_1}(n_1) \chi_{L_2}(n_2) \langle k_1 + n_{12} \rangle^{2s}  \\
&\lesssim T^2  N_1^{-4\eta-20\epsilon},
\end{align*}
which is acceptable. 

\emph{Contribution of the constant ``Gaussian chaos''  $\cG^{(0)}$:} 
Using the $\operatorname{sine}$-cancellation lemma (Lemma \ref{tools:lem_sin_cancellation}), we have that 
\begin{align*}
&\|  \cG^{(0)}(t,x) \|_{H_x^s(\bT^3)} \\
&\lesssim  \bigg| \sum_{n_1\in \bZ^3} \bigg[ 
 \rho_N(k_1+n_1) \chi_{L_1}(n_1+k_1) \rho_N(n_1) \chi_{N_1}(n_1) \chi_{L_2}(n_1) 
   \frac{1}{\langle n_1 + k_1 \rangle \langle n_1 \rangle^2}  \\
   &\times  \bigg( \int_0^t   1_{\cJ}(t^\prime)\sin((t-t^\prime)\langle k_1 + n_1 \rangle)
   \cos((t-t^\prime)\langle n_1 \rangle) \exp\big( i\lambda_1 t^\prime\big) \dtprime \bigg) \bigg] \bigg| \\
   &\lesssim  N_1^{-1+3\epsilon},
\end{align*}
which is also acceptable. 
\end{proof}

\subsection{Cubic para-controlled estimate}
In this subsection, we control the cubic para-controlled object, i.e., the first summand in the definition of $\CPara$ in \eqref{lwp:eq_CPara}.

\begin{proposition}\label{para:prop_cubic}
Let $T\geq 1$. For any interval $\cJ \subseteq [0,T]$, any $\phi \colon [0,T] \times \bT^3 \rightarrow \bC$, and $H\in \LM(\cJ)$, we define
\begin{align*}
&\Para^{(3)}_N(H,\phi;\cJ) \\
&\defe   \nparaboxld \Big( V \ast \Big( \big(P_{\leq N}^2\Duh \big) \Big[ 1_{\cJ}  \Para(H, \phi)  \Big]  \cdot \phi \Big) \cdot \phi \Big)
- \MN P_{\leq N}^2   \Duh \Big[ 1_\cJ \Para(H, \phi) \Big]. 
\end{align*}
Then, it holds that for all $p \geq 2$ that 
\begin{equation*}
\Big\|  \sup_{N\geq 1} \sup_{\cJ \subseteq [0,T]} \sup_{\substack{\| H\|_{\modulation(\cJ)}\leq 1}} 
\Big\|  \Para^{(3)}_N(H, P_{\leq N} \<1b>;\cJ) \Big\|_{\X{s_2-1}{b_+-1}([0,T])} \Big\|_{L^p_\omega(\bP)} \lesssim T^3 p^{\frac{3}{2}},
\end{equation*}
where the supremum in $\cJ$ is only taken over intervals.
\end{proposition}

\begin{remark} 
The notation $\Para^{(3)}_N(H, P_{\leq N}\<1b>;\cJ) $ will only be used in Proposition \ref{para:prop_cubic} and its proof.  The frequency-localized version of Proposition \ref{para:prop_cubic} also gains an $\eta^\prime$-power in the maximal frequency-scale. 
\end{remark}

\begin{proof}
As before, we ignore the supremum in $N$, which can be easily handled through the decay in the frequency-localized version below. Using the decay in the frequency-localized version and a crude estimate, we can also replace the $\X{s_2}{b_+-1}$-norm by the $\X{s_2}{b_--1}$-norm.  Using the definition of the restricted norms, it suffices to consider  $H\in \LM(\bR)$ with $\| H \|_{\LM(\bR)} \leq 1$. In order to use a Littlewood-Paley decomposition, we need to break up the multiplier $\MN$. We define $\MN[N_1,N_2,N_3]$ as the multiplier with the symbol
\begin{equation}
\mathscr{m}_N[N_1,N_2,N_3](n_2) = \sum_{k \in \bZ^3} \frac{\widehat{V}(k+n_2)}{\langle k \rangle^2} \rho_N^2(k) \chi_{N_1}(k) \chi_{N_2}(n_2) \chi_{N_3}(k). 
\end{equation}
We note that  $\MN[N_1,N_2,N_3]$ is only non-zero when $N_1 \sim N_3$, and hence, in particular, when $N_1 > N_3^\epsilon$. We now face a  notational nuisance; namely, that both $\PCtrl$ and $\paraboxld$ contain frequency-projections. To this end, we use $N_2$ and $N_2^\prime$ for the respective frequency-scales, but encourage the reader to mentally set $N_2 = N_2^\prime$. It then follows that 
\begin{equation}
\begin{aligned}
&\Para^{(3)}_N(H, P_{\leq N} \<1b>;\cJ) \\
&= \sum_{\substack{ N_1,N_2^\prime,N_3 \colon \\ \max(N_1,N_2^\prime) > N_3^\epsilon}} \bigg[ V \ast \Big( P_{N_1} P_{\leq N} \<1b> \cdot P_{N_2^\prime} P_{\leq N}^2 \Duh \Big[ 1_{\cJ} \PCtrl(H,P_{\leq N} \<1b>)\Big] \Big) \cdot P_{N_3} P_{\leq N} \<1b>  \\
&\quad- \MN[N_1,N_2^\prime,N_3] P_{\leq N}^2 \Duh \Big[ 1_{\cJ} \PCtrl(H,P_{\leq N} \<1b>)\Big] \bigg] . 
\end{aligned}
\end{equation}
Using the stochastic representation formula \eqref{tools:eq_stochastic_representation_1}  in Proposition \ref{tools:prop_stochastic_representation} and the expansion \eqref{para:eq_para_expansion}, we obtain that 
\begin{align*}
 &\Para^{(3)}_N(H, P_{\leq N} \<1b>;\cJ)(t,x) \\
 &= \sum_{\substack{N_1,N_2,N_2^\prime,N_3 \colon \\ \max(N_1,N_2^\prime)> N_3, \\  N_2 \sim N_2^\prime}} \sum_{k_2\in \bZ^3} \int_{\bR} \dlambda_2 ~ \widehat{H}(\lambda_2,k_2;N_2) \sum_{n_1,n_2,n_3\in \bZ^3} \bigg[ \rho_N^2(n_2+k_2) \chi_{N_2^\prime}(n_2+k_2) \\ 
 &\times \Big( \prod_{j=1}^3 \rho_N(n_j) \chi_{N_j}(n_j) \Big) \widehat{V}(n_{12}+k_2)   \Big( \int_0^t  1_{\cJ}(t^\prime) \frac{\sin((t-t^\prime)\langle n_2 + k_2 \rangle)}{\langle n_2 + k_2 \rangle} \exp\big( i t^\prime \lambda_2) \cI_1[t^\prime,n_2] \dtprime \Big) \\
 &\times \exp\big( i \langle n_{123} + k_2 ,x\rangle\big) \cI_2[t,n_1,n_3] \bigg]. 
\end{align*}
Using the product formula for multiple stochastic integrals, we can decompose the inner sum in $n_1,n_2$, and $n_3$ as 
\begin{align*}
&\sum_{n_1,n_2,n_3\in \bZ^3} \bigg[ \rho_N^2(n_2+k_2) \chi_{N_2^\prime}(n_2+k_2)  \Big( \prod_{j=1}^3 \rho_N(n_j) \chi_{N_j}(n_j) \Big) \widehat{V}(n_{12}+k_2) \cI_2[t,n_1,n_3] \\ 
 &\times  \Big( \int_0^t 1_{\cJ}(t^\prime) \frac{\sin((t-t^\prime)\langle n_2 + k_2 \rangle)}{\langle n_2 + k_2 \rangle} \exp\big( i t^\prime \lambda_2) \cI_1[t^\prime,n_2] \dtprime \Big)  \exp\big( i \langle n_{123} + k_2 ,x\rangle\big) \bigg] \\
 &= \cG^{(3)}(t,x;\lambda_2,k_2,\cJ,N_\ast)+ \cG^{(1)}(t,x;\lambda_2,k_2,\cJ,N_\ast) + \widetilde{\cG}^{(1)}(t,x;\lambda_2,k_2,\cJ,N_\ast),
\end{align*}
where the cubic and linear Gaussian chaoses are given by 
\begin{align*}
\cG^{(3)}(t,x) &= \hspace{-1ex} \sum_{\pm_1,\pm_2,\pm_3} c(\pm_j\colon  1\leq j \leq 3) \sum_{n_1,n_2,n_3\in \bZ^3} \bigg[  \rho_N^2(n_2+k_2) \chi_{N_2^\prime}(n_2+k_2)  \Big( \prod_{j=1}^3 \rho_N(n_j) \chi_{N_j}(n_j) \Big) \\ 
 &\times        \widehat{V}(n_{12}+k_2)  \Big(\int_0^t 1_{\cJ}(t^\prime) \frac{\sin((t-t^\prime)\langle n_2 + k_2 \rangle)}{\langle n_2 + k_2 \rangle} \exp\big( i t^\prime \lambda_2\pm_2 i t^\prime \langle n_2 \rangle) \dtprime \Big) \\
 &\times   \exp(\pm_1 i t \langle n_1 \rangle \pm_3 i t \langle n_3 \rangle)   \exp\big( i \langle n_{123} + k_2 ,x\rangle\big)  \cI_3[\pm_j,n_j \colon 1 \leq j \leq 3] \bigg], \allowdisplaybreaks[3]  \\ 
 \cG^{(1)}(t,x) &=\sum_{n_3\in \bZ^3}  \rho_N(n_3) \chi_{N_3}(n_3) \exp\big( \langle n_3 + k_2 , x \rangle\big) \sum_{n_1\in \bZ^3} \bigg[ \rho_N^2(n_2+k_2) \chi_{N_2^\prime}(n_2+k_2)  \rho_N^2(n_2)   \\ 
 &\times  \chi_{N_1}(n_2) \chi_{N_2}(n_2)  \Big( \int_0^t 1_{\cJ}(t^\prime) \sin((t-t^\prime)\langle n_2 + k_2 \rangle) \cos((t-t^\prime)\langle n_2  \rangle) \exp\big( i t^\prime \lambda_2) \dtprime \Big)  \\
 &\times   \widehat{V}(k_2)  \langle n_2 + k_2 \rangle^{-1} \langle n_2 \rangle^{-2}  \bigg] ~ \cI_1[t;n_3] ,\allowdisplaybreaks[3]  \\
  \widetilde{\cG}^{(1)}(t,x) &=  \sum_{n_1\in \bZ^3}  \rho_N(n_1) \chi_{N_1}(n_1) \exp\big( \langle n_1 + k_2 , x \rangle\big) \sum_{n_1\in \bZ^3} \bigg[ \rho_N^2(n_2+k_2) \chi_{N_2^\prime}(n_2+k_2)  \rho_N^2(n_2)   \\ 
 &\times  \chi_{N_2}(n_2) \chi_{N_3}(n_2)  \Big( \int_0^t 1_{\cJ}(t^\prime) \sin((t-t^\prime)\langle n_2 + k_2 \rangle) \cos((t-t^\prime)\langle n_2  \rangle) \exp\big( i t^\prime \lambda_2) \dtprime \Big)  \\
 &\times  \widehat{V}(n_{12}+k_2) \langle n_2 + k_2 \rangle^{-1} \langle n_2 \rangle^{-2}  \bigg] ~ \cI_1[t;n_1] . 
\end{align*}
We refer to $\cG^{(3)}$ as the non-resonant term and to $ \cG^{(1)}$ and $ \widetilde{\cG}^{(1)}$ as the resonant terms. Using the triangle inequality and $\| H \|_{\LM(\bR)} \leq 1$, we obtain that 
\begin{align*}
&\Big\|  \Para^{(3)}_N(H, P_{\leq N} \<1b>;\cJ) \Big\|_{\X{s_2-1}{b_+-1}(\cJ)} \\
&\lesssim  \sum_{\substack{N_1,N_2,N_2^\prime,N_3 \colon \\ \max(N_1,N_2^\prime)> N_3, \\  N_2 \sim N_2^\prime}}  N_2^{7\epsilon} ~ \sup_{\lambda_2 \in \bR} \sup_{\substack{ k_2 \in \bZ^3 \colon \\ |k_2| \lesssim N_2^\epsilon}}
\Big( \| \cG^{(3)}(\cdot;\lambda_2,k_2,\cJ,N_\ast)\|_{\X{s_2-1}{b_--1}([0,T])}   \\
&+ \| \cG^{(1)}(\cdot;\lambda_2,k_2,\cJ,N_\ast)\|_{\X{s_2-1}{b_--1}([0,T])}  + \| \widetilde{\cG}^{(1)}(\cdot;\lambda_2,k_2,\cJ,N_\ast)\|_{\X{s_2-1}{b_--1}([0,T])}  \Big). 
\end{align*}
We now use Gaussian hypercontractivity and a similar reduction as in the proof of Proposition \ref{para:prop_quadratic_object} to move the supremum outside the probabilistic moments. Then, it remains to show for all frequency scales $N_1,N_2$, and $N_3$ satisfying $\max(N_1,N_2)>N_3^\epsilon$ that  
\begin{align*}
& \sup_{\lambda_2 \in \bR} \sup_{\substack{ k_2 \in \bZ^3 \colon \\ |k_2| \lesssim N_2^\epsilon}}\Big\|  \| \cG^{(3)}(\cdot;\lambda_2,k_2,\cJ,N_\ast)\|_{\X{s_2-1}{b_--1}([0,T])}   
+ \| \cG^{(1)}(\cdot;\lambda_2,k_2,\cJ,N_\ast)\|_{\X{s_2-1}{b_--1}([0,T])}   \\
&+ \| \widetilde{\cG}^{(1)}(\cdot;\lambda_2,k_2,\cJ,N_\ast)\|_{\X{s_2-1}{b_--1}([0,T])}  \Big\|_{L^2_\omega} \\
&\lesssim T^2 \max(N_1,N_2,N_3)^{-\eta}. 
\end{align*}
We treat the estimates for the non-resonant and resonant components separately. \\

\emph{Contribution of the non-resonant terms:} To estimate the $\X{s_2-1}{b_--1}$-norm, we calculate the space-time Fourier transform of $\chi(t/T) \cG^{(3)}(t,x;\lambda_2,k_2,\cJ,N_\ast)$. We have that 
\begin{align*}
&\mathcal{F}_{t,x}\Big( \chi(t/T) \cG^{(3)}(t,x;\lambda_2,k_2,\cJ,N_\ast) \Big)(\lambda \mp \langle n \rangle, n) \\
&=  \hspace{-1ex} \sum_{\pm_1,\pm_2,\pm_3} c(\pm_j\colon  1\leq j \leq 3) \sum_{n_1,n_2,n_3\in \bZ^3} \bigg[  1\big\{ n= n_{123} + k_2 \big\} \rho_N^2(n_2+k_2)  \Big( \prod_{j=1}^3 \rho_N(n_j) \chi_{N_j}(n_j) \Big) \\ 
 &\times      \chi_{N_2^\prime}(n_2+k_2)   \widehat{V}(n_{12}+k_2)  \cI_3[\pm_j,n_j \colon 1 \leq j \leq 3] \\
 &\times   \mathcal{F}_t \bigg( \exp(\pm_1 i t \langle n_1 \rangle \pm_3 i t \langle n_3 \rangle)  \int_0^t 1_{\cJ}(t^\prime) \frac{\sin((t-t^\prime)\langle n_2 + k_2 \rangle)}{\langle n_2 + k_2 \rangle} \exp\big( i t^\prime \lambda_2\pm_2 i t^\prime \langle n_2 \rangle) \dtprime \bigg)  (\lambda\mp \langle n \rangle)\bigg]   . 
\end{align*}
Using the orthogonality of the multiple stochastic integrals and Lemma \ref{tools:lem_estimate_fouriertransform} to estimate the Fourier transform of the time-integral, we obtain that 
\begin{align*}
&\big\| \| \cG^{(3)} \|_{\X{s_2-1}{b_--1}([0,T])} \big\|_{L_\omega^2}^2 \\
&\lesssim \max_{\pm} \big\| \langle \lambda \rangle^{b_--1} \langle n \rangle^{s_2-1} \mathcal{F}_{t,x}\Big( \chi(t/T) \cG^{(3)}(t,x;\lambda_2,k_2,\cJ,N_\ast) \Big)(\lambda \mp \langle n \rangle, n) \|_{L_\lambda^2 \ell_n^2 ( \bR \times \bZ^3)} \big\|_{L_\omega^2}^2 \allowdisplaybreaks[3] \\
&\lesssim T^4 \max_{\pm,\pm_1,\pm_2,\pm_3} \max_{\iota_2= -1,0,1} \int_{\bR} \, \dlambda \, \langle \lambda  \rangle^{2(b_--1)} \sum_{n_1,n_2,n_3\in \bZ^3} \bigg[   \Big( \prod_{j=1}^3  \chi_{N_j}(n_j) \Big) 
\langle n_{123} + k_2 \rangle^{2(s_2-1)} \langle n_{12} +k_2 \rangle^{-2\beta}  \\
&\times \langle n_1 \rangle^{-2} \langle n_2 \rangle^{-4} \langle n_3 \rangle^{-2} \Big( 1+ \big| \lambda - \lambda_3 - \big( \pm \langle n_{123} + k_2 \rangle \pm_1 \langle n_1 \rangle \pm_2 \langle n_2 \rangle + \iota_2 \langle n_2 + k_2 \rangle \pm_3 \langle n_3 \rangle \big| \Big)^{-2} \bigg] \allowdisplaybreaks[3] \\
&\lesssim  T^4 N_2^{-1+5\epsilon}  \max_{\pm,\pm_1,\pm_3} \sup_{\substack{ n_2 \in \bZ^3 \colon \\ |n_2| \sim N_2 }} \sup_{m\in \bZ^3} \sum_{n_1,n_3 \in \bZ^3}\bigg[  \Big( \prod_{j=1,3}  \chi_{N_j}(n_j) \Big) \langle n_{123} \rangle^{2(s_2-1)} \langle n_{12} \rangle^{-2\beta} \langle n_1 \rangle^{-2} \langle n_3 \rangle^{-2}
 \\
&\times 1\big\{ \pm \langle n_{123} \rangle \pm_1 \langle n_1 \rangle \pm_3 \langle n_3 \rangle \in [m,m+1) \big\}  \bigg]\\
&\lesssim T^4 \max(N_1,N_2,N_3)^{2\delta_2} N_1^{-2\epsilon} N_2^{-1+7\epsilon}.
\end{align*}
In the last line, we used Lemma \ref{tools:lem_paracontrolled_counting} with $\gamma=\epsilon$.
Since $\max(N_1,N_2) >N_3^\epsilon$ and $\delta_2$ is much smaller than $\epsilon^2$, this contribution is acceptable. \\

\emph{Contribution of the resonant terms:} We only estimate $\cG^{(1)}$. Due to the factor $\widehat{V}(n_{12}+k_2)$, a simpler but similar argument also controls $\widetilde{\cG}^{(1)}$. \\

Using the inhomogeneous Strichartz estimate (Lemma \ref{tools:lem_inhomogeneous_strichartz}), we have that 
\begin{align*}
\| \cG^{(1)} \|_{\X{s_2-1}{b_--1}([0,T])} \lesssim \| \cG^{(1)} \|_{L_t^{2b_+} H_x^{s_2-1}([0,T]\times \bT^3)} \lesssim T^{\frac{1}{2}} \| \cG^{(1)} \|_{L_t^{2} H_x^{s_2-1}([0,T]\times \bT^3)}. 
\end{align*}
Using Fubini's theorem and the $\operatorname{sine}$-cancellation lemma (Lemma \ref{tools:lem_sin_cancellation}), this yields
\begin{align*}
&\big\| \| \cG^{(1)} \|_{\X{s_2-1}{b_--1}([0,T])} \big\|_{L_\omega^2}^2 \\
&\lesssim  T^2 \sup_{t\in [0,T]} \big\| \| \cG^{(1)} \|_{H_x^{s_2-1}(\bT^3)} \big\|_{L_\omega^2}^2 \\
&\lesssim T^2 \sum_{n_3\in \bZ^3} \chi_{N_3}(n_3) \langle n_3 + k_2 \rangle^{2(s_2-1)} \langle n_3 \rangle^{-2} \bigg| \sum_{n_2\in \bZ^3} \rho_N^2(n_2+k_2) \rho_N^2(n_2) \chi_{N_1}(n_2) \chi_{N_2}(n_2) \chi_{N_2^\prime}(n_2+k_2)  \\
&\times  \langle n_2 + k_2 \rangle^{-1} \langle n_2 \rangle^{-2}  \Big( \int_0^t 1_{\cJ}(t^\prime) \sin((t-t^\prime)\langle n_2 + k_2 \rangle) \cos((t-t^\prime)\langle n_2  \rangle) \exp\big( i t^\prime \lambda_2) \dtprime \Big) \bigg|^2 \\
&\lesssim T^4 \,  1\big\{N_1\sim N_2 \big\} N_1^{-2+6\epsilon}   \langle k_2 \rangle^{2(1-s_2)} \sum_{n_3\in \bZ^3} \chi_{N_3}(n_3) \langle n_3  \rangle^{2(s_2-1)} \langle n_3 \rangle^{-2} \\
&\lesssim T^4 \,  1\big\{N_1\sim N_2 \big\} N_1^{-2+8\epsilon} N_3^{2\delta_2}. 
\end{align*}
Since $\max(N_1,N_2) \gtrsim N_3^\epsilon$ and $\delta_2$ is much smaller than  $\epsilon$, this contribution is acceptable.

\end{proof}

\section{Physical-space methods}\label{section:physical}

In this section, we estimate the terms in $\Phy$. The main ingredients are para-product decompositions and Strichartz estimates. In Section \ref{section:kt}, we recall the refined Strichartz estimates for the wave equation by Klainerman and Tataru \cite{KT99}. In Section \ref{section:physical_sub}, we use the Klainerman-Tataru-Strichartz estimate to control several terms in $\Phy$. The remaining terms in $\Phy$ are estimated in Section \ref{section:hybrid}, which also requires estimates on the quartic stochastic object from Section \ref{section:quartic}.

\subsection{Klainerman-Tataru-Strichartz estimates}\label{section:kt}

We first recall the refined (linear) Strichartz estimate from  \cite[(A.59)]{KT99}. 

\begin{lemma}[Klainerman-Tataru-Strichartz estimates]\label{phy:lem_KT}
Let $\cJ$ be a compact interval. Let $Q$ be a box of sidelength $\sim M$ at a distance $\sim N$ from the origin. Let $P_Q$ be the corresponding Fourier truncation operator and let $2\leq p,q <\infty$ satisfy the sharp wave-admissibility condition $1/q+1/p=1/2$. Then,
\begin{equation}\label{phy:eq_KT}
\| P_Q u \|_{L_t^q L_x^p(\cJ\times \bT^3)} 
\lesssim (1+|\cJ|)^{\frac{1}{q}} \Big( \frac{M}{N} \Big)^{\frac{1}{2}-\frac{1}{p}} N^{\frac{3}{2}-\frac{1}{q}-\frac{3}{p}} 
\| P_Q u \|_{\X{0}{b}(\cJ)}. 
\end{equation}
\end{lemma}
\begin{remark}\label{phy:rem_lorentzian_vs_galilean}
The factor $N^{\frac{3}{2}-\frac{1}{q}-\frac{3}{p}}$ is the same as in the standard deterministic Strichartz estimate. The gain from the stronger localization in frequency space is described by the factor $ ( M/N )^{\frac{1}{2}-\frac{1}{p}}$. Naturally, there is no gain when $p=2$. \\
We emphasize that \eqref{phy:eq_KT} has a more complicated dependence on $M$ and $N$ than the corresponding result for the Schr\"{o}dinger equation. In the Schr\"{o}dinger setting, the frequency-localized Strichartz estimates for the operator $P_{Q}$ and the standard Littlewood-Paley operators $P_{\leq M}$ are equivalent, which follows from the Galilean symmetry. This difference between the Schr\"{o}dinger and wave equation already played a role in our counting estimates (Section \ref{section:counting}). 
\end{remark}

\begin{corollary}\label{phsy:cor_kt}
Let $\cJ$ be a compact interval. Let $Q$ be a box of sidelength $\sim M$ at a distance $\sim N$ from the origin. Let $P_Q$ be the corresponding Fourier truncation operator and let $q \geq 4$. Then, it holds that 
\begin{equation}
\| P_Q u \|_{L_t^q L_x^q(\cJ\times \bT^3)} 
\lesssim (1+|\cJ|)^{\frac{1}{q}}M^{\frac{3}{2}-\frac{5}{q}} N^{\frac{1}{q}} 
\| P_Q u \|_{\X{0}{b}(\cJ)}. 
\end{equation}
\end{corollary}

\begin{proof}
This follows by combining Lemma \ref{phy:lem_KT} (with $q=p=4$) and the Bernstein inequality
 \begin{equation*}
\| P_Q u \|_{L_t^\infty L_x^\infty(\cJ\times \bT^3)}\lesssim M^{\frac{3}{2}} \| P_Q u \|_{\X{0}{b}(\cJ)}.
\end{equation*}
\end{proof}
We now state a bilinear version of the Klainerman-Tataru-Strichartz estimate, which is a consequence of Lemma \ref{phy:lem_KT} (cf. \cite[Theorem 4 and 5]{KT99}). However, since we only require a special case, we provide a self-contained proof. 

\begin{lemma}[Bilinear Klainerman-Tataru-Strichartz estimate]\label{phy:lem_bilinear_kt}
Let $T\geq 1$, $q \geq 4$, let $ \gamma < 3 -10/q$ and let $N_1,N_2 \geq 1$. Then, it holds that
\begin{equation*}
\| \langle \nabla \rangle^{-\gamma} \big( P_{N_1} f \cdot P_{N_2} g\big) \|_{L_t^{\frac{q}{2}} L_x^{\frac{q}{2}}([0,T]\times \bT^3)} \lesssim T^{\frac{2}{q}} \max(N_1,N_2)^{3-2s_1 - \frac{8}{q} - \gamma} \| f \|_{\X{s_1}{b}([0,T])}  \| g \|_{\X{s_1}{b}([0,T])}.
\end{equation*}
In particular, 
\begin{equation*}
\sum_{N_1,N_2} \|   P_{N_1} f \cdot P_{N_2} g \|_{L_t^2 H_x^{-4\delta_1}([0,T]\times \bT^3)} 
\lesssim T^{\frac{1}{2}} \| f \|_{\X{s_1}{b}([0,T])}  \| g \|_{\X{s_1}{b}([0,T])}.
\end{equation*}
Furthermore, if $N_{12} \geq 1$, then 
\begin{align*}
&\| (P_{N_{12}} V) \ast \big( P_{N_1} f \cdot P_{N_2} g\big) \|_{L_t^{2} L_x^{2}([0,T]\times \bT^3)}\\
&\lesssim  T^{\frac{1}{2}} N_{12}^{\frac{1}{2}-\beta-2\delta_1}  \max(N_1,N_2)^{-\frac{1}{2}+4\delta_1} \| f \|_{\X{s_1}{b}([0,T])}  \| g \|_{\X{s_1}{b}([0,T])}.
\end{align*}

\end{lemma}

\begin{remark}
Bilinear Strichartz estimates are also important in the random data theory for nonlinear Schr\"{o}dinger equations in \cite{BOP15,BOP19}.  
In the proof of Proposition \ref{phy:prop_3b} below, we will only require the case $q=4+$ and the reader may simply think of $q$ as four. 
\end{remark}

\begin{proof}
We begin with the first estimate, which is the main part of the argument. Using the definition of the restricted $\X{s}{b}$-spaces, we may replace  $\| f \|_{\X{s_1}{b}([0,T])}$ and $\| g \|_{\X{s_1}{b}([0,T])}$ by   $\| f \|_{\X{s_1}{b}(\bR)}$ and $\| g \|_{\X{s_1}{b}(\bR)}$, respectively. The proof relies on the linear Klainerman-Tataru-Strichartz estimate (Corollary \ref{phsy:cor_kt}) and box localization. We decompose
\begin{equation*}
\| \langle \nabla \rangle^{-\gamma} \big( P_{N_1} f \cdot P_{N_2} g\big) \|_{L_t^{\frac{q}{2}} L_x^{\frac{q}{2}}([0,T]\times \bT^3)}  \lesssim \sum_{ \substack{N_{12}\colon \\ N_{12} \lesssim \max(N_1,N_2)}} 
N_{12}^{-\gamma} \| P_{N_{12}}  \big( P_{N_1} f \cdot P_{N_2} g\big) \|_{L_t^{\frac{q}{2}} L_x^{\frac{q}{2}}([0,T]\times \bT^3)}. 
\end{equation*}
If $N_1 \not \sim N_2$, then $N_{12}\sim \max(N_1,N_2)$ and the desired estimate follows from H\"{o}lder's inequality and the $L_t^q L_x^q$-estimate from Corollary \ref{phsy:cor_kt} with $M\sim N$. Thus, it remains to treat the case $N_1 \sim N_2$. Let $\cQ=\cQ(N_1,N_{12})$ be a cover of the dyadic annulus at distance $\sim N_1$ by finitely overlapping cubes of diameter $\sim N_{12}$. From Fourier support considerations and  Lemma \ref{phy:lem_KT}, it follows that 
\begin{align*}
&\| P_{N_{12}}  \big( P_{N_1} f \cdot P_{N_2} g\big) \|_{L_t^{\frac{q}{2}} L_x^{\frac{q}{2}}([0,T]\times \bT^3)} \\
&\lesssim \sum_{ \substack{Q_1, Q_2 \in \cQ\colon \\ d(Q_1,Q_2)\lesssim N_{12}}}  \| P_{Q_1} P_{N_1} f \cdot P_{Q_2} P_{N_2} g \|_{L_t^{\frac{q}{2}} L_x^{\frac{q}{2}}([0,T]\times \bT^3)} \\
& \lesssim \sum_{ \substack{Q_1, Q_2 \in \cQ\colon \\  d(Q_1,Q_2)\lesssim N_{12}}}  \| P_{Q_1} P_{N_1} f \|_{L_t^q L_x^q([0,T]\times \bT^3)}  \|P_{Q_2} P_{N_2} g \|_{L_t^q L_x^q([0,T]\times \bT^3)} \\
& \lesssim T^{\frac{2}{q}} N_{12}^{3-\frac{10}{q}} N_1^{\frac{2}{q}-2s_1} \sum_{ \substack{Q_1, Q_2 \in \cQ\colon \\  d(Q_1,Q_2)\lesssim N_{12}}}  \| P_{Q_1} P_{N_1} f \|_{\X{s_1}{b}(\bR)}  \|P_{Q_2} P_{N_2} g \|_{\X{s_1}{b}(\bR)} \\
& \lesssim T^{\frac{2}{q}} N_{12}^{3-\frac{10}{q}} N_1^{\frac{2}{q}-2s_1}
 \Big( \sum_{ \substack{Q_1, Q_2 \in \cQ\colon \\  d(Q_1,Q_2)\lesssim N_{12}}} \hspace{-2ex} \| P_{Q_1} P_{N_1} f \|_{\X{s_1}{b}(\bR)}^2 \Big)^{\frac{1}{2}}   
 \Big( \sum_{ \substack{Q_1, Q_2 \in \cQ\colon \\  d(Q_1,Q_2)\lesssim N_{12}}} \hspace{-2ex}  \| P_{Q_2} P_{N_2} g \|_{\X{s_1}{b}(\bR)}^2 \Big)^{\frac{1}{2}}   \\
& \lesssim T^{\frac{2}{q}}  N_{12}^{3-\frac{10}{q}} N_1^{\frac{2}{q}-2s_1}\| f \|_{\X{s_1}{b}(\bR)}  \| g \|_{\X{s_1}{b}(\bR)}.
\end{align*}
The desired result then follows by using the upper bound $\gamma<3-\frac{10}{q}$ and summing in $N_{12}$.\\
We now turn to the second estimate. After estimating
\begin{equation*}
\| (P_{N_{12}} V) \ast \big( P_{N_1} f \cdot P_{N_2} g\big) \|_{L_t^{2} L_x^{2}([0,T]\times \bT^3)}\lesssim N_{12}^{\frac{1}{2}-\beta-2\delta_1}
\| \langle \nabla \rangle^{-\frac{1}{2}+2\delta_1}  \big( P_{N_1} f \cdot P_{N_2} g\big) \|_{L_t^{2} L_x^{2}([0,T]\times \bT^3)},
\end{equation*}
the result follows from the first estimate.
\end{proof}

\subsection{Physical terms}\label{section:physical_sub}

In this subsection, we use the Klainerman-Tataru-Strichartz estimate and a para-product decomposition to control several terms in $\Phy$. 

\begin{proposition}\label{phy:prop_1b}
Let $\cJ$ be a bounded interval and let $f,g \in \X{s_1}{b}(\cJ)$. Then, it holds that 
\begin{align*}
& \sup_{N\geq 1} \Big\|  V\ast \big( P_{\leq N} f \cdot  P_{\leq N} g \big) \nparald P_{\leq N} \<1b> \,\Big\|_{\X{s_{2}-1}{b_+-1}(\cJ)} \\
&\lesssim (1+|\cJ|)^2 
\| f \|_{\X{s_1}{b}(\cJ)} 
\| g \|_{\X{s_1}{b}(\cJ)} 
\| \, \<1b> \, \|_{L_t^\infty \cC_x^{-1/2-\kappa}(\cJ\times \bT^3)}
\end{align*}
and 
\begin{align*}
&\sup_{N\geq 1} \Big\|  \nparaboxld \Big( V\ast \big( P_{\leq N} f \cdot  P_{\leq N} g \big) ~  P_{\leq N} \<1b> \, \Big) \Big\|_{\X{s_{2}-1}{b_+-1}(\cJ)} \\
&\lesssim  (1+|\cJ|)^2 
\| f \|_{\X{s_1}{b}(\cJ)} 
\| g \|_{\X{s_1}{b}(\cJ)} 
\| \, \<1b> \, \|_{L_t^\infty \cC_x^{-1/2-\kappa}(\cJ\times \bT^3)}
\end{align*}
In the frequency-localized versions of the two estimates, which are detailed in the proof, we gain an $\eta^\prime$-power in the maximal frequency-scale. 
\end{proposition}

\begin{proof}
After using a Littlewood-Paley decomposition, we obtain
\begin{align*}
 &\Big\|  V\ast \big( P_{\leq N} f \cdot  P_{\leq N} g \big) \nparald P_{\leq N} \<1b> \,\Big\|_{\X{s_{2}-1}{b_+-1}(\cJ)} \\
  +&\Big\|  \nparaboxld \Big( V\ast \big( P_{\leq N} f \cdot  P_{\leq N} g \big) ~  P_{\leq N} \<1b> \, \Big) \Big\|_{\X{s_{2}-1}{b_+-1}(\cJ)} \\
 \lesssim&\sum_{\substack{N_1,N_2,N_3,N_{12}\colon \\ \max(N_1,N_2)\gtrsim  N_3^\epsilon}}  \Big\|  (P_{N_{12}} V ) \ast \big( P_{\leq N} P_{N_1} f \cdot  P_{\leq N} P_{N_2} g \big)  P_{\leq N} P_{N_3} \<1b> \,\Big\|_{\X{s_{2}-1}{b_+-1}(\cJ)} ,
\end{align*}
where we also used that $N_{12}\lesssim \max(N_1,N_2)$. 
We estimate each dyadic piece separately and distinguish two cases:\\

\emph{Case 1: $N_{12} \not \sim N_3$.} Using the inhomogeneous Strichartz estimate (Lemma \ref{tools:lem_inhomogeneous_strichartz}) and Lemma \ref{phy:lem_bilinear_kt}, we obtain that 
\begin{align*}
&\Big\|  (P_{N_{12}} V ) \ast \big( P_{\leq N} P_{N_1} f \cdot  P_{\leq N} P_{N_2} g \big) P_{\leq N} P_{N_3} \<1b> \,\Big\|_{\X{s_{2}-1}{b_+-1}(\cJ)} \\
&\lesssim \Big\|  (P_{N_{12}} V ) \ast \big( P_{\leq N} P_{N_1} f \cdot  P_{\leq N} P_{N_2} g \big)  P_{\leq N} P_{N_3} \<1b> \,\Big\|_{L_t^{2b_+} H_x^{s_2-1}(\cJ\times \bT^3)} \\ 
&\lesssim (1+|\cJ|)^{\frac{1}{2}}  \max(N_{12},N_3)^{s_2-1} \Big\|  (P_{N_{12}} V ) \ast \big( P_{\leq N} P_{N_1} f \cdot  P_{\leq N} P_{N_2} g \big) \Big\|_{L_t^2 L_x^2(\cJ\times \bT^3)} \\
&\times \Big\|   P_{\leq N} P_{N_3} \<1b> \,\Big\|_{L_t^\infty L_x^\infty(\cJ\times \bT^3)} \\ 
&\lesssim  (1+|\cJ|) \max(N_{12},N_3)^{s_2-1}  N_{12}^{\frac{1}{2}-\beta+2\delta_1} \max(N_1,N_2)^{-\frac{1}{2}+4 \delta_1} N_3^{\frac{1}{2}+\kappa} 
\\
&\times  
\| f \|_{\X{s_1}{b}(\cJ)} 
\| g \|_{\X{s_1}{b}(\cJ)} 
\| \, \<1b> \, \|_{L_t^\infty \cC_x^{-1/2-\kappa}(\cJ\times \bT^3)}.
\end{align*}
Since $\max(N_1,N_2) \geq N_3^\epsilon$, we can bound the pre-factor by 
\begin{align*}
 &\max(N_{12},N_3)^{s_2-1}  N_{12}^{\frac{1}{2}-\beta+2\delta_1} \max(N_1,N_2)^{-\frac{1}{2}+4 \delta_1} N_3^{\frac{1}{2}+\kappa} 
\lesssim \max(N_{1},N_{2})^{-\beta+6\delta_1} N_3^{\delta_2+\kappa} \\
&\lesssim \max(N_1,N_2,N_3)^{-2\eta}. 
\end{align*}

\emph{Case 2: $N_{12} \sim N_3$.} By symmetry, we can assume that $N_1 \geq N_2$. Furthermore, we have that $N_3 \sim N_{12} \lesssim N_1$. Using the inhomogeneous Strichartz estimate 
(Lemma \ref{tools:lem_inhomogeneous_strichartz}), we obtain that 
\begin{align*}
&\Big\|  (P_{N_{12}} V ) \ast \big( P_{\leq N} P_{N_1} f \cdot  P_{\leq N} P_{N_2} g \big) P_{\leq N} P_{N_3} \<1b> \,\Big\|_{\X{s_{2}-1}{b_+-1}(\cJ)} \\
&\lesssim  (1+|\cJ|) \Big\| \langle \nabla  \rangle^{s_2 - \frac{1}{2} + 4 (b_+-\frac{1}{2})} \Big(  (P_{N_{12}} V ) \ast \big( P_{\leq N} P_{N_1} f \cdot  P_{\leq N} P_{N_2} g \big) P_{\leq N} P_{N_3} \<1b> \, \Big)\Big\|_{L_t^{\frac{4}{3}} L_x^{\frac{4}{3}}(\cJ\times \bT^3)} \\
&\lesssim (1+|\cJ|)^{\frac{3}{2}} N_3^{s_2 - \frac{1}{2} + 4 (b_+-\frac{1}{2})-\beta} \Big\| P_{N_1} f \Big\|_{L_t^\infty L_x^2(\cJ\times \bT^3)}  \Big\| P_{N_2} g \Big\|_{L_t^4 L_x^4(\cJ\times \bT^3)}   \Big\| P_{N_3} \, \<1b> \Big\|_{L_t^\infty L_x^\infty(\cJ\times \bT^3)} \\
&\lesssim  (1+|\cJ|)^2  N_1^{-s_1} N_2^{\frac{1}{2}-s_1}  N_3^{s_2 - \frac{1}{2} + 4 (b_+-\frac{1}{2})-\beta+ \frac{1}{2} +\kappa}  
\| f \|_{\X{s_1}{b}(\cJ)} 
\| g \|_{\X{s_1}{b}(\cJ)} 
\| \, \<1b> \, \|_{L_t^\infty \cC_x^{-1/2-\kappa}(\cJ\times \bT^3)}.
\end{align*}
Since $N_2,N_3\geq 1$, the pre-factor can be bounded by 
\begin{align*}
N_1^{-s_1} N_2^{\frac{1}{2}-s_1}  N_3^{s_2 - \frac{1}{2} + 4 (b_+-\frac{1}{2})-\beta+ \frac{1}{2} +\kappa} 
\lesssim N_1^{1-2 s_1 + s_2 - \frac{1}{2} + 4 (b_+-\frac{1}{2})-\beta +\kappa} 
= N_1^{2 \delta_1 + \delta_2 + 4 (b_+-\frac{1}{2}) +\kappa -\beta},
\end{align*}
which is acceptable. 
\end{proof}

\begin{proposition}\label{phy:prop_neq}
Let $T\geq 1$, let $J\subseteq [0,T]$ be an interval, and let $f,g\colon J \times \bT^3 \rightarrow \bR$. Then, it holds that 
\begin{align*}
 \sup_{N\geq 1} \Big\| V \ast \Big( P_{\leq N} \<1b> \paraneq  P_{\leq N}  f  \Big)  P_{\leq N} g \Big\|_{\X{s_2-1}{b_+-1}(\cJ)} 
&\lesssim (1+|\cJ|)^2  \| \, \<1b> \|_{L_t^\infty \cC^{-\frac{1}{2}-\kappa}(\cJ\times \bT^3)} \| f \|_{\X{s_1}{b}(\cJ)}  \| g \|_{\X{s_1}{b}(\cJ)}. 
\end{align*}
In the frequency-localized version of this estimate, which is detailed in the proof,  we gain an $\eta^\prime$-power in the maximal frequency-scale. 
\end{proposition}

\begin{proof}
By using a Littlewood-Paley decomposition and the definitions of $\paraneq$, we have that 
\begin{align*}
 V \ast \Big( P_{\leq N} \<1b> \paraneq  P_{\leq N}  f  \Big)  P_{\leq N} g = \sum_{\substack{N_1,N_2,N_3\colon \\ N_1 \not \sim N_2}}  V \ast \Big( P_{\leq N} P_{N_1}  \<1b>  \, \cdot P_{\leq N} P_{N_2}  f  \Big)  P_{\leq N} P_{N_3} g.
\end{align*}
We treat each dyadic block separately and distinguish two cases.

\emph{Case 1: $N_1 \gg N_2,N_3$}. 
Using the inhomogeneous Strichartz estimate (Lemma \ref{tools:lem_inhomogeneous_strichartz}), we have that 
\begin{align*}
&\Big\| V \ast \Big( P_{\leq N} P_{N_1}  \<1b>  \, \cdot P_{\leq N} P_{N_2}  f  \Big)  P_{\leq N} P_{N_3} g\Big\|_{\X{s_2-1}{b_+-1}(\cJ)} \\
&\lesssim \Big\| V \ast \Big( P_{\leq N} P_{N_1}  \<1b>  \, \cdot P_{\leq N} P_{N_2}  f  \Big)  P_{\leq N} P_{N_3} g\Big\|_{L_t^{2b_+} H_x^{s_2-1}(\cJ\times \bT^3)} \\ 
&\lesssim  (1+|\cJ|)^{\frac{1}{2}}  N_1^{s_2-1-\beta} \| P_{N_1} \<1b> \, \|_{L_t^\infty L_x^\infty(\cJ\times \bT^3)} \| P_{N_1} f \|_{L_t^4 L_x^4(\cJ\times \bT^3)} \| P_{N_2} g \|_{L_t^4 L_x^4(\cJ\times \bT^3)} \\
&\lesssim   (1+|\cJ|)  N_1^{s_2-1-\beta+ \frac{1}{2} + \kappa} N_2^{\frac{1}{2}-s_1} N_3^{\frac{1}{2}-s_1}  
\| \<1b> \, \|_{L_t^\infty \cC_x^{-\frac{1}{2}-\kappa}(\cJ\times \bT^3)}
\| f \|_{\X{s_1}{b}(\cJ)} \| g \|_{\X{s_1}{b}(\cJ)}.  
\end{align*}
Since $N_2,N_3 \ll N_1$, the pre-factor can  be bounded by 
\begin{align*}
 N_1^{s_2-1-\beta+ \frac{1}{2} + \kappa} N_2^{\frac{1}{2}-s_1} N_3^{\frac{1}{2}-s_1}   \lesssim N_1^{2 \delta_1 + \delta_2 + \kappa - \beta}, 
\end{align*}
which is acceptable. \\

\emph{Case 2.a: $N_1 \ll N_2$, $N_3 \lesssim N_2$.}
Using the inhomogeneous Strichartz estimate (Lemma \ref{tools:lem_inhomogeneous_strichartz}), we have that 
\begin{align*}
&\Big\| V \ast \Big( P_{\leq N} P_{N_1}  \<1b>  \, \cdot P_{\leq N} P_{N_2}  f  \Big)  P_{\leq N} P_{N_3} g\Big\|_{\X{s_2-1}{b_+-1}(\cJ)} \\
&\lesssim  (1+|\cJ|) \Big\| \langle \nabla \rangle^{s_2 -\frac{1}{2} + 4 (b_+-1)} \Big( V \ast \Big( P_{\leq N} P_{N_1}  \<1b>  \, \cdot P_{\leq N} P_{N_2}  f  \Big)  P_{\leq N} P_{N_3} g \Big)
\Big\|_{L_t^{\frac{4}{3}} L_x^{\frac{4}{3}}(\cJ\times \bT^3)} \\
&\lesssim  (1+|\cJ|) N_2^{s_2 -\frac{1}{2} + 4 (b_+-1)}   \| V \ast \big( P_{\leq N} P_{N_1}  \<1b>  \, \cdot P_{\leq N} P_{N_2}  f  \big) \|_{L_t^2 L_x^2(\cJ\times \bT^3)}  
\|  P_{\leq N} P_{N_3} g \|_{L_t^4 L_x^4(\cJ\times \bT^3)} \\
& \lesssim  (1+|\cJ|) N_2^{s_2 -\frac{1}{2} + 4 (b_+-1)-\beta} 
 \| \<1b> \, \|_{L_t^\infty L_x^\infty(\cJ\times \bT^3)}
  \| P_{N_2}  f  \|_{L_t^2 L_x^2(\cJ\times \bT^3)}  
\|  P_{N_3} g \|_{L_t^4 L_x^4(\cJ\times \bT^3)} \\
&\lesssim  (1+|\cJ|)^2 
N_1^{\frac{1}{2}+\kappa} N_2^{s_2 -\frac{1}{2} + 4 (b_+-1)-\beta-s_1} N_3^{\frac{1}{2}-s_1} 
\| \<1b> \, \|_{L_t^\infty \cC_x^{-\frac{1}{2}-\kappa}(\cJ\times \bT^3)}
\| f \|_{\X{s_1}{b}(\cJ)} \| g \|_{\X{s_1}{b}(\cJ)}.  
\end{align*}
The pre-factor can now be bounded as before. \\

\emph{Case 2.b: $N_1 \ll N_2$, $N_2\ll N_3$. }
Using the inhomogeneous Strichartz estimate (Lemma \ref{tools:lem_inhomogeneous_strichartz}), we have that 
\begin{align*}
&\Big\| V \ast \Big( P_{\leq N} P_{N_1}  \<1b>  \, \cdot P_{\leq N} P_{N_2}  f  \Big)  P_{\leq N} P_{N_3} g\Big\|_{\X{s_2-1}{b_+-1}(\cJ)} \\
&\lesssim  (1+|\cJ|) \Big\| \langle \nabla \rangle^{s_2 -\frac{1}{2} + 4 (b_+-1)} \Big( V \ast \Big( P_{\leq N} P_{N_1}  \<1b>  \, \cdot P_{\leq N} P_{N_2}  f  \Big)  P_{\leq N} P_{N_3} g \Big)
\Big\|_{L_t^{\frac{4}{3}} L_x^{\frac{4}{3}}(\cJ\times \bT^3)} \\
&\lesssim  (1+|\cJ|) \max(N_1,N_2)^{-\beta} N_3^{s_2 -\frac{1}{2} + 4 (b_+-1)}  \| \<1b> \, \|_{L_t^\infty L_x^\infty(\cJ\times \bT^3)}
  \| P_{N_2}  f  \|_{L_t^4 L_x^4(\cJ\times \bT^3)}  
\|  P_{N_3} g \|_{L_t^2 L_x^2(\cJ\times \bT^3)} \\
&\lesssim  (1+|\cJ|)^2   \max(N_1,N_2)^{-\beta} N_1^{\frac{1}{2}+\kappa} N_2^{\frac{1}{2}-s_1} N_3^{s_2 -\frac{1}{2} + 4 (b_+-1)-s_1}
\| \<1b> \, \|_{L_t^\infty \cC_x^{-\frac{1}{2}-\kappa}(\cJ\times \bT^3)}
\| f \|_{\X{s_1}{b}(\cJ)} \| g \|_{\X{s_1}{b}(\cJ)}.  
\end{align*}
The pre-factor can now be bounded by 
\begin{equation*}
\max(N_1,N_2)^{-\beta} N_1^{\frac{1}{2}+\kappa} N_2^{\frac{1}{2}-s_1} N_3^{s_2 -\frac{1}{2} + 4 (b_+-1)-s_1} 
\lesssim N_1^{\frac{1}{2}+\kappa-\beta} N_2^{\delta_1} N_3^{-\frac{1}{2}+\delta_1+\delta_2 + 4 (b_+-1)} 
\lesssim N_3^{2 \delta_1 + \delta_2 + \kappa - \beta},
\end{equation*}
which is acceptable. \\
\end{proof}

\begin{lemma}[Bilinear physical estimate]\label{phy:lem_bilinear_tool}
Let $\cJ \subseteq \bR$ be a bounded interval. If $\Psi, f \colon \cJ \times \bT^3 \rightarrow \bC$, then
\begin{equation*}
\| \big( V \ast \Psi \big) f \|_{\X{s_2-1}{b_+-1}(\cJ\times \bT^3)} \lesssim (1+|\cJ|)^{\frac{3}{2}} \| \Psi \|_{L_t^2 H_x^{-4\delta_1}(\cJ\times \bT^3)}
 \min\big( \| f \|_{L_t^\infty \cC_x^{\beta-\kappa}(\cJ \times \bT^3)}, \| f \|_{\X{s_1}{b}(\cJ)} \big). 
\end{equation*}
In the frequency-localized version of this estimate we also gain an $\eta^\prime$-power in the maximal frequency-scale. 
\end{lemma}

Lemma \ref{phy:lem_bilinear_tool} can be combined with our bound on $\<1p> \paraeq w_N$ in the stability theory (see Section \ref{section:stability}). In the local theory, its primary application is isolated in the following corollary. 

\begin{corollary}\label{phy:cor_bilinear}
Let $\cJ \subseteq \bR$ be a bounded interval and let $w,Y\colon \cJ \times \bT^3 \rightarrow \bR$. Then, we have uniformly in $N\geq 1$ that  
\begin{align*}
& \Big \| V \ast \Big( P_{\leq N} \<1b> \paraeq  P_{\leq N} Y\Big)  P_{\leq N} \<3DN> \Big\|_{\X{s_2-1}{b_+-1}(\cJ)} \\
& \lesssim  (1+|\cJ|)^2  \| \, \<1b>\, \|_{L_t^\infty \cC_x^{-\frac{1}{2}+\kappa}(\cJ\times \bT^3)}\| Y \|_{\X{s_2}{b}(\cJ)} \Big\|    \<3DN> \Big\|_{L_t^\infty \cC_x^{\beta-\kappa}(\cJ\times \bT^3)}, \\
&\Big \| V \ast \Big( P_{\leq N} \<1b> \paraeq  P_{\leq N}  Y   \Big)  P_{\leq N} w  \Big\|_{\X{s_2-1}{b_+-1}(\cJ)}  \\
&\lesssim  (1+|\cJ|)^2  \| \, \<1b> \,  \|_{L_t^\infty \cC_x^{-\frac{1}{2}+\kappa}(\cJ\times \bT^3)} \| Y \|_{\X{s_2}{b}(\cJ)} \| w\|_{\X{s_1}{b}(\cJ)}. 
\end{align*}
\end{corollary}

\begin{proof}[Proof of Corollary \ref{phy:cor_bilinear}:]
We have that 
\begin{align*}
\|  P_{\leq N} \<1b> \paraeq  P_{\leq N} Y\|_{L_t^2 H_x^{-4 \delta_1}(\cJ\times \bT^3)} 
&\lesssim  |\cJ|^{\frac{1}{2}} \| \, \<1b>\,\|_{L_t^\infty \cC_x^{-\frac{1}{2}-\kappa}(\cJ \times \bT^3)} \| Y\|_{L_t^\infty H_x^{s_2}(\cJ \times \bT^3)}  \\
&\lesssim   |\cJ|^{\frac{1}{2}} \| \, \<1b>\,\|_{L_t^\infty \cC_x^{-\frac{1}{2}-\kappa}(\cJ \times \bT^3)} \| Y \|_{\X{s_2}{b}(\cJ)}. 
\end{align*}
Together with Lemma \ref{phy:cor_bilinear}, this implies the corollary. 
\end{proof}

\begin{proof}[Proof of Lemma \ref{phy:lem_bilinear_tool}:] 
Let $0 \leq \theta \ll \beta$ remain to be chosen. Using the inhomogeneous Strichartz estimate and (a weaker version of) the fractional product rule, we have that 
\begin{align*}
&\| \big( V \ast \Psi \big) f \|_{\X{s_2-1}{b_+-1}(\cJ\times \bT^3)} \\
&\lesssim  (1+|\cJ|) \| \langle \nabla \rangle^{s_2 - \frac{1}{2} + 4 (b_+-\frac{1}{2})} \big(  \big( V \ast \Psi \big) f \big) \|_{L_t^{\frac{4}{3}} L_x^{\frac{4}{3}}(\cJ\times \bT^3) } \\
&\lesssim  (1+|\cJ|)  \| \langle \nabla \rangle^{s_2 - \frac{1}{2} + 4 (b_+-\frac{1}{2})} \big( V \ast \Psi \big) \|_{L_t^{2} L_x^{\frac{4}{2-\theta}}(\cJ\times \bT^3) }
\| \langle \nabla \rangle^{s_2 - \frac{1}{2} + 4 (b_+-\frac{1}{2})} f \|_{L_t^4 L_x^{\frac{4}{1+\theta}}(\cJ \times \bT^3)}. 
\end{align*}
Using Sobolev embedding, the first factor is bounded by 
\begin{align*}
 \| \langle \nabla \rangle^{s_2 - \frac{1}{2} + 4 (b_+-\frac{1}{2})} \big( V \ast \Psi \big) \|_{L_t^{2} L_x^{\frac{4}{2-\theta}}(\cJ\times \bT^3) } 
&\lesssim  \| \langle \nabla \rangle^{s_2 - \frac{1}{2} + 4 (b_+-\frac{1}{2})+\frac{3\theta}{4}-\beta} \Psi  \|_{L_t^{2} L_x^{2}(\cJ\times \bT^3) } \\
&\lesssim \| \Psi \|_{L_t^{2} H_x^{-4 \delta_1}(\cJ\times \bT^3)}.
\end{align*}
Thus, it remains to present two different estimates of the second factor. By simply choosing $\theta=0$, we see that 
\begin{equation*}
\| \langle \nabla \rangle^{s_2 - \frac{1}{2} + 4 (b_+-\frac{1}{2})} f \|_{L_t^4 L_x^{4}(\cJ \times \bT^3)} \lesssim (1+|\cJ|)^{\frac{1}{4}} \| f \|_{L_t^\infty \cC_x^{\beta-\kappa}(\cJ \times \bT^3)},
\end{equation*}
which yields the first term in the minimum. Using H\"{o}lder's inequality in time and Strichartz estimates, we also have that 
\begin{align*}
\| \langle \nabla \rangle^{s_2 - \frac{1}{2} + 4 (b_+-\frac{1}{2})} f \|_{L_t^4 L_x^{\frac{4}{1+\theta}}(\cJ \times \bT^3)}
&\lesssim (1+|\cJ|)^{\frac{\theta}{4}} \| \langle \nabla \rangle^{s_2 - \frac{1}{2} + 4 (b_+-\frac{1}{2})} f \|_{L_t^{\frac{4}{1-\theta}} L_x^{\frac{4}{1+\theta}}(\cJ \times \bT^3)} \\
&\lesssim (1+|\cJ|)^{\frac{1}{4}} \| f \|_{\X{s_1}{b}(\cJ)}, 
\end{align*}
provided that 
\begin{equation*}
s_2 - \frac{1}{2} + 4 (b_+-\frac{1}{2}) + \frac{3}{2} - \frac{1-\theta}{4} - 3 \frac{1+\theta}{4} \leq s_1. 
\end{equation*}
The last condition can be satisfied by choosing $\theta=4\delta_1$, which also satisfies $\theta \ll \beta$. 
\end{proof}

\begin{proposition}\label{phy:prop_3b}
Let $\cJ \subseteq \bR$ be a bounded interval and let $f,g,h\colon J \times \bT^3$. Then, it holds that 
\begin{equation}
\begin{aligned}
&\sup_{N\geq 1} \bigg\|  V\ast \Big( P_{\leq N}f  \cdot P_{\leq N}  g \Big) 
 P_{\leq N} h \,\bigg\|_{\X{s_{2}-1}{b_+-1}(\cJ)} \\
&\lesssim  (1+|\cJ|)^2
\prod_{\varphi=f,g,h}\min\Big( \| \varphi \|_{L_t^\infty \cC_x^{\beta-\kappa}(\cJ\times \bT^3)}, \| \varphi \|_{\X{s_1}{b}(\cJ)}\Big). 
\end{aligned}
\end{equation}
In the frequency-localized version of this estimate  we also gain an $\eta^\prime$-power in the maximal frequency-scale. 
\end{proposition}

\begin{remark}
In applications of Lemma \ref{phy:prop_3b}, we will choose $f,g$, and $h$ as either $\<3DN>$, which is contained in $L_t^\infty \cC_x^{\beta-\kappa}$, or $w_N$, which is contained in $\X{s_1}{b}$. 
\end{remark}
\begin{proof}
Since the proof is relatively standard, we only present the argument when all functions $f,g$, and $h$ are placed in the same space. The intermediate cases follow from a combination of our arguments below. \\

\emph{Estimate for $L_t^\infty \cC_x^{\beta-\kappa}$:} Using the inhomogeneous Strichartz estimate (Lemma \ref{tools:lem_inhomogeneous_strichartz}) and $s_2 \leq 1$, we have that
\begin{align*}
& \bigg\|  V\ast \Big( P_{\leq N}f  \cdot P_{\leq N}  g \Big)  P_{\leq N} h \,\bigg\|_{\X{s_{2}-1}{b_+-1}(\cJ)} 
\lesssim \Big \|  V\ast \Big( P_{\leq N}f  \cdot P_{\leq N}  g \Big)  P_{\leq N} h \,\Big\|_{L_t^{2b_+} L_x^2(\cJ\times \bT^3)}  \allowdisplaybreaks[3] \\
&\lesssim  (1+|\cJ|)  \prod_{\varphi=f,g,h} \| \varphi \|_{L_t^\infty L_x^\infty (\cJ\times \bT^3)} 
\lesssim (1+|\cJ|) \prod_{\varphi=f,g,h} \| \varphi \|_{L_t^\infty \cC_x^{\beta-\kappa}(\cJ\times \bT^3)}. 
\end{align*}

\emph{Estimate for $\X{s_1}{b}(\cJ)$:} Let $0<\theta\ll 1$ remain to be chosen. Using the inhomogeneous Strichartz estimate (Lemma \ref{tools:lem_inhomogeneous_strichartz}), we have that 
\begin{align*}
& \Big\|  V\ast \Big( P_{\leq N}f  \cdot P_{\leq N}  g \Big)  P_{\leq N} h \,\Big\|_{\X{s_{2}-1}{b_+-1}(\cJ)} \\
&\lesssim  (1+|\cJ|)  \Big\|  \langle \nabla \rangle^{s_2-\frac{1}{2} +4 (b_+-\frac{1}{2})} \Big(  V\ast \Big( P_{\leq N}f  \cdot P_{\leq N}  g \Big)  P_{\leq N} h \Big) \,\Big\|_{L_t^{\frac{4}{3}} L_x^{\frac{4}{3}}(\cJ\times \bT^3)}
\\
&\lesssim   (1+|\cJ|) 
\Big\| \langle \nabla \rangle^{s_2-\frac{1}{2} +4 (b_+-\frac{1}{2})}  \Big(  V\ast \Big( P_{\leq N}f  \cdot P_{\leq N}  g \Big)  \Big) \Big\|_{L_t^{\frac{4}{2-\theta}}L_x^{\frac{4}{2-\theta}}(\cJ\times \bT^3)} \\
&\times 
\big\| \langle \nabla \rangle^{s_2-\frac{1}{2} +4 (b_+-\frac{1}{2})}   h \big\|_{L_t^{\frac{4}{1+\theta}}L_x^{\frac{4}{1+\theta}}(\cJ\times \bT^3)} . 
\end{align*}
Using Lemma \ref{phy:lem_bilinear_kt}, the first term is bounded by $ (1+|\cJ|) ^{\frac{2-\theta}{4}} \| f\|_{\X{s_1}{b}(\cJ)}\| g\|_{\X{s_1}{b}(\cJ)}$ as long as
\begin{equation}\label{phy:eq_cond_1}
2 \delta_1 + \delta_2 + 4 \big( b_+ -\frac{1}{2} \big) + \theta < \beta. 
\end{equation}
Using H\"{o}lder's inequality in the time-variable and the linear Strichartz estimate, we have that 
\begin{align*}
\big\| \langle \nabla \rangle^{s_2-\frac{1}{2} +4 (b_+-\frac{1}{2})}   h \big \|_{L_t^{\frac{4}{1+\theta}}L_x^{\frac{4}{1+\theta}}(\cJ\times \bT^3)} 
&\lesssim  (1+|\cJ|)^{\frac{\theta}{2}} \big\| \langle \nabla \rangle^{s_2-\frac{1}{2} +4 (b_+-\frac{1}{2})}   h \big \|_{L_t^{\frac{4}{1-\theta}}L_x^{\frac{4}{1+\theta}}(\cJ\times \bT^3)} \\
&\lesssim  (1+|\cJ|)^{\frac{1+\theta}{4}} \| h \|_{\X{s_1}{b}(\cJ)},
\end{align*}
provided that 
\begin{equation}\label{phy:eq_cond_2}
\frac{\theta}{2}> \delta_1 + \delta_2 + 4 \big( b_+ - \frac{1}{2} \big). 
\end{equation}
In order to satisfy both conditions \eqref{phy:eq_cond_1} and \eqref{phy:eq_cond_2}, we can choose $\theta=4 \delta_1$. 
\end{proof}

\subsection{Hybrid physical-RMT terms}\label{section:hybrid}

In this subsection, we estimate the remaining terms in $\Phy$.
Our estimates will be phrased as bounds on the operator norm of certain random operators.  In contrast to Proposition \ref{rmt:prop1} and Proposition \ref{rmt:prop2}, however, we will not need the moment method (from \cite{DNY20}). Instead, we will rely on Strichartz estimates and the estimates for the quartic stochastic object from Section \ref{section:quartic}.

\begin{proposition}\label{phy:prop_hybrid}
Let $T\geq 1$ and $p\geq 1$. Then, we have the following three estimates: 
\begin{align}\label{phy:estimate_hybrid_1}
&\Big\| \sup_{N\geq 1} \sup_{\cJ\subseteq [0,T]} \sup_{\| w\|_{\X{s_1}{b}(\cJ)}\leq 1} \Big\|  V \ast \Big( P_{\leq N} \<1b> \cdot P_{\leq N}\<3DN> \Big) P_{\leq N} w \Big\|_{\X{s_2-1}{b_+-1}(\cJ)} \Big\|_{L^p_\omega(\bP)} 
\lesssim T^3 p^2, \\
&\Big\| \sup_{N\geq 1} \sup_{\cJ\subseteq [0,T]} \sup_{\| w\|_{\X{s_1}{b}(\cJ)}\leq 1} \Big\|  V \ast \Big( P_{\leq N} \<1b> \paraneq P_{\leq N} w \Big) P_{\leq N}\<3DN>  \Big\|_{\X{s_2-1}{b_+-1}(\cJ)} \Big\|_{L^p_\omega(\bP)} 
\lesssim T^3 p^2,  \label{phy:estimate_hybrid_2}\\
&\Big\| \sup_{N\geq 1} \sup_{\cJ\subseteq [0,T]} \sup_{\| w\|_{\X{s_1}{b}(\cJ)}\leq 1} \Big\| V \ast \Big( P_{\leq N}\<3DN> \cdot P_{\leq N} w  \Big) \nparald P_{\leq N} \<1b> \Big\|_{\X{s_2-1}{b_+-1}(\cJ)}  \Big\|_{L^p_\omega(\bP)} \label{phy:estimate_hybrid_3} \\
&\lesssim T p^2. \notag  
\end{align}
\end{proposition}

\begin{remark}
In the frequency-localized versions of \eqref{phy:estimate_hybrid_1}, \eqref{phy:estimate_hybrid_2}, and \eqref{phy:estimate_hybrid_3}, we also gain an $\eta^\prime$-power of the maximal frequency-scale. Similar as in Proposition \ref{so4:prop} and Remark \ref{so4:rem}, we may also replace $\<3DN> $ by $\<3DNtau>$.  
\end{remark}

\begin{proof}
We first prove \eqref{phy:estimate_hybrid_1}, which is the easiest part. Using the inhomogeneous Strichartz estimate (Lemma \ref{tools:lem_inhomogeneous_strichartz}), $s_2-1<-s_1$, and the (dual of) the fractional product rule, we have that 
\begin{align*}
&\Big\|  V \ast \Big( P_{\leq N} \<1b> \cdot P_{\leq N}\<3DN> \Big) P_{\leq N} w \Big\|_{\X{s_2-1}{b_+-1}(\cJ)} \\
&\lesssim \Big\|  V \ast \Big( P_{\leq N} \<1b> \cdot P_{\leq N}\<3DN> \Big) P_{\leq N} w \Big\|_{L_t^{2b_+} H_x^{s_2-1}(\cJ)} \\
&\lesssim \Big\|  V \ast \Big( P_{\leq N} \<1b> \cdot P_{\leq N}\<3DN> \Big) P_{\leq N} w \Big\|_{L_t^{2b_+} H_x^{-s_1}(\cJ)} \\
&\lesssim T \Big\|  V \ast \Big( P_{\leq N} \<1b> \cdot P_{\leq N}\<3DN> \Big) \Big\|_{L_t^\infty \cC_x^{-s_1+\eta}([0,T] \times \bT^3)} \| w \|_{L_t^\infty H_x^{s_1}(\cJ\times \bT^3)}. 
\end{align*}
Using \eqref{so4:eq_estimate_3} in Proposition \ref{so4:prop}, this implies \eqref{phy:estimate_hybrid_1}. \\

We now turn to \eqref{phy:estimate_hybrid_2} and \eqref{phy:estimate_hybrid_3}, which are more difficult. The main step consists of the following estimate: For any  $M_1,N_1,K_1,K_2 \geq 1$, we have that 
\begin{equation}\label{phy:eq_hybrid_p1}
\begin{aligned}
&\bigg\| \sup_{N\geq1} \sup_{t\in [0,T]} \sup_{\| f\|_{H_x^{s_1}}, \| g \|_{H_x^{s_1}} \leq 1} 
\Big| \int_{\bT^3} V \ast \big( P_{M_1} P_{\leq N} \<3DN> \cdot P_{K_1} P_{\leq N} f \big)  P_{N_1} P_{\leq N} \<1b> \cdot P_{K_2} P_{\leq N} g \dx \Big| \bigg\|_{L^p_\omega(\bP)} \\
&\lesssim T^3  \max(K_1,K_2,N_1,M_1)^{-\eta} \Big(1 + 1\big\{ N_1 \sim K_2 \big\} M_1^{-\beta+\kappa+\eta} K_1^{-s_1+\eta} N_1^{\frac{1}{2}+\kappa-s_1} \Big) p^2. 
\end{aligned}
\end{equation}
For notational convenience, we now omit the multiplier $P_{\leq N}$. As will be evident from the proof, the same argument applies (uniformly in $N$) with the multiplier. The proof of \eqref{phy:eq_hybrid_p1} splits into two cases. The impatient reader may wish to skim ahead to Case 2.b, which contains the most interesting part of the argument. \\ 

\emph{Case 1: $M_1 \not \sim N_1$.} From Fourier support considerations, it follows that $\max(K_1,K_2) \gtrsim \max(N_1,M_1)$. Then, we estimate the integral in \eqref{phy:eq_hybrid_p1} by 
\begin{align}
&\Big| \int_{\bT^3} V \ast \big( P_{M_1}  \<3DN> \cdot P_{K_1}  f \big)  P_{N_1}  \<1b> \cdot P_{K_2}  g \dx \Big|  \notag \\
&\lesssim \sum_{L \lesssim \max(N_1,K_2)}\Big| \int_{\bT^3} (P_L V) \ast \big( P_{M_1}  \<3DN> \cdot P_{K_1}  f \big)  \cdot \widetilde{P}_L \big(  P_{N_1}  \<1b> \cdot P_{K_2}  g \big) \dx \Big|  \notag \\
&\lesssim \sum_{L \lesssim \max(N_1,K_2)} \| P_L V \|_{L_x^1}  \Big\| P_{M_1}  \<3DN> \cdot P_{K_1}  f  \Big\|_{L_x^2}  \Big\| \widetilde{P}_L \big(  P_{N_1}  \<1b> \cdot P_{K_2}  g \big) \dx \Big\|_{L_x^2}  \notag \\
&\lesssim M_1^{-\beta+\kappa} K_1^{-s_1} \Big\|   \<3DN> \Big\|_{\cC_x^{\beta-\kappa}}
\sum_{L \lesssim \max(N_1,K_2)} L^{-\beta}  \Big\| \widetilde{P}_L \big(  P_{N_1}  \<1b> \cdot P_{K_2}  g \big) \dx \Big\|_{L_x^2}. \label{phy:eq_hybrid_p2}
\end{align}
We now further split the argument into two subcases. \\

\emph{Case 1.a: $M_1 \not \sim N_1$, $K_2 \not \sim N_1$.} Then, we only obtain a non-trivial contribution if $L\sim \max(N_1,K_2)$. Using $\max(K_1,K_2) \gtrsim \max(M_1,N_1) \geq N_1$, we obtain that 
\begin{align*}
\eqref{phy:eq_hybrid_p2} &\lesssim  M_1^{-\beta+\kappa} K_1^{-s_1}  \max(K_2,N_1)^{-\beta}  K_2^{-s_1} N_1^{\frac{1}{2}+\kappa} 
\Big\|   \<3DN> \Big\|_{\cC_x^{\beta-\kappa}} \big\| \, \<1b> \, \big\|_{\cC_x^{-\frac{1}{2}-\kappa}} \\
&\lesssim  M_1^{-\beta+\kappa} K_1^{-\eta} K_2^{-\eta} N_1^{\frac{1}{2}+\kappa+\eta-\beta-s_1} 
\Big\|   \<3DN> \Big\|_{\cC_x^{\beta-\kappa}} \big\| \, \<1b> \, \big\|_{\cC_x^{-\frac{1}{2}-\kappa}}.
\end{align*}
The pre-factor is bounded by $(M_1 K_1 K_2 N_1)^{-\eta}$, which is acceptable. \\

\emph{Case 1.b: $M_1 \not \sim N_1$, $K_2 \sim N_1$.} In this case, the worst case corresponds to $L \sim 1$. Using only H\"{o}lder's inequality, we obtain that 
\begin{align*}
\eqref{phy:eq_hybrid_p2} &\lesssim 1\big\{ K_2 \sim N_1 \big\} M_1^{-\beta+\kappa} K_1^{-s_1} N_1^{\frac{1}{2}+\kappa-s_1}  \Big\|   \<3DN> \Big\|_{\cC_x^{\beta-\kappa}} \big\| \, \<1b> \, \big\|_{\cC_x^{-\frac{1}{2}-\kappa}}.
\end{align*}
This case is responsible for the second summand in \eqref{phy:eq_hybrid_p1}. \\

\emph{Case 2: $M_1 \sim N_1$}. This case is more delicate and requires the estimates on the quartic stochastic objects from Section \ref{section:quartic}. Inspired by the uncertainty principle, we decompose 
\begin{align*}
&\Big| \int_{\bT^3} V \ast \big( P_{M_1}  \<3DN> \cdot P_{K_1}  f \big)  P_{N_1}  \<1b> \cdot P_{K_2}  g \dx \Big|   \\
&\leq \Big| \int_{\bT^3} (P_{\ll N_1} V) \ast \big( P_{M_1}  \<3DN> \cdot P_{K_1}  f \big)  P_{N_1}  \<1b> \cdot P_{K_2}  g \dx \Big|   \\
&+\Big| \int_{\bT^3} (P_{\gtrsim N_1} V) \ast \big( P_{M_1}  \<3DN> \cdot P_{K_1}  f \big)  P_{N_1}  \<1b> \cdot P_{K_2}  g \dx \Big| . 
\end{align*}
We estimate both terms separately and hence divide the argument into two subcases. \\

\emph{Case 2.a: $M_1 \sim N_1$, contribution of $P_{\ll N_1} V$.} For this term, we only obtain a non-trivial contribution if $K_1 \sim K_2 \sim N_1$. Using H\"{o}lder's inequality and Young's convolution inequality, we obtain that 
\begin{align*}
 &\Big| \int_{\bT^3} (P_{\ll N_1} V) \ast \big( P_{M_1}  \<3DN> \cdot P_{K_1}  f \big)  P_{N_1}  \<1b> \cdot P_{K_2}  g \dx \Big|   \\
&\lesssim  1\big\{ K_1 \sim K_2 \sim M_1 \sim N_1 \big\} \| P_{\ll N_1} V \|_{L_x^1} \Big\|  P_{M_1}  \<3DN>  \Big \|_{L_x^\infty}  \| P_{K_1}  f  \|_{L_x^2} 
\big \|   P_{N_1}  \<1b> \big\|_{L_x^\infty} \|  P_{K_2}  g  \|_{L_x^2} \\
&\lesssim  1\big\{ K_1 \sim K_2 \sim M_1 \sim N_1 \big\} N_1^{\frac{1}{2} + 2 \kappa - \beta - 2 s_1} 
\Big\|   \<3DN> \Big\|_{\cC_x^{\beta-\kappa}} \big\| \, \<1b> \, \big\|_{\cC_x^{-\frac{1}{2}-\kappa}}.
\end{align*}
The pre-factor is easily bounded by (and generally much smaller than)  $(M_1 K_1 K_2 N_1)^{-\eta}$. \\

\emph{Case 2.b: $M_1 \sim N_1$, contribution of $P_{\gtrsim N_1} V$.}
By expanding the convolution with the interaction potential, we obtain that 
\begin{align*}
 &\Big| \int_{\bT^3} (P_{\gtrsim  N_1} V) \ast \big( P_{M_1}  \<3DN> \cdot P_{K_1}  f \big)  P_{N_1}  \<1b> \cdot P_{K_2}  g \dx \Big|   \\
 &\leq \int_{\bT^3} |P_{\gtrsim N_1} V(y)|   \bigg|  \int_{\bT^3} \big( P_{K_1}  f (x-y)  \cdot P_{K_2}  g(x)  \big) \cdot \Big( P_{M_1}  \<3DN> (t,x-y) \cdot  P_{N_1}  \<1b> (t,x) \Big) \dx \bigg|  \dy  \\
&\lesssim \| P_{\gtrsim N_1} V(y) \|_{L_y^1}  
 \cdot  \sup_{y \in \bT^3} \| \langle \nabla_x \rangle^{\frac{1}{2}-\beta+2 \kappa} \big(  P_{K_1}  f (x-y)  \cdot P_{K_2}  g(x) \big) \|_{L_x^1}  \\
&\times \sup_{y \in \bT^3} \Big \| P_{M_1}  \<3DN> (t,x-y) \cdot  P_{N_1}  \<1b> (t,x) \Big\|_{\cC_x^{-\frac{1}{2}+\beta-\kappa}} \\
&\lesssim N_1^{-\beta} K_1^{-\eta} K_2^{-\eta}  \sup_{y \in \bT^3} \Big \| P_{M_1}  \<3DN> (t,x-y) \cdot  P_{N_1}  \<1b> (t,x) \Big\|_{\cC_x^{-\frac{1}{2}+\beta-\kappa}}  . 
\end{align*}
Using Proposition \ref{so4:prop}, this contribution is acceptable. We note that the pre-factor $N_1^{-\beta}$ is essential, since Proposition \ref{so4:prop} is not uniformly bounded over all frequency scales. \\

By combining Case 1 and Case 2, we have finished the proof of \eqref{phy:eq_hybrid_p1}. It remains to show that \eqref{phy:eq_hybrid_p1} implies \eqref{phy:estimate_hybrid_2} and \eqref{phy:estimate_hybrid_3}.
To simplify the notation, we denote the expression inside the $L_\omega^p$-norm in \eqref{phy:eq_hybrid_p1} by 
\begin{equation}
\mathcal{A}(K_1,K_2,M_1,N_1) \defe \sup_{t\in [0,T]} \sup_{\| f\|_{H_x^{s_1}}, \| g \|_{H_x^{s_1}} \leq 1} 
\Big| \int_{\bT^3} V \ast \big( P_{M_1}   \<3DN> \cdot P_{K_1} f \big)  P_{N_1} \<1b> \cdot P_{K_2}g \dx \Big|. 
\end{equation} 

To see  \eqref{phy:estimate_hybrid_2}, we use the self-adjointness of $V$, duality, and $ s_1 < 1- s_2$,  which leads to  
\begin{align*}
& \Big\|  V \ast \Big(  \<1b> \paraneq   w \Big) \<3DN>  \Big\|_{H_x^{s_2-1}} \\
&\leq \sum_{K_1,K_2,M_1,N_1} 1\big\{ K_2 \not \sim N_1 \big\}  \Big\|  P_{K_1} \Big( V \ast \big(  P_{N_1} \<1b> \cdot  P_{K_2}  w \big) P_{M_1} \<3DN> \Big)  \Big\|_{H_x^{s_2-1}} \\
&\lesssim \Big( \sum_{K_1,K_2,M_1,N_1} 1\big\{ K_2 \not \sim N_1 \big\} \mathcal{A}(K_1,K_2,M_1,N_1) \Big) \| w \|_{H_x^{s_1}}. 
\end{align*}
After using the inhomogeneous Strichartz estimate and \eqref{phy:eq_hybrid_p1}, this completes the argument. 

Finally, we turn to \eqref{phy:estimate_hybrid_3}. Using duality, we have that
\begin{align*}
&\Big\| V \ast \Big(  \<3DN> \cdot  w  \Big) \nparald  \<1b> \Big\|_{H_x^{s_2-1}} \\
&\leq \sum_{K_1,K_2,M_1,N_1} 1\big\{ \max(M_1,K_1) \geq N_1^\epsilon \big\} \Big\| P_{K_2} \Big( V \ast \Big(  P_{M_1} \<3DN> \cdot P_{K_1} w  \Big)    P_{N_1} \<1b> \Big) \Big\|_{H_x^{s_2-1}} \\
&\lesssim \sum_{K_1,K_2,M_1,N_1} 1\big\{ \max(M_1,K_1) \geq N_1^\epsilon \big\} K_2^{s_1+s_2-1} \Big\| P_{K_2} \Big( V \ast \Big(  P_{M_1} \<3DN> \cdot P_{K_1} w  \Big)    P_{N_1} \<1b> \Big) \Big\|_{H_x^{-s_1}} \\
&\lesssim  \sum_{K_1,K_2,M_1,N_1} 1\big\{ \max(M_1,K_1) \geq N_1^\epsilon \big\} K_2^{s_1+s_2-1}   \mathcal{A}(K_1,K_2,M_1,N_1) \| w\|_{H_x^{s_1}}.  
\end{align*}
We now note that $ \max(M_1,K_1) \geq N_1^\epsilon $ implies 
\begin{equation*}
1\big\{ N_1 \sim K_2 \big\} M_1^{-\beta+\kappa+\eta} K_1^{-s_1+\eta} K_2^{s_1+s_2-1} N_1^{\frac{1}{2}+\kappa-s_1}  
 \lesssim N_1^{-\epsilon \min( \beta- \kappa -\eta, 1/2-\delta_1-\eta)} N_1^{\kappa+\delta_2} \lesssim 1. 
\end{equation*}
In the last inequality, we used the parameter conditions \eqref{intro:eq_parameter_condition}. We also emphasize that the factor $K_2^{s_1+s_2-1}$ is essential for this inequality. Using inhomogeneous Strichartz estimate and \eqref{phy:eq_hybrid_p1}, we then obtain the desired estimate.
\end{proof}

\section{From free to Gibbsian random structures}\label{section:ftg}

In the previous four sections, we proved several estimates for stochastic objects, random matrices, and para-controlled structures based on $\,\bluedot\,$. In Section \ref{section:local}, these estimates were used to prove the local convergence of the truncated dynamics as $N$ tends to infinity. Unfortunately, the object $\,\bluedot\,$ only exists on the ambient probability space and the global theory requires (intrinsic) estimates for $\purpledot$ with respect to the Gibbs measure. If the desired estimate does not rely on the invariance of $\mup_M$ under the nonlinear flow, however, we can use Theorem \ref{theorem:measures} to replace the Gibbs measure $\mup_M$ by the reference measure $\nup_M$. In particular, this works for stochastic objects only depending on the linear evolution of $\purpledot$, such as $\<1p>$ or $\<3Np>$. Once we are working with the reference measure $\nup_M$, we can then use that 
\begin{equation*}
\nup_M = \Law_{\bP}\big( \, \bluedot + \reddotM\,\big). 
\end{equation*}
Since $\reddotM$ has spatial regularity $1/2+\beta-$, we expect that our estimates for $\bluedot$ will imply the same estimates for $\purpledot\,$. As a result, this section contains no inherently new estimates and only combines our previous bounds. \\

\subsection{The Gibbsian cubic stochastic object}\label{section:cubic_Gibbs}

This subsection should be seen as a warm-up for Section \ref{section:comparing} below. We explore the relationship between the two cubic stochastic objects
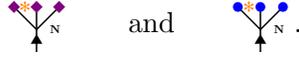
\begin{equation*}
\<3DNp> \qquad \text{and} \qquad \<3DN>. 
\end{equation*}
This is already sufficient for the structured local well-posedness in Proposition \ref{global:prop_lwp} on the support of the Gibbs measure. It will also be needed in the proof of several propositions and lemmas in Section \ref{section:multilinear_Gibbs} below. 

\begin{proposition}\label{ftg:prop_cubic}
Let $A\geq 1$, let $T\geq 1$, and let $\zeta=\zeta(\epsilon,s_1,s_2,\kappa,\eta,\eta^\prime,b_+,b)>0$ be sufficiently small. There exist two Borel sets $\Theta_{\blue}^{\cub}(A,T), \Theta_{\red}^{\cub}(A,T) \subseteq \cH_x^{-1/2-\kappa}(\bT^3)$ satisfying
\begin{equation*}
\bP\Big( \, \bluedot \in \Theta_{\blue}^{\cub}(A,T) \text{ and } \reddotM \Theta_{\red}^{\cub}(A,T)\Big) \geq 1 - \zeta^{-1} \exp(\zeta A^\zeta) 
\end{equation*}
for all $M\geq 1$ and such that the following holds for all $ \bluedot \in \Theta_{\blue}^{\cub}(A,T) \text{ and } \reddotM \Theta_{\red}^{\cub}(A,T)$: 
For all $N\geq 1$, there exist $ \HNpb, \HNbp \in \LM([0,T])$ and $\YNpb,\YNbp \in \X{s_2}{b}([0,T])$ satisfying the identities
\begin{align*}
\<3DNp> &= \<3DN> + P_{\leq N} \Duh \Big[ \PCtrl\Big( \HNpb, P_{\leq N} \<1b>\Big)\Big] + \YNpb, \\
\<3DN> &= \<3DNp> + P_{\leq N} \Duh \Big[ \PCtrl\Big( \HNbp, P_{\leq N} \<1p>\Big)\Big] + \YNbp.
\end{align*}
and the estimates 
\begin{align*}
 \| \HNbp \|_{\LM([0,T])},  \| \HNpb \|_{\LM([0,T])}\leq T^2 A 
 \end{align*}
 and 
 \begin{align*}
  \| \YNbp \|_{\X{s_2}{b}([0,T])}, \| \YNpb \|_{\X{s_2}{b}([0,T])}\leq T^3 A. 
\end{align*}
Furthermore, in the frequency-localized version of this estimate, we gain an $\eta^\prime$-power of the maximal frequency-scale. 
\end{proposition}

\begin{remark}
The results in Proposition \ref{ftg:prop_cubic} do not yield a bound on $\<3DNp>$ in $L_t^\infty \cC_x^{\beta-\kappa}$, since $\X{s_2}{b}$ does not embed into $L_t^\infty \cC_x^{\beta-\kappa}$ and we do not state any additional information on $Y_N$. However, such an estimate is possible and only requires the translation invariance of the law of $(\,\bluedot,\reddotM)$, which is a consequence of \cite[Theorem 1.4]{BB20a}. 
\end{remark}

Before we start with the proof of Proposition \ref{ftg:prop_cubic}, we record and prove the following corollary. 

\begin{corollary} \label{ftg:cor_cubic}
Let $A\geq 1$, let $T\geq 1$, let $\alpha>0$ be a large absolute constant, and let  $\zeta=\zeta(\epsilon,s_1,s_2,\kappa,\eta,\eta^\prime,b_+,b)>0$ be sufficiently small. Then, there exists a Borel set $\Theta_{\purple}^{\bil}(A,T) \subseteq \cH_x^{-1/2-\kappa}(\bT^3)$ satisfying
\begin{equation}\label{ftg:eq_bilinear_probabilistic}
\mup_M\Big( \Theta_{\purple}^{\bil}(A,T) \Big), \nup_M\Big( \Theta_{\purple}^{\bil}(A,T) \Big)\geq 1 - \zeta^{-1} \exp( \zeta A^{\zeta}) 
\end{equation}
for all $M\geq 1$ and such that the following holds for all $\purpledot \in \Theta_{\purple}^{\bil}(A,T)$:

For all intervals $\cJ\subseteq [0,T]$ and $w\in \X{s_1}{b}(\cJ)$, it holds that 
\begin{equation}\label{ftg:eq_bilinear}
\sum_{L_1,L_2} \Big\| P_{L_1} \<3DNp> \cdot P_{L_2} w \Big\|_{L_t^2 H_x^{-4\delta_1}(\cJ \times \bT^3)} \leq T^{\alpha} A \| w \|_{\X{s_1}{b}(\cJ)}. 
\end{equation}
\end{corollary}

\begin{proof}[Proof of Corollary \ref{ftg:cor_cubic}:]
We simply define $\Theta_{\purple}^{\bil}(A,T) $ as set the of initial data $\, \purpledot \in \cH_x^{-1/2-\kappa}(\bT^3)$ where \eqref{ftg:eq_bilinear} holds for a countable but dense subset of $\X{s_1}{b}(\bR)$, which is Borel measurable, and it remains to prove the probabilistic estimate \eqref{ftg:eq_bilinear_probabilistic}. Using Theorem \ref{theorem:measures}, it suffices to prove that 
\begin{equation*}
\bP( \, \bluedot + \reddotM \in  \Theta_{\purple}^{\bil}(A,T) \Big) \geq 1 - \zeta^{-1} \exp( \zeta A^\zeta). 
\end{equation*}
This follows directly from Proposition \ref{so3:prop}, Lemma \ref{phy:lem_bilinear_kt}, and Proposition \ref{ftg:prop_cubic}. 
\end{proof}

We now turn to the proof of Proposition \ref{ftg:prop_cubic}. The argument relies on the multi-linearity of the stochastic objects in the initial data. In order to use the decomposition of $\, \purpledot\,$, we define mixed cubic stochastic objects.  In Section \ref{section:gwp}, we defined stochastic objects in $\purpledot$ instead of $\bluedot$, which had the exact same renormalization constants and multipliers.  In the proof of Proposition \ref{ftg:prop_cubic}, we also work with stochastic objects that contain a mixture of both $\bluedot$ and $\reddot=\reddotM$. In this case, only factors of $\bluedot$ require a renormalization. The renormalized mixed stochastic objects are then defined by 
\begin{align*}
\<3Nbrb> &	\defe	P_{\leq N} \bigg[ V \ast \Big( P_{\leq N} \<1b> \cdot P_{\leq N} \<1r> \Big) \cdot P_{\leq N} \<1b> - \MN P_{\leq N} \<1r>\bigg]		,\\
\<3Nbbr> &	\defe	P_{\leq N} \bigg[ \Big( V \ast \<2N> \Big) \cdot P_{\leq N} \<1r> \bigg],\\	
\<3Nrrb> &	\defe	P_{\leq N} \bigg[ V \ast \Big( P_{\leq N} \<1r> \cdot P_{\leq N} \<1r> \Big) \cdot P_{\leq N} \<1b> \bigg]		,\\
\<3Nbrr> &	\defe	P_{\leq N} \bigg[ V \ast \Big( P_{\leq N} \<1b> \cdot P_{\leq N} \<1r> \Big) \cdot P_{\leq N} \<1r> \bigg]	,\\
\<3Nrrr> &	\defe	P_{\leq N} \bigg[ V \ast \Big( P_{\leq N} \<1r> \cdot P_{\leq N} \<1r> \Big) \cdot P_{\leq N} \<1r> \bigg]	.
\end{align*}
Furthermore, we define the solution to the nonlinear wave equation with forcing term $\<3Nbrb> $ by
\begin{equation*}
(-\partial_t^2 -1 + \Delta) \<3DNbrb> = \<3Nbrb>, \qquad  \<3DNbrb>[0]=0. 
\end{equation*}
The solutions for the other forcing terms above are defined similarly. Using these definitions, we obtain that identity 
\begin{equation}\label{gibbs:eq_cubic_identity}
\<3DNp>  = \<3DN>  +   2 \<3DNbrb> + \<3DNrrb> + \<3DNbbr> + 2 \<3DNbrr> + \<3DNrrr>. 
\end{equation}
Using this identity, the proof of Proposition \ref{ftg:prop_cubic} is now split into two lemmas. 

\begin{lemma}\label{ftg:lem_cubic_1}
Let $A\geq 1$, let $T\geq 1$, and let  $\zeta=\zeta(\epsilon,s_1,s_2,\kappa,\eta,\eta^\prime,b_+,b)>0$ be sufficiently small. Then, there exists two Borel sets $\Theta_{\blue}^{\cub,(1)}(A,T) , \Theta_{\red}^{\cub,(1)}(A,T) \subseteq \cH_x^{-1/2-\kappa}(\bT^3)$ satisfying
\begin{equation}
\bP\Big( \, \bluedot \in \Theta_{\blue}^{\cub,(1)}(A,T) \text{ and } \reddotM \in  \Theta_{\red}^{\cub,(1)}(A,T)\Big) \geq 1 - \zeta^{-1} \exp(\zeta A^\zeta) 
\end{equation}
for all $M\geq 1$ and such that the following holds for all  $\, \bluedot \in \Theta_{\blue}^{\cub,(1)}(A,T) \text{ and } \reddotM \in  \Theta_{\red}^{\cub,(1)}(A,T)$: 

For all $N\geq 1$, there exists a $H_N\in \LM([0,T])$ satisfying the identity 
\begin{equation}\label{ftg:eq_cubic_1}
2 ~  ( \parald ) \<3DNbrb> + ~  ( \parald ) \<3DNrrb> = P_{\leq N} \Duh[ \Para(H_N, P_{\leq N} \<1b>)]  
\end{equation}
and the estimate
\begin{equation*}
 \| H_N \|_{\modulation([0,T])}   \leq T^2 A.  
\end{equation*}
Furthermore, the difference $H_N-H_K$ gains an $\eta^\prime$-power of $\min(N,K)$. 
\end{lemma}

\begin{proof}
From Lemma \ref{para:lemma_obj_2}, it follows that there exists a (canonical) random variable $H_N\colon \Omega \rightarrow \LM([0,T])$ such that 
\begin{equation*}
2 ~  ( \parald ) \<3DNbrb> + ~  ( \parald ) \<3DNrrb> =  P_{\leq N} \Duh[ \Para(H_N, P_{\leq N} \<1b>)] 
\end{equation*}
and 
\begin{align*}
\| H_N \|_{\LM{[0,T]}} &\lesssim 
\Big( \| \,\<1b>\, \|_{\X{-s_2}{b}([0,T])} +	 \| \,\<1r> \,\|_{\X{-s_2}{b}([0,T])} \Big)  \|\, \<1r> \, \|_{\X{s_2}{b}([0,T])} \\
&\lesssim  T^2 \Big( \| \, \bluedot \, \|_{\cH_x^{-s_2}(\bT^3)} +  \| \, \reddot \, \|_{\cH_x^{s_2}(\bT^3)}  \Big) \cdot \| \, \reddot \, \|_{\cH_x^{s_2}(\bT^3)}
\end{align*}
The estimate for $H_N$ then follows from elementary properties of $\bluedot$ and the high-regularity bound for $\reddot$ in Theorem \ref{theorem:measures}. 
\end{proof}

\begin{lemma}\label{ftg:lem_cubic_2}
Let $A\geq 1$, let $T\geq 1$, and let  $\zeta=\zeta(\epsilon,s_1,s_2,\kappa,\eta,\eta^\prime,b_+,b)>0$ be sufficiently small. Then, there exists two Borel sets $\Theta_{\blue}^{\cub,(2)}(A,T) , \Theta_{\red}^{\cub,(2)}(A,T) \subseteq \cH_x^{-1/2-\kappa}(\bT^3)$ satisfying
\begin{equation}
\bP\Big( \, \bluedot \in \Theta_{\blue}^{\cub,(2)}(A,T) \text{ and } \reddotM \in  \Theta_{\red}^{\cub,(2)}(A,T)\Big) \geq 1 - \zeta^{-1} \exp(\zeta A^\zeta) 
\end{equation}
for all $M\geq 1$ and such that the following holds for all  $\, \bluedot \in \Theta_{\blue}^{\cub,(2)}(A,T) \text{ and } \reddotM \in  \Theta_{\red}^{\cub,(2)}(A,T)$: 

For all $N\geq 1$, we have that 
\begin{align*}
\max \bigg( &\Big\| \nparald \<3DNbrb> \Big\|_{\X{s_2}{b}([0,T])}  , \Big\|  \nparald \<3DNrrb> \Big\|_{\X{s_2}{b}([0,T])}, 
\Big\|  \<3DNbbr> \Big\|_{\X{s_2}{b}([0,T])} , \Big\|  \<3DNbrr> \Big\|_{\X{s_2}{b}([0,T])}  ,  \\
& \Big\|   \<3DNrrr>\Big\|_{\X{s_2}{b}([0,T])} \bigg) \leq T^3 A. 
\end{align*}
Furthermore, the difference of the cubic stochastic objects with two parameters $N$ and $K$ gains an $\eta^\prime$-power of $\min(N,K)$. 
\end{lemma}

\begin{proof}
This follows from our previous estimates for $\,\bluedot\,$ from Section \ref{section:stochastic_object}-\ref{section:physical} and the high-regularity bound for $\reddot$ in Theorem \ref{theorem:measures}. More precisely, we estimate the $L^p_\omega \X{s_2}{b}$-norm of 
\begin{itemize}
\item[$\bullet$] $ \nparald \<3DNbrb>$ by $T^2 p^{\frac{2+k}{2}}$ through Proposition \ref{rmt:prop2}, 
\item[$\bullet$] $ \nparald \<3DNrrb> $ by $T^3 p^{\frac{1+2k}{2}}$ through Proposition \ref{phy:prop_1b}, 
\item[$\bullet$] $ \<3DNbbr>$ by $T^2 p^{\frac{2+k}{2}}$ through Proposition \ref{rmt:prop1}, 
\item[$\bullet$] $  \<3DNbrr>$ by $T^3 p^{\frac{1+2k}{2}}$ through Proposition \ref{phy:prop_neq} and Corollary \ref{phy:cor_bilinear}, 
\item[$\bullet$] $  \<3DNrrr>$ by $T p^{\frac{3k}{2}}$ through Proposition \ref{phy:prop_3b}
\end{itemize}
\end{proof}

\begin{proof}[Proof of Proposition \ref{ftg:prop_cubic}:] 
The first algebraic identity and related estimates follow directly from \eqref{gibbs:eq_cubic_identity}, Lemma \ref{ftg:lem_cubic_1} and Lemma \ref{ftg:lem_cubic_2}. By using $\<1p>-\<1b>=\<1r>$ and the high regularity bound for $\reddot$, we obtain the second identity and the related estimates from the first identity. 
\end{proof}

\subsection{Comparing random structures in Gibbsian and Gaussian initial data}\label{section:comparing}

In Definition \ref{local:def_types}, we introduced the types of functions occurring in our multi-linear master estimate for $\,\bluedot\,$ (Proposition \ref{local:prop_master}). The types $w$ and $X$ in Definition \ref{local:def_types} implicitly depend on $\,\bluedot\,$ and, as already mentioned in Remark \ref{local:rem_types}, we now refer to type $w$ and $X$ as type $\wblue$ and $\Xblue$, respectively. We now introduce a similar notation for the generic initial data $\, \purpledot\,$. In order to orient the reader, we include an overview of the different types and their relationship in Figure \ref{figure:Gibbsian}.

%Different node styles for diagram
\tikzset{types/.style={rectangle, rounded corners, minimum width=1cm, minimum
height=1cm,text centered,align=center, draw=black, fill=white!10},
arrow/.style={thick,->,>=stealth}}

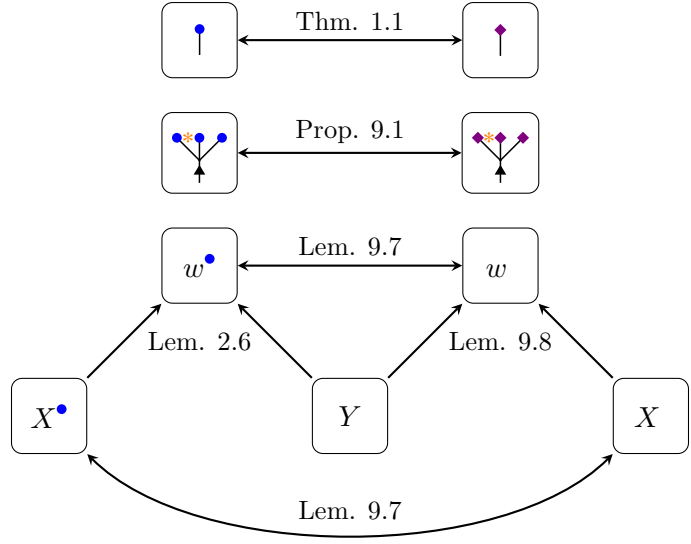
\begin{figure}
\begin{center}
\begin{tikzpicture}[node distance=2cm]
%%Nodes 
\node (1b)[types]  at (-2,0) {$\<1b>$};
\node (1p)[types]  at (2,0) {\usebox\ponebox};
\node (3b)[types] at (-2,-1.5) {\usebox\bthreebox};
\node (3p)[types] at (2,-1.5) {\usebox\pthreebox};
\node (wb)[types]  at (-2,-3) {$\wblue$};
\node (wp)[types]  at (2,-3) {$w^{\hspace{-0.5ex}\usebox\pdot}$};
\node (Xb)[types]  at (-4,-5) {$\Xblue$};
\node (Xp)[types]  at (4,-5) {$X^{\hspace{-0.5ex}\usebox\pdot}$};
\node (Y)[types] at (0,-5) {$Y$};

%%%% Arrows
\draw[thick,<->] (1b) -- (1p);
\draw[thick,<->] (3b) -- (3p);
\draw[thick,<->] (wb) -- (wp);
\draw[thick,->] (Xb) -- (wb);
\draw[thick,->] (Xp) -- (wp);
\draw[thick,->] (Y) -- (wb);
\draw[thick,->] (Y) -- (wp);
\draw[thick,<->] (Xb) to[out=315,in=225,looseness=0.75] (Xp);

%%Lemma
\node[above] at (0,0) {\small{Thm. \ref{theorem:measures}}};
\node[above] at (0,-1.5) {\small{Prop. \ref{ftg:prop_cubic}}};
\node[above] at (0,-3) {\small{Lem. \ref{ftg:lem_type_conversion_blue_purple}}};
\node[above] at (0,-6.5) {\small{Lem. \ref{ftg:lem_type_conversion_blue_purple}}};
\node at (2,-4) {\small{Lem. \ref{ftg:lem_type_conversion}}};
\node at (-2,-4) {\small{Lem. \ref{local:lem_type_conversion}}};
\end{tikzpicture}
\end{center}
\caption{\small{We display the relationship between the different types of functions used in this paper. The equivalence ``$\leftrightarrow$'' means that both types agree modulo scalar multiples and/or terms further down in the hierarchy. The implication ``$\rightarrow$'' means that, up to scalar multiples, the left type forms a sub-class of the right type. }} \label{figure:Gibbsian}
\end{figure}

\begin{definition}[Purple types]\label{ftg:def_types}
Let $\cJ \subseteq [0,\infty)$ be a bounded interval and let $\varphi \colon J \times \bT^3 \rightarrow \bR$. We say that $\varphi$ is of type
\begin{itemize}
\setlength\itemsep{1ex}
\item $\,\<1p>\,$ if $\varphi=\<1p>\,$, 
\item $\, \<3Dp>\,$ if $\varphi=\<3DNp>\,$ for some $N\geq 1$, 
\item $\wpurple$ if $\| \varphi \|_{\X{s_1}{b}(\cJ)}\leq 1$ and  $ \sum_{L_1\sim L_2} \| P_{L_1} \<1p> \, \cdot P_{L_2} w \|_{L_t^2 H_x^{-4\delta_1}(\cJ\times \bT^3)} \leq 1$ for all $N\geq 1$, 
\item $\Xpurple$ if $\varphi = P_{\leq N} \Duh \big[ 1_{\cJ_0} \PCtrl(H, P_{\leq N} \, \<1p>\, )\big] $ for a dyadic integer $N\geq 1$, a subinterval $\cJ_0\subseteq \cJ$,  and a function $H\in \LM(\cJ_0)$ satisfying $\| H \|_{\LM(\cJ_0)}\leq 1$.
\end{itemize}
\end{definition}
Since the type $Y$ in Definition \ref{local:def_types} does not depend on the stochastic object, its meaning remains unchanged. In Proposition \ref{ftg:prop_cubic}, we have already seen that the types $\<3D>$ and $\<3Dp>$ only differ by functions of type $\Xblue$ and $Y$ (or $\Xpurple$ and $Y$). In the next lemma, we clarify the relationship between the types $\wblue$ and $\wpurple$ as well as $\Xblue$ and $\Xpurple$. 

\begin{lemma}[The equivalences $\wblue\leftrightarrow \wpurple$ and $\Xblue \leftrightarrow \Xpurple$] \label{ftg:lem_type_conversion_blue_purple}
Let $A\geq 1$, let $T\geq 1$, and let $\zeta=\zeta(\epsilon,s_1,s_2,\kappa,\eta,\eta^\prime,b_+,b)>0$ be sufficiently small. Then, there exists a Borel set $\Theta_{\red}^{\type}(A,T)\subseteq \cH_x^{-1/2-\kappa}(\bT^3)$  such that 
\begin{equation*}
\bP(\, \reddotM\in \Theta_{\red}^{\type}(A,T)) \geq 1 - \zeta^{-1} \exp(-\zeta A^{\zeta})
\end{equation*}
and such that the following holds for  $\purpledot=\bluedot+\reddotM$:
\begin{itemize}
\item The types $\wblue$ and $\wpurple$ are equivalent up to multiplication by a scalar $\lambda\in \bR_{>0}$ satisfying $\lambda,\lambda^{-1}\leq T^2 A$. 
\item The types $\Xblue$ and $\Xpurple$ are equivalent up to addition/subtraction of a function in $\X{s_2}{b}$ with norm $\leq TA$. 
\end{itemize}
\end{lemma}

\begin{proof}
We will prove the desired statement on the event
\begin{equation*}
\Theta_{\red}^{\type}(A,T) = \big\{ \phi \in \cH_x^{-1/2-\kappa}(\bT^3) \colon \| \phi \|_{\cH_x^{\frac{1}{2}+\beta-\kappa}(\bT^3)}\leq cA\big\},
\end{equation*}
where $c=c(\epsilon,s_1,s_2,b)$ is a small constant. Based on Theorem \ref{theorem:measures}, this event has an acceptable probability. 

We start with the statement regarding the types $\wblue$ and $\wpurple$. Let $\varphi \in \X{s_1}{b}(\cJ)$ satisfy $\| \varphi\|_{\X{s_1}{b}(\cJ)} \leq 1$, which holds for $\varphi$ of type either $\wblue$ or $\wpurple$. For any $L\geq 1$, we have that  
\begin{align*}
&\big| \sum_{L_1 \sim L_2} \| P_{L_1} \<1p> \, \cdot P_{L_2} \varphi \|_{L_t^2 H_x^{-4\delta_1}(\cJ\times \bT^3)} - 
  \sum_{L_1 \sim L_2} \| P_{L_1} \<1b> \, \cdot P_{L_2}  \varphi \|_{L_t^2 H_x^{-4\delta_1}(\cJ\times \bT^3)}  \big| \\
 &\leq  \sum_{L_1 \sim L_2} \| P_{L_1} \<1r> \, \cdot P_{L_2}  \varphi \|_{L_t^2 H_x^{-4\delta_1}(\cJ\times \bT^3)} 
 \end{align*}
 Using Lemma  \ref{phy:lem_bilinear_kt}, it follows that 
 \begin{align*}
 &  \sum_{L_1 \sim L_2} \| P_{L_1} \<1r> \, \cdot P_{L_2}  \varphi \|_{L_t^2 H_x^{-4\delta_1}(\cJ\times \bT^3)} 
 \lesssim T^{\frac{1}{2}} \Big( \sum_{L_1} L_1^{3-s_1- \big( \frac{1}{2}+\beta-\kappa\big) - 2} \Big)
  \| \, \<1r>\, \|_{\X{\frac{1}{2}+\beta-\kappa}{b}(\cJ)} \| \varphi \|_{\X{s_1}{b}(\cJ)} \\
  &\lesssim  T^{\frac{3}{2}} \| \, \reddotM\|_{\cH_x^{\frac{1}{2}+\beta-\kappa}(\bT^3)} \| \varphi \|_{\X{s_1}{b}(\cJ)} 
  \leq \frac{1}{2} T^{\frac{3}{2}} A  \| \varphi \|_{\X{s_1}{b}(\cJ)} . 
\end{align*}
This yields the stated equivalence of the types $\wblue$ and $\wpurple$.

We now turn to the statement regarding the types $\Xblue$ and $\Xpurple$. For any $H\in \LM(\cJ)$, it holds that 
\begin{align*}
&\big \| P_{\leq N} \Duh \big[ 1_{\cJ_0} \PCtrl(H, P_{\leq N} \, \<1p>\, )\big] -P_{\leq N} \Duh \big[ 1_{\cJ_0}\PCtrl(H, P_{\leq N} \, \<1b>\, )\big]  \big \|_{\X{s_2}{b}(\cJ)} \\
&\lesssim 
\big \| P_{\leq N} \Duh \big [ 1_{\cJ_0} \PCtrl(H, P_{\leq N} \, \<1r>\, )\big]   \big \|_{\X{s_2}{b}(\cJ)}. 
\end{align*}
Using Lemma \ref{tools:lem_energy}, Lemma \ref{tools:lem_inhomogeneous_strichartz}, and Lemma \ref{para:lem_basic}, we have that 
\begin{align*}
&\| P_{\leq N} \Duh \big[  1_{\cJ_0} \PCtrl(H, P_{\leq N} \, \<1r>\, )\big]  \|_{\X{s_2}{b}(\cJ)} 
\lesssim T \| PCtrl(H, P_{\leq N} \, \<1r>\, )  \|_{L_t^\infty H_x^{s_2-1}(\cJ_0\times\bT^3)} \\
&\lesssim T \| H \|_{\LM(\cJ_0)} \|  \, \<1r>\, \|_{L_t^\infty H_x^{s_2-1+8\epsilon}(\cJ\times \bT^3)} 
\lesssim T  \| \, \reddotM\|_{\cH_x^{\frac{1}{2}+\beta-\kappa}(\bT^3)}  \leq \frac{1}{2} T A. 
\end{align*}
This yields the desired estimate. 
\end{proof}

\begin{lemma}[The implication $\Xpurple, Y \rightarrow \wpurple$]\label{ftg:lem_type_conversion}
Let  $\zeta=\zeta(\epsilon,s_1,s_2,\kappa,\eta,\eta^\prime,b_+,b)>0$ be sufficiently small, let $A\geq 1$, and let $T\geq 1$. Then, there exists a Borel set $\Theta_{\purple}^{\type}(A,T) \subseteq \cH_x^{-1/2-\kappa}(\bT^3)$ satisfying
\begin{equation}
\mup_M\Big( \Theta_{\purple}^{\type}(A,T) \Big), \nup_M\Big( \Theta_{\purple}^{\type}(A,T) \Big)\geq 1 - \zeta^{-1} \exp( \zeta A^{\zeta}) 
\end{equation}
for all $M\geq 1$ and such that the following holds for all $\purpledot \in \Theta_{\purple}^{\type}(A,T)$:
If $\varphi$ is of type $\Xpurple$ or $Y$, the scalar multiple $T^{-7} A^{-1} \varphi$ is of type $\wpurple$. 
\end{lemma}

\begin{proof}
Using a separability argument, we can define $\Theta_{\purple}^{\type}(A,T)$ through countably many bounds of the same form as in the definition of the type $\wpurple$. We first note that, after adjusting $\zeta$, we can replace $A^{-1}$ in the conclusion by $A^{-3}$. Using Theorem \ref{theorem:measures}, it suffices to prove that 
\begin{equation*}
\bP(\, \bluedot + \reddotM \in  \Theta_{\purple}^{\type}(A,T) \Big) \geq 1 - \zeta^{-1} \exp(\zeta A^\zeta). 
\end{equation*}
Thus, we may restrict both $\, \bluedot\,$ and $\, \reddotM$ to sets with acceptable probabilities under $\bP$. After these preparations, we now start with the main part of the argument. 

First, we let $w$ be of type $Y$. Using Lemma \ref{local:lem_type_conversion}, it follows that $T^{-4} A \varphi$ is of type $\wblue$. Using Lemma \ref{ftg:lem_type_conversion_blue_purple}, it follows that $T^{-6} A^{-2} \varphi$ is of type $\wpurple$. \\
Now, let $\varphi$ be of type $\Xpurple$. Using Lemma \ref{ftg:lem_type_conversion_blue_purple} and the first step in this proof, we can assume that $\varphi$ is of type $\Xblue$. Using Lemma \ref{local:lem_type_conversion}, $T^{-4} A^{-1} \varphi$ is of type $\wblue$. Finally, using Lemma \ref{ftg:lem_type_conversion_blue_purple} again, we obtain that $T^{-6} A^{-2} \varphi$ is of type $\wpurple$. 
\end{proof}

In Definition \ref{global:definition_structured} above, we introduced the function $\Z$-norms, which are used to quantify structured perturbations of the initial data. We now prove the equivalence of the $\mathscr{Z}([0,T], \scriptp;t_0,N,K)$ and $\mathscr{Z}([0,T], \scriptb;t_0,N,K)$-norms, which is similar to the statements in Lemma \ref{ftg:lem_type_conversion_blue_purple} and Lemma \ref{ftg:lem_type_conversion}. 

\begin{lemma}[Equivalence of the blue and purple structured perturbations]\label{ftg:lem_structured_perturbation}
Let $A\geq 1$, let $\alpha>0$ be a sufficiently large absolute constant, and let $\zeta=\zeta(\epsilon,s_1,s_2,\kappa,\eta,\eta^\prime,b_+,b)>0$ be sufficiently small. Then, there exist Borel sets $\Theta_{\blue}^{\tsp}(A),\Theta_{\red}^{\tsp}(A)\subseteq \cH_x^{-1/2-\kappa}(\bT^3)$ satisfying
\begin{equation}\label{ftg:eq_sp_equivalence_prop}
\bP\big( \, \bluedot \in \Theta_{\blue}^{\tsp}(A), \,\reddotM \in \Theta_{\red}^{\tsp}(A)\big)\geq 1 - \zeta^{-1} \exp(- \zeta A^\zeta).  
\end{equation}
and such that the following holds on this event: 

For all $T\geq 1$, $t_0\in [0,T]$, $N,K\geq 1$, and $Z[t_0]\in \cH_x^{s_1}(\bT^3)$, we have that 
\begin{equation}\label{ftg:eq_sp_equivalence_est}
T^{-\alpha} A^{-1} \| Z[t_0] \|_{\mathscr{Z}([0,T], \scriptb;t_0,N,K)}  \leq \| Z[t_0] \|_{\mathscr{Z}([0,T], \scriptp;t_0,N,K)} \leq T^\alpha A \| Z[t_0] \|_{\mathscr{Z}([0,T], \scriptb;t_0,N,K)} 
\end{equation}
\end{lemma}

\begin{proof}
It suffices to prove the estimate \eqref{ftg:eq_sp_equivalence_est} for events  $\Theta_{\blue}^{\tsp}(A,T)$ and $\Theta_{\red}^{\tsp}(A,T)$ satisfying the probabilistic estimate \eqref{ftg:eq_sp_equivalence_prop}, as long as the lower bound in \eqref{ftg:eq_sp_equivalence_prop} does not depend on $T$.  We can then simply take the intersection of $\Theta_{\blue}^{\tsp}(T\cdot A,T)$ and $\Theta_{\red}^{\tsp}(T \cdot A,T)$ over all integer times and increase $\alpha$ by one. \\
After using Lemma \ref{ftg:lem_type_conversion_blue_purple} to compare the high$\times$high-interaction terms (involving $L_1 \sim L_2$), it remains to prove that 
 \begin{align*}
&\Big| \| \nparaboxld \Big( \lcol  V \ast \big( P_{\leq N} \<1p>  \cdot P_{\leq N} \Zbox_N \big) \, P_{\leq N} \<1p>  \rcol  \|_{\X{s_2-1}{b_+-1}([0,T])} \Big) \\
&-  \| \nparaboxld \Big( \lcol  V \ast \big( P_{\leq N} \<1b>  \cdot P_{\leq N} \Zbox_N \big) \, P_{\leq N} \<1b>  \rcol \Big) \|_{\X{s_2-1}{b_+-1}([0,T])} \Big)\Big|\\
&\lesssim  T^\alpha A\Big( \| \| \Zbox[t_0] \|_{\cH_x^{s_1}} + \sum_{L_1\sim L_2}  \| P_{L_1} \<1b> \, \cdot  P_{L_2} Z \|_{L_t^2 H_x^{-4\delta_1}([0,T]\times \bT^3)}\Big)
\end{align*}
and
\begin{align*}
&\Big| \| \lcol V \ast \big( P_{\leq N} \<1p>  \cdot P_{\leq N} \Zcirc_N \big) \nparald P_{\leq N} \<1p> \rcol  \|_{\X{s_2-1}{b_+-1}([0,T])} \\
 &- \| \lcol V \ast \big( P_{\leq N} \<1b>  \cdot P_{\leq N} \Zcirc_N \big) \nparald P_{\leq N} \<1b> \rcol  \|_{\X{s_2-1}{b_+-1}([0,T])} \Big|\\
 &\lesssim  T^\alpha A \| \Zcirc[t_0] \|_{\cH_x^{s_2}}. 
 \end{align*}
 Regarding the first estimate, we have that 
 
  \begin{align}
&\Big| \| \nparaboxld \Big( \lcol  V \ast \big( P_{\leq N} \<1p>  \cdot P_{\leq N} \Zbox_N \big) \, P_{\leq N} \<1p>  \rcol\Big)  \|_{\X{s_2-1}{b_+-1}([0,T])}  \notag \\
&-  \| \nparaboxld \Big( \lcol  V \ast \big( P_{\leq N} \<1b>  \cdot P_{\leq N} \Zbox_N \big) \, P_{\leq N} \<1b>  \rcol \Big) \|_{\X{s_2-1}{b_+-1}([0,T])} \Big| \notag \\
&\lesssim  \| \nparaboxld \Big( V \ast \big( P_{\leq N} \<1r>  \cdot P_{\leq N} \Zbox_N \big) \, P_{\leq N} \<1b>   \Big) \|_{\X{s_2-1}{b_+-1}([0,T])}  \label{ftg:eq_sp_p1} \\
&+\| \nparaboxld \Big( V \ast \big( P_{\leq N} \<1b>  \cdot P_{\leq N} \Zbox_N \big) \, P_{\leq N} \<1r>  \Big) \|_{\X{s_2-1}{b_+-1}([0,T])}  \label{ftg:eq_sp_p2} \\ 
&+\| \nparaboxld \Big( V \ast \big( P_{\leq N} \<1r>  \cdot P_{\leq N} \Zbox_N \big) \, P_{\leq N} \<1r>   \Big) \|_{\X{s_2-1}{b_+-1}([0,T])} . 
\label{ftg:eq_sp_p3}
\end{align}
We can then control 
\begin{itemize}
\item \eqref{ftg:eq_sp_p1} through Proposition \ref{phy:prop_1b}, 
\item \eqref{ftg:eq_sp_p2} through Proposition \ref{phy:prop_neq} and Lemma \ref{phy:lem_bilinear_tool}, 
\item \eqref{ftg:eq_sp_p2} through Proposition \ref{phy:prop_3b}. 
\end{itemize}
The proof of the second estimate is similar, except that we use Corollary \ref{phy:cor_bilinear} instead of Lemma \ref{phy:lem_bilinear_tool}. 

\end{proof}

\subsection{Multi-linear master estimate for Gibbsian initial data}\label{section:multilinear_Gibbs}
In this subsection, we prove a version of the multi-linear master estimate for Gaussian data (Proposition \ref{local:prop_master}) for the purple types (Definition \ref{ftg:def_types}) instead of the blue types (Definition \ref{local:def_types}). Since we will only need this estimate in Proposition \ref{global:prop_structure_time} and Proposition \ref{global:prop_structure_cubic}, which do not involve contraction or continuity arguments, we can be less precise than in the multi-linear master estimate for Gaussian data and simply capture the size of the forcing term in the following norm. 

\begin{definition}
Let $N \geq 1$, let $\cJ \subseteq \bR$ be a compact interval, and let $R,\varphi\colon \cJ \times \bT^3 \rightarrow \bR$. Then, we define
\begin{equation*}
\| R \|_{\NL_N(J,\varphi)} \defe \inf\Big\{ \| H \|_{\LM(\cJ)} + \| F \|_{\X{s_2-1}{b_+-1}(\cJ)} \colon R = P_{\leq N} \PCtrl[ H, P_{\leq N}\varphi ] + F
\text{ on } \cJ \times \bT^3 \Big\}.
\end{equation*}
\end{definition}

\begin{remark}[Drawback of $\paraboxld$]\label{ftg:rem_drawback}
As mentioned above, the $\NL_N(J,\varphi)$-norm is less precise than our estimates in Section \ref{section:ansatz}, since it does not give an explicit description of the low-frequency modulation $H$. This allows us to circumvent a technical problem which the author was unable to resolve. In Proposition \ref{so5i:prop}, we proved that 
\begin{equation*}
\nparaboxld \Big( \lcol  V \ast \Big( P_{\leq N} \<1b> \cdot P_{\leq N} \<3DN>\Big)~ P_{\leq N} \<1b> \rcol \Big) 
\end{equation*}
lives in $\X{s_2-1}{b_+-1}$. One may therefore expect that 
\begin{equation*}
\nparaboxld \Big( \lcol  V \ast \Big( P_{\leq N} \<1p> \cdot P_{\leq N} \<3DNp>\Big)~ P_{\leq N} \<1p> \rcol \Big) 
\end{equation*}
also lives in $\X{s_2-1}{b_+-1}$. However, after using Proposition \ref{ftg:prop_cubic}, we would need an estimate for 
\begin{equation*}
\nparaboxld \Big( \lcol  V \ast \Big( P_{\leq N} \<1b> \cdot P_{\leq N} Y_N\Big)~ P_{\leq N} \<1b> \rcol \Big) 
\end{equation*}
in $\X{s_2-1}{b_+-1}$. Unfortunately, this is not covered by Proposition \ref{rmt:prop2}. In fact, without any additional assumptions on $Y_N$ other than bounds in $\X{s_2}{b}$, the high$\times$high$\rightarrow$low-interactions in $P_{\leq N} \<1b> \cdot P_{\leq N} Y_N$ rule out this estimate. 
\end{remark}

Equipped with the $\NL$-norm, we now turn to the master estimate for Gibbsian initial data. 
\begin{proposition}[Multi-linear master estimate for Gibbsian initial data]\label{ftg:prop_master}
Let $A\geq 1$, let $T\geq 1$, let $\alpha>0$ be a sufficiently large absolute constant, and let $\zeta=\zeta(\epsilon,s_1,s_2,\kappa,\eta,\eta^\prime,b_+,b)>0$ be sufficiently small. Then, there exists a Borel set $\Theta_{\purple}^{\ms}(A,T)\subseteq \cH_x^{-1/2-\kappa}$ satisfying
\begin{equation}\label{ftg:eq_master_prob}
\mup_M(\purpledot \in \Theta_{\purple}^{\ms}(A,T))\geq 1 - \zeta^{-1} \exp(-\zeta A^\zeta)
\end{equation}
for all $M\geq 1$ and such that the following estimates hold for all $\, \purpledot \in \Theta_{\purple}^{\ms}(A,T)$:\\
Let $\cJ \subseteq [0,T]$ be an interval and let $N\geq 1$. Let $\varphi_1,\varphi_2,\varphi_3 \colon \cJ \times \bT^3 \rightarrow \bR$ be as in Definition \ref{ftg:def_types} and let
\begin{equation*}
 (\varphi_1,\varphi_2;\varphi_3) \netype \big(\, \<1p>\,, \,\<1p>\,;\,\<1p>\, \big), ~  \big(\, \<1p>\,  ,\wpurple;\,\<1p>\, \big).
\end{equation*}
\begin{enumerate}[(i)]
\item \label{ftg:item_master_1} If $\varphi_3 \etype \, \<1p>$, then 
\begin{equation*}
\Big\| P_{\leq N} \Big( \lcol  V \ast \big( P_{\leq N} \varphi_1 \cdot P_{\leq N} \varphi_2 \big)  P_{\leq N} \varphi_3  \rcol \Big)  \Big\|_{\NL_N\big(\cJ,\, P_{\leq N} \, \<1p>\big)} \leq T^{\alpha} A. 
\end{equation*}
\item \label{ftg:item_master_2} In all other cases,
\begin{equation*}
\Big\|\lcol  V \ast \big( P_{\leq N} \varphi_1 \cdot P_{\leq N} \varphi_2 \big)  P_{\leq N} \varphi_3  \rcol  \Big\|_{\X{s_2-1}{b_+-1}(\cJ)} \leq T^{\alpha} A. 
\end{equation*}
\end{enumerate}
\end{proposition}

\begin{proof}
While the proof requires no new ingredients, it relies on several earlier results. For the advantage of the reader, we break up the proof into several steps.\\

\emph{Step 1: Definition of $\Theta_{\purple}^{\ms}(A,T)$ and its Borel measurability.} Using the definition of the time-restricted norms, we see that the statement for all intervals $\cJ\subseteq [0,T]$ is equivalent to the statement for only $\cJ=[0,T]$. Thus, we may simply choose $\Theta_{\purple}^{\ms}(A,T)$ as the set where \eqref{ftg:item_master_1} and \eqref{ftg:item_master_2} hold for all $N\geq 1$. To see that this leads to a Borel measurable set, we note that both $\LM([0,T])$ and $\X{s_2}{b}([0,T])$ are separable. For a fixed $N\geq 1$, we also have that the functions
\begin{equation*}
(\varphi_1,\varphi_2,\varphi_3) \mapsto \Big\| P_{\leq N} \Big( \lcol  V \ast \big( P_{\leq N} \varphi_1 \cdot P_{\leq N} \varphi_2 \big)  P_{\leq N} \varphi_3  \rcol \Big)  \Big\|_{\NL_N\big(\cJ,\, P_{\leq N} \, \<1p>\big)}
\end{equation*}
and 
\begin{equation*}
(\varphi_1,\varphi_2,\varphi_3) \mapsto \Big\|\lcol  V \ast \big( P_{\leq N} \varphi_1 \cdot P_{\leq N} \varphi_2 \big)  P_{\leq N} \varphi_3  \rcol  \Big\|_{\X{s_2-1}{b_+-1}(\cJ)} 
\end{equation*}
are continuous w.r.t. the $C_t^0 H_x^{-1/2-\kappa}([0,T]\times \bT^3)$-norm. Thus, we can represent $\Theta_{\purple}^{\ms}(A,T)$ through countably   many constraints of the same form as in \eqref{ftg:item_master_1} and \eqref{ftg:item_master_2}, and hence as a countable intersection of closed sets. In particular,  $\Theta_{\purple}^{\ms}(A,T)$ is Borel measurable. \\

\emph{Step 2: Reductions.} It therefore remains to show the probabilistic estimate \eqref{ftg:eq_master_prob}. Using the absolute continuity and representation of the reference measures from Theorem \ref{theorem:measures}, it suffices to prove that 
\begin{equation*}
\bP( \, \bluedot + \reddotM \in \Theta_{\purple}^{\ms}(A,T)) \geq 1 - \zeta^{-1} \exp(-\zeta A^\zeta)
\end{equation*}
for all $M\geq 1$. Furthermore, we can replace the upper bound $T^\alpha A$ in \eqref{ftg:item_master_1} and \eqref{ftg:item_master_2} by $CT^\alpha A^C$, where $C=C(\epsilon,s_1,s_2,\kappa,\eta,\eta^\prime,b_+,b)\geq 1$. After the estimate has been proven, this can then be repaired by adjusting $A$ and $\zeta$. Using Lemma \ref{local:lem_type_conversion}, Proposition \ref{local:prop_master}, Corollary \ref{ftg:cor_cubic}, Lemma \ref{ftg:lem_type_conversion_blue_purple}, and Lemma \ref{ftg:lem_type_conversion}, we may restrict to the event
\begin{align*}
&\Big\{ \, \bluedot \in \Theta_{\blue}^{\ms}(A,T) \medcap \Theta_{\blue}^{\type}(A,T)  \medcap \Theta_{\blue}^{\cub}(A,T) \Big\}\, \medcap
\, \Big\{ \reddotM \in  \Theta_{\red}^{\type}(A,T)  \medcap \Theta_{\red}^{\cub}(A,T) \Big\} \\ 
&\medcap \Big\{ \, \bluedot + \reddotM \in \Theta_{\purple}^{\type}(A,T) \Big\}. 
\end{align*}

\emph{Step 3: Multi-linear estimates.} The estimates for $\varphi_3 \netype \, \<1p>$ follow directly from the multi-linear master estimate for $\,\bluedot\,$ and the equivalence of the types in Corollary \ref{ftg:cor_cubic}, Lemma \ref{ftg:lem_type_conversion_blue_purple}, and Lemma \ref{ftg:lem_type_conversion}. It then remains to treat the case $\varphi_3 \etype \, \<1p>$. We further separate the proof of the estimates into two cases.\\

\emph{Step 3.1: $\varphi_1,\varphi_2 \netype \, \<1p>$.} We first remind the reader that in this case the nonlinearity does not require a renormalization. We then decompose 
\begin{align*}
&P_{\leq N} \Big( V \ast \big( P_{\leq N} \varphi_1 \cdot P_{\leq N} \varphi_2 \big)  P_{\leq N} \<1p>  \Big) \\
=& P_{\leq N} \Big(   V \ast \big( P_{\leq N} \varphi_1 \cdot P_{\leq N} \varphi_2 \big)  \parald P_{\leq N} \<1p>  \Big)\\
+& P_{\leq N} \Big(   V \ast \big( P_{\leq N} \varphi_1 \cdot P_{\leq N} \varphi_2 \big)  \nparald P_{\leq N} \<1b>  \Big)\\
+& P_{\leq N} \Big(   V \ast \big( P_{\leq N} \varphi_1 \cdot P_{\leq N} \varphi_2 \big)  \nparald P_{\leq N} \<1r>  \Big). 
\end{align*}
Using Lemma \ref{para:lemma_obj_2}, the first term is of the form $P_{\leq N} \PCtrl(H_N,P_{\leq N} \, \<1p>)$ with $\| H_N \|_{\LM([0,T])} \lesssim T^\alpha A^2$. The second and third term can be controlled through the multi-linear master estimate for Gaussian random data. \\

\emph{Step 3.2: $\varphi_1,\varphi_3\etype \,\<1p>$, $\varphi_2 \netype \, \<1p>$.} Using the equivalence of types (as in Corollary \ref{ftg:cor_cubic} and  Lemma \ref{ftg:lem_type_conversion_blue_purple}) together with the previous cases, it suffices to treat 
\begin{equation*} 
 \varphi_1,\varphi_3\etype \,\<1b>, ~ \varphi_2 \etype \<3D>, ~ \Xblue,~ Y. 
\end{equation*}
We decompose the nonlinearity 
\begin{equation*}
V \ast \big( P_{\leq N} \<1b> \cdot P_{\leq N} \varphi_2 \big)  P_{\leq N} \<1b>  
\end{equation*}
using $\paraboxld$ if $\varphi_2\etype  \<3D>, ~ \Xblue$ and using $\parald$ if $\varphi_2 \etype  Y$. Then, the bound follows from the multi-linear master estimate for Gaussian initial data, Lemma \ref{para:lemma_obj_1}, and Lemma \ref{para:lemma_obj_2}. 
\end{proof}

In Definition \ref{global:definition_structured}, we also introduced a structured perturbation of the initial data, which we briefly examined in Lemma \ref{ftg:lem_structured_perturbation} above. While the multi-linear estimate does not apply to the type $\big(\, \<1p>\,  ,\wpurple;\,\<1p>\, \big)$, we now obtain a multi-linear estimate if the second argument is a linear evolution with initial data as in Definition \ref{global:definition_structured}. Since the definition has been tailored towards this estimate, the prove will be easy and short.

\begin{lemma}[Multi-linear estimate for the structured perturbation]\label{ftg:lem_multilinear_Z}
Let $A\geq 1$, let $T\geq 1$, let $\alpha>0$ be a sufficiently large absolute constant, and let $\zeta=\zeta(\epsilon,s_1,s_2,\kappa,\eta,\eta^\prime,b_+,b)>0$ be sufficiently small. Then, there exists a Borel set $\Theta_{\purple}^{\tsp}(A,T)\subseteq \cH_x^{-1/2-\kappa}$ satisfying
\begin{equation}
\mup_M(\purpledot \in \Theta_{\purple}^{\tsp}(A,T))\geq 1 - \zeta^{-1} \exp(-\zeta A^\zeta)
\end{equation}
for all $M\geq 1$ and such that the following estimates hold for all $\, \purpledot \in \Theta_{\purple}^{\tsp}(A,T)$:\\
Let $N,K\geq 1$, let $t_0\in [0,T]$, let $Z[t_0]\in \cH_x^{-1/2-\kappa}(\bT^3)$, and let $Z(t)$ be the corresponding solution to the linear wave equation. Then, it holds that 
\begin{equation*}
 \Big\|P_{\leq N} \Big[ \lcol  V \ast \big( P_{\leq N} \<1p> \cdot P_{\leq N} Z \big)  P_{\leq N} \<1p> \rcol \Big]   \Big\|_{\NL_N\big([0,T],\, P_{\leq N} \, \<1p>\big)} \leq T^\alpha A  \| Z[t_0] \|_{\mathscr{Z}([0,T], \scriptp;t_0,N,K)}. 
\end{equation*}
\end{lemma}

\begin{proof}
Let $\Zbox[t_0]$ and $\Zcirc[t_0]$ be as in Definition \ref{global:definition_structured}. Then, we can decompose  
\begin{align*}
&P_{\leq N} \Big[ \lcol  V \ast \big( P_{\leq N} \<1p> \cdot P_{\leq N} Z \big)  P_{\leq N} \<1p> \rcol \Big]  \\
=& \paraboxld \Big( P_{\leq N} \Big[  V \ast \big( P_{\leq N} \<1p> \cdot P_{\leq N} \Zbox \big)  P_{\leq N} \<1p>  \Big] \Big)  \\
+& \nparaboxld \Big( P_{\leq N} \Big[ \lcol  V \ast \big( P_{\leq N} \<1p> \cdot P_{\leq N} \Zbox \big)  P_{\leq N} \<1p> \rcol \Big] \Big)  \\
+&P_{\leq N} \Big[ \lcol  V \ast \big( P_{\leq N} \<1p> \cdot P_{\leq N} \Zcirc \big) \parald  P_{\leq N} \<1p> \rcol \Big]  \\
+&P_{\leq N} \Big[  V \ast \big( P_{\leq N} \<1p> \cdot P_{\leq N} \Zcirc \big)  \nparald P_{\leq N} \<1p>  \Big].  
\end{align*}
The estimate then directly follows from Definition \ref{global:definition_structured}, Lemma \ref{para:lemma_obj_1}, and Lemma \ref{para:lemma_obj_2}. 
\end{proof}

\begin{appendix}
\addtocontents{toc}{\protect\setcounter{tocdepth}{1}} %%Only put sections in table of contents from here 

\section{Proofs of counting estimates}

\subsection{Cubic counting estimate}

We start with the proof of the cubic counting estimate. 

\begin{proof}[Proof of Proposition \ref{tools:prop_cubic_counting}:]
We separately prove the four counting estimates \eqref{tools:item_cubic_1}-\eqref{tools:item_cubic_4}. \\

\emph{Proof of \eqref{tools:item_cubic_1}:} By symmetry, we can assume that $N_1 \geq N_2 \geq N_3$. Using the basic counting estimate to perform the sum in $n_2 \in \bZ^3$, we obtain that 
\begin{align*}
&\# \{ (n_1,n_2,n_3) \colon |n_1|\sim N_1,|n_2|\sim N_2, |n_3|\sim N_3, |\varphi-m|\leq 1\} \\
&\lesssim \sum_{n_1,n_3\in \bZ^3} \Big( \prod_{j=1,3} 1\big\{ |n_j|\sim N_j \big\} \Big) \min\big( \langle n_{13} \rangle, N_2 \big)^{-1}  N_2^3 \\
&\lesssim N_1^2 N_2^3 N_3^3 +  N_1^3 N_2^2  N_3^3  \\
&\lesssim  N_2^{-1} (N_1 N_2 N_3)^3  , 
\end{align*}
which is acceptable. \\

\emph{Proof of \eqref{tools:item_cubic_2}:} We emphasize that $n_{123}$ is viewed as a free variable. In the variables $(n_{123},n_1,n_2)$, the phase takes the form 
\begin{equation*}
\varphi = \pm_{123} \langle n_{123} \rangle \pm_1 \langle n_1 \rangle  \pm_2 \langle n_2 \rangle  \pm_3 \langle n_{123} - n_1 - n_2  \rangle.
\end{equation*}
After changing $(n_1,n_2) \rightarrow (-n_1,-n_2)$, we obtain the same form as in \eqref{tools:item_cubic_1} and hence the desired estimate.\\

\emph{Proof of \eqref{tools:item_cubic_3}:} In the variables $(n_{123},n_{12},n_1)$, the phase takes the form
\begin{equation*}
\varphi = \pm_{123} \langle n_{123} \rangle \pm_1 \langle n_1 \rangle  \pm_2 \langle n_{12}-n_1 \rangle  \pm_3 \langle n_{123} - n_{12}  \rangle.
\end{equation*}
By first summing in $n_1$ and using the basic counting lemma, we gain a factor of $\min(N_{1},N_{12})$. Alternatively, by first summing in $n_{123}$ and using the basic counting lemma, we gain a factor of $\min(N_{123},N_{12})$. By combining both estimates, we gain a factor of 
\begin{equation*}
\max\big( \min(N_{1},N_{12}), \min(N_{123},N_{12})\big)=  \min\big( N_{12}, \max(N_{123},N_1)\big). 
\end{equation*}
While not part of the proof, we also remark that 
\begin{equation*}
\Big| \langle n_{123} \rangle   + \langle n_1 \rangle  - \langle n_{12}-n_1 \rangle -  \langle n_{123} - n_{12}  \rangle| \lesssim N_{12}. 
\end{equation*}
This shows that we cannot gain a factor of the form $\med(N_{123},N_{12},N_1)$. \\

\emph{Proof of \eqref{tools:item_cubic_4}:} In the variables $(n_{12},n_1,n_3)$, the phase takes the form 
\begin{equation*}
\varphi = \pm_{123} \langle n_{12} + n_3  \rangle \pm_1 \langle n_1 \rangle  \pm_2 \langle n_{12}-n_1 \rangle  \pm_3 \langle n_{3}  \rangle.
\end{equation*}
By first summing in $n_1$ and using the basic counting lemma, we gain a factor of $\min(N_{12},N_1)$. Alternatively, by first summing in $n_3$ and using the basic counting lemma, we gain a factor of $\min(N_{12},N_3)$. By combining both estimates, this completes the argument. The same obstruction as described in  \eqref{tools:item_cubic_3} shows that the estimate is sharp. 
\end{proof}

We now use the cubic counting estimate to prove the cubic sum estimate. 

\begin{proof}[Proof of Proposition \ref{tools:prop_cubic_sum}:]
Due to the symmetry $n_1 \leftrightarrow n_2$, we may assume that $N_1 \geq N_2$. To simplify the notation, we set 
\begin{equation*}
\begin{aligned}
\cC(m) &= \cC(N_1,N_2,N_3,N_{12},N_{123},m) \\
&= \Big\{ (n_1,n_2,n_3) \in (\bZ^3)^3\colon |n_j| \sim N_j, 1 \leq j \leq 3, |n_{12}| \sim N_{12}, |n_{123}| \sim N_{123}, |\varphi - m| \leq 1 \Big\}. 
\end{aligned}
\end{equation*}
We then have that 
\begin{equation}
\begin{aligned}
& \sum_{n_1,n_2,n_3\in \bZ^3} \bigg[ \Big( \prod_{j=1}^{3} \chi_{N_j}(n_j) \Big) \langle n_{123} \rangle^{2(s-1)} \langle n_{12} \rangle^{-2\gamma} \Big( \prod_{j=1}^{3} \langle n_j \rangle^{-2} \Big) 1\big\{ |\varphi-m|\leq 1\big\}\bigg] \\
&\lesssim \sum_{N_{123},N_{12}} N_{123}^{2(s-1)} N_{12}^{-2\gamma} \Big( \prod_{j=1}^{3} N_j^{-2} \Big) \# \cC(m). 
\end{aligned}
\end{equation}
To obtain the optimal estimate, we unfortunately need to distinguish five cases, which we listed in Figure \ref{figure:cases}. Case 1 and 2 distinguish between the high$\times$high and high$\times$low-interactions in the first two factors. This distinction is necessary to utilize the gain in $N_{12}$. The subcases mostly deal with the relation between $N_{12}$ and $N_3$, which is important to use the gain in $N_{123}$. \\

\emph{Case 1.a: $N_1\sim N_2$, $N_1 \ll N_3$.} In this case, $N_{123}\sim N_3$. Using \eqref{tools:item_cubic_4} in Proposition \ref{tools:prop_cubic_counting}, the contribution is bounded by
\begin{align*}
\sum_{ \substack{N_{12}\colon \\ N_{12} \lesssim N_1 }} N_{12}^{-2\gamma} N_1^{-4} N_3^{2s-4} \# \cC(m) 
\lesssim \sum_{ \substack{N_{12}\colon \\ N_{12} \lesssim N_1 }} N_{12}^{2-2\gamma} N_1^{-1} N_3^{2s-1} 
\lesssim  N_1^{1-2\gamma} N_3^{2s-1}, 
\end{align*} 
which is acceptable. In performing the sum, we used that $\gamma < 1$. \\

\emph{Case 1.b.i: $N_1\sim N_2$, $N_1 \gtrsim N_3$, $N_3 \ll N_{12}$.} In this case, $N_{123} \sim N_{12}$.  Using \eqref{tools:item_cubic_4} in Proposition \ref{tools:prop_cubic_counting}, the contribution is bounded by 
\begin{align*}
\sum_{ \substack{N_{12}\colon \\ N_3 \ll  N_{12} \lesssim N_1 }} \hspace{-2ex} N_{12}^{2s -2 -2\gamma} N_1^{-4} N_3^{-2} \# \cC(m)  
\lesssim \sum_{ \substack{N_{12}\colon \\ N_3 \ll  N_{12} \lesssim N_1 }} \hspace{-2ex}   N_{12}^{2s-2\gamma} N_1^{-1} N_3 
\lesssim  \sum_{ \substack{N_{12}\colon \\  N_{12} \lesssim N_1 }}  N_{12}^{2s-2\gamma+1} N_1^{-1} 
\lesssim N_1^{2(s-\gamma)},
\end{align*}
which is acceptable. In performing the sum, we used that $\gamma < s +1/2$. \\

\emph{Case 1.b.i: $N_1\sim N_2$, $N_1 \gtrsim N_3$, $N_3 \gtrsim N_{12}$.}
We note that $N_{123}\lesssim \max(N_{12},N_3) \lesssim N_3$.  Using \eqref{tools:item_cubic_3} in Proposition \ref{tools:prop_cubic_counting}, the contribution is bounded by
\begin{align*}
&\sum_{ \substack{N_{12},N_{123}\colon \\ N_{12}, N_{123} \lesssim N_3 }} \hspace{-2ex} N_{123}^{2s-2} N_{12}^{-2\gamma} N_1^{-4} N_3^{-2} \# \cC(m)  
\lesssim \sum_{ \substack{N_{12},N_{123}\colon \\ N_{12}, N_{123} \lesssim N_3 }} \hspace{-2ex} \min(N_{123},N_{12})^{-1} N_{123}^{2s+1} N_{12}^{3-2\gamma} N_1^{-1} N_3^{-2} \\
&\lesssim N_{1}^{-1} N_3^{2s -2 \gamma+1} \lesssim N_1^{2(s-\gamma)},
\end{align*}
which is acceptable. In the last inequality, we used again that $\gamma < s+1/2$. \\

\emph{Case 2.a: $N_1\gg N_2$, $N_1 \not\sim N_3$.} In this case, $N_{12}\sim N_1$ and $N_{123}\sim \max(N_1,N_3)$. 
 Using \eqref{tools:item_cubic_1} in Proposition \ref{tools:prop_cubic_counting}, the contribution is bounded by
\begin{align*}
&\max(N_1,N_3)^{2s-2} N_1^{-2-2\gamma} N_2^{-2} N_3^{-2} \# \cC(m) \lesssim \max(N_1,N_3)^{2s-2} \min(N_1,N_3)^{-1} N_1^{1-2\gamma} N_2 N_3 \\
&\lesssim \max(N_1,N_3)^{2s-2} \min(N_1,N_3)^{-1} N_1^{2-2\gamma}  N_3 =\max(N_1,N_3)^{2s-1} N_1^{1-2\gamma}. 
\end{align*}
The restriction $s\leq 1/2$ is not strictly necessary for the statement of the proposition, but ensures that the first factor does not grow in $N_1$ or $N_3$, which is essential in applications. \\

\emph{Case 2.a: $N_1\gg N_2$, $N_1 \sim N_3$.} In this case, $N_{12}\sim N_1$. Using \eqref{tools:item_cubic_2} in Proposition \ref{tools:prop_cubic_counting}, the contribution is bounded by
\begin{align*}
\sum_{\substack{N_{123}\colon \\ N_{123}\lesssim N_1}} N_{123}^{2s-2} N_1^{-4-2\gamma} N_2^{-2} \# \cC(m) 
\lesssim \sum_{\substack{N_{123}\colon \\ N_{123}\lesssim N_1}} N_{123}^{2s} N_1^{-1-2\gamma} N_2  
\lesssim N_1^{2s-2\gamma},
\end{align*}
which is acceptable. In performing the sum, we used that $s>0$. 
\end{proof}

\begin{figure}[t]
\begin{tabular}{l|l|l|l|c}
Case & $N_1 \leftrightarrow N_2$ & $N_1 \leftrightarrow N_3$ & $N_{3} \leftrightarrow N_{12}$ & Basic counting estimate \\\hline
1.a    & $N_1 \sim N_2$ &  $N_1 \ll N_3 $    & & \eqref{tools:item_cubic_4} \\ 
1.b.i    & $N_1 \sim N_2$ &  $ N_1 \gtrsim N_3  $    &  $N_3 \ll N_{12}$ & \eqref{tools:item_cubic_4} \\ 
1.b.ii  & $N_1 \sim N_2$ &  $ N_1 \gtrsim N_3 $    & $N_3 \gtrsim N_{12}$ & \eqref{tools:item_cubic_3} \\ 
2.a   & $N_1 \gg N_2$ &  $ N_1 \not \sim N_3 $    & & \eqref{tools:item_cubic_1} \\ 
2.b    & $N_1 \gg N_2$ &  $ N_1 \sim N_3 $    & & \eqref{tools:item_cubic_2} \\ 
\end{tabular}
\caption{Case distinction in the proof of Proposition \ref{tools:prop_cubic_sum}.}
\label{figure:cases}
\end{figure}

\subsection{Cubic sup-counting estimates}

\begin{proof}[Proof of Lemma \ref{tools:lem_sup}:]
We prove the four estimates separately. 

\emph{Proof of \eqref{tools:item_sup_1}:} By symmetry, we can assume without loss of generality that $N_1 \geq N_2 \geq N_3$. Using the basic counting estimate in $n_2\in \bZ^3$, we have that 
\begin{align*}
& \# \Big\{ (n_1,n_2,n_3)\colon |n_1|\sim N_1, |n_2|\sim N_2, |n_3|\sim N_3, n=n_{123}, |\varphi-m|\leq 1\Big\}\\
&\lesssim  \# \Big\{ (n_2,n_3)\colon  |n_2|\sim N_2, |n_3|\sim N_3, 
|\pm_{123} \langle n \rangle \pm_1 \langle n - n_{23} \rangle \pm_2 \langle n_2 \rangle \pm_3 \langle n_3 \rangle -m|\leq 1\Big\} \\
&\lesssim \sum_{n_3 \in \bZ^3} 1\big\{ |n_3| \sim N_3\big\} \min\big( \langle n - n_3 \rangle, N_2 \big)^{-1} N_2^3 \\
&\lesssim N_2^3 N_3^2. 
\end{align*}

\emph{Proof of \eqref{tools:item_sup_2}:} The proof is essentially the same as the proof of \eqref{tools:item_sup_1} and we omit the details. \\

\emph{Proof of \eqref{tools:item_sup_3}:} Using the basic counting estimate in $n_2\in \bZ^3$, we have that
\begin{align*}
& \# \Big\{ (n_{12},n_2,n_3)\colon |n_{12}|\sim N_{12}, |n_2|\sim N_2, |n_3|\sim N_3, n=n_{123}, |\varphi-m|\leq 1\Big\} \\
&\lesssim \# \Big\{ (n_{12},n_2) \colon |n_{12}|\sim N_{12}, |n_2|\sim N_2,  
|\pm_{123} \langle n \rangle \pm_1 \langle n_{12} - n_{2} \rangle \pm_2 \langle n_2 \rangle \pm_3 \langle n - n_{12} \rangle -m|\leq 1\Big\} \\
&\lesssim \sum_{n_{12} \in \bZ^3}  1\big\{ |n_{12}| \sim N_{12}\big\} \min(N_{12},N_2)^{-1} N_2^3 \\
&\lesssim \min(N_{12},N_2)^{-1} N_{12}^3 N_2^3.
\end{align*}

\emph{Proof of \eqref{tools:item_sup_4}:}  The proof is essentially the same as the proof of \eqref{tools:item_sup_3} and we omit the details.

\end{proof}

\subsection{Para-controlled cubic counting estimates}

\begin{proof}[Proof of Lemma \ref{tools:lem_paracontrolled_counting}:]

To simplify the notation, we set $N_{\max}=\max(N_1,N_2,N_3)$. For $0<\gamma<\beta$, we have that
\begin{equation*}
\langle n_{12} \rangle^{-2\beta} \lesssim \langle n_{12} \rangle^{-2\gamma} \lesssim \langle n_1 \rangle^{-2\gamma} \langle n_2 \rangle^{2\gamma}. 
\end{equation*}

Together with \eqref{tools:item_sup_2} from Lemma \ref{tools:lem_sup}, this yields 
\begin{align*}
&\sum_{n_1,n_3\in \bZ^3} \Big( \prod_{j=1,3} 1\big\{ |n_j| \sim N_j \big\} \Big) \langle n_{123} \rangle^{2(s_2-1)} \langle n_{12} \rangle^{-2\beta} \langle n_1 \rangle^{-2} \langle n_3 \rangle^{-2} \, 1\big\{ |\varphi-m|\leq 1\big\} \\
&\lesssim N_1^{-2-2\gamma} N_2^{2\gamma} N_3^{-2} \sum_{N_{123}} N_{123}^{2(s_2-1)} 
\# \big\{ (n_1,n_3) \colon |n_{123}| \sim N_{123}, |n_1|\sim N_1, |n_3|\sim N_3, |\varphi-m|\leq 1\big\} \\
&\lesssim N_1^{-2-2\gamma} N_2^{2\gamma} N_3^{-2} \sum_{\substack{N_{123}\colon  \\ |N_{123}| \lesssim N_{\max} } }  N_{123}^{2(s_2-1)}  \med\big(N_{123},N_1,N_3\big)^3\min\big(N_{123},N_1,N_3\big)^2. 
\end{align*}
Using that $\med\big(N_{123},N_1,N_3\big)^3\min\big(N_{123},N_1,N_3\big)^2 \lesssim N_{123} N_1^2 N_3^2$, we obtain that 
\begin{align*}
&N_1^{-2-2\gamma} N_2^{2\gamma} N_3^{-2} \sum_{\substack{N_{123}\colon  \\ |N_{123}| \lesssim N_{\max} } }  N_{123}^{2(s_2-1)}  \med\big(N_{123},N_1,N_3\big)^3\min\big(N_{123},N_1,N_3\big)^2 \\
&\lesssim N_1^{-2\gamma} N_2^{2\gamma}\sum_{\substack{N_{123}\colon  \\ |N_{123}| \lesssim N_{\max} } } N_{123}^{2s_2-1} 
\lesssim N_{\max}^{2\delta_2}  N_1^{-2\gamma} N_2^{2\gamma} . 
\end{align*}
\end{proof}

\subsection{Quartic counting estimate}

\begin{proof}[Proof of Lemma \ref{tools:lemma_nonresonant_quartic}:]
Using the upper bound on $s$, we can first sum in $n_4\in \bZ^3$ and obtain that 
\begin{align*}
&\sum_{n_1,n_2,n_3,n_4\in \bZ^3} \Big( \prod_{j=1}^{4}1\big\{ |n_j|\sim N_j \big\} \Big) \langle n_{1234} \rangle^{2s} \langle n_{123} \rangle^{-2} 
|\widehat{V}_S(n_1,n_2,n_3)|^2 
\Big(\prod_{j=1}^4 \langle n_j \rangle^{-2}\Big) 1\big\{ |\varphi-m|\leq 1\big\} \\
&\lesssim N_4^{-2\eta} \sum_{n_1,n_2,n_3\in \bZ^3} \Big( \prod_{j=1}^{3}1\big\{ |n_j|\sim N_j \big\} \Big)  \langle n_{123} \rangle^{-2} 
|\widehat{V}_S(n_1,n_2,n_3)|^2 
\Big(\prod_{j=1}^3 \langle n_j \rangle^{-2}\Big) 1\big\{ |\varphi-m|\leq 1\big\}  . 
\end{align*}
The remaining sum in $n_1,n_2$, and $n_3$ can then be estimated using Proposition \ref{tools:prop_cubic_sum}, which yields the desired estimate. 
\end{proof}

After the proof of the non-resonant quartic sum estimate (Lemma \ref{tools:lemma_nonresonant_quartic}), we now turn to the resonant quartic sum estimate. We begin with the basic resonance estimate (Lemma \ref{tools:lem_basic_resonance}), which forms the main part of the proof.

\begin{proof}[Proof of Lemma \ref{tools:lem_basic_resonance}:]
Since $n_1,n_2 \in \bZ^3$ are fixed and the phase $\varphi$ is globally Lipschitz, there are at most $\sim N_1$ non-trivial choices of $m\in \bZ$. Due to the $\log$-factor in \eqref{tools:eq_basic_resonance}, it suffices to prove 

\begin{equation*}
\sup_{m\in \bZ} \sum_{n_3 \in \bZ^3} 1\big\{ |n_3|\sim N_3\big\} 
 \langle n_{123} \rangle^{-1} \langle n_{3} \rangle^{-2} 1\big\{ |\varphi-m| \leq 1 \big\} \lesssim \langle n_{12} \rangle^{-1}. 
\end{equation*}
By inserting an additional dyadic localization, we obtain that 
\begin{equation}\label{appendix:eq_basic_resonance_p1}
\begin{aligned}
& \sum_{n_3 \in \bZ^3} 1\big\{ |n_3|\sim N_3\big\} 
 \langle n_{123} \rangle^{-1} \langle n_{3} \rangle^{-2} 1\big\{ |\varphi-m| \leq 1 \big\} \\
 &\leq N_3^{-2} \sum_{N_{123}\geq 1}  N_{123}^{-1}  \sum_{n_3 \in \bZ^3}  1\big\{ |n_{123}|\sim N_{123}\big\} 1\big\{ |n_3|\sim N_3\big\} 
1\big\{ |\varphi-m| \leq 1 \big\}. 
\end{aligned}
\end{equation}

To simplify the notation, we write $N_{12}$ for the dyadic scale of $n_{12}\in \bZ^3$. Using Lemma \ref{tools:lem_two_balls}, we have that 
\begin{align*}
&N_{123}^{-1} N_3^{-2} \sum_{n_3 \in \bZ^3} \ 1\big\{ |n_{123}|\sim N_{123}\big\} 1\big\{ |n_3|\sim N_3\big\} 
 1\big\{ |\varphi-m| \leq 1 \big\} \\
 &\lesssim N_{123}^{-1} N_3^{-2} \min(N_{123},N_{12},N_3)^{-1} \min(N_{123},N_3)^3. 
\end{align*}
We now separate the contributions of the three cases $N_{123} \ll N_3$, $N_{123} \sim N_3$, $N_{123} \gg N_3$. In the following, we implicitly restrict the sum over $N_{123}$ to values which are consistent with $|n_{123}|\sim N_{123}$, $|n_{12}| \sim N_{12}$, and $|n_{3}| \sim N_3$ for some $n_1,n_2,n_3\in \bZ^3$. 

If $N_{123} \ll N_3$, then $N_{12} \sim N_3$. Thus, 
\begin{equation*}
 \sum_{N_{123} \ll N_3} N_{123}^{-1} N_3^{-2} \min(N_{123},N_{12},N_3)^{-1} \min(N_{123},N_3)^3 \lesssim 1\big\{ N_{12} \sim N_3\big\} \sum_{N_{123} \ll N_3}  N_{123} N_3^{-2} \lesssim N_{12}^{-1}. 
\end{equation*}
If $N_{123} \sim N_3$, then $N_{12} \lesssim N_{123} \sim N_3 $. Thus, 
\begin{equation*}
\sum_{N_{123} \sim N_3} N_{123}^{-1} N_3^{-2} \min(N_{123},N_{12},N_3)^{-1} \min(N_{123},N_3)^3 \sim N_{12}^{-1}. 
\end{equation*}
Finally, if $N_{123} \gg N_3$, then $N_{123}\sim N_{12} \gg N_3$. Thus, 
\begin{equation*}
 \sum_{N_{123} \gg N_3} N_{123}^{-1} N_3^{-2} \min(N_{123},N_{12},N_3)^{-1} \min(N_{123},N_3)^3 = N_{12}^{-1} N_3^{-2} N_3^{-1} N_3^3 \sim N_{12}^{-1}. 
\end{equation*}
This completes the proof. 
\end{proof}

The resonant quartic sum estimate (Lemma \ref{tools:lemma_resonant_quartic}) is now an easy consequence of the basic resonance estimate (Lemma \ref{tools:lem_basic_resonance}). 

\begin{proof}[Proof of Lemma \ref{tools:lemma_resonant_quartic}:] 
Using Lemma \ref{tools:lem_basic_resonance}, we have that 
\begin{align*}
&\sum_{n_1,n_2 \in \bZ^3} \bigg[  \Big( \prod_{j=1}^2 1 \big\{ |n_j| \sim N_j \big\} \langle n_{12} \rangle^{2s} \langle n_1 \rangle^{-2} \langle n_2 \rangle^{-2} \\
&\times \Big( \sum_{m\in \bZ} \sum_{n_3 \in \bZ^3} \langle m \rangle^{-1} 1\big\{ |n_3|\sim N_3\big\} 
 \langle n_{123} \rangle^{-1} \langle n_{3} \rangle^{-2} 1\big\{ |\varphi-m| \leq 1 \big\} \Big)^2  \bigg] \\
 &\lesssim \log(2+N_3)^2 \sum_{n_1,n_2 \in \bZ^3} \bigg[  \Big( \prod_{j=1}^2 1 \big\{ |n_j| \sim N_j \big\} \langle n_{12} \rangle^{2s-2} \langle n_1 \rangle^{-2} \langle n_2 \rangle^{-2} \bigg] \\
 &\lesssim \log(2+N_3)^2  \max(N_1,N_2)^{2s}. 
\end{align*}
\end{proof}

\subsection{Quintic counting estimates}
Before we turn to the proof of the non-resonant quintic counting estimate, we isolate a helpful auxiliary lemma. 

\begin{lemma}[Frequency-scale estimate] \label{appendix:lem_freq_scale}
Let $N_1,N_2,N_{1345},N_{12345}$ be frequency scales which can be achieved by frequencies $n_1,\hdots,n_5\in \bZ^3$, i.e., satisfying
\begin{equation*}
1\big\{ |n_1|\sim N_1 \big\} \cdot 1\big\{ |n_2|\sim N_2 \big\} \cdot 1\big\{ |n_{1345}|\sim N_{1345} \big\}\cdot 1\big\{ |n_{12345}|\sim N_{12345} \big\}\not \equiv 0. 
\end{equation*}
Then, it holds that
\begin{equation*}
 \frac{\min(N_2,N_{12345})^2  \min(N_1,N_{1345})}{\min(N_{12345},N_{1345},N_2) } \lesssim N_2 \cdot N_{12345}. 
\end{equation*}
\end{lemma}

\begin{proof}
By using the properties of $\min$ and $\max$, we have that 
\begin{align*}
 \frac{\min(N_2,N_{12345}) \min(N_1,N_{1345})}{\min(N_{12345},N_{1345},N_2) } 
\lesssim \frac{\min(N_2,N_{12345}) N_{1345}}{\min(N_{12345},N_{1345},N_2) } 
\lesssim \max\big( \min(N_2,N_{12345}), N_{1345} \big). 
\end{align*}
Since $N_{1345}\lesssim \max(N_2,N_{12345})$, this yields 
\begin{align*}
 \frac{\min(N_2,N_{12345})^2  \min(N_1,N_{1345})}{\min(N_{12345},N_{1345},N_2) } 
\lesssim \min(N_2,N_{12345}) \cdot \max(N_2,N_{12345}) 
=N_2 N_{12345}. 
\end{align*}

\end{proof}

\begin{proof}[Proof of Lemma \ref{tools:lem_counting_quintic}:]
Let $m,m^\prime \in \bZ$ be arbitrary. We introduce $N_{12345}$ and $N_{1345}$ to further decompose according to the size of $n_{12345}$ and $n_{1345}$. Using the two-ball basic counting lemma (Lemma \ref{tools:lem_two_balls}) for the sum in $n_2 \in \bZ^3$ and summing in $n_1 \in \bZ^3$ directly, we obtain that 
\begin{align*}
& \sum_{n_1,\hdots,n_5\in \bZ^3} \bigg[ \Big( \prod_{j=1}^{5} 1\big\{|n_j|\sim N_j\} \big) 1\big\{ |n_{12345} | \sim N_{12345} \big\} 1\big\{ |n_{1345}| \sim N_{1345} \big\} \\
&\hspace{2ex}\times \langle n_{12345} \rangle^{2(s-1)} \langle n_{1345} \rangle^{-2\beta} 
\langle n_{345} \rangle^{-2} \langle n_{34} \rangle^{-2\beta} \Big( \prod_{j=1}^5 \langle n_j \rangle^{-2} \Big) \\
&\hspace{2ex}\times  1\big\{ |\psi-m|\leq 1 \big\} \cdot \Big( 1\big\{ |\varphi-m^\prime|\leq 1\} +  1\big\{ |\widetilde{\varphi}-m^\prime|\leq 1\} \Big)  \bigg] \allowdisplaybreaks[3] \\
&\lesssim N_{12345}^{2(s-1)} N_{1345}^{-2\beta}  \min(N_{12345},N_{1345},N_2)^{-1} \min(N_2,N_{12345})^3 \prod_{j=1}^5 N_j^{-2} \\
&\times \sum_{n_1,n_3,n_4,n_5\in \bZ^3}  \Big( \prod_{j=1,3,4,5} 1\big\{|n_j|\sim N_j\} \Big) 1\big\{ |n_{1345}| \sim N_{1345} \big\} 
\langle n_{345} \rangle^{-2} \langle n_{34} \rangle^{-2\beta}    1\big\{ |\psi-m|\leq 1 \big\}\allowdisplaybreaks[3] \\
&\lesssim  N_{12345}^{2(s-1)} N_{1345}^{-2\beta}  \min(N_{12345},N_{1345},N_2)^{-1} \min(N_2,N_{12345})^3  \min(N_1,N_{1345})^3 \prod_{j=1}^5 N_j^{-2} \\
 &\times \sum_{n_3,n_4,n_5\in \bZ^3}  \Big( \prod_{j=3}^5 1\big\{|n_j|\sim N_j\} \Big) 
\langle n_{345} \rangle^{-2} \langle n_{34} \rangle^{-2\beta} 1\big\{ |\psi-m|\leq 1 \big\}. 
\end{align*}
Using Proposition \ref{tools:prop_cubic_sum} with $s=0$ and $\gamma=\beta$ to bound the remaining sum in $n_3,n_4$, and $n_5$, we obtain a bound of the total contribution by
\begin{equation*}
 N_{12345}^{2(s-1)}(N_1 N_2)^{-2}  \frac{\min(N_2,N_{12345})^3  \min(N_1,N_{1345})^3}{\min(N_{12345},N_{1345},N_2) }   (N_{1345} \max(N_3,N_4,N_5))^{-2\beta}. 
\end{equation*}
As long as the contribution is non-trivial, it holds that $N_{1345} \max(N_3,N_4,N_5)\gtrsim \max(N_1,N_3,N_4,N_5)$. Thus, it remains to prove that 
\begin{equation*}
 N_{12345}^{2(s-1)}(N_1 N_2)^{-2}  \frac{\min(N_2,N_{12345})^3  \min(N_1,N_{1345})^3}{\min(N_{12345},N_{1345},N_2) } \lesssim N_2^{-2\eta},
\end{equation*}
which follows from a short calculation. Indeed, using Lemma \ref{appendix:lem_freq_scale}, we can estimate the left-hand side by 
\begin{align*}
& N_{12345}^{2(s-1)}(N_1 N_2)^{-2}  \frac{\min(N_2,N_{12345})^3  \min(N_1,N_{1345})^3}{\min(N_{12345},N_{1345},N_2) }\\
&\lesssim N_{12345}^{2s-1} \min(N_2,N_{12345}) \min(N_1,N_{1345})^2 N_1^{-2} N_2^{-1} \\
&\lesssim N_{12345}^{2s-1+2\eta} N_2^{-2\eta}. 
\end{align*}
Due to our condition on $s$, this is acceptable. 
\end{proof}

We now prove the double-resonance quintic counting estimate.

\begin{proof}[Proof of Lemma \ref{tools:lem_counting_double_resonance_quintic}:]
We also use a dyadic localization to $|n_{345}| \sim N_{345}$ and $|n_{45}|\sim N_{45}$. By paying a factor of $\log(2+\max(N_4,N_5))^2$, it suffices to estimate the maximum over   $N_{345},N_{45}$ instead of the sum.  We do not require a logarithmic loss in $N_3$, since $N_3 \gg N_4,N_5$ implies that there are only $\sim 1$ non-trivial choices for $N_{345}$. We first sum in $n_3\in \bZ^3$ using the two-ball basic counting lemma (Lemma \ref{tools:lem_two_balls}). We then sum in $n_4\in \bZ^3$ using only the dyadic constraint. This yields
\begin{align*}
&N_3^{-2} N_4^{-2}  \sup_{m\in \bZ^3} \sup_{|n_5|\sim N_5} \sum_{n_3,n_4\in \bZ^3} \bigg[ \Big( \prod_{j=3}^4 1\big\{ |n_j|\sim N_j \big \}\Big) 1\big\{ |n_{345}|\sim N_{345}\big\} 
 1\big\{ |n_{45}|\sim N_{45}\big\} \langle n_{345} \rangle^{-1} \langle n_{45} \rangle^{-\beta} \\
&\times  1\big\{ \langle n_{345} \rangle \pm_3 \langle n_3 \rangle \pm_4 \langle n_4 \rangle \pm_5 \langle n_5 \rangle \in [m,m+1) \big\} \bigg] \\
&\lesssim \min(N_{345},N_{45},N_3)^{-1} \min(N_3,N_{345})^3 N_{345}^{-1} N_{45}^{-\beta} N_3^{-2} N_4^{-2} \sum_{n_4 \in \bZ^3} 1\big\{ |n_4|\sim N_4 \big \}1\big\{ |n_{345}|\sim N_{345}\big\}\\
&\lesssim  \min(N_{345},N_{45},N_3)^{-1} \min(N_3,N_{345})^3 \min(N_{4},N_{45})^3  N_{345}^{-1} N_{45}^{-\beta}  N_3^{-2} N_4^{-2}. 
\end{align*}
Using a minor variant of Lemma \ref{appendix:lem_freq_scale}, this contribution is bounded by 
\begin{align*}
N_{45}^{-\beta} N_3^{-1} N_4^{-2} \min(N_3,N_{345})  \min(N_4,N_{45})^2 \lesssim \max(N_4,N_{45})^{-\beta} \lesssim \max(N_4,N_5)^{-\beta}. 
\end{align*}
\end{proof}
\subsection{Septic counting estimates}
\begin{proof}[Proof of Lemma \ref{tools:lem_counting_septic}:]
Using the decay of $\widehat{V}$, it suffices to prove 
\begin{equation}\label{appendix:eq_septic_p0}
\begin{aligned}
&\sum_{(n_j)_{j\not \in \scrP}}  \langle n_{\text{nr}} \rangle^{2(s-1)}
 \bigg(  \sum_{(n_j)_{j \in \scrP}}^{\ast}  1\big\{ |n_{1234567}|\sim N_{1234567}\big\} 1\big\{|n_{567}| \sim N_{567} \big\} 1\big\{|n_4|\sim N_4\big\}  \\
 &\times  \Phi(n_1,n_2,n_3) \langle n_4 \rangle^{-1} \Phi(n_5,n_6,n_7) \bigg)^2 \\
 &\lesssim \log(2+N_4)^2 \Big( N_{1234567}^{2(s-\frac{1}{2})}   N_{567}^{-2(\beta-\eta)} + N_{1234567}^{-2 (1-s+\eta)} \Big). 
\end{aligned}
\end{equation}
The argument relies on two of our previous estimates. Using the cubic sum estimate (Proposition \ref{tools:prop_cubic_sum}), we have that for all $N_{123}\geq 1$ that 
\begin{equation}\label{appendix:eq_septic_p1}
\sum_{n_1,n_2,n_3 \in \bZ^3} 1\big\{ |n_{123}| \sim N_{123} \big\} \Big( \prod_{j=1}^3 \langle n_j \rangle^{\frac{\eta}{3}} \Big) \Phi^2(n_1,n_2,n_3) \lesssim N_{123}^{-2(\beta-\eta)}. 
\end{equation}
Using the basic resonance estimate (Lemma \ref{tools:lem_basic_resonance}), we have for all $N_3 \geq 1$ that 
\begin{equation}\label{appendix:eq_septic_p2}
\sum_{n_3\in \bZ^3} 1\big\{ |n_3|\sim N_3\big\} \langle n_3 \rangle^{-1} \Phi(n_1,n_2,n_3) \lesssim \log(2+N_3) \langle n_{12} \rangle^{-1} \langle n_1\rangle^{-1} \langle n_2 \rangle^{-1}. 
\end{equation}
Using the symmetry of $\Phi$, it remains to consider the following three cases.

\emph{Case 1: $j=4$ is unpaired.}  By first using Cauchy-Schwarz, summing in $n_4$, and then using \eqref{appendix:eq_septic_p1}, we obtain that  
\begin{align*}
&\sum_{(n_j)_{j\not \in \scrP}} \langle n_{\text{nr}} \rangle^{2(s-1)}
 \bigg(  \sum_{(n_j)_{j \in \scrP}}^{\ast}  1\big\{ |n_{1234567}|\sim N_{1234567}\big\} 1\big\{|n_{567}| \sim N_{567} \big\} 1\big\{|n_4|\sim N_4\big\}  \\
 &\times  \Phi(n_1,n_2,n_3) \langle n_4 \rangle^{-1} \Phi(n_5,n_6,n_7) \bigg)^2 \\
 &\lesssim \sum_{(n_j)_{j\not \in \scrP}} \bigg[ 1\big\{ |n_{\text{nr}}|\sim N_{1234567}\big\} \langle n_{\text{nr}} \rangle^{2(s-1)} \langle n_4 \rangle^{-2} 
 \Big( \sum_{(n_j)_{j \in \scrP}}^{\ast}   \Phi(n_1,n_2,n_3)^2 \Big)     \\
 &\times   \Big( \sum_{(n_j)_{j \in \scrP}}^{\ast} 1\big\{ n_{567} \sim N_{567} \big\}\Phi(n_5,n_6,n_7)^2 \Big) \bigg]  \\
 &\lesssim N_{1234567}^{2(s-\frac{1}{2})}  \sum_{(n_j)_{ \substack{j\not \in \scrP \wedge j \neq 4}}} 
  \Big( \sum_{(n_j)_{j \in \scrP}}^{\ast}   \Phi(n_1,n_2,n_3)^2 \Big)     
  \Big( \sum_{(n_j)_{j \in \scrP}}^{\ast} 1\big\{ n_{567} \sim N_{567} \big\}\Phi(n_5,n_6,n_7)^2 \Big) \bigg]  \\
  &= N_{1234567}^{2(s-\frac{1}{2})}  \Big( \sum_{n_1,n_2,n_3\in \bZ^3}   \Phi(n_1,n_2,n_3)^2 \Big)     
  \Big( \sum_{n_5,n_6,n_7\in \bZ^3} 1\big\{ n_{567} \sim N_{567} \big\}\Phi(n_5,n_6,n_7)^2 \Big) \bigg]  \\
  &\lesssim  N_{1234567}^{2(s-\frac{1}{2})} N_{567}^{-2(\beta-\eta)}. 
 \end{align*}
 This contribution is acceptable.
 
 \emph{Case 2: $(3,4)\in \scrP$.}  We let $\scrP^\prime$ be the pairing on $\{1,2,5,6,7\}$ obtained by removing the pair $(3,4)$ from $\scrP$. We also understand the condition $j \not\in \scrP^\prime$ as a subset of $\{1,2,5,6,7\}$. By first using \eqref{appendix:eq_septic_p2} and then Cauchy-Schwarz, we have that 
 \begin{align*}
&\sum_{(n_j)_{j\not \in \scrP}} \langle n_{\text{nr}} \rangle^{2(s-1)}
 \bigg(  \sum_{(n_j)_{j \in \scrP}}^{\ast}  1\big\{ |n_{1234567}|\sim N_{1234567}\big\} 1\big\{|n_{567}| \sim N_{567} \big\} 1\big\{|n_4|\sim N_4\big\}  \\
 &\times  \Phi(n_1,n_2,n_3) \langle n_4 \rangle^{-1} \Phi(n_5,n_6,n_7) \bigg)^2 \\
 &\lesssim \log(2+N_4)^2 N_{1234567}^{2(s-1+\eta)} \\
 &\times  \sum_{(n_j)_{j\not \in \scrP^\prime}} \langle n_{\text{nr}} \rangle^{-2\eta} \Big( \sum_{(n_j)_{j \in \scrP^\prime}}^{\ast}  1\big\{|n_{567}| \sim N_{567} \big\} \langle n_{12} \rangle^{-1} \langle n_1 \rangle^{-1} \langle n_2 \rangle^{-1} \Phi(n_5,n_6,n_7) \Big)^2 \\
 &\lesssim \log(2+N_4)^2 N_{1234567}^{2(s-1+\eta)}  \sum_{(n_j)_{j\not \in \scrP^\prime}} \bigg[  
 \Big(  \sum_{(n_j)_{j \in \scrP^\prime}}^{\ast} \langle n_{\text{nr}} \rangle^{-2\eta}  \Big( \prod_{j\in \scrP^\prime} \langle n_j \rangle^{-\frac{\eta}{6}} \Big) \langle n_{12} \rangle^{-2} \langle n_1 \rangle^{-2} \langle n_2 \rangle^{-2} \Big) \\
 &\times \Big(  \sum_{(n_j)_{j \in \scrP^\prime}}^{\ast} 1\big\{|n_{567}| \sim N_{567} \big\}  \Big( \prod_{j\in \scrP^\prime} \langle n_j \rangle^{\frac{\eta}{6}} \Big) \Phi(n_5,n_6,n_7)^2 \Big) \bigg]. 
 \end{align*}
 We then use a direct calculation to bound the first inner factor and to estimate the sum in $n_5,n_6$, and $n_7$. The total contribution is bounded by $\log(2+N_4)^2 N_{1234567}^{2(s-1+\eta)} N_{567}^{-2(\beta-\eta)} \lesssim \log(2+N_4)^2 N_{1234567}^{2(s-1+\eta)}$, which is acceptable. 
 
 \emph{Case 3: $(4,5)\in \scrP$.}  We let $\scrP^\prime$ be the pairing on $\{1,2,3,6,7\}$ obtained by removing the pair $(4,5)$ from $\scrP$. We also understand the condition $j \not \in \scrP^\prime$ as a subset of $\{1,2,3,6,7\}$. By first using \eqref{appendix:eq_septic_p2} and then Cauchy-Schwarz, we have that 
 \begin{align*}
&\sum_{(n_j)_{j\not \in \scrP}} \langle n_{\text{nr}} \rangle^{2(s-1)}
 \bigg(  \sum_{(n_j)_{j \in \scrP}}^{\ast}  1\big\{ |n_{1234567}|\sim N_{1234567}\big\} 1\big\{|n_{567}| \sim N_{567} \big\} 1\big\{|n_4|\sim N_4\big\}  \\
 &\times  \Phi(n_1,n_2,n_3) \langle n_4 \rangle^{-1} \Phi(n_5,n_6,n_7) \bigg)^2 \\
 &\lesssim \log(2+N_4)^2 N_{1234567}^{2(s-1+\eta)} \sum_{(n_j)_{j\not \in \scrP^\prime}} \langle n_{\text{nr}} \rangle^{-2\eta}
 \Big( \sum_{(n_j)_{j \in \scrP^\prime}}^{\ast}  \Phi(n_1,n_2,n_3) \langle n_{67} \rangle^{-1} \langle n_6\rangle^{-1} \langle n_7 \rangle^{-1} \Big)^2. 
 \end{align*}
 Arguing similarly as in Case 2, we obtain an upper bound by $\log(2+N_4)^2 N_{1234567}^{2(s-1+\eta)}$. While this bound does not contain the gain in $N_{567}$, it is still acceptable. 
\end{proof}

\end{appendix}

\bibliography{invariant_library}
\bibliographystyle{myalpha}

\end{document}